%% file: padicMF.tex
\pdfoutput=1

\newif\ifpersonal
\newif\ifarxiv
\arxivtrue 
\RequirePackage[l2tabu, orthodox]{nag} 
\documentclass[11pt,a4paper,reqno]{amsart} 
\usepackage{amsmath, amsfonts, amsxtra, amscd, amsfonts,pigpen, eucal, amsthm,amssymb,mathrsfs,mathtools,bm,tensor} 
\usepackage{microtype,lmodern} 
\usepackage[colorinlistoftodos,prependcaption]{todonotes}
\usepackage{enumerate,comment,marginnote, braket,xspace,tikz-cd} 
\usepackage[utf8]{inputenc} 
\usepackage[T1]{fontenc} 
\definecolor{linkcolor}{HTML}{005050}
\usepackage[centering,vscale=0.7,hscale=0.7]{geometry}
\usepackage{xr-hyper}
\usepackage[hidelinks]{hyperref}

\usepackage[capitalize]{cleveref}
\usepackage[all, cmtip, 2cell]{xy}
\usepackage{graphicx}
\usepackage{color,enumerate}
\usepackage{xcolor}
\usepackage{mdframed}

\theoremstyle{plain}
\newtheorem{thm}{Theorem}[section]
\newtheorem*{thm*}{Theorem}
\newtheorem{lem}[thm]{Lemma}
\newtheorem{prop}[thm]{Proposition}

\newtheorem{cor}[thm]{Corollary}

\theoremstyle{definition}
\newtheorem{defn}[thm]{Definition}
\newtheorem{cons}[thm]{Construction}
\newtheorem{notation}[thm]{Notation}
\newtheorem{ex}[thm]{Example}
\newtheorem{rem}[thm]{Remark}
\newtheorem{context}[thm]{Context}

\theoremstyle{remark}
\newtheorem*{rem*}{Remark}
\numberwithin{equation}{subsection}

\input{newcommands_all}

\begin{document}
\title{Motivic Realizations of Singularity Categories and Vanishing Cycles}
\date{\today}

\author{Anthony Blanc}
\address{Max Planck Institute for Mathematics, Bonn, Germany}
\email{anthony.blanc@ihes.fr}

\author{Marco ROBALO}
\address{Institut de Math\'ematiques de Jussieu - Paris Rive Gauche, UPMC, France}
\email{marco.robalo@imj-prg.fr}

\author{Bertrand TÖEN}
\address{Institut de Math\'ematiques de Toulouse, Universit\'e Paul Sabatier, France}
\email{bertrand.toen@math.univ-toulouse.fr}

\author{Gabriele Vezzosi}
\address{Dipartimento di Matematica ed Informatica, Firenze, Italy}
\email{gabriele.vezzosi@unifi.it}


\maketitle

\begin{abstract}In this paper we establish a precise comparison between vanishing cycles and the singularity category of Landau--Ginzburg models over an excellent Henselian discrete valuation ring. By using noncommutative motives, we first construct a motivic $\ell$-adic realization functor for dg-categories. Our main result, then asserts that, given a Landau--Ginzburg model over a complete discrete valuation ring with potential induced by a uniformizer, the $\ell$-adic realization of its singularity category is given by the inertia-invariant part of vanishing cohomology. 
We also prove a functorial and $\infty$-categorical lax symmetric monoidal version of Orlov's comparison theorem between the derived category of singularities and the derived category of matrix factorizations for a Landau--Ginzburg model over a noetherian regular local ring.
\end{abstract}

\tableofcontents

\section{Introduction}
The main objective of this paper is to use both derived and non-commutative geometry in order to establish a precise and fairly general relation between \emph{singularity categories} and \emph{vanishing cycles}. In a first section we work with $S= \Spec \, A$ with $A$ a commutative noetherian regular local ring (e.g. a discrete valuation ring). We will consider LG-pairs or LG-models over $S$ (LG for Landau-Ginzburg), i.e. pairs $(X/S,f)$ where $X$ is a scheme flat of finite type over $S$, and $f: X \to \mathbb{A}_S^1$ is an arbitrary map. One can associate to an LG-pair $(X/S,f)$ two, a priori different, triangulated categories: the derived category of matrix factorization $\MF(X/S,f)$, and the derived category of singularities $\Sing(X/S,f)$ (\cite{MR2101296}, \cite{efimov2015coherent}). A fundamental insight of Orlov (\cite{MR2101296}) is that, whenever $X$ is regular, and $f$ is not a zero-divisor, then there is an equivalence\footnote{Orlov works with pairs $(X/S,f)$ defined over $S$ the spectrum of a field, and  where $f$ is actually flat. This is not enough for out purposes, and we refer the reader to Section \ref{section-MFandSing} of this paper for a discussion of this point.} of triangulated categories between $\MF(X/S, f)$ and  $\Sing(X/S,f)$. 

Now, both $\MF(X/S,f)$ and $\Sing(X/S,f)$ can be naturally enhanced to dg-categories over $A$ whose associated homotopy categories are the given triangulated categories of matrix factorizations and of singularities. Moreover,  Orlov's functor can be enhanced (using derived geometry) to a functor of $A$-dg-categories, which is furthermore natural in the pair $(X/S, f)$, and shown to be an equivalence under appropriate hypotheses. This is the content of Theorem \ref{t1}, below, where we also discuss a lax monoidal strengthening.  \\ 

Once we have the $A$-dg-category $\Sing(X/S,f)$ at our disposal, following \cite{MR3281141}, we may look at it as an object in the $\infty$-category $\rmSHNC_S$ of \emph{non-commutative motives} over $S$. By \cite{MR3281141}, we have an $\infty$-functor $\mathcal{M}_S: \rmSHNC_S \to \rmSH_S$, from non-commutative motives to the ($\infty$-categorical version of) Morel and Voevodsky stable category $\rmSH_S$ of commutative motives over $S$. The functor $\mathcal{M}_S$ is the lax monoidal right adjoint to the functor  $\rmSH_S  \to  \rmSHNC_S$ canonically induced by the rule $Y \mapsto \mathrm{Perf}(Y)$, where $\mathrm{Perf}(Y)$ denotes the $A$-dg-category of perfect complexes on a smooth $S$-scheme $Y$. By \cite[Theorem 1.8]{MR3281141},  $\mathcal{M}_S$ sends the image of the tensor unit $A\in \rmSHNC_S$ to the object $\BU_S$, representing homotopy algebraic K-theory, i.e. to the commutative motive identified by the fact that $\BU_S(Y)$ is the spectrum of non-connective homotopy invariant algebraic K-theory of $Y$, for any smooth $S$-scheme $Y$. As a consequence, $\BU_S$ is endowed with the structure of a commutative algebra in the symmetric monoidal $\infty$-category $\rmSH_S$. Therefore,  $\mathcal{M}_S$ actually factors, as a lax monoidal functor,   
$\mathcal{M}_S: \rmSHNC_S \to \Mod_{\BU_S}(\rmSH_S)$ via the category of $\BU_S$-modules in $\rmSH_S$.\\

The  first main idea in this paper (see Section \ref{Section-motivedgcategoryI}) is to \emph{modify the functor} $\mathcal{M}_S$ in order to obtain different informations, better suited to our goal. Instead of $\mathcal{M}_S$, we consider a somewhat dual version $$\mathcal{M}^{\vee}_S: \rmSHNC_S \longrightarrow \Mod_{\BU_S}(\rmSH_S)$$ which is, roughly speaking, defined by sending an $A$-dg-category $T$ to the commutative motive $\mathcal{M}^{\vee}_S(T)$ that sends a smooth $S$-scheme $Y$ to the spectrum $\mathrm{KH}(\mathrm{Perf(Y)}\otimes_A T)$ of (non-connective) homotopy invariant algebraic K-theory of the $A$-dg-category $\mathrm{Perf(Y)}\otimes_A T$ (see Section \ref{Section-motivedgcategoryI} for details). In particular, for $p: X \to S$, with $X$ quasi-compact and quasi-separated, we get (Proposition \ref{prop-descriptionmotivedgcategoryviapushforward}) an equivalence $\mathcal{M}^{\vee}_S(\mathrm{Perf}(X)) \simeq p_*(\BU_{X})$ in $\Mod_{\BU_S}(\rmSH_S)$, where $\BU_X$ denotes a relative version homotopy invariant algebraic K-theory, and (Proposition \ref{prop-BUmotive2periodic})  an equivalence $\mathcal{M}_S^\vee(\Sing(S,0_S))\simeq \BU_S\oplus \BU_S[1]$ in $\Mod_{\BU_S}(\rmSH_S)$. As consequence of this, the motive $\mathcal{M}_S^\vee(\Sing(X,f))$ is a module over $\BU_S\oplus \BU_S[1]$, for any LG-pair $(X,f)$ over $S$.\\

The second main idea in this paper (see Section \ref{subsection-adicreal}) is to compose the functor $\mathcal{M}^{\vee}_S: \rmSHNC_S \to \Mod_{\BU_S}(\rmSH_S)$ with the $\ell$\emph{-adic realization functor} $\mathrm{R}_S^{\ell}: \rmSH_S \to  \mathrm{Sh}_{\mathbb{Q}_{\ell}}(S)$ with values in the $\infty$-categorical version of $\mathrm{Ind}$-constructible $\ell$-adic sheaves on $S$ with $\mathbb{Q}_{\ell}$-coefficients of \cite{MR1106899}. %

 Building on results of Cisinski-Deglise and Riou, we prove (see Section \ref{subsubsection-ladicrealizationofMF}) that one can refine $\mathrm{R}_S^{\ell}$ to a functor, still denoted by the same symbol,  $$\mathrm{R}_S^{\ell}: \Mod_{\BU_S}(\rmSH_S) \longrightarrow \Mod_{\mathbb{Q}_\ell(\beta)}(\mathrm{Sh}_{\mathbb{Q}_{\ell}}(S)),$$ where $\beta$ denotes the algebraic Bott element. We then denote by $\mathcal{R}_S^{\ell}$ the composite $$\xymatrix{\mathcal{R}_S^{\ell}: \rmSHNC_S \ar[r]^-{\mathcal{M}^{\vee}_S} & \Mod_{\BU_S}(\rmSH_S) \ar[r]^-{\mathrm{R}_S^{\ell}} &  \Mod_{\mathbb{Q}_\ell(\beta)}(\mathrm{Sh}_{\mathbb{Q}_{\ell}}(S))}.$$\\

We are now in a position to state our main theorem comparing singularity categories and vanishing cycles (see Section \ref{section-maintheo}). Let us take $S=\Spec A$ to be a henselian trait with $A$ excellent \footnote{In practice we will be working with complete discrete valuation rings. In this case excellence conditions is always verified.} and with algebraically closed residue field $k$, quotient field $K$, and let us fix a uniformizer $\pi$ in $A$, so that $A/\pi =k$. We denote by $i_{\sigma}: \sigma:= \Spec \, k \to S$ the canonical closed immersion, and by $\eta$ the generic point of $S$. \\ Given now a regular scheme $X$, together with a morphism $p:X\to S$ which is proper and flat, we consider the LG-pair $(X/S,\underline{\pi})$, where $\underline{\pi}$ is defined as the composite $$\xymatrix{
\underline{\pi}: X\ar[r]^-p & S \ar[r]^-{\pi} & \mathbb{A}^1_S 
}.$$ For a prime $\ell$ different from the characteristic of $k$, we may consider the following two objects inside $\mathrm{Sh}_{\mathbb{Q}_{\ell}}(\sigma)= \mathrm{Sh}_{\mathbb{Q}_{\ell}}(k)$:\\

\begin{itemize}
\item the homotopy invariants $\mathbb{H}_{\'et}(X_{k}, \mathcal{V}(\beta)[-1])^{\mathrm{h}I}$, where $\mathbb{H}_{\'et}$ denotes $\ell$-adic \'etale hypercohomology (ie. derived global sections),   $\mathcal{V}$ is the complex of vanishing cycles relative to the map $p$, $\mathcal{V}(\beta)[-1] := \mathcal{V}[-1]\otimes \mathbb{Q}_{\ell}(\beta)$ (see Remark \ref{remark-vanishingcyclesperiodicequaltensorusual}) and $\mathrm{I}=\mathrm{Gal}(K^{\mathrm{sep}}/K)$ is the inertia group;\\
\item the derived pullback $i_{\sigma}^* (\mathcal{R}_S^{\ell}(\Sing(X, \underline{\pi})))$ of the $\ell$-adic realization of the commutative motive given by the image under $\mathcal{M}^{\vee}_{S}$ of the dg-category of singularities for the LG-pair $(X/S,\underline{\pi})$.
\end{itemize}

One way to state our main result is then\\

\noindent \textbf{Main Theorem.} (see Theorem  \ref{theorem-maincomparisonvanishingmf}) There is a canonical equivalence $$i_{\sigma}^* (\mathcal{R}_S^{\ell}(\Sing(X, \underline{\pi}))) \simeq \mathbb{H}_{\'et}(X_{k}, \mathcal{V}(\beta)[-1])^{\mathrm{h}I}$$ inside $\mathrm{Sh}_{\mathbb{Q}_{\ell}}(k)$. Moreover, this equivalence is compatible with the actions of  $i_\sigma^\ast\mathcal{M}_S^\vee(\Sing(S,0_S))$ on the l.h.s and $ \BU_\sigma\oplus \BU_\sigma[1]$ on the r.h.s.

\vspace{1cm}

What this theorem tells us is that one can recover vanishing cohomology through the dg-category of singularities, i.e. in a purely non-commutative (and derived) geometrical setting. We think of this result as both an evidence and a first step in the application of non-commutative derived geometry to problems in arithmetic geometry, that we expect to be very fruitful. It is crucial, especially for future applications, to remark that our result holds over an arbitrary base henselian trait $S$ (with perfect residue field), so that it holds both in pure and mixed characteristics. In particular, we do not need to work over a base field. The main theorem above is also at the basis of the research announcement \cite{bloch}, where a trace formula for dg-categories is established, and then used to propose a strategy of proof of Bloch's conductor conjecture (\cite{BlochConj, KaSa}). Full details will appear in \cite{ToVez}.\\

\begin{rem}
The result of Theorem \ref{theorem-maincomparisonvanishingmf} is stated for the $\ell$-adic realization but can also be given a motivic interpretation. Indeed, one can use the formalism of motivic vanishing cycles of \cite{ayoub2, MR3205601} to produce a motivic statement that realizes to our formula. We thank an anonymous referee for his comments and suggestions regarding this motivic presentation. The proof is mutatis mudantis the one presented here for the $\ell$-adic realization and we will leave it for further works.
\end{rem}

\medskip
\noindent \textbf{Related works.} The research conducted in the second part of this work has its origins in Kashiwara's computation of vanishing cycles in terms of D-modules via the Riemann-Hilbert correspondence \cite{MR726425}. A further deep and pioneering work is undoubtedly Kapranov's influential paper \cite{kaprabook} which starts with the identification of D-modules with modules over the de Rham algebra. This relation between vanishing cycle cohomology and twisted de Rham cohomology as been fully understood by Sabbah and Saito in \cite{MR3234116, 1012.3818} esablishing proofs for the conjectures of Kontsevich-Soibelman in \cite{MR2851153}. In parallel, the works of \cite{MR3084707,MR2824483, MR3108698, MR3180736} established the first link between twisted de Rham cohomology and the Hochschild cohomology of matrix factorizations and more recently, the situation has been clarified with Efimov's  results \cite{1212.2859}. The combination of these results express a link between the theory of matrix factorizations and the formalism of vanishing cycles. The recent works of Lunts and Schn\"urer's \cite[Theorem 1.2]{LuSc} built upon Efimov's work, combined with those of \cite{MR3131490}, show that this connection between the two theories can be expressed as an equivalence of classes in a certain Grothendieck group of motives. The main result (Theorem  \ref{theorem-maincomparisonvanishingmf}) of this paper might, in some sense, be seen as categorification of Lunts and Schn\"urer's results.

\bigskip

\noindent \textbf{Acknowledgements.} We warmly thank the anonymous referee for her/his exceptionally insightful and useful corrections and suggestions that led, as we hope, to a more readable paper. MR wishes to thank Mauro Porta for very useful discussions on the subject matter of this paper. BT was Partially supported by ANR-11-LABX-0040-CIMI within the program ANR-11-IDEX-0002-02.  This project has received funding from the European Research Council (ERC) under the European Union’s Horizon 2020 research and innovation programme (grant agreement ERC-2016-ADG-741501). GV wishes to thank A. Efimov for an helpful exchange about his joint work with L. Positselski. GV would like to thank the Laboratoire IMJ-PRG, Paris 6, for a very fruitful visit in May 2016, when a part of our collaboration took place. 
We would also like to thank the anonymous referees for their comments.

\medskip
\section{Matrix factorizations and derived categories of singularities}
\label{section-MFandSing}

\begin{context}
\label{context111}
Throughout this section  $A$ will be a commutative Noetherian regular local ring, $S:=\Spec\, A$ and  $\Sch_S$  the category of schemes of finite type over $S$.
\end{context}

\begin{defn}
 We introduce the category of \emph{Landau-Ginzburg models over $S$} as the subcategory of $(\Sch_S)_{/\mathbb{A}^1_S}$ spanned by those pairs 
$$(p:X\to S, f:X\to \mathbb{A}^1_S)$$
where $p$ is a flat morphism. We will denote it by $\LG_S$. We will denote by $\LG^\mathrm{aff}_S$ its full subcategory spanned by those LG-models where $X$ is affine over $S$.\\
\end{defn}

\begin{defn}
We also introduce the category of \emph{flat Landau-Ginzburg models over $S$} as the full subcategory of $\LG_S$ consisting of $(S,0: S \to \mathbb{A}_S^1)$ (where $0$ denotes the zero section of the canonical projection $\mathbb{A}_S^1 \to S$) together with those pairs 
$$(p:X\to S, f:X\to \mathbb{A}^1_S)$$
where both $p$ and $f$  are flat morphisms. We will denote it by $\LG^{\mathrm{fl}}_S$. We will denote by $\LG^\mathrm{fl, \,aff}_S$ its full subcategory spanned by those LG-models where $X$ is affine over $S$.\\
\end{defn}

\begin{cons}
The category $\LG_S$ (resp. $\LG^{\mathrm{fl}}_S$) has a natural symmetric monoidal structure $\boxplus$  given by the fact that the additive group structure on $\mathbb{A}_S^1$ \footnote{Given, on functions, by $A[T]\mapsto A[X]\otimes_A A[Y], \; T\mapsto X\otimes 1 + 1\otimes Y$} defines a monoidal structure on the category $(\Sch_S)_{/\mathbb{A}^1_S}$, 

$$\boxplus: (\Sch_S)_{/\mathbb{A}^1_S}\times (\Sch_S)_{/\mathbb{A}^1_S} \to (\Sch_S)_{/\mathbb{A}^1_S}$$
\noindent given by 
 \begin{equation}\label{eq-def-monoidalstructureLG}
(X,f), (Y,g) \mapsto (X,f) \boxplus (Y,g):= (X\times_{S}Y,f\boxplus g)
\end{equation}

\noindent where $f\boxplus g:=p_X^{*}(f) + p_Y^{*}(g)$ for 
$p_X : X\times_{S}Y \rightarrow X$ and $p_Y : X\times_{S}Y \rightarrow Y$
the two projections. Notice that as the maps to $S$ are flat, the fiber product $X\times_S Y$ is also flat over $S$, and thus it belongs to $\LG_S$. Moreover, if $f$ and $g$ are flat, then 
$f\boxplus g$ is also flat (since flatness is stable by arbitrary base-change, and the sum map $\mathbb{A}_S^1 \times_S \mathbb{A}_S^1 \to \mathbb{A}_S^1$ is flat), and if $g= 0: S \to \mathbb{A}_S^1$, then $f\boxplus g =f$, which is again flat. In particular, the unit for this monoidal structure is $(S,0)$.\\ We will denote this monoidal structure on $\mathrm{LG}_{S}$ (resp. $\LG^{\mathrm{fl}}_S$) by $\mathrm{LG}^{\boxplus}_{S}$ (resp. $\LG^{\mathrm{fl}\, \boxplus}_S$). Obviously $\LG^{\mathrm{fl}\, \boxplus}_S$  is a full symmetric monoidal subcategory of $\LG^{\boxplus}_S$.  We use similar notations for the symmetric monoidal subcategories of (flat) affine $LG$-pairs (recall that $S$ is affine, so that also the affine versions have $(S,0)$ as the unit). \\
\end{cons}

\begin{rem}\label{nonflat}
Orlov works inside $\LG^{\mathrm{fl}}_S$ in \cite{MR2101296}, while Efimov and Positselski work in the whole $\LG_S$ in \cite{efimov2015coherent}.
\end{rem}

In this section we discuss two well known constructions, namely 
matrix factorizations and categories of singularities. 
For us, these will be defined as
 $\s$-functors with values in dg-categories

$$\MF, \Sing, :\mathrm{LG}^{\mathrm{op}}_S \longrightarrow \dgcat_A^{\mathrm{idem}},$$

\noindent from the category of LG-models to the $\s$-category of 
(small) $A$-linear idempotent complete dg-categories. Our $\MF$ will be in fact a lax symmetric monoidal $\infty$-functor $\MF: \mathrm{LG}^{\boxplus, \mathrm{op}}_S \longrightarrow \dgcat_A^{\mathrm{idem}}$. The first construction we want to describe 
sends $(X,f)$ to the dg-category $\MF(X,f)$ of
matrix factorizations of $f$. The second one sends 
$(X,f)$ to $\Sing(X,f)$, the dg-category of (relative)
singularities of the scheme $X_0$ of zeros of $f$.
We compare these two constructions by means of 
the so-called \emph{Orlov's equivalence}, which for us will be 
stated as the existence of a natural transformation of $\s$-functors.
$$
\Sing\to \MF
$$
\noindent which is an equivalence when restricted to pairs $(X,f)$ with $X$ regular.\\

The results of this section consist
mainly in $\s$-categorical enhancements of well known results in the world of triangulated categories.\\

\subsection{Review of dg-categories}
\label{notations-dg-categories}
For the discussion in this section $A$ can be any commutative ring.
In this  paragraph we fix our notations for the theory of dg-categories, by recalling the main definitions and constructions used in the rest of the paper. Our references for dg-categories will be \cite{MR2762557} and \cite[Section 6.1.1 and 6.1.2]{robalo-thesis}. \\

As an $\s$-category, $\dgcat_A^{\mathrm{idem}}$ is a Bousfield localization of the $\s$-category of 
(small) $A$-linear dg-categories with respect to Morita equivalences, namely dg-functors inducing equivalences on the
the respective derived categories of perfect dg-modules. The $\s$-category
$\dgcat_A^{\mathrm{idem}}$ is naturally identified with the full
sub-$\s$-category of $\dgcat_A$ consisting of \emph{triangulated
dg-categories}\footnote{This terminology is not standard, 
as for us \emph{triangulated} also includes being idempotent complete.} in the sense of Kontsevich. Recall 
that these are small
dg-categories $T$ such that the Yoneda embedding
$T \hookrightarrow \widehat{T}_{pe}:=\{\text{Perfect } T^{op}-\text{dg-modules}\}$, is an equivalence (ie,
any perfect $T^{op}$-dg-module is quasi-isomorphic to a representable 
dg-module). With this identification the localization
$\s$-functor
\begin{equation}\label{eq-idempotentcompletionfunctor}
\dgcat_A \longrightarrow \dgcat_A^{\mathrm{idem}} 
\end{equation}
simply sends $T$ to $\widehat{T}_{pe}$. 

The $\s$-category $\dgcat_A$ can be obtained as a localization of the  $1$-category $\mathrm{dgcat}_A^\mathrm{strict}$ of small strict $A$-dg-categories with respect to Dwyer-Kan equivalences. This localization is enhanced by the existence of a model structure on  $\mathrm{dgcat}_A^\mathrm{strict}$ \cite{tabuada-quillen}. 

Moreover, both $\dgcat_A^{\mathrm{idem}}$ and $\dgcat_A$ come canonically equipped with symmetric monoidal structures induced by the tensor product of locally flat dg-categories $\mathrm{dgcat}_A^{\mathrm{strict, loc-flat}}\subseteq \mathrm{dgcat}_A^\mathrm{strict}$ - namely, those strict dg-categories whose enriching hom-complexes are flat in the category of chain complexes. The localization functor 
\begin{equation}
\label{eq-monoidallocalizationdgcategorieslocallyflat}
\mathrm{dgcat}_A^{\mathrm{strict, loc-flat}}\to \dgcat_A
\end{equation}
\noindent is monoidal with respect to these monoidal structures.  We address the reader to \cite{MR2762557} for a complete account of the dg-categories, to \cite[Prop. 2.22]{tabuada-cisinski} for the monoidal structure and \cite[Section 6.1.1 and 6.1.2]{robalo-thesis} for an $\s$-categorical narrative of these facts \footnote{where we use locally cofibrant dg-categories instead of locally flat. The two strategies are equivalent as every locally cofibrant dg-category is  locally flat \cite[2.3(3)]{Toen-homotopytheorydgcatsandderivedmoritaequivalences} and a cofibrant replacement functor is an inverse. See also the discussion in \cite[p. 222]{robalo-thesis}.}.

\medskip

\begin{notation}
For a dg-category $T$, we will denote as $[T]$ its homotopy category.
\end{notation}

At several occasions we will need to take dg-quotients: if $T_0\to T$ is a map in $\dgcat_A^{\mathrm{idem}}$, one considers its cofiber as a pushout
$$\xymatrix{
T_0\ar[r] \ar[d] & T\ar[d] \\
\{0\} \ar[r] & T \ \coprod_{T_0}\{0\}:=T'}$$
in $\dgcat_A^{\mathrm{idem}}$.  By a result of Drinfeld \cite{drinfeld1} the homotopy category of $T'$ can be canonically identified with the classical Verdier quotient of $T$ by the image of $T_0$. This  pushout is equivalent to the idempotent completion of the pushout taken in the $\dgcat_A$, thus given by $\widehat{T'}_{\mathrm{pe}}$ and its homotopy category can be identified with the idempotent completion of the Verdier quotient of $T$ by $T_0$.

\vspace{0.5cm}

To conclude this review, let us mention that for any ring $A$, $\dgcat_A^{\mathrm{idem}}$ can also be identified with the $\s$-category of small stable idempotent complete $A$-linear $\s$-categories. The proof of \cite{dgklinear} adapts to any characteristic. We address the reader to the discussion \cite[Section 6.2]{robalo-thesis} for more helpful comments.

\medskip

\subsection{Matrix factorizations}
In this section we again work under the Context \ref{context111}. We now deal with the construction of the symmetric lax monoidal
$\s$-functor
\begin{equation}
\label{equation-MFinfinityfunctor}
\MF : \LG_S^{\mathrm{aff, op}} \longrightarrow \dgcat_{A}^{\mathrm{idem}}
\end{equation}

\medskip

To define this lax monoidal $\s$-functor we will first construct an auxiliary strict version and explain its lax structure. \\

\subsubsection{} Let $(X,f) \in \LG_{S}^{\mathrm{aff}}$ and let us write $X:=\Spec\, B$, for $B$ flat of finite type over $A$. The function $f$ is thus identified
with an element $f\in B$. We associate to the pair $(B,f)$ a strict $\mathbb{Z}/2$-graded $A$-dg-category
$\MFPairs(B,f)$ as follows. 

\begin{cons}
\label{construction-MFasfunctor}
First we construct $\MFPairs(B,f)$ as an object in the theory of small strict $\mathbb{Z}/2$-graded $B$-dg-categories, meaning, small strict dg-categories enriched in $\mathbb{Z}/2$-graded complexes of $B$-modules. Its objects are pairs $(E,\delta)$, 
consisting of the following data.

\begin{enumerate}

\item A $\mathbb{Z}/2$-graded $B$-module $E=E_0\oplus E_1$, with 
$E_0$ and $E_1$ projective and of finite rank over $B$. 

\item A $B$-linear endomorphism $\delta : E \rightarrow E$ of odd degree, 
and satisfying $\delta^2= \text{multiplication by }f$. 

\end{enumerate}

In a more explicit manner, objects in $\MFPairs(X,f)$ can be
written as 4-tuples, $(E_0,E_1,\delta_0,\delta_1)$, 
consisting of $B$-modules projective and of finite type $E_i$, 
together with $B$-linear morphisms
$$\xymatrix{
E_0 \ar^-{\delta_0}[r] &  E_1  } \qquad \xymatrix{
E_1 \ar^-{\delta_1}[r] &  E_0  }$$
such that $\delta_0 \circ \delta_1=\delta_1 \circ \delta_0=\cdot f$. 

For two objects $E=(E,\delta)$
and $F=(F,\delta)$, we define a $\mathbb{Z}/2$-graded complex of $B$-modules of morphisms
$\underline{Hom}(E,F)$ in the usual manner. As a $\mathbb{Z}/2$-graded $B$-module, 
$\underline{Hom}(E,F)$ simply is the usual decomposition of $B$-linear
morphisms $E \rightarrow F$ into odd and even parts. The differential
is itself given by the usual commutator formula: for $t\in \underline{Hom}(E,F)$
homogenous of odd or even degree, we set 
$$d(t):=[t,\delta]=t \circ \delta - (-1)^{deg(t)} \delta \circ t$$
Even though $\delta$ does not square to zero, we do have $d^2=0$. This defines
$\mathbb{Z}/2$-graded complexes of $B$-modules $\underline{Hom}(E,F)$ and sets $\MFPairs(B,f)$ as a $\mathbb{Z}/2$-graded $B$-dg-category.

Composing with the structure map $\Spec \,B \to \Spec \, A =S$, one can now understand  $\MFPairs(B,f)$ as $\mathbb{Z}/2$-graded $A$-linear dg-category. Notice that as $A$ is by hypothesis a local ring, and $B$ is flat over $A$, being projective of finite rank over $B$ implies being projective of finite rank over $A$, as flat and projective become equivalent notions as soon as the modules are finitely generated.
\end{cons}

\begin{cons}
\label{construction-pseudofunctorstrictMF}
The assignment $(X,f)\mapsto \MFPairs(X,f)$  acquires a pseudo-functorial structure 
\begin{equation}
\label{eq-MFZ2gradedversion}
\LG_S^{\mathrm{aff, \mathrm{op}}}\to \mathrm{dgcat}_{\mathbb{Z}/2, A}^{strict}\end{equation}
as any morphism of $A$-algebras $q : B \longrightarrow B'$, 
with $q(f)=f'$, defines by base change from $B$ to $B'$ a $\mathbb{Z}/2$-graded $A$-linear dg-functor
$$B'\otimes_{B} -  : \MFPairs(B,f) \longrightarrow \MFPairs(B',f')$$
\end{cons}

\begin{notation}
\label{notation-identification2periodicZ2graded}
Throughout this work we will always allow ourselves to freely interchange the notions of $\mathbb{Z}/2$-graded complexes
and $2$-periodic $\mathbb{Z}$-graded complexes via an equivalence of strict categories
\begin{equation}
\label{eq-2periodicequalAuinvertible}
\xymatrix{ \mathrm{dgMod}^{\mathbb{Z}/2}_A \ar[r]^-{\theta}_-{\sim}& A[u,u^{-1}]-\mathrm{dgMod}}
\end{equation}
where $A[u,u^{-1}]$ is the free strictly commutative differential graded algebra over A with an invertible generator $u$ sitting in cohomological degree $2$. The functor $\theta$ sends a $\mathbb{Z}/2$-graded complex complex $(E_0, E_1, \delta_0 , \delta_1)$ to the $A[u,u^{-1}]$-dg-module
$$
[...\to E_1\to E_0\to E_1\to E_0\to...]
$$
where $u$ acts via the identity. The inverse equivalence to $\theta$ sends an $A[u,u^{-1}]$-dg-module $F$ to the 2-periodic complex with $F_0$ in degree $0$ and $F_1$ in degree $1$, together with the differential $F_0\to F_1$ of $F$ and the new differential $F_1\to F_2\simeq F_0$ using the action of $u^{-1}$. Moreover, $\theta$ is symmetric monoidal: the tensor product of 2-periodic complex identifies with the tensor product over $A[u,u^{-1}]$. In particular, this induces an equivalence of 1-categories between the theory of $\mathbb{Z}/2$-graded $A$-dg-categories and that of $A[u,u^{-1}]$-dg-categories :
\begin{equation}
\label{eq-Z2gradeddgcatsVS2periodicDgcats}
\xymatrix{ \mathrm{dgcat}_{\mathbb{Z}/2, A}^{strict, \otimes} \ar[r]^-{\theta}_-{\sim}&  \mathrm{dgcat}_{ A[u, u^{-1}]}^{strict, \otimes}}
\end{equation}
Applying $\theta$ to the enriching 2-periodic hom-complexes on the l.h.s of (\ref{eq-2periodicequalAuinvertible}), the equivalence (\ref{eq-2periodicequalAuinvertible}) becomes an equivalence of strict $A[u, u^{-1}]$-dg-categories.\\
This discussion descends to a monoidal equivalence of between the (Morita) homotopy theories  dg-categories, as explained in \cite[Section 5.1]{Tobias}. Moreover, the results of \cite{Toen-homotopytheorydgcatsandderivedmoritaequivalences} describing internal-homs in the Morita theory remains valid. 
\end{notation}

\vspace{0.5cm}
For our purposes we will work with the version of $\mathrm{MF}$ obtained by the composition of (\ref{eq-MFZ2gradedversion}) with (\ref{eq-Z2gradeddgcatsVS2periodicDgcats}).

\vspace{0.5cm}

\subsubsection{} We will now give a strict version of the symmetric lax monoidal structure on $\MFPairs$:

\begin{cons}
\label{construction-strictlaxstructureMF}
Let $(X,f)$ and 
$(Y,g)$ be two objects in $\LG_{S}^{\mathrm{aff}}$, with $X=\Spec\, B$ and 
$Y=\Spec\, C$ (so $f\in B$ and 
$g\in C$). We consider the pair $(D,h)$, where $D=B\otimes_{A}C$ and
$h=f\otimes 1 + 1\otimes g \in D$. We have a natural $A$-linear dg-functor
\begin{equation}
\label{eq-strictlaxstructuremorphismMF}
\boxtimes  : \MFPairs(B,f) \otimes_{A} \MFPairs(C,g) \longrightarrow \MFPairs(D,h),
\end{equation}
obtained by the external tensor product as follows. For two
objects $E=(E,\delta) \in \MFPairs(B,f)$ and 
$F=(F,\partial) \in \MFPairs(C,g)$, we define a  
projective $D$-module of finite type
$$E\boxtimes F:=E \otimes_A F,$$
with the usual induced $\mathbb{Z}/2$-graduation: the even part of 
$E\boxtimes F$ is $(E_0\otimes_A F_0) \oplus (E_1 \otimes_A F_1)$, and
its odd part is $(E_1\otimes_A F_0) \oplus (E_0\otimes_A F_1)$.
The odd endomorphism $\delta : E\boxtimes F \longrightarrow E\boxtimes F$
is given by the usual formula on homogenuous generators
\begin{equation}
\label{eq-formulaMFproduct}
\delta(x\otimes y):=\delta(x)\otimes y + (-1)^{deg x}x\otimes \partial(y).
\end{equation}
All together, this defines an object $E\boxtimes F \in \MFPairs(D,h)$, and with a bit more work
a morphism in $\mathrm{dgcat}_{\mathbb{Z}/2,A}^{\mathrm{strict}}$ giving shape to (\ref{eq-strictlaxstructuremorphismMF}).
These are clearly symmetric and associative. Finally, 
the construction $\MFPairs$ is also lax unital with unity given by the natural $\mathbb{Z}/2$-graded $A$-linear dg-functor
\begin{equation}
\label{eq-laxunitMFstrict}
A \longrightarrow \MFPairs(A,0),
\end{equation}
sending the unique point of $A$ to $(A[0],0)$ where $A[0]$ is 
$A$ considered as a $\mathbb{Z}/2$-graded $A$-module pure of even degree (and $A$ is considered
as an $\mathbb{Z}/2$-graded $A$-linear dg-category with a unique object with $A$ as endomorphism algebra). This finishes the description of the lax symmetric structure on 
\begin{equation}
\label{eq-MFstrict22}
\MFPairs : \LG_{S}^{\mathrm{aff}, op, \boxplus} \longrightarrow \mathrm{dgcat}_{\mathbb{Z}/2, A}^\mathrm{strict, \otimes}\simeq \mathrm{dgcat}_{A[u, u^{-1}]}^\mathrm{strict, \otimes}
\end{equation}
\end{cons}
\vspace{0.5cm}
\begin{cons}
We must now explain how to use the strict lax structure of the Construction \ref{construction-strictlaxstructureMF} to produce a lax structure in the homotopy theory of dg-categories. Following the discussion in the Section \ref{notations-dg-categories}, it will be enough to show that $\MFPairs$ has values in locally-flat dg-categories. But this is indeed the case as by definition the objects of $\MFPairs(B,f)$ are pairs of $B$-modules that are finitely generated and projective, hence in particular flat over $B$, and  therefore also  over $A$ (as $B$ is assumed flat over $A$). Therefore, the enriching hom-complexes are flat over $A$. Following this discussion, the lax symmetric monoidal functor (\ref{eq-MFZ2gradedversion}) factors as 
\begin{equation}
\label{eq-MFstrictlaxmonoidal}
 \LG^{\mathrm{aff, op, \boxplus}}\to \mathrm{dgCat}^{\mathrm{strict, loc-flat, \otimes}}_{\mathbb{Z}/2, A}
 \end{equation}
and by composition with $(\ref{eq-Z2gradeddgcatsVS2periodicDgcats})$ and the restriction along $A\to A[u,u^{-1}]$ we obtain a lax symmetric monoidal functor
\begin{equation}
\label{eq-MFstrictlaxmonoidal}
 \LG^{\mathrm{aff, op, \boxplus}}\to \mathrm{dgCat}^{\mathrm{strict, loc-flat, \otimes}}_{\mathbb{Z}/2, A}\underbracket{\simeq}_{(\ref{eq-Z2gradeddgcatsVS2periodicDgcats})} \mathrm{dgCat}^{\mathrm{strict, loc-flat/A, \otimes}}_{A[u, u^{-1}]}\underbracket{\to}_{rest.}\mathrm{dgCat}^{\mathrm{strict, loc-flat/A, \otimes}}_{A}
 \end{equation}
Finally, we compose this with the monoidal localization $\infty$-functor (\ref{eq-monoidallocalizationdgcategorieslocallyflat}) followed by (\ref{eq-idempotentcompletionfunctor}), to obtain a new lax monoidal $\s$-functor 
\begin{equation}
\label{eq-laxmonoidalstructurefinal33}
\MFPairs:\LG^{\mathrm{aff, op, \boxplus}}\to \dgcat_A^{\mathrm{idem}, \otimes}\end{equation}
\end{cons}

\begin{rem}
\label{remark-comparisonMonoidalstructuresMF-2-periodic}
The restriction of scalars along $A\to A[u,u^{-1}]$ forgets the 2-periodic structure. However, it is a consequence of the construction that we can recover this 2-periodic structure from the lax monoidal structure (\ref{eq-laxmonoidalstructurefinal33}).\\
 Indeed, the lax monoidal structure endows $\MFPairs(S,0)$ with a structure of object in $\CAlg(\dgcat^{\mathrm{idem}}_A)$. At the same time, as $\MFPairs(S,0)$ admits a compact generator given by $A$ in degree 0, it follows that in $\dgcat^{\mathrm{idem}}_A$, $\MFPairs(S,0)$ is equivalent to perfect complexes over the dg-algebra $\theta(\mathrm{End}_{\MFPairs(S,0)}(A))$. But an explicit computation shows that this is a strict-dg-algebra given by $A[u, u^{-1}]$ with $u$ a generator in cohomological degree $2$.  In addition to this, the symmetric monoidal structure on $\MFPairs(S,0)$, which by construction of the lax structure in fact exists in the strict theory of strict dg-categories, produces a structure of strict commutative differential graded algebra on $A[u, u^{-1}]$ which corresponds to the standard one. It follows that in $\CAlg(\dgcat^{\mathrm{idem}}_A)$, we have a monoidal equivalence 
 \begin{equation}
\label{eq-comparisonMonoidalstructuresMF-2-periodic}
\MFPairs(S,0)^{\boxplus}\simeq \Perf(A[u,u^{-1}])^{\otimes_{A[u,u^{-1}]}}
\end{equation}
where the r.h.s is equipped with the relative tensor product of $A[u,u^{-1}]$. It follows from the lax structure and the monoidal equivalence (\ref{eq-comparisonMonoidalstructuresMF-2-periodic}) that $\MF$ extends to a lax monoidal functor
\begin{equation}
\label{eq-remark-comparisonMonoidalstructuresMF-2-periodic2}
\MF:\LG^{\mathrm{aff, op, \boxplus}}\to \Mod_{\Perf(A[u,u^{-1}])}(\dgcat^{\mathrm{idem}}_A)^{\otimes}
\end{equation}
We observe that this is precisely the 2-periodic structure of  (\ref{eq-MFstrictlaxmonoidal}) before restricting scalars along $A\to A[u,u^{-1}]$. Indeed, this follows from the commutativity of the square of lax monoidal $\infty$-functors
\begin{equation}
\xymatrix{
\ar[d]\mathrm{N}(\mathrm{dgCat}^{\mathrm{strict, loc-flat/A, \otimes}}_{A[u, u^{-1}]}) \ar[r]& \mathrm{N}(\mathrm{dgCat}^{\mathrm{strict, loc-flat/A, \otimes}}_{A})\ar[d]\\
\dgcat^{\mathrm{idem}, \otimes}_{A[u, u^{-1}]}\ar[r]& \dgcat^{\mathrm{idem}, \otimes}_{A}
}
\end{equation}
and the monadic equivalence

$$\dgcat^{\mathrm{idem}, \otimes}_{A[u, u^{-1}]}\simeq \mathrm{Mod}_{\Perf(A[u,u^{-1}])}(\dgcat_A^{\mathrm{idem}})^{\otimes}$$
\end{rem}

\vspace{0.5cm}

The construction of $\MF$ as a lax monoidal functor from affine LG-pairs to $A$-dg-categories, can be extended to all LG-pairs: one way is to interpret $\MF$ as a functor $\LG^{\mathrm{aff}}\to \dgcat_A^{\mathrm{idem}, \mathrm{op}}$ and take its monoidal left Kan extension $\mathrm{Kan}\MF$ to presheaves of spaces  $\mathcal{P}(\LG^{\mathrm{aff}})$ \cite[4.8.1.10]{lurie-ha}. Now $\LG_S$ embeds fully faithfully by Yoneda inside $\mathcal{P}(\LG^{\mathrm{aff}})$ in a monoidal way with respect to the Day product.  The restriction to this full subcategory defines an $\infty$-functor
\begin{equation}
\label{eq-MFfinalfunctor}
\mathrm{Kan}\MF: \LG_S^{\mathrm{op}} \longrightarrow \dgcat^{\mathrm{idem}}_A 
\end{equation}
\noindent of matrix factorizations over $S$, naturally equipped with a  lax symmetric monoidal enhancement 
\begin{equation}
\label{eq-MFlaxfinalfunctor}
\mathrm{Kan}\MF^\boxplus : \LG_S^{\mathrm{op}, \boxplus} \longrightarrow \dgcat^{\mathrm{idem},\otimes}_A.
\end{equation}
Alternatively, there is a definition of matrix factorizations on non-affine LG-pairs $(X,f)$ under the assumption that $X$ has enough vector bundles. Indeed, one should work with matrix factorizations $(E_0, E_1, \delta_0, \delta_1)$ where $E_0$ and $E_1$ are vector bundles on $X$. See \cite{MR2101296, MR2910782, lin2011global}. Under this assumption, this second definition agrees with the Kan extension. See \cite[2.11]{lin2011global}, \cite[Section 5]{1212.2859} or \cite[Section 3]{MR3007084}.

\vspace{0.5cm}

\subsection{dg-Categories of singularities}

\subsubsection{}
 Let $(X,f)$ be an LG-model. In this section we will study an invariant that captures the singularities of $X_0 \subset X$, the closed subscheme
of zeros of $f$. As we will not impose any condition on $f$, 
for instance $f$ can be a zero divisor, we have to allow $X_0$ to be
eventually a derived scheme. More precisely, we consider the derived fiber product
\begin{equation}
\xymatrix{
\ar[d] X_0:=S \times^\mathrm{h}_{\mathbb{A}_S^1}X \ar[r]^-{\mathfrak{i}}&\ar[d]^f X\\
S\ar[r]^0 & \mathbb{A}^1_S
}
\end{equation}
where the canonical map $\mathfrak{i}$ is an lci closed immersion\footnote{In this paper we will use the terminology "lci closed immersion" to mean a quasi-smooth closed immersion, namely, a map of derived schemes whose truncation is a closed immersion and whose relative cotangent complex is perfect and of Tor-amplitude $[-1,0]$. See for instance \cite{1210.2827,MR3300415, 1802.05702}}, as it is the base change of the lci closed immersion $0_S:S\to \mathbb{A}^1_S$.

\vspace{0.5cm}

\begin{rem}
\label{remark-flatnesshypothesismatchOrlov}
Note that if $(X,f) \in \LG_S^{\mathrm{fl}}$ (i.e. $f$ is flat), then $X_0$ is just the scheme theoretic zero locus of $f$. In particular it coincides with the $X_0$ consider by Orlov in 
\cite{MR2101296}. Therefore, in this case, the derived categories of singularities considered in this paper (see Def. \ref{d2} and Remark \ref{remark-differentdefinitionsrelativesingularities} below) is (an $\infty$- or dg-categorical version of) the derived category of singularities of \cite{MR2101296}.
\end{rem}

\vspace{0.5cm}

For an LG-model $(X,f)$, with associated  derived scheme $X_0$ of zeros of $f$, 
we consider $\Qcoh(X_0)$ the $A$-linear dg-category of 
quasi-coherent complexes on $X_0$ (see \cite[Section 3.1]{MR3285853} for a survey). We consider the following full sub-dg-categories of $\Qcoh(X_0)$:

\begin{itemize}
\item $\Perf(X_0)$: perfect objects over $X_0$, meaning, objects  $E \in \Qcoh(X_0)$ such that locally it belong to the thick sub-category of $\Qcoh(X_0)$ generated by the structure sheaf of $X_0$. These are exactly the $\otimes$-dualizable objects in $\Qcoh(X_0)$. In our case, as $X_0$ is a fiber product of quasi-compact and quasi-separated schemes it has the same property, and therefore by the results in \cite[3.6]{MR2669705} (see also \cite{toen-azumaya})  perfect complexes agree with compact objects in $\Qcoh(X_0)$.
\item $\Coh(X_0)$: cohomologically bounded objects $E \in \Qcoh(X_0)$ whose cohomology  $H^*(E)$ is a coherent $H^0(\mathcal{O}_{X_0})$-module.
\item $\mathrm{Coh}^-(X_0)$: cohomologically bounded above objects $E \in \Qcoh(X_0)$ whose cohomology  $H^*(E)$ is a coherent $H^0(\mathcal{O}_{X_0})$-module\footnote{These are also known as pseudo-perfect complexes (\cite{MR0354655}) or almost perfect complexes (\cite{lurie-ha}).}.
\item $\Coh(X_0)_{\Perf(X)}$: cohomologically bounded objects $E \in \Qcoh(X_0)$ whose cohomology  $H^*(E)$ is a coherent $H^0(\mathcal{O}_{X_0})$-module and such that the direct image of $E$ under $\mathfrak{i}_*$ is perfect over $X$;
\end{itemize}

\noindent where we always have inclusions \footnote{For a general derived scheme $Y$, the inclusion $\Perf(Y)\subseteq \Coh(Y)$ is not guaranteed. Indeed, it requires $Y$ to be eventually coconnective - see \cite[Chapter 1, Def. 1.1.6]{MR3701352}. This is automatic for $Y=X_0$ using the table in the Remark \ref{rem-functorialitiesperfandcoh}.}
$$\Perf(X_0)\subseteq \Coh(X_0) \subset   \mathrm{Coh}^-(X_0)\subset \Qcoh(X_0) $$
$$\Coh(X_0)_{\Perf(X)}\subseteq \Coh(X_0)$$
and the fact the map $X_0\to X$ is a lci closed immersion (of derived schemes), thus preserving perfect complexes under pushforward (see the Remark \ref{rem-functorialitiesperfandcoh} below), gives us another inclusion
$$
\Perf(X_0)\subseteq \Coh(X_0)_{\Perf(X)}
$$

\begin{rem}
\label{rem-cohperfareidempotentcomplete}
All $\Coh(X_0)$, $\Perf(X_0)$ and $\Coh(X_0)_{\Perf(X)}$ are idempotent complete $A$-dg-categories. This is well-known for $\Coh(X_0)$ and $\Perf(X_0)$. For $\Coh(X_0)_{\Perf(X)}$ this follows because both $\Coh(X_0)$ and $\Perf(X)$ are idempotent complete and the pushforward along $X_0\to X$ is an exact functor thus preserving all retracts that exist.
\end{rem}

\begin{rem}
\label{rem-functorialitiesperfandcoh}
The constructions of $\Coh, \Perf $ and $\mathrm{Coh}^-$ possess different $\infty$-functorial properties for maps of  derived schemes, summarized in the following table:
\medskip
\begin{center}
\resizebox{1. \hsize}{!}{
\begin{tabular}{|l|c|c|}
   \hline
    & Pullbacks & Pushforwards \\
   \hline
   $\Coh$ &  Finite Tor-amplitude \cite[Lemma 3.1.3, Chapter 4] {MR3701352} & Proper locally almost finite type \cite[Lemma 5.1.4, Chapter 4]{MR3701352}\\
   \hline
   $\Perf$ & All (dualizable objects) & Proper lci \cite{1210.2827} \\
   \hline
    $\mathrm{Coh}^-$ & All \cite[2.7.3.1]{lurie-sag}   &   Proper loc. almost of finite type \cite[5.6.0.2]{lurie-sag} \\
   \hline
\end{tabular}
}
\end{center}
\end{rem}

\personal{
For a general derived schemes one can use the characterization of almost perfect complexes as almost compact objects given in \cite[7.2.4.10, 7.2.4.17]{lurie-ha} plus the fact that for a map of connective $\mathrm{E}_\infty$-rings, base change preserves the condition of being cohomologicaly bounded above as the forgetful functor is both left and right $t$-exact. Alternatively, and more immediately, use \cite[7.2.4.11-(4) and (5)]{lurie-ha} }

\medskip

We will use this table to deduce the functorialities for the construction 
$$
(X,f)\mapsto \Coh(X_0)_{\Perf(X)}
$$
This requires some observations. We thank the anonymous referee for suggesting the following arguments, simplifying the discussion we had in previous versions of the paper:
\begin{lem}
\label{lemma-inclusiont-exact}
Let $(X,f)$ be be an $\LG$-pair over $S$. Then the functor $\mathfrak{i}_\ast: \Qcoh(X_0)\to \Qcoh(X)$ preserves perfect complexes, is t-exact and conservative.
\begin{proof}
The first claim is explained in the table in the Remark \ref{rem-functorialitiesperfandcoh} as the map $\mathfrak{i}:X_0\to X$ is a lci closed immersion. For t-exactness see \cite[2.5.1.1]{lurie-sag}. Conservativity may be check on the hearts by t-exactness, but on the hears, $\mathfrak{i}_*$ induces the classical pushfoward  on the truncations $t(X_0)\to t(X)$, which is conservative.
\end{proof}
\end{lem}

\begin{prop}
\label{prop-almostperfect}
Let $(X,f)$ be be an $\LG$-pair over $S$. Then the inclusion
\begin{equation}
\label{eq-coherentperfectboundedcoherentperfect}
\Coh(X_0)_{\Perf(X)}\longrightarrow \mathrm{Coh}^-(X_0)_{\Perf(X)}
\end{equation}
\noindent is an equivalence.
\begin{proof}
By t-exactness and conservativity of $\mathfrak{i}_*$, cohomological boundedness of an object $E\in\Qcoh(X_0)$ can be checked after applying $\mathfrak{i}_*$. But if $\mathfrak{i}_*(E)$ is a perfect complex then it is cohomologically bounded.
\end{proof}
\end{prop}

\begin{rem}
It follows from \ref{prop-almostperfect} that the category $\Coh(X_0)_{\Perf(X)}$ fits in a pullback square of  idempotent complete $A$-linear dg-categories
\begin{equation}
\label{eq-pullbackcoherentboundedperfect}
\xymatrix{
\mathrm{Coh}^-(X_0)\ar[r]^{\mathfrak{i}_*}&\mathrm{Coh}^-(X)\\
\Coh(X_0)_{\Perf(X)}\ar[r]\ar@{^{(}->}[u]&\Perf(X)\ar@{^{(}->}[u]
}
\end{equation}
Indeed, by the t-exactness and conservativity in \ref{lemma-inclusiont-exact}, the diagram
\begin{equation}
\label{eq-pullbackcoherentboundedcoherent}
\xymatrix{
\mathrm{Coh}^-(X_0)\ar[r]^{\mathfrak{i}_*}&\mathrm{Coh}^-(X)\\
\Coh(X_0)\ar[r]\ar@{^{(}->}[u]&\Coh(X)\ar@{^{(}->}[u]
}
\end{equation}
\noindent is cartesian. Combining  the equivalence \ref{eq-coherentperfectboundedcoherentperfect} with the fact the diagram
\begin{equation}
\label{eq-pullbackcoherentboundedperfect757579}
\xymatrix{
\mathrm{Coh}^-(X_0)\ar[r]^{\mathfrak{i}_*}&\mathrm{Coh}^-(X)\\
\mathrm{Coh}^-(X_0)_{\Perf(X)}\ar[r]\ar@{^{(}->}[u]&\Perf(X)\ar@{^{(}->}[u]
}
\end{equation}
\noindent  is cartesian, allows us to conclude.
\end{rem}

\begin{rem}
\label{rem-lcipreservesboundedcoherent}
The construction $(X,f)\mapsto X_0$ can be presented as an $\s$-functor. We leave this as an easy exercise to the reader. Moreover, for any map of LG-pairs $u:(X,f)\to (Y,g)$ there is a well-defined pullback functor 
\begin{equation}
\label{eq-functorialityrelativielyperfectinfty}
\Coh(Y_0)_{\Perf(Y)}\to \Coh(X_0)_{\Perf(X)}
\end{equation}
Notice that the pullback map $\Coh(Y_0)\to \Coh(X_0)$ is not necessarily defined as one would need the map $X_0\to Y_0$ to be of finite Tor-amplitude. What is true in general (following the table above) is that $\mathrm{Coh}^-(Y_0)\to \mathrm{Coh}^-(X_0)$ is defined 
and this is enough to show that the restriction (\ref{eq-functorialityrelativielyperfectinfty}) is always well-defined via the equivalence \ref{eq-coherentperfectboundedcoherentperfect}. Indeed,  the proper base change formula  \cite[6.3.4.1]{lurie-sag} applied to the cartesian diagram\footnote{The diagram is cartesian by definition of morphism of LG-pairs.}
\begin{equation}
\xymatrix{
\ar[d]^{u_0} X_0\ar[r]^{\mathfrak{i}}&\ar[d]^u X\\
Y_0\ar[r] & Y
}
\end{equation}
\noindent together with the fact the pullback of perfect complexes is always perfect, tells us that $\mathrm{Coh}^-(Y_0)\to \mathrm{Coh}^-(X_0)$ restricts to a functor $\mathrm{Coh}^-(Y_0)_{\Perf(Y)}\to \mathrm{Coh}^-(X_0)_{\Perf(X)}$. \personal{if $E\in \Coh(Y_0)_{\Perf(Y)}$ then the direct image $i_*(u_0^*(E))$ in $X$ is perfect and therefore of finite Tor-amplitude. Moreover, as $X$ is an underived scheme, finite Tor-amplitude implies finite cohomological amplitude. As the map $i:X_0\to X$ is a closed immersion, $i_*: \Qcoh(X_0)\to \Qcoh(X)$ is t-exact  and therefore conservative as the induced functor $i_\ast^\heartsuit: \Qcoh(X_0)^\heartsuit \to \Qcoh(X)^\heartsuit$ is the classical pushforward functor on the truncations $t(X_0)\to t(X)$ and therefore conservative. It follows that $u_0^*(E)$ is also of finite cohomological amplitude. As we know that $u_0^*(E)$ is also in $\mathrm{Coh}^-(X_0)$ we conclude the proof}
The construction $(X,f)\mapsto \Coh(X_0)_{\Perf(X)}$ can easily be written as part of an $\s$-functor. 
\end{rem}

\vspace{0.5cm}

\begin{cor}
\label{cor-h-descent}
The assignment $(X,f)\mapsto \Coh(X_0)_{\Perf(X)}$ has descent with respect to $h$-Cech covers for the h-topology of Voevodsky (see \cite{1402.3204}). That is, for any morphism $(Y,g)\to (X,f)$ such that $Y\to X$ is an h-covering, the pullback
$$
\Coh(X_0)_{\Perf(X)}\to \mathrm{Tot}(\Coh(Y_0)_\bullet)_{\Perf(Y_\bullet)})
$$
\noindent is an equivalence in $\dgcat_A^\mathrm{idem}$. Here $Y_\bullet$ denotes the \v{C}ech nerve of $u:Y\to X$ and $(Y_0)_\bullet$ denotes the \v{C}ech nerve of $u_0:Y_0\to X_0$ (both formed in the setting of derived schemes).
\begin{proof}
This follows from \cite[Thm 4.12]{1402.3204} as  both almost perfect complexes and perfect complexes satisfy $h$-descent for derived \v{C}ech.  covers. Alternatively, notice that $Y_0\to X_0$ is an h-cover with $X_0$ eventually coconnective, then use \cite[5.6.6.1]{lurie-sag}
\end{proof}
\end{cor}

\begin{rem}\label{rem-indcompletioncoherentperfect}Let us remark that the Ind-completion $\mathrm{Ind}(\Coh(X_0)_{\Perf(X)})$ embeds fully faithfully inside the presentable $\infty$-category $\mathrm{IndCoh}(X_0)_{\mathrm{Qcoh}(X)}$ obtained via the pullback of presentable $A$-linear dg-categories
\begin{equation}
\label{eq-pullbackIndcoherentquasicoherent}
\xymatrix{
\mathrm{IndCoh}(X_0)\ar[r]^{\mathfrak{i}_*}&\mathrm{IndCoh}(X)\\
\mathrm{IndCoh}(X_0)_{\mathrm{Qcoh}(X)}\ar[r]\ar@{^{(}->}[u]^{\theta}&\mathrm{Qcoh}(X)\ar@{^{(}->}[u]^{\phi}
}
\end{equation}and the inclusion $\mathrm{Ind}(\Coh(X_0)_{\Perf(X)})\subseteq \mathrm{IndCoh}(X_0)_{\mathrm{Qcoh}(X)} $ is closed under filtered colimits.
Let us remark first that the inclusion $\Perf(X)\subseteq \Coh(X)$ is fully-faithful, so is the inclusion $\phi$ after Ind-completion $\Qcoh(X)\subseteq \mathrm{IndCoh}(X)$ and therefore, so is the map $\theta$ by definition of pullbacks in $\Prl$ \cite[5.5.3.13]{lurie-htt} and the definition of mapping spaces in a pullback. Moreover, by the description of colimits in a pullback \cite[5.4.5.5]{lurie-htt}, $\theta$ preserves filtered colimits because the same is true for $\phi$.\\
The natural inclusions of bounded coherent inside Ind-coherent and perfect inside all quasi-coherent, give us a canonical fully faithful embedding 
\begin{equation}
\label{eq-backtowork2}
\Coh(X_0)_{\Perf(X)}\subseteq \mathrm{IndCoh}(X_0)_{\mathrm{Qcoh}(X)}\end{equation}
One can easily check using the fact that $\Coh(X_0)$ are precisely the compact object of $\mathrm{IndCoh}(X_0)$, that the image of this embeding lives in the full subcategory of the r.h.s spanned by compact objects. \personal{ Indeed, let $U:=(F,E, \alpha: \phi(E)\simeq \mathfrak{i}_*(F))$ be an object of $\mathrm{IndCoh}(X_0)_{\mathrm{Qcoh}(X)}$ living in the l.h.s. and consider a filtered system $U_j=(F_j, E_j, \alpha_j: \phi(E_j)\simeq \mathfrak{i}_*(F_j)$ in $\mathrm{IndCoh}(X_0)_{\mathrm{Qcoh}(X)}$. We want to show that the natural map 
\begin{equation}
\label{eq-backtowork}
\mathrm{colim}_j \,\Map_{\,\mathrm{IndCoh}(X_0)_{\mathrm{Qcoh}(X)}}(U, U_j)\to\Map_{\,\mathrm{IndCoh}(X_0)_{\mathrm{Qcoh}(X)}}(U, \mathrm{colim}_j \,U_j)
\end{equation}
is an equivalence.
Now, as $\theta$ is fully faithful and preserves filtered colimits, the r.h.s is equivalent to the mapping space $\Map_{\mathrm{IndCoh}(X_0)}(F, \mathrm{colim}_j \,F_i)$. By the same arguments, the l.h.s of equation (\ref{eq-backtowork}) is equivalent to $ \mathrm{colim}_i \,\Map_{\mathrm{IndCoh}(X_0)}(F,  F_i)$. Finally, as $F$ is assumed to be in $\Coh(X_0)$, we conclude that (\ref{eq-backtowork}) is indeed an equivalence. This proves the claim that (\ref{eq-backtowork2}) indeed factors through the compact objects of the r.h.s. so that after passing to Ind-completion we obtain a fully faithful map that preserves filtered colimits.}
\end{rem}

\vspace{0.5cm}

We start with an \emph{absolute} version of the definition of the derived category of singularities:

\begin{defn}
Let $Z$ be a derived scheme of finite type over $S$. The (absolute) derived category of singularities of $Z$ is the dg-quotient $\Sing(Z):=\Coh(Z)/\Perf(Z)$ taken in $\dgcat_A^{\mathrm{idem}}$.
\end{defn}

We will now consider  the derived category of singularities of an LG-pair $(X,f)$. As the derived closed immersion $\mathfrak{i}_\ast:X_0\to X$ is lci and in particular, of finite Tor-dimension, the Remark \ref{rem-functorialitiesperfandcoh} guarantees well-defined operations 
$$\mathfrak{i}_*:\Perf(X_0)\to \Perf(X) \text{  and  } \mathfrak{i}_*: \Coh(X_0)\to \Coh(X)$$
$$\mathfrak{i}^*:\Perf(X)\to \Perf(X_0) \text{  and  } \mathfrak{i}^*: \Coh(X)\to \Coh(X_0)$$
\noindent and therefore, well-defined induced operations 
$$\mathfrak{i}_*\Sing(X_0)\to \Sing(X) \text{  and  } \mathfrak{i}^*:\Sing(X)\to \Sing(X_0) $$
\noindent with $\mathfrak{i}^*$ left adjoint to $\mathfrak{i}_*$. In this paper we will use the following definition:

\begin{defn}\label{d2}The \emph{dg-category of singularities of the pair $(X,f)$} is the homotopy fiber in $\dgcat_A^{\mathrm{idem}}$
$$\Sing(X,f):= \mathrm{Ker} \, ( \mathfrak{i}_*: \Sing(X_0)\to \Sing(X)).$$
\end{defn}

\begin{rem}
\label{remark-fullyfaithfulSingrelativeSingabsolute}
The canonical dg-functor $\Sing(X,f)\to \Sing(X_0)$ is fully faithful.  Indeed, being $\Sing(X,f)$ a fiber computed in $\dgcat_A^{\mathrm{idem}}$, and as the inclusion $\dgcat_A^{\mathrm{idem}}\subseteq \dgcat_A$ commutes with limits (with left adjoint the idempotent completion), we conclude the statement from the formula of the mapping spaces in the fiber product in $\dgcat_A$ and the fact that the zero dg-category $0$ is a terminal object.
\end{rem}

\vspace{0.5cm}

\begin{prop}
\label{remark-singasaquotient}
For any $(X,f) \in \LG_{S} $ the canonical functor   
$$\Coh(X_0)_{\Perf(X)}/\Perf(X_0) \simeq \Sing(X,f)$$
\noindent is an equivalence. Here the dg-quotient on the lhs is taken in $\dgcat_A^{\mathrm{idem}}$. 
\begin{proof}
We start with the observation that as $X$ is assumed to be of finite type over $S$, it is quasi-compact and quasi-separated and in particular, $\Perf(X)$ admits a compact generator \cite{bondal-vandenbergh}. Therefore $X$ is perfect (in the sense of \cite{MR2669705}) and it is then a consequence of \cite[Prop. 1.18]{MR3281141} that the exact sequence of idempotent complete dg-categories 
\begin{equation}
\xymatrix{
\Perf(X)\ar@{^{(}->}[r]\ar[d]& \Coh(X)\ar[d]\\
\ast\ar[r]& \Sing(X)
}
\end{equation}
\noindent is also a pullback in $\dgcat_A^{\mathrm{idem}}$. This cartesian diagram together with the cartesian diagram (\ref{eq-pullbackcoherentboundedperfect}) fit together in a commutative cube 
\begin{equation}
\xymatrix{
\Sing(X_0) \ar[rr]&&\Sing(X)&\\
&\Sing(X,f)\ar[ul]\ar[rr]&&0\ar[ul]\\
\Coh(X_0)\ar[uu]\ar[rr]^-(0.4){\mathfrak{i}_*}&&\Coh(X)\ar[uu]&\\
&\Coh(X_0)_{\Perf(X)}\ar@{_{(}->}[ul]\ar[rr]\ar[uu]&&\Perf(X)\ar[uu]\ar@{_{(}->}[ul]\\
}
\end{equation}
\noindent where the right, bottom and upper faces are cartesian. In particular, it follows that the face on the left is cartesian and again by \cite[Prop. 1.18]{MR3281141} applied to $\Perf(X_0)\to \Coh(X_0)\to \Sing(X_0)$\footnote{Notice again that $X_0$ is also a quasi-compact and quasi-separated derived scheme so $\Perf(X_0)$ has a compact generator - see \cite{toen-azumaya}}, combined with  the fact the face on the left is now known to be cartesian, gives us two cartesian squares
\begin{equation}
\xymatrix{
\Perf(X_0)\ar@{^{(}->}[r]\ar[d]& \Coh(X_0)_{\Perf(X)} \ar@{^{(}->}[r]\ar[d]& \Coh(X_0) \ar[d]\\
\ast\ar[r]& \Sing(X,f) \ar[r]& \Sing(X_0)
}
\end{equation}
\noindent where the vertical middle arrow is essentially surjective (being the pullback of $\Coh(X_0)\to \Sing(X_0)$ which is essentially surjective). This shows that the canonical map induced by the universal property of the quotient
\begin{equation}
\label{eq-comparisongdefinitionsquotientfibersingularities}
\Coh(X_0)_{\Perf(X)}/\Perf(X_0)\to \Sing(X,f)
\end{equation}
\noindent is essentially surjective. It remains to check it is fully faithful. For that purpose we use the commutativity of the diagram
\begin{equation}
\xymatrix{
\Coh(X_0)_{\Perf(X)}/\Perf(X_0)\ar[r]\ar[d]_{(\ref{eq-comparisongdefinitionsquotientfibersingularities})}&\Sing(X_0)\\
\Sing(X,f)\ar[ru]& \\
}
\end{equation}
\noindent and explain that both maps to $\Sing(X_0)$ are fully faithful, thus deducing the fully faithfulness of (\ref{eq-comparisongdefinitionsquotientfibersingularities}). The fact that the diagonal arrow is fully faithful has been explained in the Remark \ref{remark-fullyfaithfulSingrelativeSingabsolute}.  It remains to show that the quotient map $\Coh(X_0)_{\Perf(X)}/\Perf(X_0)\to \Sing(X_0)=\Coh(X_0)/\Perf(X_0)$ is fully-faithful. This is true as the inclusion $\Coh(X_0)_{\Perf(X)}\to \Coh(X_0)$ is fully faithful and the map induced in the quotient corresponds to a quotient by a common subcategory $\Perf(X_0)$ with a compact generator.\footnote{For the reader's convenience we explain the argument: as both categories are stable and the functors are exact it is enough to explain fully faithfulness at the level of the classical homotopy categories. This can be done using the description of Hom-sets in terms of zig-zags in the classical Gabriel-Zisman localization. The reader can check \cite[Lemma 4.7.1]{0806.1324} using the fact that the class of morphisms being inverted consists exactly those maps $E\to F$ whose cofiber is perfect.}.
\end{proof}
\end{prop}

The description of $\Sing(X,f)$ given in Prop. \ref{remark-singasaquotient} will be recurrent in this paper.

\vspace{0.5cm}

\begin{rem}
\label{remark-differentdefinitionsrelativesingularities}
Following \cite{efimov2015coherent}, one could also define the relative derived category of singularities with respect to $X_0 \to X$, $\Sing(X_0/X)$, as the dg-quotient of $\Sing(X_0)$ by the image of $\mathfrak{i}^*$ taken in $\dgcat_A^{\mathrm{idem}}$. This differs from our Definition \ref{d2} (as explained in \cite[Remark 6.9]{MR3007084}). Nevertheless, one can understand both choices of definition as variations of the situation when $X$ is regular, where both agree with $\Sing(X_0)$. Our choice has the advantage of being always equivalent to matrix factorizations of projective modules (as it is proven by \cite[Proof of Theorem 2.7, p.47]{efimov2015coherent} and we shall revisit it in Section \ref{section-comparisonMF-SING}), contrary to the one of \cite{efimov2015coherent} where one needs to use coherent matrix factorizations. 
\end{rem}

\vspace{0.5cm}
\subsubsection{} Throughout this section by default we work under the Context \ref{context111}. For some results we can actually drop the hypothesis that $A$ is local. This hypothesis will only be necessary in the construction of a strict model for $\Coh_{\Perf}$. Our goal is now to exhibit the construction of the derived category of singularities of an LG-model as lax symmetric monoidal $\s$-functor
\begin{equation}\label{eq-laxSingversion}
\Sing^{\otimes} : \LG_S^{\mathrm{op}, \boxplus} \longrightarrow \dgcat_A^{\otimes},
\end{equation}
In what follows we will first construct $\Sing$ as an $\s$-functor defined on affine LG-pairs.  Our strategy will be to build a \emph{strict model} for $\Coh(X_0)_{\Perf(X)}$ (see below) and construct the functorialities in this strict setting, transferring them later to the homotopical setting via the localization functor of Section \ref{notations-dg-categories}.

\medskip
\begin{rem}
The reader should be aware that the construction of $\Sing$ as an $\s$-functor can be done using only $\s$-categorical methods, without any rectification step, as suggested in the Remark \ref{rem-lcipreservesboundedcoherent}.  Note however that the comparison with the construction of matrix factorizations requires some steps with strict dg-categories, as our initial definition of $\MF$ (Construction \ref{construction-MFasfunctor}) was indeed given in this setting.\\
\end{rem}

\begin{cons} (Strict model for the derived intersection $X_0$.)  
Let $(X=\Spec B,f) \in \LG_{S}^{\mathrm{aff}}$, corresponding to $f\in B$ for 
$B$ a flat and finitely presented $A$-algebra. We consider $K(B,f)$, the Koszul algebra associated to the element
$f \in B$. It is the commutative $B$-dg-algebra whose underlying
complex is $\xymatrix{B \ar[r]^-{f} & B}$, with the standard multiplicative
structure \textcolor{black}{where the elements of degree -1 square to zero}. We have maps $B \rightarrow K(B,f) \rightarrow B/(f)$. 
When $f$ is not a zero divisor, 
these maps make $K(B,f)$ into a cofibrant model for $B/(f)$ as
a commutative $B$-dg-algebra (i.e. the diagram above is a 
factorization of $B \rightarrow B/(f)$ as a cofibration followed
by a trivial fibration). More generally, even if $f$ is a zero divisor, $K(B,f)$ is always a cofibrant
commutative $B$-dg-algebra which is an algebraic model
for the derived scheme $X_0$ of zeros of $f$.
\end{cons}
\medskip

\begin{ex}
\label{ex-selfintersectionzero}
Let $B=A$ and $f=0$. Then $S_0:= S\times^\mathrm{h}_{\mathbb{A}^1_S} S$ is the derived self-intersection of zero inside $\mathbb{A}^1_S$. This is explicitly given by the commutative differential graded algebra $K(A,0)=A[\epsilon]$ with $\epsilon$ a generator in cohomological degree $-1$ with $\epsilon^2=0$, with underlying complex
\begin{equation}
\xymatrix{
0\ar[r]& A.\epsilon \ar[r]^0 & A \ar[r]&0
}
\end{equation}
\end{ex}

\medskip

\begin{rem}
\label{remark-explicitmodelsforCohandPerf}
This explicit model for the derived intersection gives us explicit models for $\Perf$, $\Coh$,  and $\Qcoh$. For instance, there is
a canonical equivalence of $A$-dg-categories between the dg-category $\Qcoh(X_0)$ of quasi-coherent complexes on $X_0$, and the $A$-dg-category of cofibrant $K(B,f)$-dg-modules, which we will denote as $\widehat{K(B,f)}$. The full subcategory $\Coh (X_0) \subset \Qcoh (X_0)$ (resp. $\Perf (X_0)$ ) identifies with the full dg-subcategory of $\widehat{K(B,f)}$ spanned by those complexes which are of bounded cohomological amplitude and with coherent cohomology (resp. the full sub dg-category of $\widehat{K(B,f)}$ spanned by cofibrant dg-modules which are homotopically finitely presented). A priori, the functor $\mathfrak{i}_*$ can be described as
$$
\xymatrix{
\widehat{K(B,f)}\ar[rr]^-{\mathfrak{i}_*:=Q_B\circ \mathrm{Forget}}&& \widehat{B}
}
$$
where $Q_B$ is a cofibrant replacement functor in $B$-dg-modules and $\mathrm{Forget}$ is the restriction of scalars along $B\to K(B,f)$. But as $K(B,f)$ is already cofibrant over $B$, any cofibrant $K(B,f)$-dg-module will also be cofibrant over $B$. Thus $Q_B$ is not necessary.
\end{rem}

\vspace{0.7cm}

We now discuss a \emph{strict model} for $\Coh(X_0)_{\Perf(X)}$, for $X=\Spec B$.

\begin{cons}
We consider the full sub dg-category $\mathrm{Coh}^{\mathrm{s}}(B,f)$
of the strict dg-category of (all) $K(B,f)$-dg-modules, spanned by those whose image along the restriction of scalars along the structure map $B\to K(B,f)$
$$
K(B,f)-\mathrm{dgMod}_A\to B-\mathrm{dgMod}
$$
 are strictly perfect as complexes
of $B$-modules (i.e. strictly bounded and degreewise projective $B$-modules of finite type). 
Notice that as  $X=\Spec(B)$ is an affine scheme, the sub dg-category $\Perf(X)\subseteq\widehat{B} $ is equivalent to its full sub-dg-category spanned by strict perfect complexes (see \cite[2.4.1]{thomasonalgebraic}).
Note also
 that we do not make the assumption that objects in 
$\mathrm{Coh}^{\mathrm{s}}(B,f)$ are cofibrant as $K(B,f)$-dg-modules, 
so there is no fully faithful embedding from
$\mathrm{Coh}^{\mathrm{s}}(B,f)$ to the dg-category 
$K(B,f)$-$\mathrm{dgMod}_A^\mathrm{cof}=\Qcoh(X_0)$ of cofibrant $K(B,f)$-dg-modules.
\end{cons}

\begin{rem}
\label{remark-strictcoherentislocallyflatdgcat}
More explicitly, an object in $\mathrm{Coh}^{\mathrm{s}}(B,f)$ is the data of a strictly bounded complex $E$ of projective $B$-modules of finite type, 
together with a morphism of graded modules \textcolor{black}{$h : E \rightarrow E[1]$ of degree $1$, with $h^2=0$}
satisfying $[d,h]=dh+hd=f$. In fact, given a $B$-dg-module $E$, the datum of a $K(B,f)$-dg-module structure on $E$, restricting to the given $B$-dg-module structure via the canonical map $B \to K(B,f)$, amounts to a pair $(m_0, m_1)$ of morphisms $m_{\alpha}: E\to E[-\alpha]$ of graded $B$-modules, where $m_0$ is forced to be the identity by the fact that the $B$-dg-module structure is assigned, and $h:=m_1$ is subject to the condition $dh+hd=f$.
Note also that, as a strict $A$-dg-category, $\mathrm{Coh}^{\mathrm{s}}(B,f)$ is locally flat. This follows because by assumption $B$ is flat over $A$ and $A$ is a regular local ring as required by the Context \ref{context111}.
\end{rem}

The dg-category $\mathrm{Coh}^{\mathrm{s}}(B,f)$  is a strict model for 
the dg-category $\Coh(X_0)_{\Perf(X)}$, as stated by the following lemma.

\begin{lem}\label{l1}
Let $\mathrm{Coh}^{\mathrm{s, acy}}(B,f)\subset \mathrm{Coh}^{\mathrm{s}}(B,f)$ be the 
full sub-dg-category consisting of $K(B,f)$-dg-modules
which are acyclic as complexes of $B$-modules. Then, the
cofibrant replacement dg-functor induces an equivalence of dg-categories
\begin{equation}
\label{eq-strictmodelcoherentperfectswedenmeatballssuck}
\mathrm{Coh}^{\mathrm{s}}(B,f)[\mathrm{q.iso}^{-1}]_{\mathrm{dg}}\simeq \mathrm{Coh}^{\mathrm{s}}(B,f)/\mathrm{Coh}^{\mathrm{s, acy}}(B,f) \simeq \Coh(X_0)_{\Perf(X)}
\end{equation}
In particular, we have a natural equivalence of dg-categories
$$\mathrm{Coh}^{\mathrm{s}}(B,f)/\mathrm{Perf}^{\mathrm{s}}(B,f) \simeq \Coh(X_0)_{\Perf(X)}/\Perf(X_0) =\Sing(X,f),$$
where $\mathrm{Perf}^{\mathrm{s}}(B,f)$ is by definition the full sub-dg-category of $\mathrm{Coh}^{\mathrm{s}}(B,f)$ consisting 
of objects which are perfect as 
$K(B,f)$-dg-modules. 
\begin{proof}
The category of $K(B,f)$-dg-modules admits a combinatorial model structure inherited by the one from complexes of $B$-modules. Therefore, it admits a functorial  cofibrant replacement
$$
Q:K(B,f)-\mathrm{dgMod}_A\to K(B,f)-\mathrm{dgMod}_A^\mathrm{cof}
$$
which is not a priori a dg-functor. In our case we are interested in applying this to the inclusion $\mathrm{Coh}^{\mathrm{s}}(B,f)\subseteq K(B,f)-\mathrm{dgMod}_A$ and it happens that for objects  $E\in \mathrm{Coh}^{\mathrm{s}}(B,f)$ we can model $Q$ by a dg-functor as follows: as $E$ is strictly perfect over $B$, in particular $E$ is cofibrant over $B$. Therefore, by definition $E\otimes_B K(B,f)$ is a cofibrant $K(B,f)$-dg-module. So are the powers $E\otimes_B B^n\otimes_B K(B,f)$. This gives us a resolution of $E\simeq E\otimes_{K(B,f)}K(B,f)$ by a simplicial diagram. Extracting its totalization we obtain a cofibrant resolution of $E$ in a functorial way.
This way we get a strict cofibrant-replacement dg-functor
\begin{equation}
\label{eq-cofibrantreplacement}
Q:\mathrm{Coh}^{\mathrm{s}}(B,f)\to \widehat{K(B,f)}
\end{equation}
which by definition, sends weak-equivalences to equivalences. By the universal property of the dg-localization we have a factorization in $\dgcat_A$
\begin{equation}
\label{eq-cofibrantreplacement2}
Q:\mathrm{Coh}^{\mathrm{s}}(B,f)[\mathrm{q.iso}^{-1}]_{\mathrm{dg}}\to \widehat{K(B,f)}
\end{equation}
Notice also that by the universal property of the quotient, this dg-localization is equivalent in $\dgcat_A$ to $\mathrm{Coh}^{\mathrm{s}}(B,f)/\mathrm{Coh}^{\mathrm{s, acy}}(B,f)$ and the map (\ref{eq-cofibrantreplacement2}) is the one induced by the fact that $Q$
 sends the full subcategory $\mathrm{Coh}^{\mathrm{s, acy}}(B,f)$ to zero.\\

We show that the dg-functor (\ref{eq-cofibrantreplacement}) is fully faithful with essential image given by $\Coh(K(B,f))_{\Perf(B)}$. More precisely, we show that:

\begin{enumerate}
\item The functor (\ref{eq-cofibrantreplacement2}) factors through the full-subcategory $\Coh(K(B,f))_{\Perf(B)}$, where it is essentially surjective;
\item (\ref{eq-cofibrantreplacement2}) is fully-faithful.
\end{enumerate}

Let us start with (1). Of course, as an object $E\in \mathrm{Coh}^{\mathrm{s}}(B,f)$ is strictly bounded, any cofibrant replacement will remain cohomologicaly bounded. The cohomology groups of $E$ carry a natural structure of $\pi_0(K(B,f))=B/f$-module. Moreover, being $E$ levelwise made of projective B-modules of finite type, these same cohomology groups are coherent when seen as $B$-modules via composition with the surjective map $B\to B/f$ and therefore are coherent as $\pi_0(K(B,f))=B/f$-modules. Therefore, its cofibrant replacement $Q(E)$ is in $\Coh(K(B,f))$ \footnote{As $A$ is Noetherian and $B$ is of finite type over $A$, it is of finite presentation as an $A$-algebra. Then it is also Noetherian and therefore coherent modules are the same as finitely generated modules.}. In fact, $Q(E)$ lives in the full sub-dg-category $\Coh(K(B,f))_{\Perf(B)}$. Indeed, notice that by definition of $\mathrm{Coh}^{\mathrm{s}}(B,f)$ the image of $E$ under composition with $B\to K(B,f)$, which we will denote as $\mathrm{Forget}(E)$, is a strict perfect complex and therefore, is perfect. As the forgetful functor along $B\to K(B,f)$ preserves all weak-equivalences of dg-modules, $\mathrm{Forget}(Q(E))$ is weak-equivalent to $\mathrm{Forget}(E)$. Finally, by definition of $\mathfrak{i}_*:= Q_B\circ \mathrm{Forget}$ (see the Remark \ref{remark-explicitmodelsforCohandPerf} for notations) we find that $i_*(E)$ is quasi-isomorphic to $\mathrm{Forget}(E)$ and therefore is perfect.

To show that (\ref{eq-cofibrantreplacement2})  is essentially surjective on $\Coh(K(B,f))_{\Perf(B)}$ we notice first that as $X$ is affine, the inclusion of strictly perfect complexes over $B$, $\Perf^s(B)$, inside $\Perf(B)$ is an equivalence. In this case so is the inclusion $\Coh(K(B,f))_{\Perf^s(B)}\subseteq \Coh(K(B,f))_{\Perf(B)}$. Suppose $M\in \Coh(K(B,f))_{\Perf^s(B)}$ is in cohomological degree $0$, a $B/f$-module of finite type. In this case, take any simplicial resolution of $M$ by free $K(B,f)$-dg-modules $E\to M$. This might be unbounded because $M$ itself is not strictly perfect over $K(B,f)$. The restriction of scalars of $E$ to $B$ is cofibrant and is degreewise projective over $B$ as $K(B,f)$ itself is strictly perfect over $B$ and $M$ is by hypothesis strictly perfect over $B$. One can now truncate the resolution $\tau_{\leq b+1}E$ for $b$ the tor-amplitude of $M$ over $B$. This new resolution is now strictly bounded as $K(B,f)$-dg-module and remains quasi-isomorphic to $M$.

Let us now show (2). As $\mathrm{Coh}^{\mathrm{s}}(B,f)$ has a canonical triangulated structure (coming from the strict dg-enrichment) to show that the map  (\ref{eq-cofibrantreplacement2}) is fully faithful it is enough to show that it is fully faithful on the homotopy categories because of the triangulated nature of the dg-localization. In this case it is enough so show that for any $E\in \mathrm{Coh}^{\mathrm{s}}(B,f)$ and for any quasi-isomorphism $P \to E$ with $P$ a $K(B,f)$-dg-module, it is possible to find an object $P'\in \mathrm{Coh}^{\mathrm{s}}(B,f)$ and a second quasi-isomorphism $P'\to P\to E$. But this follows using free resolutions like in (1) above.
\end{proof}
\end{lem}

\vspace{0.1cm}
\begin{cons}
\label{construction-functorialassignementrelativeperfect}
The construction $(B,f) \mapsto \mathrm{Coh}^{\mathrm{s}}(B,f)$
is functorial in the pair $(B,f)$: if $B \rightarrow B'$ is a morphism
sending $f \in B$ to $f' \in B'$, the base change along $K(B,f)\to K(B',f')$ is induced by the base change $B'\otimes_{B}-$ given by 
an $A$-linear dg-functor
$$B'\otimes_{B}- : K(B,f)-\mathrm{dgMod}_A \longrightarrow K(B',f')-\mathrm{dgMod}_A $$
This restricts to an $A$-dg-functor
\begin{equation}
\label{eq-strictbasechange}
B'\otimes_{B}- : \mathrm{Coh}^{\mathrm{s}}(B,f) \longrightarrow \mathrm{Coh}^{\mathrm{s}}(B',f')
\end{equation}
Indeed, the base change of a strictly bounded complex remains strictly bounded and if $E$ is a $K(B,f)$-dg-module whose levels $E_i$ are projective $B$-modules of finite type, then the base change $E_i\otimes_B B'$ are $B'$-modules of finite type and again projective.
As explained in the Remark \ref{remark-strictcoherentislocallyflatdgcat} (working under the Context \ref{context111}) we get this way a pseudo-functor
\begin{equation}
\label{eq: strictcoherentasastrictfunctor}
\mathrm{Coh}^{\mathrm{s}} : \LG_{S}^{\mathrm{aff,op}} \longrightarrow \mathrm{dgCat}^{\mathrm{strict}, \mathrm{loc-flat}}_A
\end{equation}
which sends $(\Spec\, B,f)$ to $\mathrm{Coh}^{\mathrm{s}}(B,f)$.\\
One can now use  (\ref{eq: strictcoherentasastrictfunctor}) combined with the Lemma \ref{l1} to exhibit the assignment $(X,f)\mapsto \Coh(X_0)_{\Perf(X)}$ as an $\infty$-functor
\begin{equation}
\label{eq: infinityfunctorialrelativeperfect}
\mathrm{Coh}^b(-)_{\Perf(-)}: \LG_{S}^{\mathrm{aff,op}} \longrightarrow \dgcat_A^{\mathrm{idem}}
\end{equation}
For this purpose we remark that the base change maps (\ref{eq-strictbasechange}) preserve quasi-isomorphisms. Indeed, if $E\to F$ is a quasi-isomorphism between objects in $\mathrm{Coh}^{\mathrm{s}}(B,f)$ then $E\to F$ is a quasi-isomorphism between the underlying strictly perfect $B$-dg-modules. As strictly perfect $B$-complexes are cofibrant as $B$-dg-modules, and every $B$-dg-module is fibrant, $E\to F$ is an homotopy equivalence so that the base change $E\otimes_B B'\to F\otimes_B B'$ remains an homotopy equivalence and therefore a quasi-isomorphism
(alternatively, use  Brown's Lemma \cite[1.1.12]{hovey1999model}). In this case the functoriality (\ref{eq: strictcoherentasastrictfunctor}) can be refined
\begin{equation}
\label{eq: strictcoherentasastrictfunctorpairs}
\mathrm{Coh}^{\mathrm{s}} : \LG_{S}^{\mathrm{aff,op}} \longrightarrow \mathrm{PairsdgCat}^\mathrm{strict, loc-flat}_A
\end{equation}
where $\mathrm{PairsdgCat}^\mathrm{strict}_A$ is the 1-category whose objects are pairs $(T,W)$ with $T$ a strict small $A$-dg-category and $W$ a class of morphisms in $T$. This encodes the fact that weak-equivalences are stable under base change and sends a pair $(B,f)$ to the pair $(\mathrm{Coh}^{\mathrm{s}}(B,f), W_{q.iso})$ with $W_{q.iso}$ the class of quasi-isomorphisms. In the 1-category $\mathrm{PairsdgCat}^{\mathrm{strict, loc-flat}}_A$ we have a natural notion of weak-equivalence, namely, those maps of pairs $(T,W)\to (T',W')$ whose underlying strict dg-functor $T\to T'$ is a Dwyer-Kan equivalence of dg-categories. This produces a map between the $\infty$-categorical localizations
\begin{equation}
\label{eq-localizationofpairsdgcats}
\mathrm{loc}_{\mathrm{dg}}:\mathrm{N}(\mathrm{PairsdgCat}^{\mathrm{strict, loc-flat}}_A)[W_{DK}^{-1}]\to \dgcat_A\simeq \mathrm{N}(\mathrm{dgCat}^{\mathrm{strict, loc-flat}}_A)[W_{DK}^{-1}]
\end{equation}
\noindent sending a pair $(T,W)$ to its dg-localization $T[W^{-1}]_{\mathrm{dg}}$ in $\dgcat_A$. To give a concrete description of  this $\infty$-functor, we remark the existence of another 1-functor
$$
\mathrm{dgCat}^{\mathrm{strict, loc-flat}}_A \to\mathrm{PairsdgCat}^{\mathrm{strict, loc-flat}}_A
$$
sending a strict small $A$-dg-category $T$ to the pair $(T, W_T)$ where $W_T$ is the class of equivalences in $T$. By definition, this functor sends weak-equivalences of dg-categories to weak-equivalences of pairs  therefore induces a functor between their $\infty$-localizations
$$\dgcat_A^\otimes \simeq \mathrm{N}(\mathrm{dgCat}^{\mathrm{strict, loc-flat}}_A)[W_{DK}^{-1}]^\otimes \to \mathrm{N}(\mathrm{PairsdgCat}^{\mathrm{strict, loc-flat}}_A)[W_{DK}^{-1}]^\otimes$$
We now claim that this functor admits a left adjoint, which will be our model of (\ref{eq-localizationofpairsdgcats}). By a dual version of  \cite[Lemma 5.2.4.9]{lurie-htt} it suffices to check that for every pair $(T,S)$, the left fibration
$$
\dgcat_A\times_{\mathrm{N}(\mathrm{PairsdgCat}^{\mathrm{strict, loc-flat}}_A)[W_{DK}^{-1}]} \mathrm{N}(\mathrm{PairsdgCat}^{\mathrm{strict, loc-flat}}_A)[W_{DK}^{-1}]_{(T,S)/.} \to \dgcat_A
$$
\noindent is co-representable. But this follows because of the existence of dg-localizations - see \cite[Corollary 8.7]{Toen-homotopytheorydgcatsandderivedmoritaequivalences} \footnote{see also the higher categorical comments in \cite[Section 6.1]{robalo-thesis}.}.
\end{cons}

\vspace{0.3cm}

\begin{cons}
\label{constructionlaxmonoidalstructurestrictperfect}
First we construct a lax symmetric monoidal structure on  $\mathrm{Coh}^{\mathrm{s}}$ (\ref{eq: strictcoherentasastrictfunctor}). Given two $\LG$-pairs $(X:=\Spec(B), f)$ and $(Y:=\Spec(C),g)$ one must specify a functor
\begin{equation}
\label{eq-laxstructurestrictrelativeperfect}
\mathrm{Coh}^{\mathrm{s}}(B,f)\otimes_A \mathrm{Coh}^{\mathrm{s}}(C,g)\to \mathrm{Coh}^{\mathrm{s}}(B\otimes_A C, f\otimes 1 + 1\otimes g)
\end{equation}
\noindent verifying the conditions of lax symmetric structure. To construct (\ref{eq-laxstructurestrictrelativeperfect}) let us start by introducing some notation. For an $\LG$-pair $(X,f)$ we denote by $Z^\mathrm{h}(f)$ the derived zero locus of $f$ so that in the affine case, with $X=\Spec(B)$, we have $Z^\mathrm{h}(f)=\Spec(K(B,f))$. By construction, given two affine $\LG$-pairs as above, one obtains a commutative diagram
\begin{equation}
\label{eq-cartesianLGlax}
\xymatrix{
Z^\mathrm{h}(f)\times_S Z^\mathrm{h}(g)\ar[d]\ar[r]^-{(1)}& Z^\mathrm{h}(f\boxplus g:=+\circ (f,g))\ar[d]\ar[r]^-{(2)}&X\times_S Y\ar[d]^{(f,g)}\\
S\ar[r]^0& \mathbb{A}^1_S \ar[r]^{(id,-id)}\ar[d]&  \mathbb{A}^1_S \times_S  \mathbb{A}^1_S \ar[d]^{+}\\
&S\ar[r]^0& \mathbb{A}^1_S 
}
\end{equation}
\noindent where each face is cartesian and all the horizontal maps are lci closed immersions (as a consequence of the same property for the zero section $S\hookrightarrow \mathbb{A}^1_S$). Moreover, we remark that the arrows (1) and (2) in the diagram can be given strict models
\begin{equation}
\label{eq-japanesehousestill1}
\xymatrix{
&B\otimes_A C\ar[dr]^{(1) \,\circ\, (2)}\ar[dl]_{(2)}&\\
K(B\otimes_A C, f\otimes1 + 1\otimes g)\ar[rr]^{(1)} && K(B,f)\otimes_A K(C,g)
}
\end{equation}
\noindent where (1) is completely determined by an element $\alpha$
of degree $-1$ in $K(B,f) \otimes_{A} K(C,g)$ satisfying
$$d(\alpha)=f\otimes 1 + 1 \otimes g \qquad \alpha^2=0.$$
We set 
\begin{equation}
\label{eq-formulaCohstrictproduct}
\alpha:= h\otimes 1 + 1 \otimes k
\end{equation}
where $h$ and $k$ are the canonical element in $K(B,f)$ and $K(C,g)$ respectively, 
of degree $-1$ with 
\begin{equation}
\label{eq-explicitactionepsilonproduct}
d(h)=f \qquad d(k)=g \qquad h^2=k^2=0.
\end{equation}
To define the lax symmetric structure (\ref{eq-laxstructurestrictrelativeperfect}) one is reduced to explain that the composition
\begin{equation}
\label{eq-laxstructurestrictrelativeperfect2}
\resizebox{1. \hsize}{!}{
\xymatrix{
\mathrm{Coh}^{\mathrm{s}}(B,f)\otimes_A \mathrm{Coh}^{\mathrm{s}}(C,g)\ar[r]^{\boxtimes}&\mathrm{Coh}^{\mathrm{s}}(K(B,f)\otimes_A K(C,g)) \ar[r]^{(1)_*} &\mathrm{Coh}^{\mathrm{s}}(B\otimes_A C, f\otimes 1 + 1\otimes g)
}
}
\end{equation}
\noindent is well-defined. $(1)_*$ is given by the forgetful functor and as such it is well-defined the level of the categories $\mathrm{Coh}^{\mathrm{s}}$: indeed, if $E$ is  strictly bounded  $K(B,f)\otimes_A K(C,g)$-dg-module whose image under  the forgetful functor $(1)\, \circ \, (2)$ is strictly perfect over $B\otimes_A C$, then by commutativity of the diagram (\ref{eq-japanesehousestill1}), $(1)_\ast(E)$ is in $\mathrm{Coh}^{\mathrm{s}}(B\otimes_A C, f\boxplus g)$. \\
It remains to provide an argument for the  box product $\boxtimes$: is defined by sending a pair $(E, F)$ to $\pi_f^\ast(E)\otimes \pi_g^\ast(F)$ with $\pi_f$ and $\pi_g$ the projections of $Z^\mathrm{h}(f)\times_S Z^\mathrm{h}(g)$ in each coordinate. Using the projection formulas and base change for affines\footnote{Notice that by definition of $\LG$-pairs, both $B$ and $C$ are flat over $A$. In particular, the derived tensor product is the usual one.}, the underlying $A$-module of $E\boxtimes F$ is just the $A$-tensor product $E\otimes_A F$. One must show that if $E$ (resp. $F$) is strictly perfect over $B$ (resp. $C$) then $E\boxtimes F$ is strictly perfect over $B\otimes_A C$. The fact that $E\boxtimes F$ is strictly bounded follows immediately from the definition of the strict tensor product of complexes and the fact both $E$ and $F$ are strictly bounded. The fact that level of the complex $E\boxtimes F$ is projective over $B\otimes_A C$ follows because each level $E^i$ (resp. $F^k$) is by assumption projective over $B$ (resp. over $C$) so that each graded piece of the tensor product $E^i\otimes_A F^k$ is projective over $B\otimes_A C$: $E^i$ (resp. $F^k$) being projective over $B$ (resp. $C$) gives us a retract via a map of $B$-modules (resp. $C$-modules) of an inclusion of $B$-modules (resp. $C$-modules) $E^i\subseteq B^{\oplus_l}$ for some $l$ (resp. $F^k\subseteq C^{\oplus_s}$). Via base change we obtain the graded piece $E^i\otimes_A F^k$ as a retract of $(B\otimes_A C)^{\oplus_{l+s}}$ via a map of $B\otimes_A C$-modules, for some $l,s$. This proves the claim. To conclude, we define the lax unit via the map
\begin{equation}
\label{eq-laxunitCohstrict}
A \longrightarrow \mathrm{Coh}^{\mathrm{s}}(K(A,0)),
\end{equation}
\noindent sending the unique point to $A$ itself (with its trivial structure
of $K(A,0)$-dg-module). The construction (\ref{eq-laxstructurestrictrelativeperfect}) is clearly symmetric and associative and this concludes the construction of a lax symmetric monoidal enhancement of (\ref{eq: strictcoherentasastrictfunctor})
\begin{equation}
\label{eq-locallyflatlaxmonoidalversionCohstrict}
\mathrm{Coh}^{\mathrm{s},\boxplus} : \LG_{S}^{\mathrm{aff,op}, \boxplus} \longrightarrow \mathrm{dgCat}^{\mathrm{strict}, \mathrm{loc-flat}, \otimes}_A
\end{equation}
\end{cons}
\medskip

\begin{rem}
\label{remsymmetristrictstructure}
In particular, we obtain a symmetric monoidal structure on $\mathrm{Coh}^{\mathrm{s}}(A,0)$.
\end{rem}

 \medskip
\begin{cons}
\label{construction-laxmonoidalstructurerelativeperfect}
One now proceeds as in the Construction \ref{construction-functorialassignementrelativeperfect} to obtain a lax symmetric monoidal structure on (\ref{eq: infinityfunctorialrelativeperfect}): one remarks that   the category of pairs $\mathrm{PairsdgCat}^\mathrm{strict}_A$ introduced in the Construction \ref{construction-functorialassignementrelativeperfect} comes naturally equipped with a tensor structure: if $(T,W)$ and $(T', W')$ are two pairs, the pair $(T, W)\otimes (T', W')$ is defined by $(T\otimes T', W\otimes W')$. 
The lax structure of (\ref{eq-locallyflatlaxmonoidalversionCohstrict}) can be lifted to pairs
\begin{equation}
\label{eq-laxmonoidalstrictCohPairs}
\mathrm{Coh}^{\mathrm{s}, \boxplus} : \LG_{S}^{\mathrm{aff, op}, \boxplus} \longrightarrow \mathrm{PairsdgCat}^{\mathrm{strict, loc-flat}, \otimes}_A
\end{equation}
Indeed, one checks that the composition (\ref{eq-laxstructurestrictrelativeperfect2}) sends the product of quasi-isomorphisms to a quasi-isomorphism. For $(1)_\ast$ this is by definition. For $\boxplus$ this follows because it is explicitly computed as a tensor product over $A$ and strictly perfect complexes are, as we have seen before, flat over $A$.
To conclude, it follows from the definition of locally-flat dg-categories that the tensor structure in $\mathrm{PairsdgCat}^{\mathrm{strict, loc-flat}}_A$ is compatible with weak-equivalences in each variable so that the localization functor along Dwyer-Kan equivalences of pairs is a monoidal $\infty$-functor
\begin{equation}
\label{eq-monoidallocalizationofpairs}
\mathrm{N}(\mathrm{PairsdgCat}^{\mathrm{strict, loc-flat}, \otimes}_A) \longrightarrow  \mathrm{N}(\mathrm{PairsdgCat}^{\mathrm{strict, loc-flat}}_A)[W_{DK}^{-1}]^\otimes
\end{equation}
 It remains to check that  (\ref{eq-localizationofpairsdgcats}) is strongly monoidal. This follows from \cite[7.3.2.12]{lurie-ha} as the required hypothesis follow from the definition of the tensor structure on pairs, together with the fact that for any two pairs $(T,S), (T', S')$ the canonical morphism
$$
(T\otimes T')[S\otimes S'^{-1}]_{\mathrm{dg}}\to  T[S^{-1}]_{\mathrm{dg}}\otimes  T'[S'^{-1}]_{\mathrm{dg}}
$$
is an equivalence in $\dgcat_A$ (this is an immediate consequence of the universal property of dg-localizations combined with the existence of internal-homs in $\dgcat_A^\otimes$).\\
 Finally, the composition of the lax monoidal $\infty$-functors (\ref{eq-laxmonoidalstrictCohPairs}), (\ref{eq-monoidallocalizationofpairs}), (\ref{eq-localizationofpairsdgcats}) and idempotent completion, combined with the result of Lemma \ref{l1}, achieve the construction of the lax monoidal structure on (\ref{eq: infinityfunctorialrelativeperfect}).
\end{cons}

\vspace{0.2cm}

\begin{rem}
\label{rem-symmonoidalstructurerelativeperfectSzero}
The lax symmetric monoidal structure the Remark (\ref{remsymmetristrictstructure}) and the Construction \ref{construction-laxmonoidalstructurerelativeperfect} produces a symmetric monoidal structure on $\Coh(K(S,0))$, which we shall denote as $\Coh(K(S,0))^{\boxplus}$. Its monoidal unit is the $K(A,0)$-dg-module $A$ in degree $0$ with zero $\epsilon$-action. Via the identification of $K(S,0)$ as a strict model for the derived tensor product $S\times_{\mathbb{A}^1_S} S$, this symmetric monoidal structure corresponds to the convolution product induced by the additive group structure on $\mathbb{A}^1_S$.
This symmetric monoidal structure has a geometric origin: in fact $S\times_{\mathbb{A}^1_S} S$ is a derived group scheme with operation induced by the additive group structure on the affine line. By unfolding the definition, given $E, F\in \Coh(K(S,0))$, $E\boxplus F$ is given by the underlying tensor  $E\otimes_A F$ equipped with an action of $K(S,0)$ via the map $K(S,0)\to K(S,0)\otimes_A K(S,0)$ of (1) in (\ref{eq-japanesehousestill1}). In the case when $A$ is a field of characteristic zero this recovers the monoidal structure described in \cite[Construction 3.1.2]{1101.5834}.\\ 
Moreover, given an LG-pair $(X,f)$, the action of $\Coh(K(S,0))^{\boxplus}$ on $\Coh(X_0)_{\Perf(X)}$ also has a geometric interpretation: indeed, the derived fiber product $X_0$ carries a canonical action of the derived group scheme $S\times_{\mathbb{A}^1_S} S$. This is obtained via the cartesian cube
\begin{equation}
\label{eq-geometricinterpretationlaxaction}
\xymatrix{
&\ar[dl]_-{pr}X_0\times_S (S\times_{\mathbb{A}^1_S} S) \ar[dd]^<<<<<{pr_{(S,0)}}\ar[rr]^-{v\,=\, \textrm{action}}&&X_0\ar[dl]^{\mathfrak{i}}\ar[dd]\\
X_0\ar[dd]\ar[rr]^(.3){\mathfrak{i}} &&X\ar[dd]^(.3){f}&\\
&S\times_{\mathbb{A}^1_S} S\ar[dl]\ar[rr]&&S\ar[dl]\\
S\ar[rr]&&\mathbb{A}^1_S&
}
\end{equation}
Let us describe this action more precisely. In the affine case this is given by the formula (\ref{eq-laxstructurestrictrelativeperfect2}). In geometric terms this is explained by the derived fiber product in the diagram (\ref{eq-geometricinterpretationlaxaction}) whose top face is the self-intersection square
\begin{equation}
\label{eq-geometricinterpretationlaxaction2}
\xymatrix{
\ar[d]^{pr} X_0\times_S (S\times_{\mathbb{A}^1_S} S) \ar[r]^-{v}&X_0\ar[d]^{\mathfrak{i}}\\
X_0\ar[r]^-{\mathfrak{i}} &X\\
}
\end{equation}
and the action of $F\in\Coh(K(A,0))$ on $M\in\Coh(X_0)_{\Perf(X)}$ is given by $F\boxplus M := v_\ast(pr^\ast(M)\otimes pr_{(S,0)}^*(F))$. In particular, by derived base change, we have 
\begin{equation}
\label{eq-explicitdescriptionaction}
K(A,0)\boxplus M\simeq \mathfrak{i}^*\mathfrak{i}_*(M)
\end{equation}
Moreover, the action of $A$ (as a $K(A,0)$-module in degree $0$ with a trivial action of $\epsilon$) is given by 
\begin{equation}
\label{eq-explicitdescriptionaction2}
A\boxplus M:= v_\ast(pr^\ast(M)\otimes pr_{(S,0)}^*(A)))\simeq M
\end{equation}
To show this last formula we remark that $A$ as a trivial $K(A,0)$-module is given by $t_\ast (A)$ where $t: S=Spec(A)\to Spec(K(A,0)= S\times_{\mathbb{A}^1_S} S$ is the inclusion of the classical truncation. Using the pullback diagram 
\begin{equation}
\xymatrix{
 X_0\simeq \ar[d]X_0\times_S (S) \ar[r]^-{Id\times t}&X_0\times_S (S\times_{\mathbb{A}_S^1}S)\ar[d]^{pr}\\
S\ar[r]^-{t} &S\times_{\mathbb{A}^1_S}S\\
}\end{equation}
\noindent and derived base-change, we get that 
$$A\boxplus M\simeq v_\ast(pr^\ast(M)\otimes pr_{(S,0)}^*t_\ast A))\simeq  v_\ast ( pr^*(M)\otimes (Id_{X_0}\times t)_\ast \mathcal{O}_{X_0})$$ 
which by the projection formula, is equivalent to
\begin{equation}
\label{eq-explicitdescriptionaction3}
v_\ast \circ (Id_{X_0}\times t)_\ast ((Id_{X_0}\times t)^\ast pr^*(M))\simeq M
\end{equation} 
\noindent We conclude that $A$, as a trivial $K(A,0)$-module, acts via the identity map on $\Coh(X_0)_{\Perf(X)}$.
\end{rem}

\vspace{0.5cm}

We now provide an explicit description of $\Coh(S\times^\mathrm{h}_{\mathbb{A}^1_S} S)_{\Perf(A)}$ with the symmetric monoidal structure of the previous remark. This is essentially the observation that for the computation performed in the proof of \cite[Prop. 3.1.4]{1101.5834} to work we don't need $A$ to be a field of characteristic zero. In fact, it works whenever $A$ is regular:

\medskip

\begin{lem}
\label{SingS0} 
\label{lemma-monoidalequivalencerelativeperfectandsing}
Let $A$ be a regular commutative ring. Then we have an equivalence in $\CAlg(\dgcat^{\mathrm{idem}}_A)$
\begin{equation} 
\label{eq-kudegree2Koszulequivalencemonoidal}
\Coh(S\times^\mathrm{h}_{\mathbb{A}^1_S} S)_{\Perf(S)}^{\boxplus}\simeq \Perf(A[u])^{\otimes_{A[u]}}
\end{equation}
\noindent where on the r.h.s we have the standard tensor product over $A[u]$ (where $u$ has degree $2$) induced by the fact $A[u]$ is naturally a commutative algebra-object in $\Mod_{\mathbb{Z}}(\Sp)^{\otimes}$.
\begin{proof}
Let us first explain the equivalence between the underlying categories. Since $A$ is regular, we have $\Coh(S\times^\mathrm{h}_{\mathbb{A}^1_S} S)_{\Perf(S)} \simeq \Coh(S\times^\mathrm{h}_{\mathbb{A}^1_S} S)$, where $S\times^\mathrm{h}_{\mathbb{A}^1_S} S$ is the derived zero locus of the zero-section $0: S \to \mathbb{A}_S^1$. Now, this derived zero-locus is the spectrum of the simplicial commutative ring $\mathrm{Sym}^s_A(A[1])$, whose normalization is the commutative differential graded ring $K(A,0)$ of Example \ref{ex-selfintersectionzero}. Therefore $\Coh(S\times^\mathrm{h}_{\mathbb{A}^1_S} S)$ is equivalent to $\Coh(K(A,0))$, i.e. to dg-modules over $K(A,0)$ which are coherent on the truncation $H^0(K(A,0))=A$. It is easy to verify that $\Coh(K(A,0))$ is generated by the $A$-dg-module $A$,  via the homotopy cofiber-sequence
\begin{equation}
\label{eq-resolutionK(A,0)}
\xymatrix{
A\ar[r]^0\ar[d]& A\ar[d]\\
0\ar[r]& K(A,0)
}
\end{equation}
so that $\mathrm{Ind}(\Coh(K(A,0)))$ is equivalent to dg-modules over $\mathbb{R}\mathrm{Hom}_{K(A,0)}(A, A)$ and $\Coh(K(A,0))$ to perfect dg-modules over $\mathbb{R}\mathrm{Hom}_{K(A,0)}(A, A)$. Now, we remark the existence of an infinite resolution
\begin{equation}
\label{eq-infiniteresolution}
\xymatrix{(\cdots \ar[r]^{\mathrm{id}} & \underbracket{A}_{-3} \ar[r]^{0} & \underbracket{A}_{-2} \ar[r]^{\mathrm{id}} & \underbracket{A}_{-1} \ar[r]^{0} & \underbracket{A}_{0}) \ar[rr]^{\mathrm{id}}  && \underbracket{A}_{0}}
\end{equation}
\noindent of $A$ as a $K(A,0)$-dg-module. This can be obtained as an homotopy colimit in $\Qcoh(K(A,0))$ induced by the multiplication by $\epsilon$ as follows: let $K(A,0)\{1\}$ denote the cofiber of $\epsilon: K(A,0)[1]\to K(A,0)$ and by induction, we construct $K(A,0)\{n+1\}$ by the cofiber
\begin{equation}
\label{equation-inductionlevelii}
\xymatrix{
\ar[d]K(A,0)[2n-1]\ar[r]& K(A,0)\{n\}\ar[d]\\
0\ar[r] & K(A,0)\{n+1\}
}
\end{equation}
and obtain the infinite resolution (\ref{eq-infiniteresolution}) as the homotopy colimit  in $\Qcoh(K(A,0))$
\begin{equation}
\label{eq-infiniteresolutionformulacolimit}
\xymatrix{\mathrm{colim}\, ( K(A,0) \ar[r]&  K(A,0)\{1\} \ar[r]&  K(A,0)\{2\} \ar[r]& \cdots)&\simeq A}
\end{equation}
Using this resolution we can directly compute 
\begin{equation}
\label{eq-explicitcomputationAdegree2}
\mathbb{R}\mathrm{Hom}_{K(A,0)}(A, A) \simeq A[u]
\end{equation}
\noindent with $\deg (u)=2$. Let us briefly describe this computation. It is clear that as $A$-modules, we get an isomorphism of complexes
\begin{equation}
\label{eq-explicitcomputationAdegree2-2}
\resizebox{1. \hsize}{!}{\xymatrix{\mathbb{R}\mathrm{Hom}_A(\cdots \ar[r]^-{\mathrm{id}} & \underbracket{A}_{-3} \ar[r]^{0} & \underbracket{A}_{-2} \ar[r]^{\mathrm{id}} & \underbracket{A}_{-1} \ar[r]^{0} & \underbracket{A}_{0}, \underbracket{A}_{0}) \ar[r]^{\sim}  & (\underbracket{A}_{0}\ar[r]^{0} & \underbracket{A}_{1}\ar[r]^{\mathrm{id}}&\underbracket{A}_{2} \ar[r]^{0} & \underbracket{A}_{3} \ar[r]^{\mathrm{id}} &\cdots }}
\end{equation}
\noindent where each degree $\underbracket{A}_{i}$ on the r.h.s is a disguise of $\mathrm{Hom}_A( \underbracket{A}_{-i}, A)$.
 The extra demand for a $K(A,0)$-linear compatibility forces every map $f$ to verify the relation $f(\epsilon. (-))=\epsilon f$ with $\epsilon$ corresponding to the unity of $A$ in degree $-1$ in $K(A,0)$. As the action of $\epsilon$ is zero on the trivial $K(A,0)$-module $A$ concentrated in degree $0$, the $K(A,0)$-linear structure gives $f(\epsilon.-)=0$ imposing that for odd $i$ only the zero map in $\mathrm{Hom}_A( \underbracket{A}_{-i}, A)$ is allowed. This shows (\ref{eq-explicitcomputationAdegree2}) as a map of dg-modules,  under which $u$ corresponds to $1\in A\simeq \mathrm{Ext}^2_{K(A,0)}(A,A)$.  \\
 To explain why (\ref{eq-explicitcomputationAdegree2}) is an equivalence of dg-algebras we argue as follows: Since $\Coh(S\times^\mathrm{h}_{\mathbb{A}^1_S} S)_{\Perf(S)}^{\boxplus}$ is symmetric monoidal with tensor unit $A$, $\mathbb{R}\mathrm{Hom}_{K(A,0)}(A, A) $ is actually a commutative algebra object (endomorphisms of the unit). The element $u$ then defines (\ref{eq-explicitcomputationAdegree2}) as map of commutative algebra objects, where $A[u]$ is endowed with its usual algebra structure. This concludes the proof of the equivalence of underlying dg-categories 
\begin{equation}\label{eq-kudegree2Koszulequivalence}
\Coh(S\times^\mathrm{h}_{\mathbb{A}^1_S} S)_{\Perf(A)}\simeq \Perf(A[u])
\end{equation}
\medskip
We now discuss the symmetric monoidal equivalence. As a preliminary step  we describe the computation of $\mathbb{R}\mathrm{Hom}_{K(A,0)}(A, E)$ for $E\in \Coh(K(A,0))$. One shows that level $n$ of $\mathbb{R}\mathrm{Hom}_{K(A,0)}(A, E)$ is the level $n$ of the complex $E\otimes_A A[u]$. However, the differential on $\mathbb{R}\mathrm{Hom}_{K(A,0)}(A, E)$ is not the naive tensor product differential. 
 Indeed, using the same infinite resolution of $A$ as a $K(A,0)$-module and from the relation $f(\epsilon. (-))=\epsilon f$ one obtains that the elements of odd degree are determined by the antecedent element even degree under multiplication by $\epsilon$. Therefore,  the level  $n$ of $\mathbb{R}\mathrm{Hom}_{K(A,0)}(A, E)$ is the direct sum $\bigoplus_{i\geq 0} E_{n-2i}$ and the differential $\bigoplus_{i\geq 0} E_{n-2i}\to \bigoplus_{i\geq 0} E_{n+1-2i}$ is given by $d+ \epsilon.(-)$ where $d$ is the native differential of $E$. It is clear now that the resulting level $n$ of $\mathbb{R}\mathrm{Hom}_{K(A,0)}(A, E)$ identifies with the resulting level $n$ of the naive tensor product $E\otimes_A A[u]$ but the differential is twisted by the action of $\epsilon$. To encode the result of this computation we will write
\begin{equation}
\label{ex-explicitcomputationmainkeytomonoidal}
\mathbb{R}\mathrm{Hom}_{K(A,0)}(A, (E,d))\simeq (E\otimes_A A[u], d+ \epsilon)
\end{equation}
This formula is the starting point to show that the equivalence (\ref{eq-kudegree2Koszulequivalence}) is monoidal.  We show that the arguments given in \cite[Prop. 3.1.4]{1101.5834} work for a general regular ring $A$ using our infinite resolution (\ref{eq-infiniteresolution}) instead of the Koszul-Tate resolution used in loc.cit. We start by showing that the strict dg-functor
\begin{equation}
\label{eq-strictdgKoszulcomputation}
\mathcal{E}:K(A,0)-\mathrm{dgMod}^{strict}_A\to A[u]-\mathrm{dgMod}^{strict}_A
\end{equation}
\noindent sending
\begin{equation}
(E,d_E)\mapsto \mathcal{E}(E, d_E):=(E\otimes_A A[u], d_E+ \epsilon)
\end{equation}
\noindent is symmetric monoidal with respect to the convolution $\boxplus$ of the Remark \ref{constructionlaxmonoidalstructurestrictperfect} on the l.h.s and the usual tensor product over $A[u]$ on the r.h.s.  This follows essential from the definition of $\boxplus$: given two pairs $(E,d_E)$ and $(F, d_F)$ in $K(A,0)-\mathrm{dgMod}_A$ their box product is given by the pair that consists of the usual tensor product over $A$,$E\otimes_A F$ equipped with the action of $K(A,0)$ given by the formula (\ref{eq-explicitactionepsilonproduct}) which in this case is explicitly given by $\epsilon\otimes Id + Id\otimes \epsilon$. In this case the natural lax structure is an equivalence as 
$$
\mathcal{E}((E, d_E)\boxplus (F, d_F))=\mathcal{E}(E\otimes_A F, d_E\otimes Id + Id\otimes d_F)\simeq
$$
$$
 ((E\otimes_A F)\otimes_A A[u]), (d_E\otimes Id + Id\otimes d_F) + (\epsilon\otimes Id + Id\otimes \epsilon))\simeq
$$
 $$
((E\otimes_A F)\otimes_A A[u]),(d_E + \epsilon)\otimes Id + Id\otimes (d_F+\epsilon) \simeq (E\otimes_A A[u], d_E+\epsilon)\otimes_{A[u]}(F\otimes_A A[u], d_F + \epsilon)
 $$
 $$
 \simeq  \mathcal{E}(E, d_E)\otimes_{A[u]} \mathcal{E}(F, d_F)
 $$
 As explained in the Remark \ref{constructionlaxmonoidalstructurestrictperfect} this restricts to a symmetric monoidal functor
 \begin{equation}
 \label{17ago1419}
 \mathrm{Coh}^s(K(A,0))^{\boxplus}\subseteq K(A,0)-\mathrm{dgMod}^{strict, \boxplus}_A\to A[u]-\mathrm{dgMod}^{strict, \otimes}_A
 \end{equation}
 As a second step we notice that $\mathcal{E}$ preserves quasi-isomorphisms: Indeed, it is clear that if $E\to F$ is a quasi-isomorphism of strictly perfect complexes over $K(A,0)$, then $E\otimes_A A[u]\to F\otimes_A A[u]$ with the standard differentials on both source and target, is a quasi-isomorphism (again, this is because $E$, resp. $F$, being strictly perfect over $K(A,0)$, it is, in particular, degreewise projective over $A$, therefore flat). We have to argue why the induces map
 $$
(E\otimes_A A[u], d_E+ \epsilon) \to (F\otimes_A A[u], d_F+ \epsilon) 
$$
\noindent is also a quasi-isomorphism for the differentials perturbed by $\epsilon$. To see this we remark that $(E\otimes_A A[u], d_E+ \epsilon)$ is obtained as the totalization of the a complex with given by differentials $d_E$ and $\epsilon$. This is obtained from the canonical filtration by degree on $A[u]$. The associated graded pieces are all given by $E$. Therefore, the convergence of the associated spectral sequences then tells us that the cohomology of these graded pieces converges to the cohomology of the total complex. In particular if $E\to F$ is a quasi-isomorphisms, the associated graded pieces of the two spectral sequences are quasi-isomorphic and therefore, so are the total complexes.
 Finally, we use the compatibility of $\mathcal{E}$ with quasi-isomorphisms to  pass to the monoidal localizations. As seen in the Construction \ref{construction-laxmonoidalstructurerelativeperfect}, the strict lax structure on $\mathrm{Coh}^s(-)$, $\mathrm{Coh}^s(K(A,0))$ is stable under the $\boxplus$ tensor product on $K(A,0)-\mathrm{dgMod}^{strict}_A$. Therefore, the dg-localization along quasi-isomorphisms is monoidal. Moreover, as $A[u]-\mathrm{dgMod}^{strict}_A$ is a symmetric monoidal model category, it follows from \cite[Thm A.7]{1707.01799} that the localization of $A[u]-\mathrm{dgMod}^{strict}_A$ along quasi-isomorphisms is also a lax symmetric monoidal functor. By the universal properties of localizations, combined with the equivalence (\ref{eq-strictmodelcoherentperfectswedenmeatballssuck}) we obtain a  lax symmetric monoidal dg-functor in $\dgcat^{big}_A$
\begin{equation}
\label{eq-swedenswedensofarfromeden}
\mathcal{E}: \Coh(K(A,0))^{\boxplus}\to  \Qcoh(A[u])^{\otimes_{A[u]}}
\end{equation}
\noindent which one can easily check to be strongly monoidal. Moreover, as seen above in the proof of the equivalence (\ref{eq-kudegree2Koszulequivalence}), every $E \in \Coh(K(A,0))$ can be obtained from the $K(A,0)$-dg-module $A$ under finite shifts and cones so that  $\mathcal{E}(E)$ will be a perfect $A[u]$-module (as $\mathcal{E}(A)$ is free as A[u]-module because the action of $\epsilon$ on $A$ is trivial.) Therefore  (\ref{eq-swedenswedensofarfromeden}) factors through $\Perf(A[u])^{\otimes_{A[u]}}$ and this factorization in $\dgcat_A^{\mathrm{idem}}$ recovers the equivalence (\ref{eq-kudegree2Koszulequivalence}).
\end{proof}
\end{lem}

\begin{rem}\label{rem-colimitformulaonlyworksforquasi-coherent}
Notice that the equivalence (\ref{eq-infiniteresolutionformulacolimit}) is valid only in $\Qcoh(K(A,0))$ and not in $\mathrm{IndCoh}(K(A,0))$. 
\end{rem}

\begin{rem}
\label{rem-comparisonmultiplicationsonAu}
As seen in the proof, it is a consequence of the symmetric monoidal equivalence (\ref{eq-kudegree2Koszulequivalencemonoidal}) that the equivalence (\ref{eq-explicitcomputationAdegree2}) identifies the algebra-structure of composition of endomorphisms with the standard multiplication on $A[u]$.
\end{rem}
\begin{rem}
As in the Remark \ref{remark-comparisonMonoidalstructuresMF-2-periodic}, using the equivalence (\ref{eq-kudegree2Koszulequivalencemonoidal}) we recover the lax symmetric monoidal structure on the $\infty$-functor (\ref{eq: infinityfunctorialrelativeperfect}) 
\begin{equation}
\label{eq-laxmonoidalAumodules}
\xymatrix{\LG^{\mathrm{aff, op}, \boxplus}\ar[rr]^-{\Coh(-)_{\Perf(-)}}&& \Mod_{\Perf(A[u])}(\dgcat_A^{\mathrm{idem}})^{\otimes}}
\end{equation}
restricted to affine LG-pair $(X,f)$ and the induced action of the small stable idempotent complete symmetric monoidal $(\infty,1)$-category $\Perf(A[u])^{\otimes_{A[u]}}$ on the $\infty$-category $\Coh(X_0)_{\Perf(X)}$.
\end{rem}

\medskip

\begin{prop}\label{prop-technicalintermediate}
Assume $A$ is a regular ring. Under the symmetric monoidal equivalence (\ref{eq-kudegree2Koszulequivalencemonoidal}):
\begin{enumerate}
\item the full subcategory $\Perf(S\times^\mathrm{h}_{\mathbb{A}^1_S} S)\subseteq  \Coh(S\times^\mathrm{h}_{\mathbb{A}^1_S} S)$ is identified with the full subcategory $\Perf(A[u])^{u-\text{Torsion}}$ of $u$-torsion modules, i.e, those perfect dg-modules $M$ over $A[u]$ such that there exists an $N\geq 0$ such that the multiplication by $u^N$, $M[-2n]\to M$ is null-homotopic. 
\medskip
\item the quotient map 
\begin{equation}
\label{19ago1149}
\Coh(S,0)\to \Sing(S,0)
\end{equation}
 identifies  with the symmetric monoidal base change map
\begin{equation}\label{eq-symmetricmonoidalstandardinvertionofu}
-\otimes^{\mathbb{L}}_{A[u]} A[u, u^{-1}]: \Perf(A[u])\to \Perf(A[u, u^{-1}]).
\end{equation}
\noindent yielding
\begin{equation}\label{eq-kudegree2Koszulequivalence2}
\Sing(S,0)\simeq \Perf(A[u, u^{-1}])
\end{equation}
\end{enumerate}
\begin{proof}
Let us start with the first claim. The argument is similar to the one of \cite[Lemma 3.1.9 ]{1101.5834}. Using the formula (\ref{ex-explicitcomputationmainkeytomonoidal}) one obtains 
\begin{equation}
\label{ex-explicitcomputationmainkeytomonoidal2}
\xymatrix{\mathbb{R}\mathrm{Hom}_{K(A,0)}(A, K(A,0)) \ar[r]^-{\sim}& (\underbracket{A}_{-1}\ar[r]^{0} & \underbracket{A}_{0}\ar[r]^{\epsilon.=\mathrm{id}}&\underbracket{A}_{1} \ar[r]^{0} & \underbracket{A}_{2} \ar[r]^{\epsilon.=\mathrm{id}} &\cdots )}
\end{equation}
\noindent and observe that the r.h.s is quasi-isomorphic to $A[1]$. In this case, the equivalence (\ref{eq-kudegree2Koszulequivalence}) maps the full subcategory $\Perf(S\times^\mathrm{h}_{\mathbb{A}^1_S} S)\subseteq  \Coh(S\times^\mathrm{h}_{\mathbb{A}^1_S} S)$, by definition, generated by $K(A,0)$ under finite colimits and retracts, to the full subcategory of 
$\Perf(A[u])$ generated by the object $A[1]$. This is equivalent to the (stable and idempotent complete) subcategory generated by $A$ as a trivial $A[u]$-module.
We remark that as an $A[u]$-dg-module, $A$ fits in a cofiber-fiber sequence
\begin{equation}
\label{eq-cofiber-fiberAasA[u]}
\xymatrix{
\ar[d] A[u][-2]\ar[r]^{u.}& A[u]\ar[d]\\
0\ar[r]& A
}
\end{equation}

\noindent which can be obtained using the explicit model for the cone of the multiplication by $u$ (given by the identity on each level). \\
We use this to conclude that the thick subcategory generated by $A$ in $\Perf(A[u])$ is exactly the full subcategory spanned by the $u$-torsion dg-modules. Indeed, $A$ is by construction $u$-torsion, as $u$ acts null-homotopically. Moreover, the full subcategory of $u$-torsion modules is by its nature a thick stable and idempotent complete subcategory of $\Perf(A[u])$. It remains to show that every $M\in \Perf(A[u])^{u-\text{Torsion}}$ can be obtained as a retract of a homotopy finite cellular object built from $A$ under shifts and cones. For this purpose we use the cofiber/fiber-sequence (\ref{eq-cofiber-fiberAasA[u]}): given $M\in \Perf(A[u])^{u-\text{Torsion}}$, using the relative tensor product over $A[u]$ (which as explained in the Remark \ref{rem-comparisonmultiplicationsonAu} carries its standard structure of $\mathrm{E}_\infty$-algebra) we obtain a cofiber-fiber sequence
\begin{equation}
\label{eq-cofiber-fiberAasA[u]3}
\xymatrix{
\ar[d] M[-2n]\simeq M\otimes_{A[u]}A[u][-2n]\ar[r]^-{u^n\sim 0.}& M\simeq M\otimes_{A[u]}A[u]\ar[d]\\
0\ar[r]&  M[-2n+1]\oplus M\simeq M\otimes_{A[u]}(\bigoplus_{0\leq i \leq n-1}A[-2i])
}
\end{equation}
The assumption that $M$ is perfect over $A[u]$ means by definition that it is obtained under finite shits and cones of $A[u]$. In particular, $M\otimes_{A[u]}(\bigoplus_{0\leq i \leq n-1}A[-2i])$ is then obtained as a finite cell-object from $A$ and has $M$ as a direct factor.\\
\vspace{0.5cm}
Let us now address the second claim. First notice that since $A$ is regular, we have $\Sing(S,0) \simeq \Sing(S\times^\mathrm{h}_{\mathbb{A}^1_S} S)$. 
Thanks to the half of the proposition already proved, to establish this identification one is reduced to present the base change $-\otimes^{\mathbb{L}}_{A[u]} A[u, u^{-1}]$ as a Verdier quotient with respect to the thick subcategory of $u$-torsion dg-modules. For this purpose we remark that at the level of the presentable $\infty$-categories of modules, the base-change in $\mathrm{Pr}^{\mathrm{L}}$
$$-\otimes^{\mathbb{L}}_{A[u]} A[u, u^{-1}]: \Mod_{A[u]}(\Sp)\to \Mod_{A[u, u^{-1}]}(\Sp)$$
admits the restriction of scalars along the map $A[u]\to A[u,u^{-1}]$ as a fully faithful right adjoint, whose image is the full subcategory of $\Mod_{A[u]}(\Sp)$ spanned by those dg-modules where the multiplication by $u$ is invertible. In other words, such objects become the local objects for the presentation of $\Mod_{A[u, u^{-1}]}(\Sp)$ as a Bousfield localization of $\Mod_{A[u]}(\Sp)$. As a consequence of this fact, $-\otimes^{\mathbb{L}}_{A[u]} A[u, u^{-1}]$ has an alternative description in terms of a colimit in (the big category of) $A[u]$-modules in spectra given by multiplication by $u$:
$$
M\otimes^{\mathbb{L}}_{A[u]} A[u, u^{-1}]\simeq \mathrm{colim}_{n}( \cdots \to M\to M[2]\to M[4]\to \cdots)
$$
This follows from the combination of the universal properties of base change, of colimits in $A[u]$-modules and fully faithfulness along restriction of scalars. The formula remains valid for perfect complexes because base-change preserves perfect complexes (the colimit being always taken in spectra). In this case, if an $A[u]$-module $M$ is $u$-torsion the colimit by multiplication by $u$ is by cofinality equivalent to the colimit of the zero diagram so that $M\otimes_{A[u]}A[u,u^{-1}]\simeq 0$. Conversely, if $M\otimes_{A[u]}A[u,u^{-1}]\simeq 0$ and $M\in \Perf(A[u])$, then $M$ is compact and we have
$$
\ast\simeq \mathrm{Map}_{A[u]}(M, \underbracket{\mathrm{colim}_n M[2n]}_{\simeq 0}))\simeq \mathrm{colim}_n\mathrm{Map}_{A[u]}(M,  M[2n]))
$$
in $\Mod_{A[u]}(\Sp)$, so that there is an $n$ such that the power $u^n$ is the zero map.
\end{proof}
\end{prop}
\vspace{0.4cm}

\begin{rem}
\label{formulauactionbasic}
Given $T\in \Perf(A[u])$ and $M\in \Coh(X_0)_{\Perf(X)}$ let us introduce the notation $T\otimes_{A[u]}M$ for the action of $T$ on $M$ via the monoidal equivalence  (\ref{eq-kudegree2Koszulequivalencemonoidal}). By definition of (\ref{eq-kudegree2Koszulequivalencemonoidal}), we get that $M\simeq A[u]\otimes_{A[u]}M\simeq A\boxtimes M$ and that $A[1]\otimes_{A[u]}M\simeq K(A,0)\boxtimes M\simeq \mathfrak{i}^*\mathfrak{i}_*M$. In particular, as the action by construction commutes with colimits in each variable, the cofiber sequence (\ref{eq-cofiber-fiberAasA[u]}) produces a cofiber sequence 
\begin{equation}\label{eq-formulauactionbasic}
\xymatrix{
M[-2]\simeq A[u][-2]\otimes_{A[u]}M\ar[r]^{u}\ar[d]& M\simeq A[u]\otimes_{A[u]}M\ar[d]\\
0\ar[r]&  A\otimes_{A[u]}M\simeq \mathfrak{i}^*\mathfrak{i}_*(M)[-1]
}
\end{equation}
\end{rem}

\vspace{1cm}
The following proposition achieves the main goal of this section of exhibiting a lax symmetric structure on the $\infty$-functor $\Sing$ (\ref{eq-laxSingversion}). In particular this extends \cite[Prop. 3.4.3]{1101.5834} to a base ring which we only require to be regular local, instead of a field of characteristic zero.

\medskip

\begin{prop} \label{Singismonoidal}
\label{prop-laxstructureSingviabasechange}
Assume the Context \ref{context111}. There is a natural equivalence between the composition
\begin{equation}
\label{eq-laxmonoidalAumodules}
\resizebox{1. \hsize}{!}{
\xymatrix{\LG^{\mathrm{aff, op}}\ar[rr]^-{\Coh(-)_{\Perf(-)}}&& \Mod_{\Perf(A[u])}(\dgcat_A^{\mathrm{idem}})\ar[rr]^-{-\otimes_{A[u]}A[u, u^{-1}]}&&\Mod_{\Perf(A[u, u^{-1}])}(\dgcat_A^{\mathrm{idem}}) \ar[rr]^-{\mathrm{Forget}} && \dgcat_A^{\mathrm{idem}}}
}
\end{equation}
\noindent and the $\infty$-functor $\Sing$. By transfer under this equivalence, the functor $\Sing$ acquires a lax symmetric monoidal enhancement. In particular, $\Sing(S,0)$ acquires a symmetric monoidal structure equivalent to the natural one on 2-periodic complexes.
\begin{proof}
The proof is similar to \cite[3.4.1]{1101.5834}. As $\mathrm{Coh}^b(K(A,0)^{\boxplus}\simeq \Perf(A[u])^{\otimes_{A[u]}}$ is a rigid symmetric monoidal idempotent-complete dg-category and 
$$
\Perf(A[u])^{u-\text{Torsion}}\subseteq \Perf(A[u])\to \Perf(A[u,u^{-1}])
$$
is an exact sequence of $\Perf(A[u])$-linear dg-categories, one can use exactly the same arguments as in \cite[Lemma 3.4.2]{1101.5834}  to deduce that for any LG-pair $(X,f)$ the base-change sequence
$$
\resizebox{1. \hsize}{!}{
\xymatrix{
\mathrm{Coh}^{b}(X_0)_{\Perf(X)}\otimes_{\Perf(A[u])} \Perf(A[u])^{u-\text{Torsion}}\subseteq \mathrm{Coh}^{b}(X_0)_{\Perf(X)}\to \mathrm{Coh}^{b}(X_0)_{\Perf(X)}\otimes_{\Perf(A[u])}\Perf(A[u,u^{-1}])}}
$$
remains a cofiber-fiber sequence in $\Mod_{\Perf(A[u])}(\dgcat_A^{\mathrm{idem}})$. It follows from the definition of the tensor product in $\dgcat_A^{\mathrm{idem}}$ that the localization functor $\mathrm{Coh}^{b}(X_0)_{\Perf(X)}\to \mathrm{Coh}^{b}(X_0)_{\Perf(X)}\otimes_{\Perf(A[u])}\Perf(A[u,u^{-1}])$ can be described as in the proof of the Prop. \ref{Singismonoidal} by the $\infty$-functor sending $M\in \mathrm{Coh}^{b}(X_0)_{\Perf(X)}$ to the colimit in $\mathrm{Ind}(\mathrm{Coh}^{b}(X_0)_{\Perf(X)}$) \footnote{Notice also that thanks to the Remark \ref{rem-indcompletioncoherentperfect}, this filtered colimit can also be taken in $\mathrm{IndCoh}(X_0)_{\Qcoh(X)}$.} of the sequence given by the action of $u$:
$$
M\mapsto \mathrm{colim}_{n}( \cdots \to M\to M[2]\to M[4]\to \cdots)
$$
Moreover, by definition, $\mathrm{Coh}^{b}(X_0)_{\Perf(X)}\otimes_{\Perf(A[u])} \Perf(A[u])^{u-\text{Torsion}}$ identifies with the full subcategory of $\mathrm{Coh}^{b}(X_0)_{\Perf(X)}$ spanned by those objects $M$ such that there exists $N\geq 0$ such that $u^N: M\to M[2n]$ is null-homotopic. In this case we are reduced to show that $M$ is in $\Perf(X_0)$ if and only there exists $N\geq 0$ such that $u^N\sim 0$. For this we follow the steps of \cite[3.4.1 (i) $\Leftrightarrow$ (ii)]{1101.5834}. We use the description of the action of $\Coh(K(A,0))$ on $\Coh(X_0)_{\Perf(X)}$ given in the remark \ref{rem-symmonoidalstructurerelativeperfectSzero}, namely the formulas (\ref{eq-explicitdescriptionaction}), (\ref{eq-explicitdescriptionaction2}).

Combining these formulas with our resolution (\ref{eq-infiniteresolutionformulacolimit}) for $A$ as a $K(A,0)$-module,  we get a diagram of natural transformations
$$
\cdots \mathfrak{i}^*\mathfrak{i}_*[n]\to  \mathfrak{i}^*\mathfrak{i}_*[n-1]\to \cdots \to  \mathfrak{i}^*\mathfrak{i}_*[1]\to  \mathfrak{i}^*\mathfrak{i}_*
$$
induced by the multiplication by $\epsilon$. As the resolution works only in $\Qcoh(K(A,0))$ (Remark \ref{rem-colimitformulaonlyworksforquasi-coherent}), the formula (\ref{eq-infiniteresolutionformulacolimit}) yields a canonical equivalence of $\infty$-functors
\begin{equation}
\label{eq-colimitresolutionidentityfunctor}
\mathrm{colim}_n \,( \mathfrak{i}^*\mathfrak{i}_* \to  \mathfrak{i}^*\mathfrak{i}_*\{1\}\to  \mathfrak{i}^*\mathfrak{i}_*\{2\}\to \cdots )\simeq \mathrm{Id}_{\Qcoh(X_0)}
\end{equation}
\noindent where the colimit is taken in $\Qcoh(X_0)$ (by definition of $\Coh_{\Perf}$, $\,\,\,\mathfrak{i}^*\mathfrak{i}_*$ has values in the quasi-coherent category). We now remark that $M$ is in $\Perf(X_0)$ if and only if there exists an $N\geq 0$ such that $M$ is an homotopy retract of some $\mathfrak{i}^*\mathfrak{i}_*\{N\}(M)$.  Indeed, suppose that $M$ is perfect. Then it is a compact object in $\Qcoh(X_0)$ and the identity map
$M\to M\simeq \mathrm{colim}_n\, \ \mathfrak{i}^*\mathfrak{i}_*\{n\}(M)$ factors through some finite stage $\mathfrak{i}^*\mathfrak{i}_*\{n\}(M)$. Conversely, suppose that $M\in \Coh(X_0)_{\Perf(X)} $ is an homotopy retract of some $\mathfrak{i}^*\mathfrak{i}_*\{n\}(M)$. Then because of the definition of $\Coh(X_0)_{\Perf(X)}$, for any $M$, $\mathfrak{i}^*\mathfrak{i}_*(M)[n]$ is always a perfect complex so that the finite colimit $\mathfrak{i}^*\mathfrak{i}_*\{n\}(M)$ is also perfect. Being a retract, $M$ will also be perfect.\\
It remains to identity modules obtained as homotopy retracts of some $\mathfrak{i}^*\mathfrak{i}_*\{N\}(M)$ exactly with those modules where the action of $u$ is torsion.

The cofiber-sequence \ref{eq-formulauactionbasic}, tells us that 
$$
\mathrm{cofib}\, u\simeq \mathfrak{i}^*\mathfrak{i}_*(M)[-1]
$$
Using it one can easily construct a cofiber-fiber sequence
$$
\xymatrix{
\mathfrak{i}^\ast \mathfrak{i}_* M[1] \ar[d] \ar[r]& 0\ar[d]\\
\mathfrak{i}^\ast \mathfrak{i}_* M\ar[r]&  \mathrm{cofib}(u^2)[3]
}
$$
By induction, one shows that 
$$
\xymatrix{
\mathrm{cofib}\,u[-1]\ar[d]\ar[r]& 0\ar[d]\\
\mathrm{cofib}\, u^n[-2]\ar[r]& \mathrm{cofib}\, u^{n+1}
}
$$
are cofiber-diagrams and more generally, using the diagrams (\ref{equation-inductionlevelii}) we get
$$\mathrm{cofib}(u^N)[2N-1]\simeq  \mathfrak{i}^*\mathfrak{i}_*\{N-1\}(M)$$
In particular, if the action is torsion, we have $u^N\sim 0$ for some $N\geq 0$, and $\textrm{cofib}(u^N)\simeq M[-2N+1]\oplus M$ and $\mathrm{Cofib}(u^N)[2N-1]\simeq M\oplus M[2N-1]$. In this case we get $M\oplus M[2N-1]\simeq  \mathfrak{i}^*\mathfrak{i}_*\{N-1\}(M)$ so that $M$ is a retract of the finite colimit. Conversely, if $M$ is a retract of the cofiber of $u^N$ then this $u^N$ is null-homotopic as the retract gives a splitting of the cofiber sequence $M[-2N]\to M\to cof\, u^N$.
\end{proof}
\end{prop}

\medskip

\begin{rem}\label{rem-noneedforaffineformonoidalstructrureSIng}Notice that in the previous proof it is never used that $(X,f)$ is an affine LG-pair. In fact, the proof of the previous proposition can be used to conclude that the functor $(X,f)\mapsto \Sing(X,f)$ is lax monoidal on non-affine LG-pairs, independently of the strict model for bounded coherent on $X_0$ perfect on $X$ of \ref{l1} in the affine case. Indeed, the lax monoidal structure on $(X,f)\mapsto \Sing(X,f)$ can be obtained using the fact that $\Qcoh$ is a lax monoidal $\infty$-functor (obtained from the lax monoidal structure on the construction $A\mapsto \Mod_A(\Sp)$ - see \cite{lurie-ha}), combined with the cartesian property of the diagram (\ref{eq-pullbackcoherentboundedperfect}), the Proposition \ref{prop-technicalintermediate}, the Lemma \ref{lemma-monoidalequivalencerelativeperfectandsing}, the lax monoidality of inverting $u$ and the proof of the proposition \ref{prop-laxstructureSingviabasechange}. 
\end{rem}

\medskip

\begin{rem} The fact that $\Sing(S,0)$ is \emph{monoidal} equivalent to 2-periodic complexes has been proved in the case where $A$ is a field of characteristic zero, see for instance \cite[Prop. 3.1.9]{1101.5834}  and \cite[Section 5.1]{MR3300415}. This is an instance of Koszul duality for modules. 
\end{rem}

\medskip

\begin{rem}
One should also remark that as for $\MF$, under some hypothesis, the functor $\Sing$ on non-affine LG-pairs matches the result of the Kan extension of its restriction to affine LG-pairs. This follows from a combination of Cech descent for $\mathrm{Coh}^b_{\Perf}$ of the Cor. \ref{cor-h-descent}, together with the fact that for Noetherien schemes of finite Krull dimension the Zariski topos is hypercomplete \cite[3.7.7.3]{lurie-sag}. Knowing Zariski descent for  $\mathrm{Coh}^b_{\Perf}$ it suffices, after the Prop. \ref{prop-laxstructureSingviabasechange} and the Remark \ref{rem-noneedforaffineformonoidalstructrureSIng}, to remark that the base-change $-\otimes_{\Perf( A[u]) }\Perf( A[u, u^{-1}])$ in idempotent complete $A$-dg-categories, preserves finite limits because both $\Perf( A[u])$ and $\Perf( A[u, u^{-1}])$ have single compact generators, and the localization $A[u, u^{-1}] $ can be obtained as a filtered colimit under multiplication by $u$ in $\Sp$, and filtered colimits preserve finite limits. 
\end{rem}

\vspace{1cm}

\subsection{Comparison}
\label{section-comparisonMF-SING}
Consider the Context \ref{context111}. In this \S \, we  prove the following 
\begin{thm}\label{t1}
There is a lax symmetric natural transformation of 
$\s$-functors
$$\mathrm{Orl}^{-1,\,\otimes}:\Sing \to \MF : \LG_S^{\mathrm{op}} \longrightarrow \dgcat^{\mathrm{idem}}_A$$
\noindent with the following properties:

\begin{enumerate}
\item $\mathrm{Orl}^{-1,\,\otimes}$ identifies the  \emph{symmetric monoidal} structure of $\Sing(S,0)$ given in Prop. \ref{prop-laxstructureSingviabasechange} with the one of $\MF(S,0)$ as in  Remark \ref{remark-comparisonMonoidalstructuresMF-2-periodic}.

\item $\mathrm{Orl}^{-1,\,\otimes}$ is an equivalence when restricted to the sub-category of LG-models $(X,f)$ where $f$ is a non-zero divisor on $X$  (i.e. the induced morphism $\mathcal{O}_X \to \mathcal{O}_X$ is a monomorphism), $X/S$ is separated, and $X$ has the resolution property (i.e. every coherent $\mathcal{O}_X$-Module is a quotient of a vector bundle, e.g. $X$ regular).
\item In particular, from (1) and (2),  $\MF(X,f)$ and $\Sing(X,f)$ are then equivalent
\emph{as $A[u,u^{-1}]$-linear }idempotent complete dg-categories.
\end{enumerate}
\end{thm}

Theorem \ref{t1} provides a $\infty$-functorial, dg-categorical lax symmetric monoidal version of the so-called Orlov's comparison theorem,
comparing matrix factorizations and categories of singularities (see \cite[Theorem 2.7]{efimov2015coherent} and \cite[Thm 3.5]{MR2910782}).  

\medskip

\begin{rem}\label{monoidal-leftover}(Derived vs Classical zero locus) Note however, that our natural transformation $\mathrm{Orl}^{-1}$ is defined also for non-flat LG-models $(X,f)$, and this was made possible by considering the derived zero locus of $f$ instead of the classical scheme-theoretic zero locus in the definition of the functor $\Sing$ (while $\MF$ is defined using only non-derived ingredients).
Note that for flat LG-models $(X,f)$, $f$ is indeed a non-zero divisor on $X$.
If we restrict the functor $\Sing$ to LG-models $(X,f)$ where  $f$ is a non-zero divisor on $X$, i.e. $f_{|U}$ is a non-zero divisor for all Zariski open affine subschemes $U\subseteq X$ (this case is of particular interest for us, see Remark \ref{rem-choice-uniformizer}), then the derived fiber $X_0$ coincides with the classical scheme-theoretic fiber $X^{\textrm{cl}}_0$ (i.e. the truncation of $X_0$), and one does not need to use derived algebraic geometry at all in the definition of $\Sing$ (Definition \ref{d2}). Note, however, that if these conditions are not met, there is no way to avoid taking the derived fiber $X_0$. In fact, the push-forward along the closed immersion $X^{\textrm{cl}}_0 \to X$ does not necessarily preserve perfect complexes, so that a purely classical analogue $\Sing^{\textrm{cl}}$ of our definition of $\Sing$ is simply impossible. And this, regardless, the fact that $X$ may or may not enjoy the resolution property. Moreover, even when the pushforward along $X^{\textrm{cl}}_0 \to X$ does preserve perfect complexes, so that both our definition of $\Sing$ and its purely classical analogue $\Sing^{\textrm{cl}}$ make sense, then they might differ. As an important example, one could take $(X,f):=(S,0)$: here  $X^{\textrm{cl}}_0 = X$, so that $\Sing^{\textrm{cl}}(S,0)$ is defined and is trivial, while $\Sing(S,0)$ is equivalent to the dg-category of 2-periodic complexes $\mathrm{Perf}(A[u,u^{-1}])$ (Proposition \ref{prop-laxstructureSingviabasechange}), and is therefore equivalent to $\MF(S,0)$, as an object in $\CAlg(\dgcat^{\mathrm{idem}}_A)$. In particular there is no hope for $\MF$ to be equivalent to $\Sing^{\textrm{cl}}$, when $f$ is allowed to be a zero-dvisor.
\end{rem}

\begin{rem}
\label{rem-Orlovworksproj}
The considerations of the previous remarks lead us to believe that the $\infty$-functor $\mathrm{Orl}^{-1}$ of Theorem \ref{t1} is an equivalence even without restricting to flat or non-zero divisors LG-pairs: we think this generalization of Theorem \ref{t1} is important, and will be discussed elsewhere. Granting this fact, we can make a few more observations. First, note that flat LG-pairs $(X,f)$ where $X/S$ is separated and $X$ is \emph{regular} belong to the subcategory for which $\mathrm{Orl}^{-1}$ is an equivalence. But unfortunately, the property of being regular is not preserved under base-change, so that these regular flat LG-pairs do not form a monoidal subcategory of the category of flat LG-pairs (recall from Section \ref{section-MFandSing} that $(X,f)\boxplus (Y,g):= (X\times_S Y, f\boxplus g)$). However, if we denote by $\LG^{\textrm{fl-qproj}}_S$ the subcategory of flat LG-pairs $(X,f)$ over $S$ where $X/S$ is quasi-projective (hence separated), then $\LG^{\textrm{fl-qproj}}_S$ is a symmetric monoidal subcategory of $\LG^{\textrm{fl}}_S$, and $\mathrm{Orl}^{-1}$ will remain an equivalence when restricted to $\LG^{\textrm{fl-qproj}}_S$ (granting the validity of 
Theorem \ref{t1} without the flatness hypothesis) since any quasi-projective scheme over an affine scheme has the resolution property, see e.g. \cite[2.1]{thomasonalgebraic}).

\end{rem}

\vspace{0.5cm}
We now address the proof of the Theorem \ref{t1}. By Kan extension and descent, it is enough to perform the construction of
$\mathrm{Orl}:\Sing\to\MF$ for affine LG-models.\\
For $(\Spec\, B,f)$ an affine LG-model, we first define a strict $A$-linear dg-functor
$$\psi : \mathrm{Coh}^{\mathrm{s}}(B,f) \longrightarrow \MFPairs(B,f)$$
as follows.

\begin{cons}
\label{construction-OrlovComparison}
Recall the description of objects of $\mathrm{Coh}^{\mathrm{s}}(B,f)$ from Remark \ref{remark-strictcoherentislocallyflatdgcat}: they are pairs $(E,h)$ consisting of a strictly bounded complex $E$ of projective $B$-modules of finite type, 
together with a morphism of graded modules $h : E \rightarrow E$ of degree $-1$,
satisfying the equation $[d,h]\equiv dh+hd=f$. 
Given such a pair $(E,h)$,  we define $\psi(E)$ to be
the $\mathbb{Z}/2$-graded $B$-module associated to $E$, that is
\begin{equation}
\label{eq-definitiontwoperiodizationtoMF}\psi(E)_0 = \oplus_{n} E_{2n} \qquad \psi(E)_1 = \oplus_{n} E_{2n+1}.
\end{equation}
We endow $\psi(E)$ with the odd endomorphism 
\begin{equation}
\label{eq-definitiondifferentialOrlovfunctor}
\delta:=h+d : \psi(E) \longrightarrow \psi(E).
\end{equation}
As $h^2=0$, we have $\delta^2 = f$, so this defines an object $(\psi(E),\delta)$ in 
$\MFPairs(B,f)$. Indeed, as $E$ is strictly bounded and each $E_i$ is a projective $B$-module, each sum above is finite and remains projective over $B$. This defines an $A$-linear dg-functor
$$\psi_{(B,f)} : \mathrm{Coh}^{\mathrm{s}}(B,f) \longrightarrow \MFPairs(B,f).$$

The $\psi_{(B,f)}$ are part of a natural transformation between the pseudo-functors
$$\psi: (\ref{eq-locallyflatlaxmonoidalversionCohstrict})\to (\ref{eq-MFstrict22})$$  
This is clear from the pseudo-functorial structure on (\ref{eq-laxmonoidalstrictCohPairs}) described in the beginning of the Construction \ref{construction-functorialassignementrelativeperfect} and the pseudo-functorial behavior of (\ref{eq-MFstrictlaxmonoidal}) described in the Construction \ref{construction-pseudofunctorstrictMF}. Moreover, $\psi$ has a lax symmetric monoidal enhancement $\psi^\otimes$ with respect to the lax monoidal enhancements (\ref{eq-locallyflatlaxmonoidalversionCohstrict}) and (\ref{eq-MFstrictlaxmonoidal}). Indeed, given affine LG-pairs $(B,f)$ and $(C,g)$ the commutativity of the diagram in $\mathrm{dgCat}_A^{\mathrm{strict, loc-flat}}$
\begin{equation}
\xymatrix{
\mathrm{Coh}^{\mathrm{s}}(B,f)\otimes_A \mathrm{Coh}^{\mathrm{s}}(C,g)\ar[d]^{(\ref{eq-laxstructurestrictrelativeperfect})}\ar[rrr]^{\psi_{(B,f)}\otimes \psi_{(C,g)}}&&&\MFPairs(B,f)\otimes_A \MFPairs(C,g)\ar[d]^{(\ref{eq-strictlaxstructuremorphismMF})}\\
\mathrm{Coh}^{\mathrm{s}}(B\otimes_A C,f\otimes 1 +1 \otimes g)\ar[rrr]^{\psi_{(B\otimes_A C,f\otimes 1 +1 \otimes g)}}&&&\MFPairs(B\otimes_A C,f\otimes 1+ 1\otimes g)
}
\end{equation}
\noindent comes from the explicit descriptions of each composition: if $E\in \mathrm{Coh}^{\mathrm{s}}(B,f)$ and $F\in \mathrm{Coh}^{\mathrm{s}}(C,g)$, the composition $\psi_{(B\otimes_A C,f\otimes 1 +1 \otimes g)}\circ (\ref{eq-laxstructurestrictrelativeperfect})(E,F)$ gives a 2-periodic complex
\begin{equation}
\label{valis1}
\xymatrix{\bigoplus_{(\alpha, \beta): \alpha + \beta \, \text{even}}E_\alpha\otimes_A F_\beta \ar@<1ex>[r]& \ar@<1ex>[l]\bigoplus_{(\alpha, \beta): \alpha + \beta \, \text{odd}}E_\alpha\otimes_A F_{\beta}}
\end{equation}
If $h$ (resp. $k$) denotes the element of degree $-1$ in $K(B,f)$ (resp. $K(C,g)$) explained in the Remark \ref{remark-strictcoherentislocallyflatdgcat}, then the formula (\ref{eq-formulaCohstrictproduct}) and the formula for the differential of the tensor product of complexes combined, describe the differential $\delta$ of (\ref{valis1}) (defined by (\ref{eq-definitiondifferentialOrlovfunctor}))  as
\begin{equation}
\label{eq-comparisondifferentialsproductlax}
\delta_E\otimes \mathrm{Id}_F + \mathrm{Id}_F\otimes \delta_F
\end{equation}
By re-indexing (\ref{valis1}) 
$$
\resizebox{1. \hsize}{!}{
\xymatrix{(\bigoplus_{\alpha \, \text{even},\beta \, \text{even} }E_\alpha\otimes_A F_\beta) \bigoplus (\bigoplus_{\alpha \, \text{odd},\beta \, \text{odd} }E_\alpha\otimes_A F_\beta) \ar@<1ex>[r]& \ar@<1ex>[l](\bigoplus_{\alpha \, \text{even},\beta \, \text{odd} }E_\alpha\otimes_A F_\beta) \bigoplus (\bigoplus_{\alpha \, \text{odd},\beta \, \text{even} }E_\alpha\otimes_A F_\beta)}}
$$
\noindent we recover the composition $(\ref{eq-strictlaxstructuremorphismMF})\circ \psi_{(B,f)}\otimes \psi_{(C,g)}(E,F)$, as the definition of the product differential  (\ref{eq-formulaMFproduct}) also gives (\ref{eq-comparisondifferentialsproductlax}). To conclude we have to check that $\psi$ is compatible with the lax units, meaning, that it makes the diagram 
$$
\xymatrix{
A\ar[r]^-{(\ref{eq-laxunitCohstrict})} \ar[dr]_-{(\ref{eq-laxunitMFstrict})}& \mathrm{Coh}^{\mathrm{s}}(A,0)\ar[d]^{\psi_{(A,0)}}\\
& \MFPairs(A,0)
}
$$
\noindent commute. This is immediate from the definitions.

\end{cons}
\begin{lem}\label{l2}
The dg-functor defined above
$$\psi_{(B,f)} : \mathrm{Coh}^{\mathrm{s}}(B,f) \longrightarrow \MFPairs(B,f)$$
sends quasi-isomorphisms to equivalences.
\begin{proof}
We will prove the equivalent statement that $\psi$
sends the full sub-dg-category $\mathrm{Coh}^{\mathrm{s, acy}}(B,f)$
of acyclic complexes, to zero. Again, recall from the Remark \ref{remark-strictcoherentislocallyflatdgcat} the description of objects in $\mathrm{Coh}^{\mathrm{s}}(K(B,f))$ as pairs $(E,h)$. Such a pair $(E,h)$ sits in $\mathrm{Coh}^{\mathrm{s,acy}}(K(B,f))$ if and only if there exists a degree $-1$ endomorphism $k$ of 
$E$, with $kd+dk=id$. The endomorphism $k$ defines 
an odd degree endomorphism of $\psi(E)$ as a $\mathbb{Z}/2$-graded 
$B$-module, so an element $\psi(k)$ of degree -1 in the complex of endomorphism
$\underline{End}_{\MF(B,f)}(\psi(E))$ (i.e. $\psi(k) \in \underline{End}_{\MF(B,f)}(\psi(E))_{-1}$).
By construction this element is a homotopy between $0$ and $id+hk+kh$. 
The endomorphism $u=hk+hk$ is of degree $-2$ and thus, because 
$E$ is bounded, we have $u^n=0$ for some integer $n$. We see in particular
that the identity of $\psi(E)$ becomes a nilpotent 
endomorphism in the homotopy category $[\MF(B,f)]$. This implies that 
$\psi(E)\simeq 0$ in $[\MF(B,f)]$ as stated.
\end{proof}
\end{lem}

\medskip

A consequence of the Lemma \ref{l2} is that $\psi^\otimes$ has an enhancement as a lax symmetric monoidal natural transformation between
\begin{equation}
\label{eq-OrlovlaxmonoidalstrictPairs}
\UseTwocells  
\xymatrix{\LG_{S}^{\mathrm{aff, op}, \boxplus} \rrtwocell^{(\ref{eq-laxmonoidalstrictCohPairs})}_{\MFPairs^{\otimes}}{_} && \mathrm{PairsdgCat}^{\mathrm{strict, loc-flat}, \otimes}_A}
\end{equation}
\noindent where $\MFPairs$ is seen as an object in $\mathrm{PairsdgCat}^{\mathrm{strict, loc-flat}}_A$ taking the equivalences as the distinguish class of morphisms. 
Finally, composing with the (symmetric monoidal) functors (\ref{eq-monoidallocalizationofpairs}) and (\ref{eq-localizationofpairsdgcats}) we obtain a lax symmetric monoidal transformation
\begin{equation}
\label{eq-Orlovlaxmonoidalrelativeperfect}
\UseTwocells  
\xymatrix{\LG_{S}^{\mathrm{aff, op}, \boxplus} \rrtwocell^{\mathrm{Coh}(-)_{\Perf(-)}^{\boxplus}}_{\MFPairs^{\boxplus}}{_} && \dgcat^{\mathrm{idem}, \otimes}_A}
\end{equation}

\begin{rem}
\label{rem-symmetricmonoidalSzerorelativeperfect}
As part of the lax symmetric monoidal enhancement (\ref{eq-Orlovlaxmonoidalrelativeperfect}),  we obtain a symmetric monoidal $\infty$-functor
\begin{equation}
\label{eq-symmetricmonoidalSzerorelativeperfect}
\psi_{(A,0)}^\otimes: \Coh(K(A,0))^\boxplus \to \MFPairs(A,0)^\boxplus
\end{equation}
It is now immediate to check, by just unraveling the definitions, that under the symmetric monoidal equivalences (\ref{eq-kudegree2Koszulequivalencemonoidal}) and (\ref{eq-comparisonMonoidalstructuresMF-2-periodic}) , $\psi_{(A,0)}^\otimes$ identifies with the symmetric monoidal base change functor (\ref{eq-symmetricmonoidalstandardinvertionofu}). 
\end{rem}

\vspace{0.7cm}

It is a consequence of the Remark \ref{rem-symmetricmonoidalSzerorelativeperfect} that the lax natural transformation (\ref{eq-Orlovlaxmonoidalrelativeperfect}) has in fact values in $\Mod_{\Perf(A[u])}(\dgcat^{\mathrm{idem}}_A)^\otimes$. Moreover, as the lax symmetric monoidal $\infty$-functor $\MFPairs^{\boxplus}$ has values in $\Mod_{\Perf(A[u, u^{-1}])}(\dgcat^{\mathrm{idem}}_A)^\otimes$ (because of the equivalence (\ref{eq-comparisonMonoidalstructuresMF-2-periodic})), by base-change, (\ref{eq-Orlovlaxmonoidalrelativeperfect}) is in fact equivalent to the data of a lax symmetric monoidal transformation
\begin{equation}
\label{eq-Orlovlaxmonoidalrelativeperfect2periodic}  
\xymatrix{\mathrm{Coh}(-)_{\Perf(-)}^{\boxplus}\otimes_{A[u]}A[u, u^{-1}]\ar[r]&\MFPairs^{\boxplus}}
\end{equation}
Finally, composing with the equivalence of the Prop. \ref{prop-laxstructureSingviabasechange} we obtain a lax symmetric monoidal transformation
\begin{equation}
\label{eq-Orlovlaxmonoidalfinal}
\psi^\otimes:\Sing^{\otimes}\to \MFPairs^\boxplus
\end{equation}

\vspace{0.5cm}
\textit{Proof of Theorem \ref{t1}.}
We set $\mathrm{Orl}^{-1,\otimes}:=\psi^{\otimes}  : \Sing^{\otimes} \longrightarrow \MF^{\boxplus}.$ The property (1) is now a consequence of the Remark \ref{rem-symmetricmonoidalSzerorelativeperfect}. Let us prove (2). The explicit description of $\psi$ given in the Construction \ref{construction-OrlovComparison} is all we need to conclude. Indeed, as observed in \cite[p. 47]{efimov2015coherent}, for each fixed $(B,f)$, the induced triangulated functor
$$[\mathrm{Orl}^{-1}] : [\Sing(B,f)] \longrightarrow [\MF(B,f)]$$
(denoted as $\Delta$ in loc. cit.) is an inverse to the functor $\Sigma$ described in \cite[Theorem 2.7]{efimov2015coherent}  (which is an analogue of Orlov's ``Cok'' functor in \cite[Thm 3.5]{MR2910782}), and thus is an equivalence (by \cite[Theorem 2.7]{efimov2015coherent}) on those LG-pairs $(X,f)$ where $f$ is flat (so that the derived fiber $X_0$ coincides with the scheme theoretic fiber considered in \cite[Theorem 2.7]{efimov2015coherent}), $X/S$ is separated (hence $X$ is), and $X$ has the resolution property, so that the standing hypotheses of \cite{efimov2015coherent}) are met. See also \cite[Theorem 6.8.]{MR3007084}.

As all the categories involved are stable, this implies the equivalence of the dg-enrichments.  

\medskip
(3) is a consequence of (1) and (2).
\hfill  $\Box$
\vspace{0.5cm}

\section{Motivic Realizations of dg-categories}
\label{section-realizationdgcategories}

In this section we explain how to associate to every dg-category $T$ a motivic $\BU$-module, where $\BU$ is the motivic ring-object representing algebraic $K$-theory. At first we describe some general features of this motivic incarnation of $T$ and then we will study several of its realizations. If $R$ is any realization of motives (e.g. $\ell$-adic, étale, Hodge, de Rham, etc), the realization $R(T)$ will carry a structure of $R(\BU)$-module.

\subsection{Motives, $\BU$-modules and noncommutative motives}
\label{Section:motives}

\begin{context}
\label{context3333}
Throughout this section $S:=\Spec\, A$ is any affine scheme
\end{context}

By \cite{ MR3281141} we have
a symmetric monoidal $\s$-category $\rmSH_S^\otimes$, 
which is an $\s$-categorical version of Morel-Voevodsky's
stable homotopy category of schemes over $S$ \cite{voevodsky-morel}. We let $\rmSm_S$ be the category of 
smooth  schemes over $S$. It is a symmetric monoidal
category for the cartesian product. By definition, 
$\rmSH_S$ is a presentable stable symmetric monoidal $\s$-category together with a
symmetric monoidal $\s$-functor
$$\Sigma_+^\infty: \rmSm_S^\times \longrightarrow \rmSH_S^\otimes$$
universal with respect to the following properties (see 
\cite[Cor. 1.2]{MR3281141}) :

\begin{enumerate}

\item The image of an elementary Nisnevich square in $\rmSm_S$ is a pushout square in $\rmSH_S$.\footnote{\label{17ago1242} One could also start with smooth and affine schemes over $S$ instead of $\rmSm_S$. The Nisnevich topology defined only for affine schemes agrees with the usual Nisnevich topology. See \cite[2.3.2]{MR3748687}, \cite[A.2]{1711.03061}}

\item (Homotopy invariance) The natural projection
$\mathbb{A}^{1}_S \longrightarrow S$ is sent to an equivalence.

\item (Stability) Let $S \longrightarrow \mathbb{P}^{1}_S$ be the point at infinity and consider its image in $\rmSH_S$. The cofiber of this map in $\rmSH_S$, denoted as $(\mathbf{P}^1_S, \infty)$, is $\otimes$-invertible.

\end{enumerate}

\begin{notation}
In the following, we will denote by $1_{S} \in \rmSH_S$ the unit of the tensor structure in $\rmSH_S^\otimes$. As an object in $\rmSH_S$, this is equivalent to the tensor product of the topological circle $S^1$ and the algebraic circle $\mathbb{G}_{m,S}$. We will also be using the standard notation $1_S(1):= (\mathbf{P}^1_S, \infty)[-2]= \Omega \, \mathbb{G}_{m,S}$, and $(-)(d):= (-)\otimes 1_S(1)^{\otimes d}$ for the motivic Tate $d$-twist, $d\in \mathbb{Z}$, where, as usual, we denote by $[1]$ the shift given by smashing with the topological circle $S^1$.
\end{notation}

\subsubsection{$\BU_S$ and non-commutative motives}
There exists an object  $\BU_S\in \rmSH_S$ representing homotopy invariant algebraic $\K$-theory of Weibel \cite{weibel-homotopyinvariantktheory, cisinski-descentpar}, in the sense that for any smooth scheme $Y$ over $S$,  the hom-spectrum
$$
\Map_{\rmSH_S}(\Sigma^\infty_+ Y, \BU_S)\simeq \mathrm{KH}(Y)$$
\noindent is the (non-connective) spectrum
of homotopy invariant algebraic $\K$-theory of $Y$ \footnote{\label{17ago1244}This motive is usually denoted as $\mathrm{KGL}_S$ but we will denote it here as $\BU_S$, inspired by the topological analogy.} . The relation between the motive $\BU_S$ and the theory of non-commutative motives was studied in \cite{MR3281141}. Let us briefly recall it. First of all, to every $Y\in \rmSm_S$  we can assign a \emph{dg-category over $S$}, $\Perf(Y)$, of perfect complexes on $Y$. This dg-category is of \emph{finite type} in the sense of \cite{toen-vaquie}. This assignment can be organized into a symmetric monoidal  $\infty$-functor
$$\Perf: \rmSm_S^\times \to \dgcat^\mathrm{idem, ft, op, \otimes}_S$$ where $\dgcat^\mathrm{idem, ft, op, \otimes}_S$ denotes the monoidal full $\infty$-subcategory of $\dgcat^\mathrm{idem, op, \otimes}_S$ consisting of $S$-dg-categories of finite type.
One can mimic the construction of motives starting from the theory of dg-categories. More precisely, one constructs a presentable stable symmetric monoidal $\s$-category $\rmSHNC_S^\otimes$ together with a symmetric monoidal functor 
$$\iota: \dgcat^\mathrm{idem, ft, op, \otimes}_S\to \rmSHNC_S^\otimes$$
satisfying a universal property analogous to (1),(2), (3) above for $\rmSH$, namely:
\begin{enumerate}[(1')]
\item every \emph{Nisnevich square of dg-categories} (see \cite[Section 33.1]{MR3281141}) is sent to a pushout diagram;
\item  the pullback along the canonical projection $\Perf(S)\to\Perf(\mathbb{A}^1_S)$ is sent to an equivalence;
\item the image of the cofiber $\Perf((\mathbb{P}_S^1, \infty))$ is $\otimes$-invertible. 
\end{enumerate}
More concretely, the objects of $\rmSHNC_S$ can be identified with functors $\dgcat^\mathrm{idem, ft, \otimes}_S\to \Sp$ satisfying the conditions (1'),(2'), (3'). Moreover, from the universal property of $\rmSH_S^\otimes$, one then obtains a symmetric monoidal $\s$-functor
$$
\mathrm{R}_{\Perf}: \rmSH_S^\otimes\to \rmSHNC_S^\otimes
$$
informally defined by sending a motive $Y$ to the motive of its dg-category $\Perf(Y)$. 
For formal reasons, this admits a lax monoidal right adjoint $\mathcal{M}_S^\otimes$. By \cite[Theorem 1.8]{MR3281141} this adjoint sends the image of the tensor unit in $\rmSHNC_S^\otimes$ to the object $\BU_S$, thus endowing it with a structure of commutative algebra in the $\infty$-category $\rmSH_S$ \footnote{See also \cite[Section 5.2]{gepner2009motivic}), \cite{1711.03061}, \cite{MR3385689} for the discussion on $\mathrm{E}_\infty$-algebra structures on $\BU_S$.}. Formal reasons then imply that $\mathcal{M}_S^\otimes$ factors as a lax monoidal functor via the theory of $\BU$-modules
$$\rmSHNC_S^\otimes \to \Mod_{\BU_S}(\rmSH_S)^\otimes$$
which we will again denote as $\mathcal{M}_S^\otimes$.\\

\subsubsection{Algebraic Bott periodicity}
\label{rem-algebraicbott}
The object $\BU_S$ reflects the projective bundle theorem in algebraic K-theory in the form of a periodicity given by the Bott isomorphism
\begin{equation}
\label{eq-bottfirstequivalence}
\xymatrix{
\ar[r]^-{\sim} \BU_S & \mathbf{R}\mathrm{Hom}_{\rmSH_S}((\mathbb{P}_S^1,\infty), \BU_S)\simeq  \BU_S(-1)[-2] 
}
\end{equation}
One can find this as a consequence of the fact that the non-commutative motive of $(\mathbb{P}_S^1,\infty)$ is a tensor unit (see \cite[Lemma 3.25]{MR3281141}):
$$
\xymatrix{
\BU_S\simeq \mathcal{M}_S(\mathbf{R}\mathrm{Hom}_{\rmSHNC_S}( 1_S^{\mathrm{nc}} , 1_S^{\mathrm{nc}}))\simeq \mathcal{M}_S(\mathbf{R}\mathrm{Hom}_{\rmSHNC_S}(\mathrm{R}_{\Perf}(\mathbb{P}_S^1,\infty), 1_S^{\mathrm{nc}}))\ar[r]^-{\sim}&
}
$$
$$\xymatrix{
\mathbf{R}\mathrm{Hom}_{\rmSH_S}((\mathbb{P}_S^1,\infty), \mathcal{M}_S(1_S^{\mathrm{nc}})) \simeq 
\BU_S(-1)[-2] 
}
$$
Notice that as all the functors used here are lax monoidal and the tensor unit in non-commutative motives has a unique structure of commutative algebra object, the map in the equivalence (\ref{eq-bottfirstequivalence}) is in fact $\BU_S$-linear. Therefore, it is completely determined by a map in $\rmSH_S $
\begin{equation}
\label{eq-uinverseofBott}
\nu:1_{S}(1)[2]\to \BU_S
\end{equation}
\noindent which, unwinding the argument in the proof of \cite[Lemma 3.25]{MR3281141} one sees, corresponds to the element $\nu=:[\mathcal{O}(1)]-[\mathcal{O}]$ in $\tilde{K}_0(\mathbb{P}^1_S)$. 
Moreover, as $(\mathbb{P}_S^1,\infty)\simeq 1_{S}(1)[2]$ is $\otimes$-invertible, we can tensor (\ref{eq-bottfirstequivalence}) on both sides by $1_{S}(1)[2]$ and obtain
\begin{equation}
\label{eq-bottperiodicity}
\xymatrix{\BU_S(1)[2] \ar[r]^-{\sim} &\BU_S}
\end{equation}
 The map (\ref{eq-bottperiodicity}) corresponds to the composition
$$
\xymatrix{
\ar[r]^-{Id\otimes \nu}\BU_S(1)[2] & \BU_S\otimes \BU_S\ar[r]& \BU_S
}
$$ 
where the last map is the multiplication map of the commutative algebra structure on $\BU_S$. The element $\nu$ is invertible with inverse corresponding to a map
\begin{equation}
\label{eq-bottelement}
\nu^{-1}=\beta: 1_S(-1)[-2] \to \BU_S
\end{equation}
To conclude, let us remark that the Bott periodicity of \ref{eq-bottfirstequivalence} can now be extended to any $\BU_S$-modules $\mathrm{M}$; indeed, we have equivalences of $\BU_S$-modules
\begin{equation}
\mathbf{R}\mathrm{Hom}_{\rmSH_S}((\mathbb{P}_S^1,\infty), \mathrm{M})\simeq  \mathbf{R}\mathrm{Hom}_{\BU_S}((\mathbb{P}_S^1,\infty)\otimes \BU_S, \mathrm{M})\simeq \mathbf{R}\mathrm{Hom}_{\BU_S}( \BU_S, \mathrm{M})\simeq \mathrm{M}
\end{equation}
\noindent or, by duality
\begin{equation}
\label{eq-bottfirstequivalenceanymodule}
\mathrm{M}\simeq \mathrm{M}(1)[2]
\end{equation}

\medskip

\begin{rem}
\label{rem-compactgeneratorsBUmodules}
Recall (for instance, from \cite[Prop. 5.3.3]{robalo-thesis}) that  $\rmSH_S$ is compactly generated by the family of objects $(\mathbb{P}_S^1,\infty)^{-n}\otimes\Sigma^\infty_+(Y)$ with $Y$ smooth over $S$ and $n\in \mathbb{N}$. Furthermore, $\Mod_{\BU_S}(\rmSH_S)$ is compactly generated by the objects of the form $(\mathbb{P}_S^1,\infty)^{-n}\otimes\Sigma^\infty_+(Y)\otimes \BU_S$ which by (\ref{eq-bottperiodicity}) are  equivalent to $\Sigma^\infty_+(Y)\otimes \BU_S$.
\end{rem}

\vspace{0.5cm}

\subsection{The realization of dg-categories as $\BU$-modules}
\label{Section-motivedgcategoryI}

Throughout this section we work under the Context \ref{context3333}. The construction $\mathcal{M}_S^\otimes$ gives us a way to assign a motive to a dg-category of finite type via the composition with the universal map $\dgcat^{\mathrm{idem, ft, op,\otimes}}_S\to \rmSHNC_S^\otimes$. We will use it to produce a more interesting assignment. By construction, $\rmSHNC_S^\otimes$ is a stable presentable symmetric monoidal $\infty$-category. As such, it admits internal-hom objects $\mathbf{R}\mathrm{Hom}_{\rmSHNC}$ and in particular, there exists an $\infty$-functor
$$\mathbf{R}\mathrm{Hom}_{\rmSHNC}(-,1^\mathrm{nc}_S):\rmSHNC^{\mathrm{op}}\to \rmSHNC$$
where $1^\mathrm{nc}_S$ is the tensor unit. Of most importance to us is the fact that this functor can be endowed with a lax monoidal structure \cite[5.2.2.25, 5.2.5.10, 5.2.5.27 ]{lurie-ha}
$$\mathbf{R}\mathrm{Hom}_{\rmSHNC}(-,1^\mathrm{nc}_S):\rmSHNC_S^{\mathrm{op},\otimes}\to \rmSHNC_S^{\otimes}$$
The composition 
$$\xymatrix{\dgcat^{\mathrm{idem, ft, \otimes}}_S\ar[r]^-{\iota} & \rmSHNC_S^{\mathrm{op},\otimes}\ar[rrr]^-{\mathbf{R}\mathrm{Hom}_{\rmSHNC}(-,1^\mathrm{nc}_S)} & & &\rmSHNC_S^\otimes }$$
is lax monoidal. We now recall that dg-categories of finite type generate all the Morita theory of dg-categories under filtered colimits i.e. $\mathrm{Ind}(\dgcat^{\mathrm{idem, ft}}_S)^\otimes\simeq \dgcat^{\mathrm{idem, \otimes}}_S$ \footnote{See \cite{toen-vaquie} and the $\infty$-categorical narrative in \cite[6.1.27]{robalo-thesis}}. As $\rmSHNC_S$ is presentable,  we obtain via the universal property of the convolution product in $\mathrm{Ind}$-objects \cite[4.8.1.10]{lurie-ha}, an induced lax symmetric monoidal functor 

\begin{equation}
\label{eq-dualityfunctor}
\xymatrix{
\dgcat^{\mathrm{idem, ft, \otimes}}_S \ar[r]\ar@{^{(}->}[d]& \rmSHNC_S^{\otimes}\\
\dgcat^{\mathrm{idem, \otimes}}_S \ar@{-->}[ru]_{\mu_S^\otimes}&
}
\end{equation}

Informally, if $T$ is an $S$-dg-category of finite type,  and we write again $T$ for its image in $\rmSHNC_S$, the object $\mu^{\otimes}_S(T)=\mathbf{R}\mathrm{Hom}_{\rmSHNC}(T,1^\mathrm{nc}_S)$ can be described \footnote{It is helpful to remind the reader that $\rmSHNC_S$ can be constructed as a localization of the of presheaves on $(\dgcat_A^\mathrm{idem, ft})^{\textrm{op}}$ with values in the $\infty$-category of spectra $\Sp$, by forcing Nisnevich descent and $\Perf(\mathbb{A}^1_S)$-invariance - see \cite[Section 33]{MR3281141}} as the $\infty$-functor sending a dg-category of finite type $T'$ over $S$ to the mapping spectrum 
$$\Map_{\rmSHNC_S}(T', \mathbf{R}\mathrm{Hom}_{\rmSHNC}(T,1^\mathrm{nc}_S)) \simeq \Map_{\rmSHNC_S}(T'\otimes_S T, 1^\mathrm{nc}_S)$$
which following \cite[Theorem 1.8 (ii) and Cor 4.8]{MR3281141} is the spectrum of homotopy invariant $\mathrm{K}$-theory 
$$\mathrm{KH}(T'\otimes_S T)$$
More generally, for $T\in \dgcat_S^\mathrm{idem}$ we can write $T$ as a filtered colimit of dg-categories of finite type and as $\mathrm{KH}$ commutes with filtered colimits (this is well known but see \cite[Prop. 2.8]{Anthony-thesis}), and the same holds for tensor product of dg-categories, we conclude that $T$ is sent to the object in $\rmSHNC_S$, defined by the $\infty$-functor $\mathrm{KH}(-\otimes_S T)$.\\

We will denote by $\mathcal{M}^\vee_S$ the composition of the lax monoidal functors 
\begin{equation}
\label{equation-definition-Msdual}
\xymatrix{\dgcat^{\mathrm{idem, \otimes}}_S \ar[r]^-{\mu_S^\otimes} &\rmSHNC_S^{\otimes}\ar[r]^-{\mathcal{M}_S^\otimes}& \Mod_{\BU_S}(\rmSH_S)^\otimes}
\end{equation}

\begin{rem}
Notice that as $\mathcal{M}_S$ commutes with filtered colimits, this composition is also the functor obtained by the monoidal universal property of the $\mathrm{Ind}$-completion. To see that $\mathcal{M}_S$ commutes with filtered colimits it is enough to test on compact generators \cite[Prop. 5.3.3 and 6.4.24]{robalo-thesis} and use the fact $\mathrm{R}_{\Perf}$ preserve compact generators.
\end{rem}

By the definition of $\mathcal{M}_S$ as a right adjoint to $\mathrm{R}_{\Perf}$ and following the previous discussion, the motive $\mathcal{M}_S^\vee(T)$ in $\rmSH_S$ represents the $\infty$-functor sending a smooth scheme $X$ over $S$ to the spectrum $\mathrm{KH}(\Perf(X)\otimes_S T)$.

\begin{cor}
\label{cor-Mpreservesexactsequences}
$\mathcal{M}_S^{\vee}$ sends exact sequences of dg-categories  to cofiber-fiber sequences in the stable $\infty$-category of $\BU_S$-modules. 
\begin{proof}
Indeed, this follows because cofiber sequences of dg-categories are stable under tensor products and because homotopy K-theory sends exact sequences of dg-categories (see the discussion in \cite[Section 1.5.4]{MR3281141}) to cofiber-fiber sequences in spectra (see the details in \cite[Prop. 417 and Prop. 3.19]{MR3281141}).
\end{proof}
\end{cor}

\begin{rem}
In \cite{MR2822869,tabuada-higherktheory, tabuada-cisinski} Cisinski and Tabuada introduced an alternative category of non-commutative motives $\mathcal{M}_{Tab}^\mathrm{Loc, \otimes}$ which is dual to the one used here.  Indeed, there is a duality blocking a direct comparison between $\mathcal{M}_{Tab}^\mathrm{Loc, \otimes}$ and $\rmSH^\otimes$. The category $\rmSHNC^\otimes$ was designed to avoid this obstruction (see \cite[Appendix A]{MR3281141} for the comparison between the two approaches). It is exactly this duality that we encode in our construction of motivic realizations of dg-categories via the functor $\mu^\otimes$. As in \cite[Appendix A]{MR3281141}, let $\mathcal{M}_{Tab}^\mathrm{Nis, \otimes}$ be the Nisnevich version of the construction of Tabuada-Cisinski. Then, by definition of $\mathcal{M}_{Tab}^\mathrm{Nis, \otimes}$, the functor $\mu^\otimes$ of the diagram (\ref{eq-dualityfunctor}) factors in a unique way as a lax monoidal functor 
\begin{equation}
\label{eq-compTabuada}
\xymatrix{
\dgcat^{\mathrm{idem, ft, \otimes}}_S \ar[r]^-{\mu_S^\otimes}\ar[d]& \rmSHNC_S^{\otimes}\\
 \mathcal{M}_{Tab}^\mathrm{Nis, \otimes} \ar@{-->}[ru]&
}
\end{equation}
This exhibits $\mathcal{M}_{Tab}^\mathrm{Nis, \otimes}$ as a universal motivic realization of dg-categories (see next section for more on realizations).
\end{rem}

\subsection{The six operations in $\BU$-modules and realizations}
\label{section-sixopBUmod}

\begin{context}
\label{context2222}
Throughout this section $S:=\Spec\, A$  is an henselian trait and  $\bsch$  denotes the category of Noetherian $S$-schemes of finite Krull dimension. See \cite[2.0]{cisinski-tcmm}  and \cite[2.0, footnote 35]{cisinski-tcmm}.
\end{context}

Before continuing towards our main goals we will need to discuss some functorial aspects. In the last two section we worked with motives over a fixed base $S$. It is possible to work in a relative setting. For every $X\in \bsch$ we can construct a stable presentable symmetric monoidal $\s$-category $\rmSH_X^\otimes$ encoding the motivic homotopy theory of Morel-Voevodsky over $X$. Moreover, we can make the assignment $X\mapsto \rmSH_X^\otimes$ functorial in $X$, given by an $\s$-functor
$$\rmSH^\otimes:\bsch^{\mathrm{op}}\to \CAlg(\Prlstb)$$ 
where $\Prlstb$ denotes the $\infty$.category of stable presentable $\infty$-categories with colimit preserving functors, and $\CAlg(\Prlstb)$ is the $\infty$-category of  \cite[4.8]{lurie-ha} symmetric monoidal stable presentable $\infty$-categories such that the tensor product preserves colimits in each argument. This is done in \cite[Section 9.1]{robalo-thesis}. The $\s$-functor $\rmSH_X^\otimes$ comes together with a more complex system of functorialities encoding the six operations of Grothendieck (see Appendix \ref{Section:formalismsixoperations}).

We will be interested in several \emph{motivic realizations}. For us, a motivic realization consists of an $\infty$-functor 
$$\mathrm{D}^\otimes:\bsch^{\mathrm{op}}\to \CAlg(\Prlstb)$$
enriched with a system of six operations, plus the data of a monoidal natural transformation 
$$\rmSH^\otimes \to \mathrm{D}^\otimes$$
and a system of compatibilities between the systems of six operations on $\rmSH_X^\otimes$ and  $\mathrm{D}^\otimes$ (see  Prop. \ref{prop-compsixoperations}).  Of major importance to us are the \emph{étale} and the \emph{$\ell$-adic} realizations which we will explore later in this section.

\begin{cons}\label{17ago1246}(Operations $f^\ast, f_\ast$ on $\rmSHNC$)
In \cite[Chapter 9]{robalo-thesis} it is shown that the theory of non-commutative motives  admits relative versions encoded by an $\infty$-functor
$$\rmSHNC^\otimes: \bsch^\mathrm{op}\to \CAlg(\Prlstb)$$
In the affine case, given a map $f: \Spec( R)\to \Spec(R')$, $f^\ast$ is induced by the map $f^*(T):= T\otimes_R \Perf(R')$ from $R$ to $R'$-dg-categories. In the general non-affine case, it is obtained by Kan extension from  $\rmSHNC^\otimes$ defined for (underived) affine schemes, namely by the formula
$$
\rmSHNC^\otimes(X):=\lim_{\Spec(A)\to X}\rmSHNC^\otimes(A)
$$
This way, for any map of schemes $f:X\to Y$ we get functorialities $(f^\ast, f_\ast) $, where $f^\ast$ is obtained by by Kan extension, and $f_\ast$ by the Adjoint functor theorem. Moreover, $(\rmSHNC^\otimes, (-)^\ast)$ is known to be a Zariski sheaf \cite[Chapter 9]{robalo-thesis}. Also by Kan extension, we get a natural transformation
$$\mathrm{R}_{\Perf}: \rmSH^\otimes \to \rmSHNC^\otimes$$ compatible with pullbacks. See also \cite{1711.03061} \footnote{Notice however that it is not known if the necessary localization and proper base change properties hold for $\rmSHNC^\otimes$  and therefore it is not known if it admits the six operations. In any case this won't be necessary in this paper.}.
\end{cons}

\medskip

Another important example is that of $\BU$-modules. For each $X\in \bsch$ there exists a commutative algebra object $\BU_X\in \CAlg(\rmSH_X)$ representing a relative version of homotopy invariant algebraic $\mathrm{K}$-theory. This commutative algebra structure can also be obtained from a relative version of the results in \cite{MR3281141} which the reader can consult in \cite[Chapter 9]{robalo-thesis}. For this, it is crucial that these relative versions $\BU_{-}$ are compatible under pullbacks (see \cite[3.8]{cisinski-descentpar}). This allows us to construct an $\infty$-functor
$$\Mod_{\BU}(\rmSH)^\otimes:\bsch^\mathrm{op}\to \CAlg(\Prlstb) \, : (X/S) \longmapsto \Mod_{\BU_X}(\rmSH_X)$$
together with a natural transformation
$$ - \otimes \BU: \rmSH^\otimes \to \Mod_{\BU}(\rmSH)^\otimes$$
which for each $X\in \bsch$ admits a conservative right adjoint $\Mod_{\BU_X}(\rmSH_X)\to \rmSH_X$ that forgets the module structure. 
As explained in \cite[13.3.3]{cisinski-tcmm} (see also the discussion in \cite[pg 260, 9.4.38, 9.4.39]{robalo-thesis}), the conservativity of the forgetful functor and the fact it commutes with the functorialities $(-)_*$, $(-)^*$ and $(-)_\sharp$, and verifies the projections formulas, are enough to deduce the conditions endowing $\Mod_{\BU}(\rmSH)^\otimes$ with a system of six operations (see Prop. \ref{prop-Ayoub-Cisinski}), to make the natural transformation $-\otimes \BU$ compatible with the operations in the sense of Prop.\ref{prop-compsixoperations}, and to make the forgetful functor $\Mod_{\BU}\to \rmSH$ compatible with all the operations (meaning that the natural transformations at the end of \ref{prop-compsixoperations} are natural isomorphisms; see \cite[Section 7.2]{cisinski-tcmm}).

\begin{rem}
\label{rem-smoothbottperiodicitysharpimage}
Notice that the algebraic Bott isomorphism of  \ref{rem-algebraicbott} forces  the functorialities $(-)_\sharp$ and $(-)_!$ to be the same for smooth maps (see equation (\ref{eq-smoothsharpisshriek}) in Appendix A).\\
\end{rem}

To conclude this preliminary section, we must also remark that if $\mathrm{R}^\otimes :\rmSH^\otimes \to \mathrm{D}^\otimes$ is a motivic realization, being monoidal, it preserves algebra-objects and thus sends $\BU$ to an algebra object $\mathrm{R}(\BU)$. Therefore, it produces a new realization
$$\mathrm{R}_{\Mod}^\otimes: \Mod_{\BU}(\rmSH)^\otimes \to \Mod_{\mathrm{R}(\BU)}(\mathrm{D})^\otimes$$
Throughout the next sections we will analyse several realizations of dg-categories, all obtained by pre-composition with $\mathcal{M}_S^{\vee, \otimes}$
$$\xymatrix{\dgcat_S^{\mathrm{idem, \otimes}} \ar[r]^-{\mathcal{M}_S^{\vee, \otimes}} & \Mod_{\BU_S}(\rmSH_S)^\otimes \ar[r]^-{\mathrm{R}_{\Mod,S}^\otimes} & \Mod_{\mathrm{R}(\BU_S)}(\mathrm{D}_S)^\otimes}$$

\begin{rem}
By Corollary \ref{cor-Mpreservesexactsequences}, every realization of dg-categories sends exact sequences to exact sequences.
\end{rem}

\begin{ex}
\label{ex-topologicalKtheory}
When the base ring is $\mathbb{C}$, we have a \emph{Betti realization}
$\rmSH_\mathbb{C}^\otimes\to \Sp^\otimes$. In \cite[Section 4.6]{topk} it is shown that the realization of $\BU_\mathbb{C}$ is  the spectrum representing topological $K$-theory $\BU^{\mathrm{top}}$. The composite realization $$\dgcat_S^{\mathrm{idem}} \to \Mod_{\BU_S}(\rmSH_S)\to \Mod_{\BU^{\mathrm{top}}}(\Sp)$$ recovers what in \cite{topk} is called the \emph{topological K-theory} of dg-categories. 
\end{ex}

\subsection{The $\BU$-motives of $\Perf$, $\Coh$ and $\Sing$.}
\label{section-motivesPerfCohgeneralities}
 
\subsubsection{Motive of $\Perf$}
\label{subsubsection-eq-XschemeS}

\begin{context}
\label{context5555}
Throughout this section $S:=\Spec\, A$ is an affine scheme and \begin{equation}
\label{eq-XschemeS}
\xymatrix{
X\ar[d]^p\\
S
}
\end{equation}
is a quasi-compact quasi-separated $S$-scheme.
\end{context}

The dg-category $\Perf(X)$ is an object in $\dgcat_S^\otimes$ and via the construction explained in Section \ref{Section-motivedgcategoryI} it produces a $\BU_S$-module $\mathcal{M}_S^\vee(\Perf(X))$. At the same time, $\BU_X$ is an object in $\rmSH_X$ and, as in the previous section, we can consider its direct image $p_*(\BU_X)\in \Mod_{\BU_S}(\rmSH_S)$.

\begin{prop}
\label{prop-descriptionmotivedgcategoryviapushforward}
Assume $p:X\to S$ as in the Context \ref{context5555}. Then the two objects $\mathcal{M}_S^\vee(\Perf(X))$ and $p_*(\BU_X)$ are canonically equivalent as $\BU_S$-modules.
\begin{proof}
The first ingredient is the fact that by \cite[Theorem 1.8]{MR3281141} and its extension to general basis in \cite[Cor. 9.3.4]{robalo-thesis}, we have canonical equivalences
$$p_*(\BU_X)\simeq p_*(\mathcal{M}_X(1^\mathrm{nc}_X))$$
By formal adjunction reasons, $\mathcal{M}$ and $(-)_*$ are compatible, so that $p_*(\mathcal{M}_X(1^\mathrm{nc}_X))$ is canonically equivalent to $\mathcal{M}_S(p_*(1^\mathrm{nc}_X))$, where now $p_*$ denotes the direct image functoriality in $\rmSHNC$. 
By definition of $\mathcal{M}_S^\vee$ (\ref{equation-definition-Msdual}), we are reduced to show that $p_*(1^\mathrm{nc}_X)$ is equivalent to the object in $\rmSHNC_S$ given by $\mu_S(\Perf(X))$ (where $\mu_S$ is defined in diagram (\ref{eq-dualityfunctor}) ). 
Unwinding the adjunctions, the first corresponds to the $\s$-functor sending an  $S$-dg-category of finite type $T$ to the homotopy $\mathrm{K}$-theory spectrum $\mathrm{KH}_X(p^*(T))$ where $p^\ast(T)$ is the pullback of $T$ seen as object of $\rmSHNC_S$. The second corresponds to the $\infty$-functor sending a dg-category of finite type $T$ to the spectrum $\mathrm{KH}_S(T\otimes_S\Perf(X))$.\\
The case where $X=\Spec \, B$ is an affine scheme over $S=\Spec \, A$, the equivalence between the two follows from the argument in the proof of \cite[Prop. 10.1.4]{robalo-thesis}\footnote{Notice that the arguments in \cite[Prop. 10.1.4]{robalo-thesis} are written for $p$ a closed immersion but in fact work in the general map between affines. } which exhibit a natural equivalence between the underlying Waldhausen's S-constructions via the projection formula
$$
\mathrm{KH}_S(T\otimes_S\Perf(X))\simeq \mathrm{KH}_X(p^*(T))
$$

We now deduce the general case from the affine case using the Zariski descent property for $X\mapsto \rmSHNC_X$ of \cite[9.21]{robalo-thesis} and Zariski descent for homotopy K-theory.  Indeed, the constructions
$$
X=\Spec(B)\mapsto p_*(1^\mathrm{nc}_B) \text{   and  } X=\Spec(B)\mapsto \mu_S(\Perf(B))
$$
\noindent are Zariski sheaves on the category of affine schemes over $S$ with values in $\rmSHNC_S$ and by the argument above we just constructed a natural isomorphism between them. Now any Zariski sheaf on $\mathrm{Sch}/S$ is the right Kan extension of its restriction to affines $\mathrm{affSch}/S$. Therefore, for non-affine $X$'s the result follows by Kan extension using the Zariski descent property for $\rmSHNC$ and for homotopy K-theory of dg-categories, implying that
$$
X\mapsto p_*(1^\mathrm{nc}_X) \text{   and  } X\mapsto \mu_S(\Perf(X))
$$
\noindent are Zariski sheaves.
\end{proof}
\end{prop}

\begin{rem}
As a consequence of the six operations for $\BU$-modules, if $p$ is a smooth map, then $p_*\BU_X$ is also equivalent to $\BU_S^X:= \mathbf{R}\mathrm{Hom}_{\rmSH_S}(p_{\sharp} 1_X, \BU_S)$ (by projection formula).
\end{rem}

\begin{rem}
\label{rem-algebrastructuremotiveofperf}
For $X$ as in Proposition \ref{prop-descriptionmotivedgcategoryviapushforward}, $\Perf(X)$ carries a symmetric monoidal structure given by the tensor product of perfect complexes. This can be understood as a commutative algebra object $\Perf(X)^\otimes \in \CAlg(\dgcat_S^{\mathrm{idem}})$. As $\mathcal{M}_S^\vee$ is lax monoidal, $\mathcal{M}_S^\vee(\Perf(X))$ is an object in $\CAlg(\Mod_{\BU_S}(\rmSH_S))$.
\end{rem}

\vspace{0.5cm}

\subsubsection{Motive of $\Coh$ and K-theory with support}
\label{subsubsection-contextopenclosed3}

\begin{context}
\label{context4444}
We now extend Context \ref{context5555} by considering 
\begin{equation}\label{eq-contextopenclosed3}
\xymatrix{ U:=X-Z\, \ar@{^{(}->}[r]^-{j}& X \ar[d]^-{p}&  \ar@{_{(}->}[l]_-{i} \,  Z \\
&S&}
\end{equation}
\noindent with $X$ a \emph{regular scheme}, quasi-compact quasi-separated, $i$ a closed immersion and $j$ its open complementary. $S$ is again any affine scheme.
\end{context}

\medskip

 It follows also from Prop. \ref{prop-descriptionmotivedgcategoryviapushforward} that $(p\circ j)_* \BU_U$ is equivalent to $\mathcal{M}_S^\vee(\Perf(U))$ 
where $\Perf(U)$ is seen as an $S$-dg-category via the composition $p\circ j$. In the same way we have that $(p\circ i)_* \BU_Z\simeq \mathcal{M}_S^\vee(\Perf(Z))$. Moreover,  pullback along $j$ produces a map of $\BU_S$-modules $\mathcal{M}_S^\vee(j^*): \mathcal{M}_S^\vee(\Perf(X))\to \mathcal{M}_S^\vee(\Perf(U))$ which via the equivalence of Prop. \ref{prop-descriptionmotivedgcategoryviapushforward} is identified with the map induced by the unit of the adjunction
\begin{equation}\label{eq:BUofXandBUofU}
p_*(\BU_X) \to p_* j_* j^* \BU_X
\end{equation}
This morphism fits into an exact sequence in $\Mod_{\BU_S}\rmSH_S$
\begin{equation}
\label{eq-localizationexactsequence}
p_* i_* i^{!} \BU_X \to p_* \BU_X \to p_* j_* j^* \BU_X
\end{equation}
\noindent consequence of the localization property of  \cite{voevodsky-morel} (see \ref{prop-Ayoub-Cisinski} and \ref{prop-AdeelAyoubCisinski} in Appendix A for references of this localization property and the six operations for $\rmSH$).

\vspace{0.5cm}
As $X$ is assumed to be regular, and $U$ is open, $U$ is also regular and we have 
$$\Perf(X)=\Coh(X) \text{    and    } \Perf(U)=\Coh(U)$$
and by the Prop. \ref{prop-descriptionmotivedgcategoryviapushforward},
$$p_* \BU_X \simeq \mathcal{M}_S^\vee(\Perf(X))\simeq \mathcal{M}_S^\vee(\Coh(X)) $$
\medskip
$$p_* j_* j^* \BU_X \simeq \mathcal{M}_S^\vee(\Perf(U))\simeq \mathcal{M}_S^\vee(\Coh(U)) $$
\medskip
In this case one can write (\ref{eq-localizationexactsequence}) as
\begin{equation}
\label{eq-localizationexactsequence222}
\xymatrix{p_* i_* i^{!} \BU_X \ar[r]& p_* \BU_X \ar[r] &  p_* j_* j^* \BU_X\\
& \ar[u]^\sim  \mathcal{M}_S^\vee(\Coh(X)) \ar[r]^{\mathcal{M}_S^\vee(j^\ast)} &  \ar[u]^\sim \mathcal{M}_S^\vee(\Coh(U)) }
\end{equation}
Finally, regarding the dg-category $\Coh(Z)$ of bounded coherent complexes on $Z$ as a $S$-dg-category via the composition $p\circ i$, the composition of dg-functors
$$
\xymatrix{\Coh(Z)\ar[r]^{i_\ast}&\Coh(X)\ar[r]^{j^\ast}& \Coh(U)}
$$
\noindent is the zero functor. In this case we find  a canonical factorization in $\BU_S$-modules
\begin{equation}
\label{eq-localizationexactsequence2223}
\xymatrix{p_* i_* i^{!} \BU_X \ar[r]& p_* \BU_X \ar[r] &  p_* j_* j^* \BU_X\\
\ar@{-->}[u]  \mathcal{M}_S^\vee(\Coh(Z)) \ar[r] & \ar[u]^\sim  \mathcal{M}_S^\vee(\Coh(X)) \ar[r] &  \ar[u]^\sim \mathcal{M}_S^\vee(\Coh(U)) }
\end{equation}

\begin{prop}
\label{prop:BU-motiveofclosedisCoh}
Assume the conditions and notations as in the Context \ref{context4444}. Then the canonical map of $\BU_S$-modules
\begin{equation}
\label{eq-canonicalmapsupprtlasttime}
\xymatrix{\mathcal{M}_S^\vee(\Coh(Z))\ar@{-->}[r]&p_* i_* i^{!} \BU_X}
\end{equation}
is an equivalence
\begin{proof}
Following the Remark \ref{rem-compactgeneratorsBUmodules}, the collection of objects of the form $\Sigma^\infty_+(Y)\otimes \BU_S$ with $Y$ smooth over $S$, forms a family of compact generators for $\Mod_{\BU_S}(\rmSH_S)$. Therefore, to show that (\ref{eq-canonicalmapsupprtlasttime}) is an equivalence in $\BU_S$-modules,  it is enough to show that the composition map
$$
\Map_{\BU_S}(\Sigma^\infty_+(Y)\otimes \BU_S, \mathcal{M}_S^\vee(\Coh(Z)))\to \Map_{\BU_S}(\Sigma^\infty_+(Y)\otimes \BU_S, p_* i_* i^{!} \BU_X)
$$
\noindent is an equivalence for every $Y$ smooth over $S$. By definition 
of $\mathcal{M}_S^\vee(\Coh(X))$ and $\mathcal{M}_S^\vee(\Coh(U))$, we have an identification of mapping spectra
\begin{equation}\label{eq-humannaturejackson}
\Map_{\BU_S}(\Sigma^\infty_+(Y)\otimes \BU_S, \mathcal{M}_S^\vee(\Coh(X)))\simeq \mathrm{KH}(\Perf(Y)\otimes_S \Coh(X))
\end{equation}
\begin{equation}
\label{eq-humannaturejackson2}
\Map_{\BU_S}(\Sigma^\infty_+(Y)\otimes \BU_S, \mathcal{M}_S^\vee(\Coh(U)))\simeq \mathrm{KH}(\Perf(Y)\otimes_S \Coh(U))
\end{equation} 
Let now $X$ be an $S$-scheme of finite type and let $Y$ be smooth over $S$. We have an equivalence of $S$-dg-categories
\begin{equation}\label{eq-cohperf=coh}
\Coh(X)\otimes_S \Perf(Y) \simeq \Coh(X\times_S Y)
\end{equation}
This is \cite[Prop. B.4.1]{1101.5834} together with the fact that as $Y$ is smooth over $S$ and $S$ is assumed to be regular, $Y$ is regular and therefore $\Perf(Y)=\Coh(Y)$\personal{Here in fact we need to assume S is excellent}. The equivalence (\ref{eq-cohperf=coh}) holds for $X$ and also for both $U$ and $Z$ \footnote{Notice that this works without any hypothesis on the regularity of $Z$.}. Using (\ref{eq-cohperf=coh}), (\ref{eq-humannaturejackson}) is equivalent to $\mathrm{KH}(\Coh(X\times_S Y))$ which is equivalent to the $\mathrm{G}$-theory spectrum of $X\times_S Y$ by $\mathbb{A}^1$-invariance of G-theory. Mutatis-mutandis for $U$ and (\ref{eq-humannaturejackson2}).
Therefore, the map \ref{eq:BUofXandBUofU} can be identified with the $\mathrm{G}$-theory pullback along $j$
$$\mathrm{G}(X\times_S - )\to \mathrm{G}(U\times_S -)$$
whose fiber is well-known from Quillen's localization theorem for $\mathrm{G}$-theory \cite[\S 7 Prop 3.2]{MR0338129} to be the homotopy invariant $\mathrm{K}$-theory of the dg-category of bounded coherent sheaves in $Z$,  $\Coh(Z\times_S Y)$ which again by the lemma is equivalent to  $\Perf(Y)\otimes_S \Coh(Z)$ from which the propositions follows.
\end{proof}
\end{prop}

\begin{rem}
\label{rem-Ktheorywithsupport}
Combining the result of the Prop. \ref{prop:BU-motiveofclosedisCoh} with the discussion in \cite[Section 13.4.1]{cisinski-tcmm} one finds that $\mathcal{M}_S^\vee(\Coh(Z))$ can also be described as $\mathrm{K}$-theory with support in $Z$.
\end{rem}

\vspace{0.5cm}

\begin{cor}
\label{newcorreferee2} Assume the Context \ref{context4444}. Then we have a cofiber-fiber sequence of $\BU_S$-modules
\begin{equation}
\label{eq-purity}
p_* i_*\BU_Z\to p_* i_* i^! \BU_X\to \mathcal{M}_S^\vee(\Sing(Z))
\end{equation}
\begin{proof}
Following \ref{cor-Mpreservesexactsequences}, the exact sequence of $S$-dg-categories
$$\Perf(Z) \to\Coh(Z) \to \Sing(Z)$$
creates a cofiber-fiber sequence of $\BU_S$-modules
$$\mathcal{M}_S^\vee(\Perf(Z)) \to \mathcal{M}_S^\vee(\Coh(Z)) \to \mathcal{M}_S^\vee(\Sing(Z))$$
which, thanks to Proposition \ref{prop-descriptionmotivedgcategoryviapushforward} and Proposition \ref{prop:BU-motiveofclosedisCoh} (applied to $p\circ i$), can now be identified with the cofiber-fiber sequence (\ref{eq-purity}).
\end{proof}
\end{cor}

\begin{rem}
\label{rem-algebrastructuremotiveofcohandperf}
Following the Remark \ref{rem-algebrastructuremotiveofperf},  
$$\Perf(Z)\in  \CAlg(\dgcat_S^{\mathrm{idem}})$$
\noindent and therefore 
$$\mathcal{M}_S^\vee(\Perf(Z))\in \CAlg(\Mod_{\BU_S}(\rmSH_S)).$$
Now, since the tensor product of coherent by perfect is coherent, the inclusion $\Perf(Z)\subseteq \Coh(Z)$ makes $\Coh(Z)$ an object in $\Mod_{\Perf(Z)^\otimes}(\dgcat_S^{\mathrm{idem},\otimes})$. In this case $\mathcal{M}_S^\vee(\Coh(Z))$ defines an object in $\Mod_{\mathcal{M}_S^\vee(\Perf(Z))}(\Mod_{\BU_S}(\rmSH_S))$. Moreover, as the inclusion $\Perf(Z)\subseteq \Coh(Z)$ is a map of $\Perf(Z)$-modules, the induced map 
\begin{equation}
\label{fundamentalclassKtheory}
u:p_* i_*\BU_Z\simeq \mathcal{M}_S^{\vee} (\Perf(Z)) \to \mathcal{M}_S^{\vee} (\Coh(Z))\simeq p_* i_* i^! \BU_X
\end{equation}
\noindent is defined in $\Mod_{\mathcal{M}_S^\vee(\Perf(Z))}(\Mod_{\BU_S}(\rmSH_S))$. By adjunction, this is the same as a map $1_S\to \mathcal{M}_S^{\vee} (\Coh(Z))$ in $\rmSH_S$. Whenever $X$ is regular, under the equivalence of the Prop. \ref{prop:BU-motiveofclosedisCoh}, this map corresponds to an element $u$ in the Grothendieck group of $\mathrm{K}$-theory with support $\mathrm{K}_Z(X)$ corresponding to $i_\ast(\mathcal{O}_Z)$. In particular, when $Z$ itself is regular, this element identifies with $\lambda_{-1}$ of the conormal bundle of $Z$ in $X$.  See  \cite[13.4.1, 13.5.4, 13.5.5]{cisinski-tcmm}. We will again discuss this element $u$ in the Remark \ref{remark-fundamentalclasslci}. 
\end{rem}

\vspace{0.5cm}

\begin{rem}
In particular, when $Z$ is itself regular we recover the \emph{purity isomorphism} in algebraic K-theory of \cite[13.6.3]{cisinski-tcmm} or \cite[7.14]{MR3205601}:
\begin{equation}
\label{16ago1554purityalgebraicKtheory}
\xymatrix{
p_*i_*\BU_Z\ar[r]^-{u}_-{\sim}& p_*i_*i^! \BU_X}
\end{equation}
More generally, the motive $\mathcal{M}_S^\vee(\Sing(Z))$ measures the obstruction to purity.
\end{rem}

\vspace{0.5cm}

Combining (\ref{eq-localizationexactsequence}) and (\ref{eq-purity}) we get two cofiber-fiber sequences
\begin{equation}
\label{eq-twocofibersequencesmotivicSIng}
\xymatrix{
p_*i_* \BU_Z \ar[r]^u& p_*i_*i^! \BU_X \ar[d] \ar[r]& \mathcal{M}_S^\vee(\Sing(Z))\\
&p_*\BU_X\ar[d]&\\
&p_*j_* j^* \BU_X&
}
\end{equation}

\subsubsection{Motive of $\Sing$}
\label{subsubsection-contextderivedzerolocus}

\begin{context}
\label{newcontextreferee1}
We now consider $X$ as in Context \ref{context4444}, together with a function $f:X\to \mathbb{A}^1_S$. In this case we get derived fiber products over $S$
\begin{equation}
\label{eq-derivedzerolocusfiberproduct}
\xymatrix{
 \ar @{} [dr] | \lrcorner U \ar@{^{(}->}[r]^j \ar[d] &X\ar[d]^f& \ar[l]_{\mathfrak{i}} \ar[d] X_{0} \ar @{} [dl] | \llcorner  \\
\mathbb{G}_{m,S}\ar@{^{(}->}[r]& \mathbb{A}^1_S&\ar@{_{(}->}[l]_{i_0}S
}
\end{equation}
\noindent where $i_0$ is the zero section, map $\mathfrak{i}$ is an $lci$ closed immersion and $j$ is its open complementary. The classical truncation of this diagram brings us to the setting of diagram (\ref{eq-contextopenclosed3}) with $Z:=t(X_0)$ the classical underived zero locus of $f$. We denote by $i$ the composition
$$
\xymatrix{Z\ar[r]^{t}& X_0\ar[r]^{\mathfrak{i}} &X}
$$
\noindent with $t$ the classical truncation.
\end{context}

\vspace{0.5cm}

\begin{cor}
\label{newcorreferee1}
Consider the notation of Context \ref{newcontextreferee1}. Suppose that the right square in diagram (\ref{eq-derivedzerolocusfiberproduct}) is Tor-independent (i.e. $t(X_0)\simeq X_0)$. Then there is a fiber sequence of $\BU_S$-modules
$$\xymatrix{
p_*i_* i^\ast \BU_X \ar[r]^u & p_*i_*i^! \BU_X  \ar[r]& \mathcal{M}_S^\vee(\Sing(X_0))
}$$
\begin{proof}
Apply Corollary \ref{newcorreferee2}.
\end{proof}
\end{cor}

The next result explains why the hypothesis of tor-independence is not necessary: the $\BU$-motive $\mathcal{M}_S^\vee(\Coh(-))$ is in fact invariant under derived thickenings.

\begin{prop}
\label{prop-invariantunderderivedthickenings}
Let $\tilde{Z}$ be a derived scheme over $S$ with classical underlying scheme $Z$ and canonical closed immersion $Z\hookrightarrow \tilde{Z}$. Then the push forward along the inclusion 
$$\mathcal{M}_S^\vee(\Coh(Z))\to \mathcal{M}_S^\vee(\Coh(\tilde{Z})) $$
is an equivalence of $\BU_S$-modules.

\begin{proof}
Analyzing the definitions, and using the Remark \ref{rem-compactgeneratorsBUmodules} as in the proof of the Prop. \ref{prop:BU-motiveofclosedisCoh},  we are reduced to show that for any smooth scheme  $Y\to S$, the induced map of K-theory spectra
$$\mathrm{KH}(\Perf(Y)\otimes_S \Coh(Z))\to \mathrm{KH}(\Perf(Y)\otimes_S \Coh(\tilde{Z}))$$
is an equivalence. But again, thanks to \cite[Prop. B.4.1]{1101.5834} we have the formula (\ref{eq-cohperf=coh}) so that it is enough to show that the map
$$\mathrm{KH}(\Coh(Y\times_S Z))\to \mathrm{KH}(\Coh(Y\times_S\tilde{Z}))$$
is an equivalence. Here we mean the derived fiber product, which as $Y$ is flat over $S$, equals the usual fiber product.
But now this equivalence follows from the theorem of the heart: for a any derived scheme $V$ with truncation $t(V)$,  $\Coh(t(V))$ and $\Coh(V)$ both carry $t$-structures with the same heart \cite[2.3.20]{lurie-DAGVIII}, so that by \cite{1212.5232}, their K-theory spectra are equivalent via pushforward.\\

\end{proof}
\end{prop}

\vspace{0.5cm}
We will now use the invariance under derived thickenings to give a formula for the motive of $\Sing$. In the setting of the diagram (\ref{eq-derivedzerolocusfiberproduct}), for the derived scheme $X_0$, the exact sequence of $S$-dg-categories
$$\Perf(X_0) \to\Coh(X_0) \to \Sing(X_0)$$
creates, by Cor. \ref{cor-Mpreservesexactsequences}, a cofiber-fiber sequence of $\BU_S$-modules
\begin{equation}
\label{eq-cofibersequenceSingderivedmotivic}
\mathcal{M}_S^\vee(\Perf(X_0)) \to \mathcal{M}_S^\vee(\Coh(X_0)) \to \mathcal{M}_S^\vee(\Sing(X_0))
\end{equation}
where, thanks to Prop. \ref{prop-invariantunderderivedthickenings} and Prop. \ref{prop:BU-motiveofclosedisCoh}, the middle term is canonically identified with $p_*i_*i^! \BU_X$ via the commutative diagram
\begin{equation}\label{spotorno}
\xymatrix{
p_*i_*i^! \BU_X\ar[rr]\ar[d]^{\sim}_{\ref{prop:BU-motiveofclosedisCoh}}&& p_*\BU_X\ar[d]^{\sim}_{\ref{prop-descriptionmotivedgcategoryviapushforward}}\\
\mathcal{M}_S^\vee(\Coh(t(X_0)))\ar[r]^{t_*, \sim}_{\ref{prop-invariantunderderivedthickenings}}& \mathcal{M}_S^\vee(\Coh(X_0))\ar[r]^{\mathcal{M}_S^\vee(\mathfrak{i}_\ast)}& \mathcal{M}_S^\vee(\Coh(X))
}
\end{equation}
\noindent where $\mathfrak{i}_\ast$ is the pushforward of bounded coherent sheaves, along the derived closed immersion $\mathfrak{i}:X_0\to X$ and the lower horizontal composition does  identify with pushforward along $i:Z\to X$.\footnote{To prove this one could also use the six operations and nil-invariance theorem for motives over derived schemes as developed in \cite{1610.06871}. However, this is not strictly necessary for our discussion.}

\vspace{0.5cm}

Using the fact that the closed immersion $\mathfrak{i}:X_0\to X$ is lci, we know that the push-forward $\mathfrak{i}_\ast$ preserves perfect complexes \cite{1210.2827}, and thus provides a map $ \mathcal{M}_S^{\vee}(\Perf(X_0))\to \mathcal{M}_S^{\vee}(\Perf(X)) \simeq p_* \BU_X$ that we will still denote as $\mathcal{M}_S^\vee(\mathfrak{i}_\ast)$. By projection formula this is a map of $p_* \BU_X$-modules, and moreover  using the identifications in (\ref{spotorno}), it fits into a commutative 2-simplex

\begin{equation}
\label{eq-twocofibersequencesmotivicSIng2}
\xymatrix{
\mathcal{M}_S^{\vee}(\Perf(X_0)) \ar[dr]_{\mathcal{M}_S^\vee(\mathfrak{i}_\ast)} \ar[r]^u & p_*i_*i^! \BU_X \ar[d] \\
&p_*\BU_X\\
}
\end{equation}

Combining the exact sequence (\ref{eq-cofibersequenceSingderivedmotivic}), the localization sequence for the $j^*$-pullback (\ref{eq-localizationexactsequence}) and the localization sequence for the $i^*$-pullback
\begin{equation}
p_*j_\sharp \BU_U\to p_*\BU_X \to p_*i_*\BU_Z
\end{equation}
\noindent one obtains
\begin{lem}
\label{17ago1323}
Set $Z=\mathrm{t}(X_0)$. The considerations above lead to a commutative diagram of $\BU_S$-modules
\begin{equation}
\label{eq-exactsequenceMF}
\xymatrix{
\mathcal{M}_S^{\vee}(\Perf(X_0))\ar[dr]_{\mathcal{M}_S^\vee(\mathfrak{i}_\ast)}\ar[r]^u & p_*i_*i^! \BU_X \ar[d]^-{(\ref{spotorno})} \ar[r]& \mathcal{M}_S^\vee(\Sing(X_0))\\
p_*j_\sharp \BU_U \ar[dr]_{ h=:} \ar[r]&p_*\BU_X\ar[d]^{j^*-pullback}\ar[r]^{i^*-pullback}& p_*i_*\BU_Z\\
&p_*j_* j^* \BU_X& 
}
\end{equation}
\end{lem}

In the next lemma, we will denote by $\mathcal{M}_S^\vee(\mathfrak{i}^\ast): p_* \BU_X \simeq \mathcal{M}_S^{\vee}(\Perf(X)) \to \mathcal{M}_S^{\vee}(\Perf(X_0))$ the morphism of $\BU_S$-modules $\mathcal{M}_S^{\vee}(\mathfrak{i}^*: \Perf(X) \to \Perf (X_0))$.
\begin{lem}
\label{lemma-null}
The composition $\mathcal{M}_S^\vee(\mathfrak{i}^\ast\mathfrak{i}_\ast) : \mathcal{M}_S^{\vee}(\Perf(X_0)) \to \mathcal{M}_S^{\vee}(\Perf(X_0))$ is null-homotopic in $\BU_S$-modules.  In particular, we have a commutative triangle
\begin{equation}
\label{eq-exactsequenceMFker}
\xymatrix{
\mathcal{M}_S^{\vee}(\Perf(X_0))\ar[dr]_{\mathcal{M}_S^\vee(t^\ast\mathfrak{i}^\ast\mathfrak{i}_\ast)\sim 0}\ar[r]^u & p_*i_*i^! \BU_X \ar[d] \\
& p_*i_*\BU_Z\\
}
\end{equation}
\begin{proof}
Recall that in the context  of the diagram (\ref{eq-derivedzerolocusfiberproduct}) we have the cartesian cube (\ref{eq-geometricinterpretationlaxaction}), and that, by Remark \ref{rem-symmonoidalstructurerelativeperfectSzero}, we have 
$$\mathfrak{i}^*\mathfrak{i}_*\simeq v_*(pr)^*\simeq  K(A,0)\boxplus -$$
\noindent and
$$Id_{X_0}\simeq A\boxplus - $$
Finally, using the cofiber-sequence (\ref{eq-resolutionK(A,0)}) we get a cofiber sequence of dg-functors
\begin{equation}
\label{eq-resolutionK(A,0)2}
\xymatrix{
Id_{X_0}\ar[r]^0\ar[d]& Id_{X_0}\ar[d]\\
0\ar[r]&\mathfrak{i}^*\mathfrak{i}_*
}
\end{equation}
which shows that $\mathfrak{i}^*\mathfrak{i}_*$ induces the zero map in the K-theory of $X_0$. Finally, to conclude that the induced map $\mathcal{M}_S^\vee(\mathfrak{i}^\ast\mathfrak{i}_\ast)$  in $\BU_S$-modules is zero, we argue that the map in non-commutative motives $\mu_S(\Perf(X_0))\to \mu_S(\Perf(X_0))$ is zero, before applying $\mathcal{M}_S$. Indeed,  this follows by applying the same argument to the homotopy $K$-theory of $T\otimes \Perf(X)$ for any dg-category $T$ over $S$.\\

\end{proof}
\end{lem}

\medskip

\begin{prop}
\label{17ago1324}
The commutative diagram (\ref{eq-exactsequenceMFker}) produces a cofiber-fiber sequence of $\BU_S$-modules
\begin{equation}
\label{eq-exactsequenceMF2334}
 \mathcal{M}_S^\vee(\Sing(X_0))\to\mathcal{M}_S^{\vee}(\Perf(X_0))[1]\oplus p_*i_*\BU_Z\to \mathrm{cofib}(h: p_*j_\sharp \BU_U\to p_*j_* \BU_U)
\end{equation}
\begin{proof}
Apply the octahedral property (\cite[Thm 1.1.2.15 (TR4)]{lurie-ha})  to (\ref{eq-exactsequenceMFker}).
\end{proof}
\end{prop}

\medskip

\begin{rem}
\label{remark-fundamentalclasslci22}
\label{remark-fundamentalclasslci}
Assume that $Z=t_0(X_0)=X_0$ \footnote{as will be the case considered later in section \ref{section-comparisonvanishingcyclesandMF}}.  Then the Lemma \ref{lemma-null} guarantees the existence of a 2-cell providing a factorization as a map of $p_* i_* \BU_Z$-modules
\begin{equation}
\label{eq-exactsequence1classcycleSGA1}
\xymatrix{& \textrm{cofib}(p_*j_\sharp \BU_U\to p_*j_* \BU_U)[-1]\ar[d]^{\partial}\\
\ar[dr]_{\mathcal{M}_S^\vee(t^\ast\mathfrak{i}^\ast\mathfrak{i}_\ast)\sim 0} p_* i_* \BU_Z\ar@{-->}[ur]^-{\theta^{(X,Z)}_\mathrm{K}} \ar[r]^u &p_*i_*i^!\BU_X\ar[d]\\
& p_\ast i_\ast \BU_Z
}
\end{equation}
\noindent where $\partial$ is the boundary map of the vertical cofiber sequence.
Using Bott periodicity (Section \ref{rem-algebraicbott}), the map $\theta^{(X,Z)}_\mathrm{K}$ can also be interpreted as a map of $p_* \BU_X$-modules
\begin{equation}
\label{eq-exactsequence2classcycleSGA1}
\xymatrix{p_* i_* \BU_Z(-1)[-2]\simeq p_* i_* \BU_Z\ar[r]^-{\theta^{(X,Z)}_\mathrm{K}}& \textrm{cofib}(p_*j_\sharp \BU_U\to p_*j_* \BU_U)[-1]}
\end{equation}
This is the same as a map of $p_* \BU_X$-modules
\begin{equation}
\label{eq-exactsequence233classcycleSGA1}
\xymatrix{p_* i_* \BU_Z \ar[r]^-{\theta^{(X,Z)}_\mathrm{K}}& \textrm{cofib}(p_*j_\sharp \BU_U\to p_*j_* \BU_U)(1)[1]}
\end{equation}
In particular, following the Remark \ref{rem-algebrastructuremotiveofcohandperf}, when $(X,Z)$ is a regular pair, the commutativity of the diagram (\ref{eq-exactsequence1classcycleSGA1}) says that  $\partial \circ \theta^{(X,Z)}_\mathrm{K}$ is $\lambda_{-1}$ of the conormal bundle of $Z$ in $X$.. We will see in section \ref{remark-fundamentalclasslci22ladic} that $\theta^{(X,Z)}_\mathrm{K}$ is a K-theoretic version of the cycle class defined in \cite[Cycle \S 2.1]{MR0463174} associated to the closed pair $(Z, X)$. 
\end{rem}

\vspace{1cm}

\subsection{The $\BU$-motives of 2-periodic complexes}
\label{section-BU2complexes}

\begin{context}
In this section $S=\Spec\, A$ with $A$ a regular ring.
\end{context}

As a first application of the cofiber-fiber sequence (\ref{eq-exactsequenceMF2334}) we compute the motivic $\BU$-module of the dg-category of 2-periodic complexes. 

\medskip
The symmetric monoidal functor (\ref{19ago1149}) yields a map of commutative algebra objects in $\BU_S$-modules

\begin{equation}
\label{19ago1159}
\mathcal{M}_S^\vee(\Coh(A[\epsilon])^{\boxplus})\to\mathcal{M}_S^\vee(\Sing(S,0)^{\otimes})
\end{equation}

\noindent for the convolution structure on the l.h.s induced by the group structure on $S\times_{\mathbb{A}^1_S}S\simeq \Spec A[\epsilon]$. The unit of of this group is given by the truncation map $t: S\to S\times_{\mathbb{A}^1_S}S$. In particular, the unit of the commutative algebra $\mathcal{M}_S^\vee(\Coh(A[\epsilon])^{\boxplus})$ is given by the pushfoward map 
$t_\ast:  \BU_S\simeq \mathcal{M}_S^\vee(\Coh(S))\to \mathcal{M}_S^\vee(\Coh(A[\epsilon])^{\boxplus})$. Being the unit, $t_\ast$ is a map of algebras, and by the Prop. \ref{prop-invariantunderderivedthickenings}, is an equivalence. This tells us that the convolution product becomes the standard product on $\BU_S$. Combining these observations we deduce that (\ref{19ago1159}) can be written as a map of algebras
\begin{equation}
\label{19ago1200}
\BU_S\to\mathcal{M}_S^\vee(\Sing(S,0)^{\otimes})
\end{equation}

The next proposition describes the underlying object of the commutative $\BU_S$-algebra $\mathcal{M}_S^\vee(\Sing(S,0)^{\otimes})$:

\begin{prop}
\label{prop-BUmotive2periodic}
There is a canonical equivalence of $\BU_S$-modules
\begin{equation}
\mathcal{M}_S^\vee(\Sing(S,0))\simeq \BU_S\oplus \BU_S[1]
\end{equation}
\end{prop}

\medskip

As explained in Prop. \ref{prop-laxstructureSingviabasechange}, $\Sing(S,0)$ is equivalent to $\Perf(A[u, u^{-1}])$ and obtained as the cofiber sequence in $\dgcat_A^{\mathrm{idem}}$

\begin{equation}
\Perf(A[\epsilon])\subseteq \Coh(A[\epsilon]) \to \Sing(S,0)
\end{equation}

In this case, as $U$ is empty and $S$ is regular (so that $\Sing(S,0) \simeq \Sing (S_0)$ where $S_0$ is the derived pullback of $0_S: S \to \mathbb{A}_S^1$ along itself), the cofiber-fiber sequence of (\ref{eq-exactsequenceMF2334}) gives an equivalence

\begin{equation}
\mathcal{M}_S^\vee(\Sing(S,0))\simeq \mathcal{M}_S^\vee(\Sing(S_0)) \simeq \mathcal{M}_S^\vee(\Perf(A[\epsilon]))[1]\oplus \BU_S 
\end{equation}

\noindent and we are left to show that $\mathcal{M}_S^\vee(\Perf(A[\epsilon]))$ is equivalent to $\BU_S$. But this follows from Prop. \ref{prop-descriptionmotivedgcategoryviapushforward}, and the following remark applied to the graded algebra $A[\epsilon]$.

\begin{rem}
\label{lemma-gradedalgebrasarecontractibletozerolevel}
Let $R=\bigoplus_{i\leq 0} R_i$ be a graded algebra over $A$ concentrated in non-positive degrees. Then the canonical inclusion and projection
\begin{equation}\label{eq-homotopygraded1}
q: R_0\to R \,\,\,\,\,\,\,\,\, \,\,\,\,\,\,\,\,\,  \,\,\,\,\,\,\,\,\, pr: R\to R_0
\end{equation}
seen as maps of $A$-dg-categories with single objects, are $\mathbb{A}^1_A$-homotopy inverse. Indeed,  the composition $R_0\to R\to R_0$ is the identity. For the other composition notice that by definition the grading in $R$ is the data of a map of $A$-modules
$$
R\to R\otimes_A A[t]
$$
sending an element $r\in R$ of degree $i\leq 0$ to $r\otimes t^{-i}$. This map provides the required $\mathbb{A}^1_A$-homotopy between the composition $q\circ pr$ and the identity via the respective evaluations at $0$ and $1$.
\end{rem}

\vspace{0.5cm}

\subsection{Rational Coefficients and Tate-2-periodicity}\label{subsection-rationalcoefficientsrealizations} Throughout this subsection we work under the Context \ref{context2222}. $\rmSH^\otimes$ carries a canonical action of the $\infty$-category of spectra $\mathrm{Sp}^\otimes$ seen as a constant system of monoidal categories indexed by $\bsch$. More precisely, since for any $X$, $\rmSH_X^{\otimes}$ is a symmetric monoidal $\infty$-category and $\rmSH_X$ is stable and presentable, there exists a unique (up to a contractible space of choices) natural transformation
\begin{equation}
\label{17ago1250}
a: \Sp^\otimes \to \rmSH^\otimes
\end{equation}
of $\infty$-functors $\bsch^{\mathrm{op}}\to \CAlg(\Prlstb)$: this follows from the universal property of the smash product symmetric monoidal structure on spectra \cite[Cor. 4.8.2.19]{lurie-ha}. This provides for each $\Lambda\in \CAlg(\Sp)$, a family of commutative algebra objects $\Lambda_X\in \CAlg(\rmSH_X)$ indexed by $X\in \bsch$ and stable under pullbacks. A perhaps more concrete, though less structured way of understanding the natural transformation $a: \Sp^\otimes \to \rmSH^\otimes$ is as follows. If $\Lambda\in \CAlg(\Sp)$ and $X \in \bsch$, then we can identify $a(\Lambda)$ with the constant object $\Lambda_{X}^s$ in the stable homotopy category $\rmSH_{X}^{s}$ of schemes where only the simplicial suspension $S_s^1$ has been inverted (this is because we already have bonding maps $S_s^1 \wedge \Lambda_n \to \Lambda_{n+1}$, as $\Lambda \in \Sp$). Now, $\rmSH_X$ can be constructed by further inverting the Tate suspension in $\rmSH_{X}^{s}$, and we define  $\Lambda_X:= a(X)(\Lambda)$ as the image of $\Lambda_X^s$ via the canonical functor $\rmSH_{X}^{s} \to \rmSH_{X}$.\\
Note that, in particular, via the natural transformation $a$, $\rmSH$ is tensored over $\Sp$ (see  \cite[Rmk. 4.8.2.20]{lurie-ha}), and by the same discussion as for $\BU$-modules above, we thus have a system of categories $\Mod_{\Lambda}(\rmSH)^\otimes$ together with a realization map 
\begin{equation}
\label{eq-universallinearrealization}
-\otimes \Lambda: \rmSH^\otimes \to \Mod_\Lambda(\rmSH)^\otimes.
\end{equation}

\vspace{0.5cm}

\begin{rem}\label{rem-changecoef}
For each $X\in \bsch$, the category $\Mod_{\Lambda_X}(\rmSH_X)$ can be identified with the tensor product in $\Prlstb$ 
$$\rmSH_X\otimes_{\Sp} \Mod_\Lambda(\Sp)$$
This is follows from \cite[Thm. 4.8.4.6 and Section 4.5.1]{lurie-ha}. 
\end{rem}

\vspace{0.5cm}

\begin{rem}
The natural transformation of (\ref{eq-universallinearrealization}) is the universal $\Lambda$-linear realization. Indeed, recall that, by definition, $\Lambda$-linear stable $\infty$-categories are objects in $\Mod_{\Mod_{\Lambda}(\Sp)^\otimes}(\Prlstb)$. In particular, if $\mathrm{R}:\rmSH^\otimes \to \mathrm{D}^\otimes$ is a realization where $\mathrm{D}^\otimes$ takes values in $\Lambda$-linear categories, the universal property of base change \cite[4.5.3.1]{lurie-ha} tells us that $\mathrm{R}$ factors in a unique way by a $\Lambda$-linear realization $\mathrm{R}:\rmSH^\otimes \otimes_{\Sp} \Mod_\Lambda(\Sp)^\otimes \to \mathrm{D}^\otimes$.
\end{rem}
\vspace{0.3cm}
Let $\Lambda=\mathrm{H}\mathbb{Q}$ be the Eilenberg-Maclane spectrum representing rational singular cohomology \footnote{This is equivalent to the rational sphere spectrum. }. It has the structure of algebra-object in $\CAlg(\Sp)$ given by the cup product in cohomology. This is an idempotent ring-object, in the sense that the multiplication map $\mathrm{H}\mathbb{Q}\otimes \mathrm{H}\mathbb{Q}\simeq \mathrm{H}\mathbb{Q}$ is an equivalence (i.e. the localization at $\mathrm{H}\mathbb{Q}$ is smashing). Therefore the universal $\mathbb{Q}$-linear realization
\begin{equation}
\label{eq-rationalizationfunctor}
-\otimes \mathbb{Q}:= -\otimes \mathrm{H}\mathbb{Q} :  \rmSH^\otimes \to \Mod_{\mathrm{H}\mathbb{Q}}(\rmSH)^\otimes
\end{equation}
\noindent  identifies $\Mod_{\mathrm{H}\mathbb{Q}}(\rmSH)$ with the full subcategory $\rmSH_{\mathbb{Q}}$ of $\rmSH$ spanned by non torsion objects. 

\vspace{0.3cm}

\begin{prop}
Assume the Context \ref{context2222} (in particular $S$ is of finite Krull dimension). Then rationalization (\ref{eq-rationalizationfunctor}) is strongly compatible with all the six operations in the sense that the natural transformations (\ref{eq-transferlowerstartuppershriek}) are natural isomorphisms.
\begin{proof}
Follows from the arguments in the proof of \cite[A.14]{MR3205601}.
\end{proof}
\end{prop}

\vspace{0.3cm}

Following the discussion in  Section \ref{section-sixopBUmod}, rationalization carries over to $\BU$-modules
\begin{equation}\label{eq-rationalrealizationBU}
-\otimes \mathbb{Q}: \Mod_{\BU}(\rmSH)^\otimes \to \Mod_{\BU_\mathbb{Q}}(\rmSH_{\mathbb{Q}})^\otimes\simeq \Mod_{\BU_\mathbb{Q}}(\rmSH)^\otimes
\end{equation}

\noindent where $\BU_\mathbb{Q}:= \BU_S\otimes \mathrm{H}\mathbb{Q}$ and the last equivalence follows from \cite[3.4.1.9]{lurie-ha}. 

\vspace{0.3cm}

Thanks to the strong compatibility with the six operations, all the constructions performed in the previous sections at the level of $\BU_S$-modules can now be repeated after rationalization, without changing the results. This follows because the realization at the level of modules is determined by the underlying realization via the forgetful functors. 

\vspace{0.3cm}

The main reason why we are interested in passing to rational coefficients is the following result:

\begin{prop}\label{prop-rationalKtheoryspectrum}
Let $X$ be any scheme of finite Krull-dimension. Then the morphism $u:1_X(1)[2]\to \BU_X\otimes \mathbb{Q}$ of (\ref{eq-uinverseofBott}) induces an equivalence of commutative algebra objects

\begin{equation}
\label{eq-BUrationalis2periodic}
\xymatrix{\mathrm{M}\mathbb{B}_X(\beta):= \mathrm{Free}(\mathrm{M}\mathbb{B}_X(1)[2])[\nu^{-1}]\ar[r]^-{\sim}& \BU_{X,\mathbb{Q}}:= \BU_X\otimes \mathbb{Q} }
\end{equation}

\noindent where $\mathrm{M}\mathbb{B}_X$ is the commutative algebra-object representing Beilinson's motivic cohomology of \cite[Def. 14.1.2]{cisinski-tcmm}\footnote{The structure of commutative algebra object of $\mathrm{M}\mathbb{B}$ in the symmetric monoidal $\infty$-category $\rmSH$ follows from the equivalences \cite[(5.3.35.2)]{cisinski-tcmm}, the definition \cite[15.2.1]{cisinski-tcmm} and \cite[14.2.9]{cisinski-tcmm}. The combination of these results characterizes $\mathrm{M}\mathbb{B}$-modules as a monoidal reflexive localization of $\rmSH_{\mathbb{Q}}$, so that as explained in \cite[14.2.2]{cisinski-tcmm}, $\mathrm{M}\mathbb{B}$ is the image of the monoidal unit in $\rmSH_{\mathbb{Q}}$ under a monoidal localization functor and a lax monoidal inclusion, so, it acquires a natural structure of commutative algebra object.}. In particular, if $X$ is a geometrically unibranch excellent scheme, then we can replace  Beilinson's motivic cohomology by the spectrum $\mathrm{M}\mathbb{Q}_X$ representing rational motivic cohomology, via the equivalence of commutative algebra objects 

\begin{equation}
\label{eq-cherncharacter}
\mathrm{M}\mathbb{B}\to \mathrm{M}\mathbb{Q}_X
\end{equation}

\noindent induced by the Chern character.
\begin{proof}
This is already proven in  \cite[14.2.17, 16.1.7]{cisinski-tcmm} using the results of Riou in \cite{riou2010algebraic} on the $\gamma$-filtration. The only remaining issue is to construct this as an equivalence of $E_\infty$-algebras in motives. But this follows just by describing the image of $\nu$ and showing it is invertible. But again this is verified at the classical level. 

The fact that the map (\ref{eq-cherncharacter}) corresponds to the Chern character is explained in \cite[Def 6.2.3.9 and Rem 6.2.3.10]{riou2010algebraic}.
\end{proof}
\end{prop}

\medskip

\begin{rem}
In practice the hypothesis that the schemes are excellent and geometrically unibranch will be verified in the cases that interests us, namely, when $S$ is a complete Henselian trait and $X$ is a regular scheme of finite type over $S$. This follows because complete local rings are excellent \cite[Scholie 7.8.3 page 214]{MR0199181} and $S$ being a discrete valuation ring it is regular, so it is normal (in fact in dimension 1 the two are equivalent) and therefore by a direct checking of the definitions, geometrically unibranch. See \cite[Thm 8.3.30]{cisinski-tcmm}. Moreover, following \cite[Prop. 7.8.6 page 217]{MR0199181} if $X\to S$ is a scheme of finite type over $S$ with $S$ excellent then $X$ is excellent. Again if  $X$ is regular, it is normal \cite[Cor 3 , IV-39]{MR0201468} and therefore it is geometrically unibranch.
\end{rem}

\medskip

Following from  \cite[14.1.6]{cisinski-tcmm} one obtains for any $X$ a canonical isomorphism 
$$\BU_{X, \mathbb{Q}}\otimes \mathrm{M}\mathbb{B}_X \simeq \BU_{X, \mathbb{Q}}$$
and  the realization (\ref{eq-rationalrealizationBU}) is equivalent to

\begin{equation}\label{eq-rationalrealizationBU2}
\Mod_{\BU}(\rmSH) \to \Mod_{\BU_{\mathbb{Q}}}(\Mod_{\mathrm{M}\mathbb{B}}(\rmSH))
\end{equation}

\noindent which is therefore strongly compatible with all the six operations.

\subsection{$\ell$-adic realization} \label{subsection-adicreal}
In this section we discuss the $\ell$-adic realization of $\BU$-modules and dg-categories.

\begin{context}
\label{assumption-excellentdim1}
Throughout this section we assume that $S$ is an excellent scheme of dimension less or equal than one and we denote by $\bsch$ the category of schemes of finite type over $S$. Notice that by  \cite[Prop. 7.8.6 page 217]{MR0199181} such schemes are also excellent. We fix $\ell$ a prime invertible in $S$.
\end{context}

We describe below, working over schemes under the context \ref{assumption-excellentdim1}, the construction of a monoidal realization functor
\begin{equation}\label{eq-ladic1}
\xymatrix{
\mathrm{R}^\ell: \Mod_{\mathrm{M}\mathbb{B}}(\rmSH)^\otimes\ar[r] & \mathrm{Sh}_{\mathbb{Q}_\ell}(-)^\otimes
}
\end{equation}
\noindent where for each $X\in \bsch$,  $\mathrm{Sh}_{\mathbb{Q}_\ell}(X)^{\otimes}$ denotes the symmetric monoidal $\infty$-category of \emph{Ind-constructible $\mathbb{Q}_\ell$-adic sheaves on $X$}. Let us first explain what is the definition of $\mathrm{Sh}_{\mathbb{Q}_\ell}(X)$ and the reason for the context \ref{assumption-excellentdim1}. 
We will need to consider for each $n\geq 0$, $\mathrm{Sh}(X_{ét}, \mathbb{Z}/\ell^n)$ the $\infty$-category  of étale sheaves with $\mathbb{Z}/\ell^n$-coefficients. We will denote by $\mathrm{Sh}^{c}(X_{ét}, \mathbb{Z}/\ell^n)$ the full subcategory of étale sheaves with $\mathbb{Z}/\ell^n$-coefficients spanned by constructible sheaves as in \cite[4.2.5]{lurie-gaitsgory-tamagawa} or in \cite[page 598]{cisinski2016etale}. We have the following crucial result:

\begin{prop}
\label{prop-GabberSGA}
Let $S$ be a base scheme in the context of \ref{assumption-excellentdim1}. Then for any scheme $X$ of finite type over $S$, the étale topos of $X$ is of finite cohomological dimension and the étale cohomological dimension of its points are uniformly bounded. In this case  $\mathrm{Sh}^{c}(X_{ét}, \mathbb{Z}/\ell^n)$  is compactly generated and its compact objects are exactly the constructible sheaves.
\begin{proof}
The first statement follows as in \cite[Prop. 1.1.5, Remark 1.1.6]{cisinski2016etale}. The second statement now follows exactly as in \cite[4.2.2]{lurie-gaitsgory-tamagawa} replacing the use of \cite[Lemma 4.1.13]{lurie-gaitsgory-tamagawa} by the first statement.
\end{proof}
\end{prop}

\medskip

In this case we consider the limit of $\infty$-categories 
\begin{equation}
\label{eq-ladicc}
\mathrm{Sh}^{c}_{\ell}(X):= \mathrm{lim}_{n\geq 0} \mathrm{Sh}^{c}(X_{ét}, \mathbb{Z}/\ell^n)
\end{equation}
and it follows that because of Prop. \ref{prop-GabberSGA} and the same arguments as in  \cite[4.3.17]{lurie-gaitsgory-tamagawa} that $\mathrm{Sh}^{c}_{\ell}(X)$ is an $\infty$-categorical enhancement of the derived category of constructible $\ell$-adic sheaves on $X$, in the sense that its homotopy category is equivalent (as a triangulated category) to the so-called \emph{constructible derived $\ell$-adic category} $D_c(X, \mathbb{Z}_{\ell}) $ of \cite{MR1106899} and \cite{MR751966}.\\
As a second step, we define the $\infty$-category of $\ell$-adic sheaves on $X$ as the Ind-completion of  $\mathrm{Sh}^{c}_{\ell}(X)$ (\cite[4.3.26]{lurie-gaitsgory-tamagawa})  
\begin{equation}
\label{eq-ladicconstructible}
\mathrm{Sh}_{\ell}(X):= \mathrm{Ind}(\mathrm{Sh}_{\ell}^c(X))
\end{equation}
\noindent Therefore one should think of $\mathrm{Sh}_{\ell}(X)$ as the \emph{derived $\infty$-category of Ind-constructible $\ell$-adic sheaves} on $X$. Finally, we define $\mathrm{Sh}_{\mathbb{Q}_\ell}(X):=  \mathrm{Sh}_{\ell}(X)\otimes_{\mathbb{Z}_{\ell}}\mathbb{Q_{\ell}}$. Note that $\mathrm{Sh}_{\mathbb{Q}_\ell}(X)$ can also be identified as the full subcategory of $\mathrm{Sh}_{\ell}(X)$ spanned by those objects $F$ such that the natural morphism $F\to F[\ell^{-1}]$ is an equivalence. Also note that the limit (\ref{eq-ladicconstructible}) can be taken inside the theory of symmetric monoidal small idempotent-complete stable and $\mathbb{Z}_{(\ell)}$-linear $\infty$-categories. Therefore, by working inside $\CAlg(\Mod_{\Mod_{\mathbb{Z}_{(\ell)}}}(\Prl))$, and taking first Ind-completion, and then applying $- \otimes_{\mathbb{Z}_{\ell}}\mathbb{Q_{\ell}}$ (or, equivalently, inverting $\ell)$), we get a symmetric monoidal small idempotent-complete stable and $\Ql$-linear structure on $\mathrm{Sh}_{\mathbb{Q}_\ell}(X)$. We will denote this monoidal structure by $\mathrm{Sh}_{\mathbb{Q}_\ell}(X)^{\otimes}$ and its tensor unit by $\mathbb{Q}_{\ell, X}$.

\medskip

\begin{rem}
\label{rem-finiteetaledimension}
The discussion in \cite[Section 4]{lurie-gaitsgory-tamagawa} is written for quasi-projective schemes over a field as it requires the étale topos of $X$ to be of finite cohomological dimension \cite[Lemma 4.1.13]{lurie-gaitsgory-tamagawa}. In our case this follows from the assumptions in \ref{assumption-excellentdim1} and Prop. \ref{prop-GabberSGA}.
\end{rem}

\medskip

\begin{rem}
The construction of an $\infty$-functor $X\mapsto \mathrm{Sh}_{\mathbb{Q}_\ell}(X)^{\otimes}$ can be obtained using the arguments of \cite[Chapter 9]{robalo-thesis}.
\end{rem}

\medskip

The monoidal realization (\ref{eq-ladic1}) could be obtained using the universal property of $\rmSH$ proved in \cite{MR3281141}. However, we will need to show that it is strongly compatible with all the six operations and with the classical notion of Tate twists. For this purpose, we describe an alternative construction of (\ref{eq-ladic1}) using results of \cite{cisinski2016etale}. Once (\ref{eq-ladic1}) is available and strongly compatible with all operations, we can then pass to $\BU_\mathbb{Q}$-modules and deduce that the composition
\begin{equation}\label{eq-ladic2}
\xymatrix{
\Mod_{\BU}(\rmSH)\ar[r]^-{(\ref{eq-rationalrealizationBU2})}& \Mod_{\BU_\mathbb{Q}}(\Mod_{\mathrm{M}\mathbb{B}}(\rmSH))\ar[r]^-{(\ref{eq-ladic1})} & \Mod_{R^\ell(\BU_\mathbb{Q})}(\mathrm{Sh}_{\mathbb{Q}_\ell}(-))
}
\end{equation}
\noindent is again compatible with all the six operations and twists.
\vspace{0.5cm}
 Let us now review the construction of  (\ref{eq-ladic1}). This can be done as in \cite{cisinski2016etale} using the theory of h-motives. Recall from \cite[Def 3.1.2]{MR1403354} that the $\mathrm{h}$-topology on Noetherian schemes is the topology whose covers are the universal topological epimorphisms. It is the minimal topology generated by open coverings and proper surjective maps (see for the case of excellent schemes \cite[Def 3.1.2]{MR1403354}). In \cite{cisinski2016etale} the authors constructed for any noetherian scheme $X$ and any ring $R$ a theory of $h$-motives, $\mathrm{DM}_\mathrm{h}(X, R)$. See \cite[Section 5.1]{cisinski2016etale}. The constructions in loc.cit can be formulated in the language of higher categories, using the arguments and steps of \cite{MR3281141} and an $\infty$-functor
$$\mathrm{DM}_\mathrm{h}(-, R)^{\otimes}:\bsch^{\mathrm{op}}\to \CAlg(\Prl)$$
\noindent can be provided as in \cite[Chapter 9]{robalo-thesis}. Following \cite[5.6.2]{cisinski2016etale} and Prop. \ref{prop-Ayoub-Cisinski} this $\infty$-functor satisfies all the formalism of the six operations over Noetherian schemes of finite Krull dimension.
\medskip
We now recall how to relate $h$-motives both to the r.h.s and l.h.s of (\ref{eq-ladic1}). Let $R$ be the localization of $\mathbb{Z}$ at the prime $\ell$. 
\vspace{0.3cm}
To understand the r.h.s of (\ref{eq-ladic1}) we use a form of rigidity theorem given by \cite[Thm 5.5.3 and Thm 4.5.2]{cisinski2016etale}: for any Noetherien scheme $X$, $\ell$ invertible in $\mathcal{O}_X$ and for each $n\geq0$, we have a monoidal equivalence 
\begin{equation}\label{eq-rigidityGabber}
\mathrm{DM}_\mathrm{h}( X, R/\ell^n)\simeq \mathrm{Sh}(X_{ét}, R/\ell^n)
\end{equation}
\noindent with the last being the standard $\infty$-category of  $\ell^n$-torsion étale sheaves on $X$. This equivalence is compatible with all the six operations over Noetherian schemes of finite Krull dimension. Following \cite[Thm 6.3.11]{cisinski2016etale} for all Noetherian schemes of finite dimension, (\ref{eq-rigidityGabber}) restricts to an equivalence
\begin{equation}\label{eq-rigidityGabberconstructible}
\mathrm{DM}_{\mathrm{h}, lc}( X, R/\ell^n)\simeq \mathrm{Sh}^c((X)_{ét}, R/\ell^n)
\end{equation}
\noindent where on the l.h.s we have the full subcategory of locally constructible objects of \cite[Def. 6.3.1]{cisinski2016etale}. Thanks to the uniformization results of Gabber  (see \cite[6.3.15]{cisinski2016etale} for the l.h.s and \cite[Exposé 0 Thm 1]{MR3329769} for the r.h.s)  the constructions $X\mapsto \mathrm{DM}_{\mathrm{h}, lc}( X, R/\ell^n)$ are stable under the six operations when restricted to quasi-excellent noetherian schemes of finite dimension. 

\medskip

\begin{rem}\label{tousualtwist}
Via (\ref{eq-rigidityGabber})  motivic Tate twists are sent to the usual $\ell$-adic twists given by the roots of unity. This is a consequence of the Kummer exact sequence as explained in \cite[Section 3.2]{cisinski2016etale}.
\end{rem}

\medskip

As a result, the equivalences (\ref{eq-rigidityGabberconstructible}) assemble to an equivalence of $\infty$-functors, strongly compatible with all the six operations and twists
\begin{equation}\label{eq-rigidityetaleh}
\mathrm{DM}_{\mathrm{h}, lc}( -, R/\ell^n)\simeq \mathrm{Sh}^c((-)_{ét}, R/\ell^n)
\end{equation}
\noindent whenever $\ell$ is invertible in $S$.

\medskip

\begin{rem}\label{whyfincohdim}
Under the assumptions in \ref{assumption-excellentdim1} and because of Prop. \ref{prop-GabberSGA} and \cite[Prop 6.3.10]{cisinski2016etale}, the notion of locally constructible objects in h-motives coincides with the notion of constructible of \cite[5.1.3]{cisinski2016etale} which also coincides with the notion of compact object \cite[Thm 5.2.4]{cisinski2016etale}.
\end{rem}

\medskip

As a conclusion to this discussion, (\ref{eq-rigidityetaleh}) provides an equivalence
\begin{equation}
\label{eq-ladicsheavesindfinal}
\mathrm{lim}_{n\geq 0}\mathrm{DM}_{\mathrm{h}, lc}( -, R/\ell^n)\simeq \mathrm{lim}_{n\geq 0}\, \mathrm{Sh}^c((-)_{ét}, R/\ell^n)
\end{equation}
\medskip

\vspace{0.5cm}

Now, the l.h.s of (\ref{eq-ladic1}) is related to the theory of h-motives via the combination of \cite[Thm 5.2.2]{cisinski2016etale} and \cite[14.2.9]{cisinski-tcmm}: when $R=\mathbb{Q}$ we have an equivalence of $\infty$-functors defined on Noetherian schemes of finite Krull dimension
\begin{equation}
\label{thm-cisinskiDMBeil}
\mathrm{DM}_\mathrm{h}(-,\mathbb{Q})\simeq \Mod_{\mathrm{M}\mathbb{B}}(\rmSH)
\end{equation}
\noindent strongly compatible with the six operations.  We recall that for each scheme $X$ the $\infty$-category $\rmSH_X$ is compactly generated and so is $\Mod_{\mathrm{M}\mathbb{B}_X}(\rmSH_X)$. See \cite[Section 4.4. and Prop.3.8.3]{robalo-thesis}. Thanks to \cite[Prop. 6.3.3 and Thm 5.2.4]{cisinski2016etale}, the equivalence (\ref{thm-cisinskiDMBeil}) identifies the compact objects of $\Mod_{\mathrm{M}\mathbb{B}_X}(\rmSH_X)$ with the subcategory of locally constructible objects $\mathrm{DM}_{\mathrm{h},lc}(-,\mathbb{Q})$ as defined in \cite[5.1.3]{cisinski2016etale}. By \cite[6.2.14]{cisinski2016etale} it is stable under all the six operations.\\

\medskip

Having these caracterizations of both the r.h.s and l.h.s of (\ref{eq-ladic1}), in order to achieve the construction of the natural transformation (\ref{eq-ladic1}), we need to exhibit a natural transformation of $\infty$-functors with values in small stable idempotent complete $R[\ell^{-1}]\simeq \mathbb{Q}$-linear $\infty$-categories
$$
\mathrm{DM}_{\mathrm{h},c}(-,\mathbb{Q})\simeq \mathrm{DM}_{\mathrm{h},c}(-,R)\otimes_R R[\ell^{-1}]\to (\mathrm{lim}_{n\geq 0}\mathrm{DM}_{\mathrm{h}, lc}( -, R/\ell^n))\otimes_R R[\ell^{-1}]
$$

\medskip

For that purpose we use the results of \cite{cisinski2016etale} that explain the $\ell$-adic realization functor (\ref{eq-ladic1}) as an $\ell$-adic completion of h-motives. The system of base changes along the maps of rings $R\to R/\ell^n$ produces natural transformations
$\mathrm{DM}_\mathrm{h}(-,R)\to \mathrm{DM}_{\mathrm{h}}( -, R/\ell^n)$ and by the standard procedure one can  construct the data of a cone over the diagram indexed by $n\geq0$ and obtain a natural transformation between the  $\infty$-functors with values in presentable stable R-linear $\infty$-categories
\begin{equation}
\mathrm{DM}_\mathrm{h}(-,R)\to \mathrm{lim}_{n\geq 0}\mathrm{DM}_{\mathrm{h}}( -, R/\ell^n)
\end{equation}
It follows from the same arguments as in \cite[4.3.9]{lurie-gaitsgory-tamagawa} that this homotopy limit identifies with the construction $\mathrm{DM}_{\mathrm{h}}(-, \widehat{R}_{\ell})$ of \cite[Def 7.2.1]{cisinski2016etale} and from \cite[Thm 7.2.11]{cisinski2016etale} that it commutes with all the six operations over Noetherian schemes of finite Krull dimension. Moreover, by \cite[7.2.16]{cisinski2016etale}, it restricts to a natural transformation between locally constructible objects
\begin{equation}
\label{eq-ladiccompletionconstructible}
\mathrm{DM}_{\mathrm{h}, lc}(-,R)\to \mathrm{lim}_{n\geq 0}\mathrm{DM}_{\mathrm{h}, lc}( -, R/\ell^n)
\end{equation}
\noindent again compatible with all operations. Following the discussion in \cite[Section 7.2.18, Prop. 7.2.19, Thm 7.2.21]{cisinski2016etale}, using the Prop. \ref{prop-GabberSGA}, one can mimic the arguments of \cite[4.3.17]{lurie-gaitsgory-tamagawa} to deduce that the homotopy category of the r.h.s recovers the classical derived category of constructible $\ell$-adic sheaves of \cite{MR751966} and \cite{MR1106899}.
\vspace{0.5cm}

Finally, the realization $\mathrm{R}^\ell$ of (\ref{eq-ladic1}) is defined via the composition 

\begin{equation}
\label{eq-ladicfirstrealizationDMrational}
\xymatrix{
 \Mod_{\mathrm{M}\mathbb{B}}(\rmSH) \ar[r]^-{\sim}_-{(\ref{thm-cisinskiDMBeil})}& \mathrm{Ind}(\mathrm{DM}_{\mathrm{h}, lc}(-,R)\otimes_{R}\mathbb{Q})\ar[r]^-{(\ref{eq-ladiccompletionconstructible})}& \mathrm{Ind}(\mathrm{lim}_{n\geq 0}\mathrm{DM}_{\mathrm{h}, lc}( -, R/\ell^n)\otimes_{R}\mathbb{Q})\ar[d]^-{\sim}_-{(\ref{eq-ladicsheavesindfinal})}\\
 && \mathrm{Sh}_{\mathbb{Q}_\ell}(-)
 }
\end{equation}

\noindent and by the preceding discussion it is strongly compatible with all the six operations and Tate twists.

\medskip

\begin{rem}
As the $\ell$-adic monoidal realization functor $\mathrm{R}^\ell: \Mod_{\mathrm{M}\mathbb{B}}(\rmSH)^\otimes\to \mathrm{Sh}_{\mathbb{Q}_\ell}(-)^\otimes$ is monoidal, for any scheme $X$ we have $\mathrm{R}^\ell(\mathrm{M}\mathbb{B}_X)\simeq \mathbb{Q}_{\ell,X}$, the monoidal unit of $\mathbb{Q}_\ell$-adic sheaves over $X$. As it is a left adjoint and commutes with Tate twists,  Prop. \ref{prop-rationalKtheoryspectrum} implies that that 
\begin{equation}
\label{eq-ladicrealizationBUQ}
\mathrm{R}^\ell(\BU_{X, \mathbb{Q}})\simeq \mathrm{Free}(\mathbb{Q}_{\ell,X}(1)[2])[\nu^{-1}]\simeq \bigoplus_{n\in \mathbb{Z}}\mathbb{Q}_{\ell,X}(n)[2n]=: \mathbb{Q}_{\ell, X}(\beta)
\end{equation}
By Remark \ref{tousualtwist}, $\mathbb{Q}_{\ell, X}(i)$ are the usual $\ell$-adic Tate twists given by the roots of unity. In particular, the extension (\ref{eq-ladic2}) of $\mathrm{R}^\ell$ to $\BU$-modules  takes values in Tate-twisted $2$-periodic objects inside $\mathrm{Sh}_{\mathbb{Q}_{\ell}}(X)$, i.e. objects $E$ together with an equivalence $E \simeq E(1)[2]$.
\end{rem}

\vspace{0.5cm}

\begin{notation} Throughout the rest of this paper we will write
\begin{equation}
\label{eq-ladicrealizationBUmodules}
\mathrm{R}^\ell: \Mod_{\BU}(\rmSH)\to \Mod_{\mathbb{Q}_{\ell}(\beta)}(\mathrm{Sh}_{\mathbb{Q}_\ell}(-))
\end{equation}
to denote the natural transformation obtained via the composition (\ref{eq-ladic2}). As already observed, it is strongly compatible with all six operations and Tate twists.
\end{notation}

\medskip

\begin{rem} Note that if $p:X \to S$ is a smooth finite type morphism of schemes, and we denote by $[X]:= p_{\sharp}(1_X) \in \rmSH_S$ its motive over $S$, then $$\mathrm{Map}_{\mathrm{Sh}_{\mathbb{Q}_\ell}(S)}(\mathrm{R}^{\ell}([X] \otimes \BU_S ), \mathbb{Q}_{\ell, S}) \simeq \mathbf{H}^{\bullet}_{\ell}(X, \mathbb{Q}_{\ell}) \otimes \mathbb{Q}_{\ell}(\beta),$$ where $\mathbf{H}^{\bullet}_{\ell}(X, \mathbb{Q}_{\ell})\simeq \mathrm{Map}_{\mathrm{Sh}_{\mathbb{Q}_\ell}(X)}(\mathbb{Q}_{\ell, X}, \mathbb{Q}_{\ell, X})$ denotes the $\ell$-adic cohomology of $X$. In other words, the $\mathbb{Q}_{\ell, S}$-dual of $\mathrm{R}^{\ell}([X] \otimes \BU_S)$ is a Tate 2-periodized version of the $\ell$-adic cohomology of $X$. Note that, instead, if we denote by $\mathrm{R}^{' \ell}: \Mod_{\mathrm{M}\mathbb{B}}(\rmSH)^\otimes \to  \mathrm{Sh}_{\mathbb{Q}_\ell}(-)^\otimes$ the realization functor (\ref{eq-ladic1}), we have $$\mathrm{Map}_{\mathrm{Sh}_{\mathbb{Q}_\ell}(S)}(\mathrm{R}^{' \ell}([X] \otimes \mathrm{M}\mathbb{B}_S ), \mathbb{Q}_{\ell, S}) \simeq \mathbf{H}^{\bullet}_{\ell}(X, \mathbb{Q}_{\ell}).$$
In the same situation, we have $$\mathrm{R}^{\ell}(\mathcal{M}_S^\vee(\Perf(X))) \simeq p_*(\mathbb{Q}_{\ell , X}(\beta)).$$ This follows from Prop. \ref{prop-descriptionmotivedgcategoryviapushforward}, projection formula, and the fact that $p_*$ commutes with $\ell$-adic realization. In other words, the $\ell$-adic realization of rationalized $\mathcal{M}_S^\vee(\Perf(X))$ is equivalent to the Tate 2-periodized version of the $\ell$-adic cohomology of $X$ relative to $S$.

\end{rem}

\medskip

\section{Vanishing Cycles and Singularity Categories}\label{section-maintheo}
In this section we prove our main Theorem \ref{theorem-maincomparisonvanishingmf} establishing the link between the results of the previous section and the theory of vanishing cycles. Namely, it says that the motive of the dg-category of singularities is a model for the $\ell$-adic cohomology of the 2-periodized sheaf of vanishing cycles. Before coming to the precise statement, we will first recall the theory of nearby and vanishing cycles in the motivic setting. 

\medskip

\subsection{Nearby and Vanishing Cycles}
In this section we recall the formalism of nearby and vanishing cycles in the $\ell$-adic setting as presented in \cite{MR0354656}. More recently, a motivic formulation was developed by Ayoub in \cite{ayoub2}. There are several technical steps required to express the formalism of \cite{ayoub2} in a higher categorical setting. We provide these details in Appendix A of this paper.

\vspace{0.5cm}

\begin{context}
Throughout this section we fix a diagram of schemes
\begin{equation}
\label{trait}
\xymatrix{\eta \ar[r]^{j_\eta} &S&  \ar[l]_{i_{\sigma}}\sigma}
\end{equation}
\noindent with $S$ an excellent henselian trait, namely, the spectrum of an excellent henselian discrete valuation ring $A$,  with uniformizer $\pi$, generic point $\eta=\Spec(K)$ and closed point $\sigma=\mathrm{Spec}(A/m)$ with $m=(\pi)$ the maximal ideal and $k:=A/m$ is a perfect field of characteristic $p\geq 0$. The pair $(\eta, \sigma)$ forms a closed-open complement pair (and the maps are, respectively, an open and a closed immersion). In practice we will take $S$ to be the spectrum of a complete discrete valuation ring. Roughly speaking, the scheme $S$ plays the role of a formal disk, $\sigma$ of the center of the disk and $\eta$ of the punctured disk (this is quite precise in the equicharacteristic zero case). We will say that a henselian trait is \emph{strictly local} if $k$ is algebraically closed. 
\end{context}

\begin{rem}
\label{rem-choice-uniformizer}
The choice of a uniformizer $\pi\in A$ defines a map $\pi:S\to \mathbf{A}^1_S$. In this case we have two cartesian diagrams
\begin{equation}
\label{choice-uniformizer}
\xymatrix{
\ar @{} [dr] | \lrcorner \eta \ar[d] \ar[r]^{j_\eta}& S\ar[d]^\pi&\sigma\ar[l]_{i_\sigma}\ar[d] \ar @{} [dl] | \llcorner \\
\mathbb{G}_{m,S}\ar@{^{(}->}[r]_{j_0}& \mathbb{A}^1_S&\ar@{_{(}->}[l]^{i_0}S
}
\end{equation}
\noindent which allow us to reduce ourselves to working over an affine line, even in mixed characteristic. Notice that both diagrams are in fact derived fiber products. Indeed, for the left diagram this is immediate because the inclusion of $\mathbb{G}_{m,S}$ is an open immersion, while for the diagram on the right we find that the derived tensor product is given by the spectrum of the commutative differential graded algebra 
\begin{equation}
\xymatrix{
0\ar[r]& \underbracket{A}_{-1} \ar[r]^{\pi \cdot} & \underbracket{A}_{0} \ar[r]&0
}
\end{equation}
which is, in fact, quasi-isomorphic to $A/\pi$, since $\pi$ is a non-zero divisor.

\end{rem}

\vspace{0.5cm}

\begin{notation}
In what follows, we fix:
\begin{itemize}
\item a separable closure $\bar{k}$ of $k$ (inside a fixed algebraic closure), and denote by $\bar{S}:=S_{(\bar{\sigma})}$ the strict localization of $S$ at the corresponding geometric point $\bar{\sigma}= \Spec(\bar{k}$) (which is localized at $\sigma$): in other words $\bar{S}$ the spectrum of the strict henselization of $A$ along $k \hookrightarrow \bar{k}$. Note that $\bar{S}$ is now a strictly henselian trait, with closed point $\bar{\sigma}$ and fraction field $K^\mathrm{unr}$ a maximal unramified extension of $K$ (see \cite[Chapitre II,  \S 2 Prop 3 and Ex.4]{MR0150130}). We set $\eta^{\mathrm{unr}}=\Spec(K^\mathrm{unr})$ and denote as $j_{\eta^{\mathrm{unr}}}:\eta^{\mathrm{unr}}\to \bar{S}$ the corresponding open immersion. 
\item a separable closure $\bar{K}$ of $K^{\mathrm{unr}}$ (inside a fixed algebraic closure), and put $\bar{\eta}:=\Spec\, \bar{K}$ and $j_{\bar{\eta}}: \bar{\eta}\to \bar{S}$.

\item $\eta^t=\Spec(K^t)$ a maximal tamely ramified extension of $K$ inside $\bar{K}$. 

\end{itemize}

All this information fits in a commutative diagram
\begin{equation}
\label{eq-diagramintegralclosures}
\xymatrix{
\ar[d]\bar{\eta}\ar@/_2pc/[ddd]_{u_\eta}\ar[ddr]^{j_{\bar{\eta}}}&&\\
\eta^t\ar[d]\ar[dr]&&\\
\eta^{\mathrm{unr}}\ar[d]\ar@{^{(}->}[r]_{j_{\eta^{\mathrm{unr}}}}& \bar{S} \ar[d]^u & \ar@{_{(}->}[l]_{i_{\bar{\sigma}}} \bar{\sigma}\ar[d]^{u_{\sigma}}\\
\eta \ar@{^{(}->}[r]_{j_\eta} & S & \ar@{_{(}->}[l]^{i_\sigma} \sigma
}
\end{equation}

\noindent which the reader should keep in mind throughout this section. \\ 
\end{notation}

\begin{rem}
We also recall the existence of an exact sequences of groups (see \cite[Chap. III \S 5 Thm 2, Thm 3, Cor I]{MR0150130})

\begin{equation}
\label{defofinertia}
1\to \mathrm{I}:=\Gal(\bar{\eta}/\eta^\mathrm{unr})\to \Gal(\bar{\eta}/\eta) \to \Gal(\bar{\sigma}/\sigma)\to 1
\end{equation}

\noindent where $\mathrm{I}$ is the inertia group, which fits in an exact sequence (see \cite{MR0150130})

\begin{equation}
\label{eq-exactsequencegaloisgroups}
1\to \Gal(\bar{\eta}/\eta^t) \to \mathrm{I} \to \mathrm{I}_t:=\Gal(\eta^t/\eta^\mathrm{unr})\to 1
\end{equation}
where $\mathrm{I}_t$ is the \emph{tame inertia group}, isomorphic to $\mathrm{lim}_{(n,p)=1} \, \mu_n$ where $\mu_n$ is the group of $n^{th}$-roots of unit in $K^{\mathrm{unr}}$. 
\end{rem}

\personal{
\begin{equation}
\label{eq-diagramintegralclosures}
\xymatrix{
\ar[d]\bar{\eta}\ar[d]_{u_\eta}\ar[dr]^{j_{\bar{\eta}}}& &\\
\eta \ar[r]_{j_\eta} & S & \ar[l]^{i_\sigma} \sigma
}
\end{equation}
$$Sh(S)= lax lim (i^*j_*:Sh(\eta) \to Sh(\sigma))$$
In other words an object over S is the same thing as an object $A_\sigma$ over $\sigma$, an object $A_\eta$ over $\eta$ and a morphism $A_\sigma \to i^*j_* A_\eta$ in $Sh(\sigma)$. If there is étale descent, then this is the same as pairs $\bar{A}_\sigma\in Sh(\bar{\sigma})^{Gal(\bar{\sigma}/\sigma)}$,  $\bar{A}_\eta\in Sh(\bar{\eta})^{Gal(\bar{\eta}/\eta)}$ and a $Gal(\bar{\eta})$-equivariant morphism $\bar{A}_\sigma\to \bar{A}_\eta$ where $Gal(\bar{\eta}/\eta)$ acts on $\bar{A}_\sigma$ via the morphism of groups $Gal(\eta)\to Gal(\sigma)$. This is the same as a $Gal(\sigma)$-equivariant map $\bar{A}_\sigma\to (\bar{A}_\eta )^I$.
\vspace{0.5cm}
In fact it seems that S itself is the lax gluing of $\eta$ and $\sigma$ along the specialization map $\eta \to \sigma$. $S$ is a correspondence $\eta\to \sigma$. Then for any scheme $X_\sigma$ over $\sigma$ one can define $X_\sigma \times_\sigma$ S whose sheaves over are pairs of sheaves on $X_\sigma$, one with an action of $Gal(\sigma)$, another with $Gal(\eta)$ and an equivariant map between the two.
\vspace{0.5cm}
What we really want is to define a category 
$$Gluing( Sh(X_{\bar{\sigma}})^{Gal(\eta)}, Sh(X_{\bar{\sigma}})^{Gal(\sigma)})$$
and a functor 
$$NearbyCycles: Sh(X)\to Gluing( Sh(X_{\bar{\sigma}})^{Gal(\eta)}, Sh(X_{\bar{\sigma}})^{Gal(\sigma)})$$
}

\vspace{0.5cm}

Consider a map $p:X\to S$. We recall  the definition of the nearby and vanishing cycles in the $\ell$-adic setting. Consider the commutative diagram:

\begin{equation}
\label{eq-diagramintegralclosures2}
\xymatrix{
&\ar[dd]^(0.3){p_{\bar{\eta}}}\ar[dl]_{v_\eta}X_{\bar{\eta}}\ar[rr]^{\bar{j}}&&\ar[dd]^(0.3){\bar{p}}\bar{X}\ar[dl]^v&&\ar[dd]^(0.3){p_{\bar{\sigma}}} \ar[ll]_{\bar{i}} X_{\bar{\sigma}}\ar[dl]^{v_{\sigma}}\\
\ar[dd]^(0.3){p_\eta}X_\eta\ar[rr]_(0.3){j}&&X\ar[dd]^(0.3){p}&&\ar[ll]^(0.3){i}X_\sigma\ar[dd]^(0.3){p_\sigma}&\\
&\bar{\eta}\ar[dl]_{u_\eta}\ar[rr]^(0.3){j_{\bar{\eta}}}&& \bar{S} \ar[dl]^u && \ar[ll]_(0.3){i_{\bar{\sigma}}} \bar{\sigma}\ar[dl]^{u_{\sigma}}\\
\eta \ar[rr]_{j_\eta} && S && \ar[ll]^{i_\sigma} \sigma
}
\end{equation}
where the right base square, the two front-faces, the two back-faces, and the central transverse square are all cartesian (the maps $v_{\sigma}$ and $v_{\eta}$ are then uniquely induced).\\
\begin{defn}
The object of \emph{nearby cycles} associated to $p$ and $E\in \mathrm{Sh}_{\mathbb{Q}_\ell}(X)$ is given by 
$$\Psi_p(E):=\bar{i}^*\bar{j}_* E_{\bar{\eta}}\in \mathrm{Sh}_{\mathbb{Q}_\ell}(X_{\bar{\sigma}})^{\Gal(\bar{\eta}/\eta)}$$ 
\noindent where $E_{\bar{\eta}} :=v_{\eta}^*j^* E\simeq \bar{j}^*v^*E$ and $\mathrm{Sh}_{\mathbb{Q}_\ell}(X_{\bar{\sigma}})^{\Gal(\bar{\eta}/\eta)}$ is the $\s$-category  of objects in $\mathrm{Sh}_{\mathbb{Q}_\ell}(X_{\bar{\sigma}})$ equipped with an equivariant structure with respect to the \emph{continuous} action of $\Gal(\bar{\eta}/\eta)$ on $X_{\bar{\sigma}}$ via the canonical  map $\Gal(\bar{\eta}/\eta) \to \Gal(\bar{\sigma}/\sigma)$. \personal{For $\Psi_p(E)$ this action is induced via $v_\eta$, by transfer of structure (``transport de structure'', in french). Moreover, this  $\Gal(\bar{\eta}/\eta)$-action on $\Psi_p(E)$ is compatible with the $\Gal(\bar{\sigma}/\sigma)$-action on $X_{\bar{\sigma}}$ via the canonical map $\Gal(\bar{\eta}/\eta) \to \Gal(\bar{\sigma}/\sigma)$}. 
\end{defn}
When the map $p$ is uniquely determined by our context, we will often write $\Psi(E)$ for $\Psi_p(E)$.

\begin{rem}
We will not give here the details for a precise definition of the \emph{continuous} Galois-equivariant $\infty$-category $\mathrm{Sh}_{\mathbb{Q}_\ell}(X_{\bar{\sigma}})^{\Gal(\bar{\eta}/\eta)}$ of the previous paragraph. A precise construction  can be obtained  using the $\infty$-categorical analogue of the Deligne topos described in \cite[Exp. XIII]{MR0354657}. Indeed the étale topos of $S$ can be described as a lax limit (in the sense of $\infty$-categories) of the diagram given by the specialization map between the étale topos of the generic point and the étale topos of the closed point. See \cite[A.8]{lurie-ha}.
\end{rem}

\begin{rem}
\label{remark-inertiainvariantscanbetaken}
Notice that by definition, the inertia group $I$ acts trivially on $X_{\bar{\sigma}}$. In this case, every object of $\mathrm{Sh}_{\mathbb{Q}_\ell}(X_{\bar{\sigma}})^{\Gal(\bar{\eta}/\eta)}$ can be seen as equipped with a continuous action of $I$ together a compatible $\Gal(\bar{\sigma}/\sigma)$-equivariant structure. More precisely, the equivalence of stacks 
$$
X_{\bar{\sigma}}/\Gal(\bar{\eta}/\eta)\simeq (X_{\bar{\sigma}}/I)/\Gal(\bar{\sigma}/\sigma)
$$
\noindent establishes an equivalence
$$
\mathrm{Sh}_{\mathbb{Q}_\ell}(X_{\bar{\sigma}})^{\Gal(\bar{\eta}/\eta)}\simeq  \mathrm{Rep}^{\mathrm{cont}}(I, \mathrm{Sh}_{\mathbb{Q}_\ell}(X_{\bar{\sigma}}))^{\Gal(\bar{\sigma}/\sigma)}
$$
Moreover, one can check that taking I-invariants renders the commutativity of the diagram
$$
\xymatrix{\mathrm{Sh}_{\mathbb{Q}_\ell}(X_{\bar{\sigma}})^{\Gal(\bar{\eta}/\eta)}\ar[rd]_{q_\ast} \ar[rr]^{\sim}&& \mathrm{Rep}^{\mathrm{cont}}(I, \mathrm{Sh}_{\mathbb{Q}_\ell}(X_{\bar{\sigma}}))^{\Gal(\bar{\sigma}/\sigma)}\ar[dl]^{(-)^{\mathrm{h}I}}\\
&\mathrm{Sh}_{\mathbb{Q}_\ell}(X_{\bar{\sigma}})^{\Gal(\bar{\sigma}/\sigma)}&}
$$
\noindent with $q_\ast$ being the pushforward along $q:X_{\bar{\sigma}}/\Gal(\bar{\eta}/\eta)\to X_{\bar{\sigma}}/\Gal(\bar{\sigma}/\sigma)$
\end{rem}

\medskip
\vspace{0.5cm}

Set $E_{\bar{\sigma}}:=v_{\sigma}^*i^* E$. There is a canonical adjunction morphism
\begin{equation}
\label{equation-definitionvanishing}
sp:E_{\bar{\sigma}}\to \Psi_p(E)
\end{equation}
that is compatible with the action of $\Gal(\bar{\eta}/\eta)$. On the source, this action comes via (\ref{defofinertia}). In other words, 
(\ref{equation-definitionvanishing})
is a morphism in the $\infty$-category $\mathrm{Sh}_{\mathbb{Q}_{\ell}}(X_{\bar{\sigma}})^{\mathrm{Gal}(\bar{\eta}/\eta)}$ of $\ell$-adic sheaves on $X_{\bar{\sigma}}$ endowed with a $\mathrm{Gal}(\bar{\eta}/\eta)$-equivariant structure. 
\medskip

\begin{defn}
\label{17ago1329} Given $E\in \mathrm{Sh}_{\mathbb{Q}_\ell}(X)$, the object of \emph{vanishing cycles} $\mathcal{V}_p (E)$ is defined as the cofiber
$$
E_{\bar{\sigma}}\to \Psi_p(E)\to \mathcal{V}_{p}(E)
$$

\noindent in the $\infty$-category $\mathrm{Sh}_{\mathbb{Q}_{\ell}}(X_{\bar{\sigma}})^{\mathrm{Gal}(\bar{\eta}/\eta)}$. To shorten notations, we will write $\mathcal{V}_{p}:= \mathcal{V}_{p}(\mathbb{Q}_{\ell, X})$. 
\end{defn}

Again, when the map $p$ is uniquely determined by our context, we will often write $\mathcal{V}(E)$ for $\mathcal{V}_p(E)$\footnote{A more standard notation for $\mathcal{V}_p(E)$ is $\Phi_{p}(E)$. }.

\vspace{1cm}

\subsection{$\ell$-adic Realization of Singularities Categories }
\label{subsubsection-ladicrealizationofMF}

\begin{context}
\label{newlabelreferee144}
Throughout this section we fix an excellent henselian trait $S=\Spec A$ with uniformizer $\pi$. We also fix $p:X\to S$ a proper flat scheme over $S$ with $X$ regular. We consider the LG-pair $(X, f)$ where $f$ is defined as the composite 
\begin{equation}
\label{eq-Xwithuniformizer}
\xymatrix{
f:=(X\ar[r]^-p & S \ar[r]^-{\pi} & \mathbb{A}^1_S ),
}
\end{equation}
$\pi$ being our fixed uniformizer. We will simply denote denote this LG-pair by $(X,\pi)$. Also, consider the following commutative diagram with pullback squares

\begin{equation}
\label{eq-Xwithuniformizervanishingcycles}
\xymatrix{
 \ar @{} [dr] | \lrcorner  X_\eta \ar@{^{(}->}[r]^j \ar[d]^{p_\eta} &X\ar[d]^p& \ar[l]_i \ar[d]^{p_\sigma} X_{\sigma}=X_0 \ar @{} [dl] | \llcorner  \\
\ar @{} [dr] | \lrcorner  \eta \ar[d] \ar@{^{(}->}[r]_{j_\eta}& S\ar[d]^\pi&\sigma\ar[l]^{i_\sigma}\ar[d] \ar @{} [dl] | \llcorner \\
\mathbb{G}_{m,S}\ar@{^{(}->}[r]_{j_0} \ar[dr]_q& \mathbb{A}^1_S\ar[d]&\ar@{_{(}->}[l]^{i_0}S \ar@{=}[dl]\\
&S&
}
\end{equation}
which by Remark \ref{rem-choice-uniformizer} and under the hypothesis that $p$ is flat, are also derived fiber products. In particular, the canonical inclusion $t(X_0)\to X_0$ is an equivalence.
\end{context}

\vspace{0.5cm}

Corollary \ref{newcorreferee1} and Lemma \ref{lemma-null} provide cofiber-fiber sequences of $\BU_S$-modules 
\begin{equation}
\mathcal{M}_S^\vee(\Sing(X, \pi))\to \mathcal{M}_S^{\vee}(\Perf(X_0))[1]\oplus p_*i_*i^*\BU_X\to  \textrm{cofib}(p_*j_\sharp \BU_{X_\eta}\to p_*j_* \BU_{X_\eta})
\end{equation}
\begin{equation}\xymatrix{
p_*i_* i^\ast \BU_X \ar[r]^u & p_*i_*i^! \BU_X  \ar[r]& \mathcal{M}_S^\vee(\Sing(X_0))
}\end{equation}
\noindent where we have $\mathcal{M}_S^{\vee}(\Perf(X_0)))\simeq p_*i_*\BU_{X_0}$.

\vspace{0.5cm}

\begin{prop}
\label{prop-exactsequenceMFoverSlocalring}
In the notations of diagram (\ref{eq-Xwithuniformizervanishingcycles}), the canonical map
$$
\mathcal{M}_S^\vee(\Sing(X, \pi))\to (i_\sigma)_* (i_\sigma)^*(\mathcal{M}_S^\vee(\Sing(X, \pi)))
$$
\noindent is an equivalence of $\BU_S$-modules. Furthermore, there is a cofiber-fiber sequence of $\BU_{\sigma}$-modules
\begin{equation}
\label{eq-cofibersequencethatdeterminsSingoverresiduefield}
(i_\sigma)^*(\mathcal{M}_S^\vee(\Sing(X,\pi)))\to (p_\sigma)_*(\BU_{X_0}[1]\oplus\BU_{X_0})\to (p_\sigma)_*i^*j_*\BU_{X_\eta}
\end{equation}
\begin{proof}
The localization property for $\rmSH$ (see Proposition \ref{prop-AdeelAyoubCisinski} and Proposition \ref{prop-Ayoub-Cisinski}) gives us a cofiber-fiber sequence of $\BU_S$-modules
\begin{equation}
\label{eq-exactsequenceMFoverSlocalring}
(j_\eta)_\sharp\circ (j_\eta)^*(\mathcal{M}_S^\vee(\Sing(X, \pi)))\to \mathcal{M}_S^\vee(\Sing(X, \pi))\to (i_\sigma)_*\circ (i_\sigma)^*(\mathcal{M}_S^\vee(\Sing(X, \pi)))
\end{equation}
We show that the first term in the cofiber-sequence (\ref{eq-exactsequenceMFoverSlocalring}) is a zero object.  In this case the last map in (\ref{eq-exactsequenceMFoverSlocalring}) is an equivalence of $\BU_S$-modules and the motivic $\BU_S$-module of $\mathcal{M}_S^\vee(\Sing(X, \pi))$ is completely determined by its restriction to the residue field. More precisely, it is determined by a cofiber-fiber sequence of $\BU_k$-modules (\ref{eq-cofibersequencethatdeterminsSingoverresiduefield}).
To show that the first term is zero, as $j_\sharp$ is fully faithful\footnote{Follows from smooth base change for $j$ (see Prop. \ref{prop-AdeelAyoubCisinski} and \ref{prop-Ayoub-Cisinski}).}, it is enough to apply $j^*$ to the first row of the diagram (\ref{eq-exactsequenceMF}) and check it is sent to zero. Indeed, under the hypothesis that $p$ is proper, proper base change (see Prop. \ref{prop-AdeelAyoubCisinski} and \ref{prop-Ayoub-Cisinski}) gives us a natural equivalence $(j_\eta)^*p_*\simeq (p_\eta)_*j^*$. But again, the localization property tells us $j^*\circ i_*\simeq 0$, so that the first two terms in the first row of $(\ref{eq-exactsequenceMF})$ become zero. So does the cofiber $(j_\eta)^*(\mathcal{M}_S^\vee(\Sing(X_0)))$. \\
To describe $(i_{\sigma})^*(\mathcal{M}_S^\vee(\Sing(X_0)))$ we apply $(i_\sigma)^*$ to the whole diagram (\ref{eq-exactsequenceMF}). Again because of proper base change, we have a natural equivalence $(i_\sigma)^*p_*\simeq (p_\sigma)_*i^*$. Moreover, as the counit is an equivalence $(i_\sigma)^*\circ (i_\sigma)_*\simeq Id$ \footnote{See \cite[Rmk 9.4.19]{robalo-thesis}} and because 
$i_\sigma^\ast p_\ast j_\sharp \BU_{X_\eta}\simeq (p_\sigma)_\ast i^*j_\sharp\BU_{X_\eta}$ with $i^*j_\sharp$ being always zero, we recover
\begin{equation}
\label{eq-exactsequenceMFoverresiduefield}
\xymatrix{
(p_\sigma)_*\BU_{X_0}\ar[dr]_{0}\ar[r]& (p_\sigma)_*i^! \BU_X \ar[d] \ar[r]& i^*(\mathcal{M}_S^\vee(\Sing(X_0)))\ar[dd]\\
&(p_\sigma)_*\BU_{X_0}\ar[d] \ar[dr]&\\
(i_\sigma)^* (j_\eta)_* (p_\eta)_* \BU_{X_\eta}\ar@{=}[r]&(p_\sigma)_*i^*j_* \BU_{X_\eta}& \ar[l] (p_\sigma)_*\BU_{X_0}[1]\oplus (p_\sigma)_*\BU_{X_0}
}
\end{equation}

\end{proof}
\end{prop}

\begin{rem}
It follows from the localization sequence and the same arguments used in the proof of Prop. \ref{prop-exactsequenceMFoverSlocalring} that the adjunction map
$$
[\textrm{cofib}(p_*j_\sharp \BU_{X_\eta}\to p_*j_* \BU_{X_\eta})]\to (i_\sigma)_\ast i_\sigma^\ast [\textrm{cofib}(p_*j_\sharp \BU_{X_\eta}\to p_*j_* \BU_{X_\eta}]\simeq (i_\sigma)_\ast i_\sigma^\ast p_*j_* \BU_{X_\eta}
$$
is an equivalence.
\end{rem}

\vspace{1cm}
Finally, we study the image of $\mathcal{M}_S^\vee(\Sing(X, \pi))$ under the $\mathbb{Q}_\ell$-adic realization functor $\mathrm{R}_{\ell}$  of (\ref{eq-ladicrealizationBUmodules}).
\medskip

\begin{cor}
\label{newcorreferee3}
Consider the same notations as in Proposition \ref{prop-exactsequenceMFoverSlocalring}. Then:
\begin{enumerate}
\item The canonical map
$$
\mathrm{R}^\ell(\mathcal{M}_S^\vee(\Sing(X, \pi)))\to (i_\sigma)_*\circ (i_\sigma)^*(\mathrm{R}^\ell(\mathcal{M}_S^\vee(\Sing(X, \pi))))
$$
\noindent is an equivalence of $\mathrm{R}^\ell(\BU_S)$-modules. 
\medskip
\item There is a cofiber-fiber sequence of $\mathrm{R}^\ell(\BU_{\sigma})$-modules 
\medskip
\begin{equation}
\label{eq-cofibersequencethatdeterminsSingoverresiduefieldladic}
(i_\sigma)^*\mathrm{R}^\ell(\mathcal{M}_S^\vee(\Sing(X,\pi)))\to (p_\sigma)_*(\mathrm{R}^\ell(\BU_{X_0})[1]\oplus\mathrm{R}^\ell(\BU_{X_0}))\to (p_\sigma)_*i^*j_*\mathrm{R}^\ell(\BU_{X_\eta})
\end{equation}
\medskip
\end{enumerate}
\begin{proof}
 By construction $\mathrm{R}^\ell(\mathcal{M}_S^\vee(\Sing(X, \pi)))$ carries the structure of a Tate-twisted-2-periodic object in $\mathrm{Sh}_{\mathbb{Q}_\ell}(S)$. Moreover, given the fact that $p$ is proper, the combination of proper base change, of the strong compatibility between the six operations and  of Prop. \ref{prop-exactsequenceMFoverSlocalring}, implies that the $\mathbb{Q}_\ell$-adic sheaf $\mathrm{R}^{\ell}_S(\mathcal{M}_S^\vee(\Sing(X, \pi)))$ is again determined by its restriction to the residue field via a cofiber-fiber sequence of $\mathrm{R}^{\ell}_{\sigma}(\BU_{\sigma})$-modules. 
\end{proof}
\end{cor}

\medskip

\begin{rem}
\label{17ago1221}
 The combination of Prop. \ref{remark-fundamentalclasslci} with proper base change for $p$ and the fact $i^*j_\sharp=0$, tells us that the restriction of the 2-cell  (\ref{eq-exactsequence1classcycleSGA1}) to $\sigma$, provides a commutative diagram of $(p_\sigma)_\ast \BU_Z$-modules
\begin{equation}
\label{eq-exactsequenceclasscyclethatmattersabitlessbefore}
\xymatrix{ & i_\sigma^\ast  (j_\eta)_\ast (p_\eta)_* \BU_{X_\eta}[-1]\ar[d]\\
 (p_\sigma)_\ast \BU_Z \ar@{-->}[ur]^{i_\sigma^\ast\theta_{\mathrm{K}}^{(X,Z)}} \ar[r] & (p_\sigma)_\ast i^! \BU_{X_\eta}
 }
\end{equation}
Transferring this diagram along the $\ell$-adic realization, we obtain a factorization \begin{equation}
\label{eq-exactsequenceclasscyclethatmattersabitless}
 \xymatrix{& i_\sigma^\ast  (j_\eta)_\ast (p_\eta)_* \mathrm{R}^\ell(\BU_{X_\eta})[-1]\ar[d]\\
(p_\sigma)_\ast \mathrm{R}^\ell(\BU_Z)\ar@{-->}[ur]^{i_\sigma^\ast \mathrm{R}^\ell(\theta_{\mathrm{K}}^{(X,Z)})}\ar[r] & (p_\sigma)_\ast i^! \mathrm{R}^\ell(\BU_{X_\eta})}
\end{equation}
as maps of $(p_\sigma)_\ast \mathrm{R}^\ell \BU_Z$-modules.
\end{rem}

\vspace{0.5cm}

\begin{rem}
Notice that, since  $\mathrm{R}^\ell_{X_0}(\BU_{X_0})\simeq \mathbb{Q}_{\ell, X_0}(\beta)$, one has:

$$(p_\sigma)_\ast(\mathrm{R}^\ell(\BU_{X_0}))\simeq (p_\sigma)_\ast p_\sigma^\ast (\mathbb{Q}_{\ell, \sigma}(\beta))\simeq (p_\sigma)_\ast( \mathbb{Q}_{\ell, X_0})\otimes \mathbb{Q}_{\ell, \sigma}(\beta)\simeq ((p_\sigma)_\ast(\mathbb{Q}_{\ell, X_0}))(\beta)$$

\noindent where the second equivalence follows from the projection formula (as $p$ is proper), and the last equivalence follows from the definitions and from the fact that the tensor product commutes with colimits separately in each variable.

\end{rem}

\vspace{1cm}

\subsection{The action of the punctured disk $\eta$}
\label{remark-fundamentalclasslci22ladic}

Throughout this section assume the Context \ref{newlabelreferee144}. In this section we introduce a key player - the algebra structure on the cohomology of the punctured disk $\eta$. Our final goal is describe the $\ell$-adic realization of the motive of $\Sing$ computed in the previous section in terms of an action of this algebra.

\medskip
\begin{defn}
\label{definition-cohomologypunctureddisk}
The object $\mathbb{H}_{\mathbb{Q}_\ell}(\eta):= i_\sigma^\ast (j_\eta)_\ast \mathbb{Q}_{\ell,\eta}$ in  $\mathrm{Sh}_{\mathbb{Q}_\ell}(\sigma)$ is called the \emph{cohomology of the punctured disk $\eta$}.
\end{defn}

\medskip
The first observation is that $\mathbb{H}_{\mathbb{Q}_\ell}(\eta)$ carries a canonical algebra structure and its underlying object can be described in concrete terms using absolute purity for $(S,\sigma)$:

\medskip
\begin{lem}
\label{newlabelreferee141}
The object $\mathbb{H}_{\mathbb{Q}_\ell}(\eta)$ carries a canonical structure of commutative algebra object in $\mathrm{Sh}_{\mathbb{Q}_{\ell}}(\sigma)$. Furthermore, at the level of the underlying objects, we have an equivalence
$$\mathbb{H}_{\mathbb{Q}_\ell}(\eta) \simeq \mathbb{Q}_{\ell,\sigma}  \bigoplus \mathbb{Q}_{\ell,\sigma}(-1)[-1]$$ 
\begin{proof}
The algebra structure follows from the fact $i_\sigma^\ast (j_\eta)_\ast$ is lax monoidal. The computation requires two steps. The first is motivic. The second uses purity for $\ell$-adic sheaves.

\begin{enumerate}[Step 1)]
\item We claim that, given the diagram,
\begin{equation}
\xymatrix{
\mathbb{G}_{m, S}\ar[r]^{j_0}\ar[dr]_{q}& \mathbb{A}^1_S\ar[d]^{h} &\ar[l]_{i_0} S\ar[dl]^{Id}\\
&S&}
\end{equation}
\noindent we have 
\begin{equation}
\label{eq-novaboladeberlim}
(i_0)^*(j_0)_* 1_{\mathbb{G}_{m, S}}\simeq q_\ast 1_{\mathbb{G}m_S}\simeq 1_S\bigoplus 1_S(-1)[-1]
\end{equation}
\medskip
Here these operations hold for $\rmSH_S$ but also for any realization compatible with the six operations (like the $\ell$-adic one with $\mathbb{Q}_\ell$ coefficients). 
\medskip
The second equivalence in (\ref{eq-novaboladeberlim}) follows the fact that \begin{equation}\label{17ago1030}
q_\sharp 1_{\mathbb{G}_{m, S}}= 1_S\bigoplus 1_S(1)[1]
\end{equation}
Let us first show (\ref{17ago1030}): by definition of the Tate motive we obtain $1_S(1)[1]$ as the cofiber of the map in motives over $S$ of $e:S\to q_\sharp 1_{\mathbb{G}_{m, S}}$ given by the unit of the multiplicative group structure. The map $e$ admits a retract induced by the projection $q:\mathbb{G}_{m, S}\to S$. The formula (\ref{17ago1030}) follows then from the fact we are working in a stable setting.
\medskip

We can use this to establish the second equivalence in (\ref{eq-novaboladeberlim}): by adjunction and the projection formula for $q_\sharp$, we get that for any $E\in \rmSH_S$
\begin{align}
\label{17ago1347}
\Map_{\rmSH_S}( E, q_\ast q^\ast 1_S)\simeq \Map_{\rmSH_S}( q_\sharp q^\ast E, 1_S)\simeq\\
\simeq \Map_{\rmSH_S}( q_\sharp1_S\otimes E, 1_S)\simeq\Map_{\rmSH_S}( E\oplus E(1)[1], 1_S)\simeq\\
\simeq \Map_{\rmSH_S}( E, 1_S)\times \Map_{\rmSH_S}(  E(1)[1], 1_S)\simeq \\
\Map_{\rmSH_S}( E, 1_S)\times \Map_{\rmSH_S}(  E, 1_S(-1)[-1])\simeq \\
\simeq  \Map_{\rmSH_S}( E, 1_S\oplus 1_S(-1)[-1])
\end{align}
\medskip
\noindent showing that $q_\ast 1_{\mathbb{G}_{m, S}}\simeq 1_S\bigoplus 1_S(-1)[-1]$. In particular, the unit of the adjunction
$$
1_S\to q_\ast1_S\simeq 1_S\bigoplus 1_S(-1)[-1]
$$
identifies with the inclusion of the first factor in the direct sum. 
\medskip

It remains to show the first equivalence in the formula (\ref{eq-novaboladeberlim}). For that purpose it will be enough to show that 
\begin{equation}
\label{16ago1001}
(i_0)^*(j_0)_* q^\ast \simeq q_\ast q^\ast.
\end{equation}
For that, we notice that the projection $p_\sharp p^\ast \to \mathrm{Id}$ admits a section given by the $1:S\to \mathbb{G}_{m, S}$. Now we recall that $q=h\circ j_0$ so that (\ref{16ago1001}) is equivalent to
\begin{equation}
\label{new-refereeformula11}
(i_0)^*(j_0)_* j_0^\ast h^\ast \simeq h_\ast (j_0)_\ast j_0^\ast h^\ast
\end{equation}
From the localization property we have an exact sequence
\begin{equation}
\label{16ag01002}
\xymatrix{
h_\ast (i_0)_\ast i_0^! h^\ast\ar[r]& h_\ast \mathrm{Id} \, h^\ast \ar[r]& h_\ast (j_0)_\ast j_0^\ast h^\ast
}
\end{equation}
Because of $\mathbb{A}^1$-invariance we have $h_\ast \mathrm{Id} \, h^\ast \simeq \mathrm{Id}$. Moreover, as $h_\ast (i_0)_\ast = \mathrm{Id}_S$,  (\ref{16ag01002}) is equivalent to
\begin{equation}
\label{16ago1003}
\xymatrix{
\ar[d]^{\simeq}h_\ast (i_0)_\ast i_0^! h^\ast\ar[r]& \ar[d]^{\simeq} h_\ast \mathrm{Id} \, h^\ast \ar[r]&\ar[d]^{\simeq} h_\ast (j_0)_\ast j_0^\ast h^\ast\\
 i_0^! h^\ast\ar[r]& \mathrm{Id}_S\ar[r]& h_\ast (j_0)_\ast j_0^\ast h^\ast
}
\end{equation}
Finally, because $(i_0)_*$ is fully faithful and because $i_0^\ast h^\ast=\mathrm{Id}_S$ we have a commutative diagram
\begin{equation}
\label{16ago1004}
\xymatrix{
\ar[d]^{\simeq} h_\ast (i_0)_\ast i_0^! h^\ast\ar[r]& \ar[d]^{\simeq}h_\ast \mathrm{Id} \, h^\ast \ar[r]&\ar[d]^{\simeq} h_\ast (j_0)_\ast j_0^\ast h^\ast\\
 i_0^! h^\ast\ar[r] \ar[d]^{\sim}&\ar[d]^{\sim} \mathrm{Id}_S\ar[r]& h_\ast (j_0)_\ast j_0^\ast h^\ast\\
 i_0^\ast (i_0)_\ast i_0^! h^\ast\ar[r]& i_0^\ast h^\ast&
}
\end{equation}
\noindent so that, from the localization sequence 
\begin{equation}
\label{16ago1007}
\xymatrix{
 i_0^\ast \,\, [\,\,(i_0)_\ast i_0^! h^\ast\ar[r]&  h^\ast\ar[r]& (j_0)_\ast j_0^\ast h^\ast \,\,]
}
\end{equation}
\noindent we deduce (\ref{new-refereeformula11}).
Combining (\ref{16ago1004}) and (\ref{16ago1007}) we obtain a cofiber-fiber sequence
\begin{equation}
\label{16ago1013}
\xymatrix{
 i_0^! 1_{\mathbb{A}^1_S}\ar[r] & 1_S\ar[r]& q_\ast 1_S\simeq 1_S\oplus 1_S(-1)[-2]
}
\end{equation}
\noindent where the last map is the inclusion of the first factor. In particular, we find
\begin{equation}\xymatrix{1_S(-1)[-2]\ar[r]^-{\sim}& i_0^! 1_{\mathbb{A}^1_S}}\end{equation}
\medskip
\item Now we transfer this discussion along the $\mathbb{Q}_\ell$-realization: Absolute purity for $\ell$-adic sheaves, namely, the result of \cite[7.4]{MR3205601} using the fact that the hypothesis \cite[7.3]{MR3205601} holds (as proved in \cite[XVI 3.5.1]{MR3329769}), says that the exchange map

$$
\mathrm{Ex}:t^\ast i_0^!\to i_\sigma^! \pi^\ast
$$

\noindent associated to the pullback diagram

\begin{equation}
\xymatrix{
\sigma \ar[r]^{i_\sigma} \ar[d]^t& S\ar[d]^\pi\\
S\ar[r]^{i_0}& \mathbb{A}^1_S}
\end{equation}

\noindent with $\pi$ the uniformizer, is an equivalence. Therefore we obtain an equivalence of sequences

\begin{equation}
\label{16ago1015}
\xymatrix{
i_\sigma^! \pi^* \mathbb{Q}_{\ell, \mathbb{A}^1_S} \ar[r] & i_\sigma^* \pi^*\mathbb{Q}_{\ell,\mathbb{A}^1_S}  \ar[r]& i_\sigma^\ast (j_\eta)_\ast \mathbb{Q}_{\ell,\eta} = \mathbb{H}_{\mathbb{Q}_\ell}(\eta)\\
t^* i_0^! 1_{\mathbb{A}^1_S}\ar[r] \ar[u]^{\sim}_{\mathrm{Ex}}& \ar[u]^{\sim} t^* i_0^*1_{\mathbb{A}^1_S}\ar[r]& t^*( i_0)^\ast (j_0)_\ast \mathbb{Q}_{\ell, \mathbb{G}m_S}  \ar[u]
}
\end{equation}

\noindent so that the last terms on the right are also equivalent, namely, by (\ref{16ago1013}) we find
\begin{equation}
\mathbb{H}_{\mathbb{Q}_\ell}(\eta)\simeq \mathbb{Q}_{\ell,\sigma} \oplus \mathbb{Q}_{\ell,\sigma} (-1)[-1]
\end{equation}

\noindent and the top sequence in (\ref{16ago1015}) reads as

$$
i_\sigma^! \mathbb{Q}_{\ell,S} \to \mathbb{Q}_{\ell,\sigma}  \to \mathbb{Q}_{\ell,\sigma} \oplus \mathbb{Q}_{\ell,\sigma} (-1)[-1]
$$

\noindent where the last map has a splitting. The exchange map gives the purity isomorphism for $(\sigma,S,\eta)$

\begin{equation}
\label{eq-purityisofor}
\xymatrix{\mathbb{Q}_{\ell,\sigma} (-1)[-2]\ar[r]^{\sim}& i_\sigma^! \mathbb{Q}_{\ell,S}}
\end{equation}
\end{enumerate}
\end{proof}
\end{lem}

\medskip

\begin{rem}
\label{16ago1716}
Using the description of $\mathrm{R}^\ell(\BU_\sigma)$ as an infinite direct sum of Tate twists (\ref{eq-ladicrealizationBUQ}) we can consider the map
\begin{equation}
\label{16ago1557}
\xymatrix{\mathbb{Q}_{\ell,\sigma} (-1)[-2] \ar[r]^-{\mathrm{inc}}& \mathrm{R}^\ell(\BU_\sigma)\ar[r]^u& i_\sigma^! \mathrm{R}^\ell(\BU_S)\ar[r]^-{\mathrm{proj}_0}& i_\sigma^! \mathbb{Q}_{\ell,S} }
\end{equation}
\noindent where:
\begin{itemize}
\item  the first map is the canonical inclusion. Under the equivalence (\ref{eq-ladicrealizationBUQ}) this inclusion can also be described as the Bott element, $\beta$ (\ref{eq-bottelement}).
\item  the last map is the projection at level $0$.
\item the middle map is the purity morphism for algebraic K-theory (\ref{16ago1554purityalgebraicKtheory}). 
\end{itemize}
See \cite[14.4.1]{cisinski-tcmm}. We claim that the purity map (\ref{eq-purityisofor}) coincides with the composition (\ref{16ago1557}). \textcolor{black}{ This follows because of the compatibility of the exchange map under the six operations, before and after passing to $\mathrm{R}^\ell(\BU)$-modules. }
\medskip
In particular, by the universal property of base-changing to $\mathrm{R}^\ell(\BU_\sigma)$, (\ref{16ago1557}) provides a $\mathrm{R}^\ell(\BU_\sigma)$-linear commutative square 
\begin{equation}
\label{16ago2152}
\xymatrix{
 \mathrm{R}^\ell(\BU_\sigma) \ar[rr]^-{u}_-{\sim} && i_\sigma^! \mathrm{R}^\ell(\BU_S)\\
\ar[u]^{-\ast \beta}_-{\sim}\mathrm{R}^\ell(\BU_\sigma)\otimes \mathbb{Q}_{\ell,\sigma} (-1)[-2] \ar[rr]^-{Id\otimes (\ref{eq-purityisofor})}_-{\sim} && \ar[u]^-{\mathrm{can}}_-{\sim} \mathrm{R}^\ell(\BU_\sigma)\otimes  i_\sigma^! \mathbb{Q}_{\ell,S}
}
\end{equation}
\noindent which under (\ref{eq-ladicrealizationBUQ}) says that $u$ is the multiplication by $\beta^{-1}$.
\end{rem}

\medskip

Let us now extract some consequences of the Lemma.

\medskip

\begin{rem}
\label{newrefereelabel141713}
In light of the lemma, $\mathbb{H}_{\mathbb{Q}_\ell}(\eta)$ can be understood as the cohomology of a circle but with a generator in Tate degree $(-1,-1)$. Multiplication by this generator, defined by the composition
\begin{equation}
\label{korokseed46464}
\xymatrix{\mathbb{Q}_{\ell, \sigma}(-1)[-1]\otimes \mathbb{H}_{\mathbb{Q}_\ell}(\eta)\ar[rr]^-{incl\otimes Id}& &  \mathbb{H}_{\mathbb{Q}_\ell}(\eta)\otimes \mathbb{H}_{\mathbb{Q}_\ell}(\eta) \ar[rr]&&\mathbb{H}_{\mathbb{Q}_\ell}(\eta)}
\end{equation}
\noindent provides a map of $\mathbb{H}_{\mathbb{Q}_\ell}(\eta)$-modules
\begin{equation}
\label{korokseed5757575}
\mathbb{H}_{\mathbb{Q}_\ell}(\eta)\to \mathbb{H}_{\mathbb{Q}_\ell}(\eta)(1)[1]
\end{equation}
Rewinding base change along $\mathbb{Q}_{\ell, \sigma}\to \mathbb{H}_{\mathbb{Q}_\ell}(\eta)$, this map corresponds to a map of $\mathbb{Q}_{\ell, \sigma}$-modules
\begin{equation}
\label{16ago1353}
\mathbb{Q}_{\ell, \sigma}\to \mathbb{H}_{\mathbb{Q}_\ell}(\eta)(1)[1]
\end{equation}
\noindent recovering the element of Tate degree (-1,-1).
\end{rem}

\medskip

\begin{prop}
\label{propclassdelignefundamentalcycle}(\cite[Cycle \S 2.1]{MR0463174} or \cite[Exp. XVI Def. 2.3.1, Prop. 2.3.4]{MR3329769} and \cite[\S 1, 1.1.2, 1.1.5]{fujiwara2000proof})
There exists an $\ell$-adic class 
\begin{equation}
\label{korokseed6757575040}
[\theta_\ell^{(S,\sigma)}: \mathbb{Q}_{\ell, \eta}\to \mathbb{Q}_{\ell, \eta}(1)[1]]\in \mathrm{H}^1(\eta, \mathbb{Q}_{\ell}(1))
\end{equation} such that:
\begin{enumerate}
\item its image under $i_\sigma^\ast (j_\eta)_\ast$ is the map (\ref{korokseed5757575});
\item  its image under the boundary map of the open-closed pair with support $(\sigma, S, \eta)$
\begin{equation}
\label{korokseed66764242470}
\mathrm{H}^1(\eta, \mathbb{Q}_{\ell}(1))\to \mathrm{H}^2_\sigma(S, \mathbb{Q}_{\ell}(1))
\end{equation}
\noindent is the first Chern class of the conormal bundle of $\sigma$ in $S$.
\end{enumerate}
\end{prop}

\vspace{1cm}

\begin{cons}
\label{newlabelrefereediaquasenofim}
The algebra structure on the punctured disk can be transferred along pullbacks. Let $p:X\to S$ be as in the Context \ref{newlabelreferee144} and set
\begin{equation}
\label{newref1467748339}
\mathbb{H}_{\mathbb{Q}_\ell}(X_\eta):= (p_\sigma)_\ast i^\ast j_\ast  \mathbb{Q}_{\ell, X_\eta}\simeq i_\sigma^\ast (j_\eta)_\ast (p_\eta)_\ast \mathbb{Q}_{\ell, X_\eta}
\end{equation}
The unit of the adjunction $\mathbb{Q}_{\ell, \eta}\to (p_\eta)_\ast \mathbb{Q}_{\ell, X_\eta}$, being a map of algebras, produces under the lax-monoidal functor $i_\sigma^\ast (j_\eta)_\ast$, a map of commutative algebra objects
\begin{equation}
\label{newref146774833943}
m_p: \mathbb{H}_{\mathbb{Q}_\ell}(\eta)\to \mathbb{H}_{\mathbb{Q}_\ell}(X_\eta)
\end{equation}
\medskip
In particular, this guarantees the commutativity of the diagram "multiplication by $\theta_\ell^{(S,\sigma)}$"
\begin{equation}
\label{newrefereelabel1776459}
\xymatrix{\mathbb{H}_{\mathbb{Q}_\ell}(\eta)\ar[rr]^{-\ast \theta_\ell^{(S,\sigma)}} \ar[d]^{m_p}&&\mathbb{H}_{\mathbb{Q}_\ell}(\eta)(1)[1] \ar[d]^{m_p}\\
\mathbb{H}_{\mathbb{Q}_\ell}(X_\eta)\ar[rr]^{-\ast m_p(\theta_\ell^{(S,\sigma)})}&& \mathbb{H}_{\mathbb{Q}_\ell}(X_\eta)(1)[1]}
\end{equation}
\noindent and implies that the bottom map can also be obtained by pulling $\theta_\ell^{(S,\sigma)}$ back to $X_\eta$: 
$$
\xymatrix{ -\ast m_p(\theta_\ell^{(S,\sigma)}) = (p_\sigma)_\ast i^*j_\ast p_\eta^\ast \big[ \,\,  \mathbb{Q}_{\ell, X_\eta}\ar[rr]^-{\theta_\ell^{(S,\sigma)}}&& \mathbb{Q}_{\ell, X_\eta}(1)[1] \,\,\big] }
$$
Notice at the same time that $\mathbb{H}_{\mathbb{Q}_\ell}(X_\eta)$ is also a $(p_\sigma)_\ast \mathbb{Q}_{\ell, X_\sigma}$-algebra via a map
\begin{equation}
\label{newrefereelabel146459}
(p_\sigma)_\ast \mathbb{Q}_{\ell, X_\sigma}\to \mathbb{H}_{\mathbb{Q}_\ell}(X_\eta)
\end{equation}
\noindent obtained from the map of algebras $\mathbb{Q}_{\ell, X}\to j_\ast j^\ast \mathbb{Q}_{\ell, X}$. As coproducts of commutative algebras are given by tensor products \cite[3.2.4.7]{lurie-ha}, the combined actions of 
$(p_\sigma)_\ast \mathbb{Q}_{\ell, X_\sigma}$ and $\mathbb{H}_{\mathbb{Q}_\ell}(\eta)$ are encoded by a single map of commutative algebras
\begin{equation}
\label{newref146774jajjcaclnl833943}
\mathbb{H}_{\mathbb{Q}_\ell}(\eta)\otimes_{\mathbb{Q}_{\ell,\sigma}} (p_\sigma)_\ast \mathbb{Q}_{\ell, X_\sigma} \to \mathbb{H}_{\mathbb{Q}_\ell}(X_\eta)
\end{equation}
\end{cons}

\vspace{1cm}

Let us now transfer this discussion to $\mathrm{R}^\ell(\BU)$-modules.

\begin{defn}
The \emph{$\mathrm{R}^\ell(\BU)$-valued cohomology of the punctured disk $\eta$} is the object $\mathbb{H}^{\mathrm{R}^\ell(\BU)}_{\mathbb{Q}_\ell}(\eta):= i_\sigma^\ast (j_\eta)_\ast \mathrm{R}^\ell(\BU_\eta)$ in  $\Mod_{\mathrm{R}^\ell(\BU_\sigma)}\mathrm{Sh}_{\mathbb{Q}_\ell}(\sigma)$.
\end{defn}

\medskip
\begin{cons}
\label{17ago1220}
The standard lax monoidal argument tells us that $\mathbb{H}^{\mathrm{R}^\ell(\BU)}_{\mathbb{Q}_\ell}(\eta)$ is a commutative algebra in  $\Mod_{\mathrm{R}^\ell(\BU_\sigma)}\mathrm{Sh}_{\mathbb{Q}_\ell}(\sigma)$. The map of algebras 
$$
\mathbb{Q}_{\ell, X_\eta}\to \mathrm{R}^\ell(\BU_\eta)
$$
\noindent induces a map of $\mathbb{Q}_{\ell,\sigma}$-algebras
\begin{equation}
\label{18ago2217}
 \mathbb{H}_{\mathbb{Q}_\ell}(\eta)\to \mathbb{H}^{\mathrm{R}^\ell(\BU)}_{\mathbb{Q}_\ell}(\eta)
\end{equation}
Furthermore, the fact that $\mathbb{H}^{\mathrm{R}^\ell(\BU)}_{\mathbb{Q}_\ell}(\eta)$ is a $\mathrm{R}^\ell(\BU_\sigma)$-algebra, tells us that under base-change (\ref{18ago2217}) produces a map of $\mathrm{R}^\ell(\BU_\sigma)$-algebras
\begin{equation}
\label{18ago2217}
 \mathbb{H}_{\mathbb{Q}_\ell}(\eta)\otimes_{\mathbb{Q}_{\ell,\sigma}} \mathrm{R}^\ell(\BU_\sigma) \to \mathbb{H}^{\mathrm{R}^\ell(\BU)}_{\mathbb{Q}_\ell}(\eta)
\end{equation}
\noindent which we can show to be an equivalence: indeed, it is enough to check that it is an equivalence between the underlying objects and here we have
$$
\mathbb{H}_{\mathbb{Q}_\ell}(\eta)\otimes_{\mathbb{Q}_{\ell,\sigma}} \mathrm{R}^\ell(\BU_\sigma)\simeq  i_\sigma^\ast (j_\eta)_\ast (\mathbb{Q}_{\ell, \eta}) \otimes_{\mathbb{Q}_{\ell,\sigma}} \bigoplus_{i\in \mathbb{Z}}\mathbb{Q}_{\ell}(i)[2i]\simeq \bigoplus_{i\in \mathbb{Z}}  i_\sigma^\ast (j_\eta)_\ast (\mathbb{Q}_{\ell, \eta}) (i)[2i]\simeq 
$$

$$
\simeq \bigoplus_{i\in \mathbb{Z}}  i_\sigma^\ast (j_\eta)_\ast (\mathbb{Q}_{\ell, \eta}(i)[2i]) \simeq i_\sigma^\ast (j_\eta)_\ast (\bigoplus_{i\in \mathbb{Z}}   (\mathbb{Q}_{\ell, \eta}(i)[2i])\simeq \mathbb{H}^{\mathrm{R}^\ell(\BU)}_{\mathbb{Q}_\ell}(\eta)
$$

\medskip
\noindent where we used the fact that both $^*$-pullbacks and $_*$-pushforwards preserve arbitrary colimits and Tate-twists \footnote{This will be explained in the proof of Prop. \ref{remark-vanishingcyclesperiodicequaltensorusual} below.}

\medskip

Finally, the map of algebras (\ref{18ago2217}) renders the commutavity of the diagram
"multiplication by $\theta_\ell^{(S,\sigma)}$"
\begin{equation}
\label{newrefereelabel1776459}
\xymatrix{\mathbb{H}_{\mathbb{Q}_\ell}(\eta)\ar[rr]^{-\ast \theta_\ell^{(S,\sigma)}} \ar[d]&&\mathbb{H}_{\mathbb{Q}_\ell}(\eta)(1)[1] \ar[d]\\
\mathbb{H}^{\mathrm{R}^\ell(\BU)}_{\mathbb{Q}_\ell}(\eta)\ar[rr]^{-\ast \theta_\ell^{(S,\sigma)}}&& \mathbb{H}^{\mathrm{R}^\ell(\BU)}_{\mathbb{Q}_\ell}(\eta)(1)[1]}
\end{equation}
\noindent and implies that the bottom map can be obtained by tensoring with $\theta_\ell^{(S,\sigma)}$ 
$$
\xymatrix{i_\sigma^\ast (j_\eta)_\ast \big[ \,\,\mathrm{R}^\ell(\BU_\eta)\ar[rr]^{Id\otimes \theta_\ell^{(S,\sigma) }}&& \mathrm{R}^\ell(\BU_\eta)(1)[1] \,\,\big] }
$$
\end{cons}

\vspace{1cm}
\begin{cons}
\label{17ago1219}
Under the same hypothesis as in \ref{newlabelrefereediaquasenofim}, the algebra structure on the $\mathrm{R}^\ell(\BU)$-valued punctured disk can be transferred along maps $p:X\to S$.  Set
$$
\mathbb{H}^{\mathrm{R}^\ell(\BU)}_{\mathbb{Q}_\ell}(X_\eta):= (p_\sigma)_\ast i^\ast j_\ast \mathrm{R}^\ell(\BU_\eta)\simeq i_\sigma^\ast (j_\eta)_\ast (p_\eta)_\ast \mathrm{R}^\ell(\BU_\eta)
$$
Repeating the same arguments as above, we find a map of algebra objects
$$
\mathbb{H}^{\mathrm{R}^\ell(\BU)}_{\mathbb{Q}_\ell}(\eta)\to\mathbb{H}^{\mathrm{R}^\ell(\BU)}_{\mathbb{Q}_\ell}(X_\eta)
$$
\noindent which, composed with $
 \mathbb{H}_{\mathbb{Q}_\ell}(\eta)\to \mathbb{H}^{\mathrm{R}^\ell(\BU)}_{\mathbb{Q}_\ell}(\eta)
$, renders the compatibility with "multiplication by $\theta_\ell^{(S,\sigma)}$"
\begin{equation}
\label{newrefereelabel1776459hhaf9}
\xymatrix{\mathbb{H}_{\mathbb{Q}_\ell}(\eta)\ar[rr]^{-\ast \theta_\ell^{(S,\sigma)}} \ar[d]^{m_p}&&\mathbb{H}_{\mathbb{Q}_\ell}(\eta)(1)[1] \ar[d]^{m_p}\\
\mathbb{H}^{\mathrm{R}^\ell(\BU)}_{\mathbb{Q}_\ell}(X_\eta)\ar[rr]^{-\ast m_p(\theta_\ell^{(S,\sigma)})}&& \mathbb{H}^{\mathrm{R}^\ell(\BU)}_{\mathbb{Q}_\ell}(X_\eta)}
\end{equation}
\noindent and implies that the bottom map can also be obtained as
$$
\xymatrix{i_\sigma^\ast (j_\eta)_\ast p_\eta^\ast \big[ \,\,\mathrm{R}^\ell(\BU_\eta)\ar[rr]^{Id\otimes \theta_\ell^{(S,\sigma) }}&& \mathrm{R}^\ell(\BU_\eta)(1)[1] \,\,\big] }
$$
\noindent 
At the same time, as in \ref{newlabelrefereediaquasenofim},  $\mathbb{H}^{\mathrm{R}^\ell(\BU)}_{\mathbb{Q}_\ell}(X_\eta)$ is also a $(p_\sigma)_\ast \mathrm{R}^\ell(\BU_\sigma)$-algebra through a map of $\BU_\sigma$-algebras
\begin{equation}
\label{17ago2131}
(p_\sigma)_\ast \mathrm{R}^\ell(\BU_\sigma)\to \mathbb{H}^{\mathrm{R}^\ell(\BU)}_{\mathbb{Q}_\ell}(X_\eta)
\end{equation}
 Using \cite[3.2.4.7]{lurie-ha}, the combined actions can be assembled as a map of $\mathbb{Q}_{\ell, \sigma}$-algebras
\begin{equation}
\label{newref146774jajfunnotfunjcaclnl833943}
\xymatrix{\mathbb{H}_{\mathbb{Q}_\ell}(\eta)\otimes_{\mathbb{Q}_{\ell,\sigma}} (p_\sigma)_\ast \mathrm{R}^\ell(\BU_\sigma) \ar[r]^-{\mathrm{can}}&\mathbb{H}^{\mathrm{R}^\ell(\BU)}_{\mathbb{Q}_\ell}(X_\eta)}
\end{equation}
\noindent which can equivalently be written as a map of $\BU_\sigma$-algebras
\begin{equation}
\label{newref146774jajfunnotfunjcaclnl833943}
\xymatrix{[\,\BU_\sigma\otimes_{\mathbb{Q}_{\ell,\sigma}}\mathbb{H}_{\mathbb{Q}_\ell}(\eta)\,]\,\otimes_{\BU_\sigma} (p_\sigma)_\ast \mathrm{R}^\ell(\BU_\sigma) \ar[r]&\mathbb{H}^{\mathrm{R}^\ell(\BU)}_{\mathbb{Q}_\ell}(X_\eta)}
\end{equation}
To conclude, let us remark one can re-write $-\ast m_p( \theta_\ell^{(S,\sigma)})$ as map of $(p_\sigma)_\ast \mathrm{R}^\ell(\BU_{X_\sigma})$-modules 
\begin{equation}
\label{newlabelfinalidentificationclassesbasesigmacas23e}
\,\,(p_\sigma)_\ast \mathrm{R}^\ell(\BU_{X_\sigma})\to \mathbb{H}_{\mathbb{Q}_\ell}(X_\eta)(1)[1]
\end{equation}
\end{cons}

\medskip

\begin{rem}
In the case where $p:X\to S$ is smooth and proper, smooth base-change for $ p$ gives $j_\ast p_\eta^\ast \simeq p^\ast (j_0)_\ast$ and tells us that the map $\mathrm{can}$ of  (\ref{newref146774jajfunnotfunjcaclnl833943}) is an equivalence of algebras.
\end{rem}

\medskip

Finally, we can achieve the main goal of this section and explain the link between the action of the punctured disk encoded by the map of algebras (\ref{newref146774jajfunnotfunjcaclnl833943}) and the $\ell$-adic realization of $\Sing$.

\vspace{1cm}

\begin{prop}
\label{16ago2244}
Forgetting the commutative algebra structures to modules, the fiber sequence
\begin{equation}
\label{newref146774jajfunnotfunnotfunagainjcaclnl833943}
\xymatrix{\mathrm{fib}(\mathrm{can})\ar[r]&\mathbb{H}_{\mathbb{Q}_\ell}(\eta)\otimes_{\mathbb{Q}_{\ell,\sigma}} (p_\sigma)_\ast \mathrm{R}^\ell(\BU_{X_\sigma}) \ar[r]^-{\mathrm{can}}&\mathbb{H}^{\mathrm{R}^\ell(\BU)}_{\mathbb{Q}_\ell}(X_\eta)}
\end{equation}
is equivalent to the fiber sequence (\ref{eq-cofibersequencethatdeterminsSingoverresiduefieldladic}). In particular
$$
(i_\sigma)^*\mathrm{R}^\ell(\mathcal{M}_S^\vee(\Sing(X,\pi)))\simeq  \mathrm{fib}(\mathrm{can})
$$
\personal{Notice that as 0 is not a map of algebras, this limit is of course not a limit diagram in algebras, only in modules.}
\begin{proof}
Proper base change for perfect complexes gives us a diagram of $S$-dg-categories
$$
\xymatrix{
\Perf(X_\sigma)\ar[r]^-{i_*}&\Perf(X)\ar[r]^-{i^*}& \Perf(X_\sigma)\\
\Perf(\sigma) \ar[u]^{(p_\sigma)^\ast}\ar[r]^-{(i_\sigma)_*}&\ar[u]^{p^\ast} \Perf(S)\ar[r]^-{ i_\sigma^*}& \Perf(\sigma) \ar[u]^{(p_\sigma)^\ast}
}
$$
\noindent where:

\begin{itemize}
\item from the projection formula for $i$, the maps in the top line are $\Perf(X_\sigma)^\otimes$-linear;
\item   from the projection formula for $i_\sigma$ the maps in the bottom line are $\Perf(\sigma)^\otimes$-linear;
\item $(p_\sigma)^\ast$ is symmetric monoidal.
\end{itemize}

\medskip

Transferring this diagram to motives and using the identifications of Sections \ref{subsubsection-eq-XschemeS} and \ref{subsubsection-contextopenclosed3}, we find a commutative diagram of $\BU_\sigma$-motives

\begin{equation}
\label{16ago1754}
\resizebox{1. \hsize}{!}{
\xymatrix{
(p_\sigma)_\ast \BU_{X_\sigma} \simeq i_\sigma^\ast \mathcal{M}_S^\vee (\Perf(X_\sigma)) \ar[drrr]_(.35){i_\sigma^\ast \mathcal{M}_S^\vee (i^\ast i_\ast)\sim 0}\ar[rr]^-{u^{(X,X_\sigma)}}&& i_\sigma^\ast \mathcal{M}_S^\vee (\Coh(X_\sigma))\simeq (p_\sigma)_\ast i_\sigma^!\BU_S \ar[dr]&\\   
&&&(p_\sigma)_\ast \BU_\sigma \simeq i_\sigma^\ast \mathcal{M}_S^\vee (\Perf(X_\sigma)) \\ 
(p_\sigma)_\ast \BU_{X_\sigma}\simeq i_\sigma^\ast \mathcal{M}_S^\vee (\Perf(\sigma))  \ar[drrr]_-{i_\sigma^\ast \mathcal{M}_S^\vee (i_\sigma^\ast (i_\sigma)_\ast)\sim 0}\ar[rr]^-{u^{(S,\sigma)}}_-{\sim}\ar[uu]^-{i_\sigma^\ast \mathcal{M}_S^\vee (p_\sigma^\ast)}&& i_\sigma^! \BU_S\simeq i_\sigma^\ast \mathcal{M}_S^\vee (\Coh(\sigma))  \ar[dr]^0 \ar[uu]^-{i_\sigma^\ast \mathcal{M}_S^\vee (p_\sigma^\ast)}&\\ 
&&&  \BU_\sigma\simeq i_\sigma^\ast \mathcal{M}_S^\vee (\Perf(\sigma)) \ar[uu]^-{i_\sigma^\ast \mathcal{M}_S^\vee (p_\sigma^\ast)}\\ 
}
}
\end{equation}

\noindent where we have:
\begin{itemize}
\item the top face is the $(p_\sigma)_\ast \BU_{X_\sigma}$-linear commutative diagram (\ref{eq-exactsequenceMFker}).
\item the bottom face is the $\BU_\sigma$-linear version of the diagram (\ref{eq-exactsequenceMFker}) for $p=Id:S\to S$.
\item the maps $i_\sigma^\ast \mathcal{M}_S^\vee (p_\sigma^\ast):\BU_\sigma\to (p_\sigma)_\ast \BU_{X_\sigma}$ coincide with the natural unit of the adjunction. On the extreme left and right of the diagram, these are maps of algebra-objects. The middle one is a map of modules.
\end{itemize}

\medskip

In this case, by the universal property of base change along $\BU_\sigma\to (p_\sigma)_\ast \BU_{X_\sigma}$ the diagram (\ref{16ago1754}) is equivalent to $(p_\sigma)_\ast \BU_{X_\sigma}$-linear commutative diagram

\begin{equation}
\label{16ago1757}
\resizebox{1. \hsize}{!}{
\xymatrix{
&& (p_\sigma)_\ast i^!\BU_X \ar[dddr]&\\   
&&& \\ 
\ar @/_2.0pc/  [drrr] |{0\sim \mathrm{Id}\otimes i_\sigma^\ast \mathcal{M}_S^\vee (i_\sigma^\ast (i_\sigma)_\ast) \sim i_\sigma^\ast \mathcal{M}_S^\vee (i^\ast (i)_\ast)}  (p_\sigma)_\ast \BU_{X_\sigma}\simeq (p_\sigma)_\ast \BU_{X_\sigma}\otimes_{\BU_\sigma} \BU_\sigma  \ar[rr]^-{Id\otimes u^{(S,\sigma)}}_-{\sim}\ar[uurr]^-{u^{(S,\sigma)}}&& (p_\sigma)_\ast \BU_{X_\sigma}\otimes_{\BU_\sigma} i_\sigma^!\BU_S \ar[dr]^{Id\otimes 0} \ar@{-->}[uu]&\\ 
&&&  (p_\sigma)_\ast \BU_{X_\sigma}\simeq (p_\sigma)_\ast \BU_{X_\sigma}\otimes_{\BU_\sigma} \BU_\sigma   \\ 
}
}
\end{equation}

\medskip

Finally, we transfer this diagram along the $\ell$-adic realization and explain how to conclude the proof. First notice the information of the fiber sequence (\ref{eq-cofibersequencethatdeterminsSingoverresiduefieldladic}) is already present in the diagram (\ref{16ago1757}) by passing to cofibers. Indeed, passing to cofibers in the top face, we get 

\begin{equation}
\label{17ago1052}
\xymatrix{\mathrm{Cofib}\,\,[\,  i_\sigma^\ast \mathcal{M}_S^\vee (i^\ast (i)_\ast)\,]\ar[r]^-{(\ref{eq-cofibersequencethatdeterminsSingoverresiduefieldladic})}& \mathbb{H}^{\mathrm{R}^\ell(\BU)}_{\mathbb{Q}_\ell}(X_\eta)}
\end{equation}

Now, the commutativity of (\ref{16ago1757}) together with the commutativity of

\begin{equation}
\label{16ago1900}
\resizebox{1. \hsize}{!}{
\xymatrix{
\ar[d]^-{Id\otimes 0}(p_\sigma)_\ast \mathrm{R}^\ell(\BU_{X_\sigma})\otimes_{\mathrm{R}^\ell(\BU_\sigma)} i_\sigma^!\mathrm{R}^{\ell}(\BU_S) \ar[r]^-{\sim}  &
 (p_\sigma)_\ast \mathrm{R}^\ell(\BU_{X_\sigma})\otimes_{\mathrm{R}^\ell(\BU_\sigma)} (\mathrm{R}^\ell(\BU_\sigma)\otimes_{\mathbb{Q}_{\ell, \sigma}} i_\sigma^!\mathbb{Q}_{\ell, S} )\ar[r]^-{\sim} &
  (p_\sigma)_\ast \mathrm{R}^\ell(\BU_{X_\sigma})\otimes_{\mathbb{Q}_{\ell, \sigma}} i_\sigma^!\mathbb{Q}_{\ell, S} \ar[d]^-{Id\otimes 0}\\
\mathrm{R}^\ell(\BU_{X_\sigma})\otimes_{\mathrm{R}^\ell(\BU_\sigma)} \mathrm{R}^{\ell}(\BU_\sigma) \ar[r]^-{\sim}  &
(p_\sigma)_\ast \mathrm{R}^\ell(\BU_{X_\sigma})\otimes_{\mathrm{R}^\ell(\BU_\sigma)} (\mathrm{R}^\ell(\BU_\sigma)\otimes_{\mathbb{Q}_{\ell, \sigma}} \mathbb{Q}_{\ell, \sigma} )\ar[r]^-{\sim} &
 (p_\sigma)_\ast \mathrm{R}^\ell(\BU_{X_\sigma})\otimes_{\mathbb{Q}_{\ell, \sigma}} \mathbb{Q}_{\ell, \sigma}
}
}
\end{equation}

\noindent combined with the commutativity of (\ref{16ago2152}), tells us that after taking cofibers in (\ref{16ago1757}) we obtain a new commutative diagram establishing the identification of cofiber-sequences

\begin{equation}
\label{16ago1908}
\resizebox{1. \hsize}{!}{
\xymatrix{
\mathrm{Cofib}\,\,[\,  i_\sigma^\ast \mathcal{M}_S^\vee (i^\ast (i)_\ast)\,]\ar[dr]_{(\ref{eq-cofibersequencethatdeterminsSingoverresiduefieldladic})} \ar[r]^-{\sim} & (p_\sigma)_\ast \mathrm{R}^\ell(\BU_{X_\sigma})\otimes_{\mathbb{Q}_{\ell, \sigma}} \mathrm{Cofib}\,\,[\,i_\sigma^!\mathbb{Q}_{\ell, S}\to \mathbb{Q}_{\ell, \sigma}\,]\,  \ar[r]^-{\sim} \ar[d] & (p_\sigma)_\ast \mathrm{R}^\ell(\BU_{X_\sigma})\otimes_{\mathbb{Q}_{\ell, \sigma}}\mathbb{H}_{\mathbb{Q}_\ell}(\eta)  \ar[dl]^{}\\ 
&\mathbb{H}^{\mathrm{R}^\ell(\BU)}_{\mathbb{Q}_\ell}(X_\eta)&
}
}
\end{equation}

Finally, we argue that the right-diagonal map in (\ref{16ago1908})  is the map $\mathrm{can}$ appearing in (\ref{newref146774jajfunnotfunnotfunagainjcaclnl833943}). Indeed, $\mathrm{can}$ is constructed as a coproduct of commutative-algebra maps and therefore determined by its restriction to $(p_\sigma)_\ast \mathrm{R}^\ell(\BU_{X_\sigma})$ and to $\mathbb{H}_{\mathbb{Q}_\ell}(\eta)$. Some diagram chasing shows that the restrictions of right-diagonal map in (\ref{16ago1908}) coincide with the ones of $\mathrm{can}$.
\end{proof}
\end{prop}

\medskip

\begin{rem}
\label{17ago1218}
A consequence of the Corollary \ref{16ago2244} obtained through some diagram chasing is that the classes $-\ast m_p( \theta_\ell^{(S,\sigma)})$ of (\ref{newlabelfinalidentificationclassesbasesigmacas23e})
 and $i_\sigma^\ast \mathrm{R}^\ell(\theta_{\mathrm{K}}^{(X,Z)})$ of 
(\ref{eq-exactsequenceclasscyclethatmattersabitless}) coincide.
\end{rem}

\vspace{1cm}

We conclude this section relating the algebra $\mathbb{H}_{\mathbb{Q}_\ell}(\eta)\otimes_{\mathbb{Q}_{\ell,\sigma}} \mathrm{R}^\ell(\BU_\sigma)$ to the algebra  $i_\sigma^\ast \mathrm{R}^\ell(\mathcal{M}_S^\vee(\Sing(S,0))$ of (\ref{19ago1200}).

\medskip

\begin{prop}
\label{20ago2141}
There is a canonical equivalence of commutative $\mathrm{R}^\ell(\BU_\sigma)$-algebra objects
\begin{equation}
\label{20ago2120}
\mathbb{H}_{\mathbb{Q}_\ell}(\eta)\otimes_{\mathbb{Q}_{\ell,\sigma}} \mathrm{R}^\ell(\BU_\sigma)\simeq i_\sigma^\ast \mathrm{R}^\ell(\mathcal{M}_S^\vee(\Sing(S,0))
\end{equation}
\begin{proof}
As a $\mathrm{R}^\ell(\BU_\sigma)$-module, we know that the l.h.s is equivalent to $\mathrm{R}^\ell(\BU_\sigma)\oplus \mathrm{R}^\ell(\BU_\sigma)[1]$ (Prop. \ref{prop-BUmotive2periodic}) and the r.h.s is equivalent to $\mathrm{R}^\ell(\BU_\sigma)\oplus \mathrm{R}^\ell(\BU_\sigma)(-1)[-1]$ (Prop. \ref{newlabelreferee141} ). Bott periodicity imposes $\mathrm{R}^\ell(\BU_\sigma)(-1)[-1]\simeq \mathrm{R}^\ell(\BU_\sigma)[1]$ and the inclusion of this factor on both the left and right terms, gives under the universal property of the free symmetric $\mathrm{R}^\ell(\BU_\sigma)$-algebra, maps of commutative algebra objects

\begin{equation}
\label{20ago2138}
\xymatrix{
\mathbb{H}_{\mathbb{Q}_\ell}(\eta)\otimes_{\mathbb{Q}_{\ell,\sigma}} \mathrm{R}^\ell(\BU_\sigma)& \ar[l] \mathrm{Sym}_{\mathrm{R}^\ell(\BU_\sigma)}(\mathrm{R}^\ell(\BU_\sigma)[1])\ar[r]& i_\sigma^\ast \mathrm{R}^\ell(\mathcal{M}_S^\vee(\Sing(S,0))
}
\end{equation}

Finally, as we are working with $\mathbb{Q}_\ell$-coefficients and the cohomology of the symmetric groups are zero in characteristic zero, we find that as $\mathrm{R}^\ell(\BU_\sigma)$-modules we have

$$
 \mathrm{Sym}_{\mathrm{R}^\ell(\BU_\sigma)}(\mathrm{R}^\ell(\BU_\sigma)[1])\simeq  \mathrm{R}^\ell(\BU_\sigma)\oplus \mathrm{R}^\ell(\BU_\sigma)[1]
$$

\noindent showing that both maps in the diagram (\ref{20ago2138}) are equivalences.
\end{proof}
\end{prop}

\vspace{1cm}
\subsection{$\ell$-adic inertia-invariant vanishing cycles }
\label{subsubsection-descriptionofmotiveofvanishingcycles}
In this section we investigate more closely the sequence defining vanishing cycles in Def. \ref{17ago1329}, with the final goal of relating this sequence to the one characterizing the motive of the singularity category. \medskip

\vspace{0.5cm}
Consider the Context \ref{newlabelreferee144}. Let $\mathcal{V}_p(\beta) := \mathcal{V}_{p}(E)$ where $E:=\mathrm{R}^{\ell}_X(\BU_{X})\simeq \mathbb{Q}_{\ell}(\beta)_X$, and recall (Definition \ref{17ago1329}) our convention $\mathcal{V}_{p}:=\mathcal{V}_{p}(\mathbb{Q}_{\ell, X})$ As a first observation, we prove the following

\begin{prop}
\label{remark-vanishingcyclesperiodicequaltensorusual}
We have a canonical equivalence 
$$
\mathcal{V}_{p}(\beta)\simeq \mathcal{V}_{p}\otimes \mathbb{Q}_{\ell}(\beta).
$$
\begin{proof}
Given $E$ a $\mathbb{Q}_\ell$-adic sheaf on $X$,  we claim first to have a canonical equivalence between $\mathcal{V}_{p}(E(1))$ and $\mathcal{V}_{p}(E)\otimes \mathbb{Q}_{\ell}(1)$. To check this we can look at the cofiber sequence defining vanishing cycles

\begin{equation}
\label{equation-cofibersequencetwisted}
v_\sigma^* i_*(E(1))\to \bar{i}^*\bar{j}_* \bar{j}^* v^* (E(1))\to \mathcal{V}_{p}(E(1))
\end{equation}

\noindent and notice that pullbacks commute with Tate-twists (by definition) and that one has canonical equivalences

\begin{equation}
\label{equation-commutetwistspushforward}
\bar{j}_* \bar{j}^* (F(1))\simeq  (\bar{j}_* \bar{j}^* F)(1)
\end{equation}

\noindent for any $F$ over $\bar{X}$. This equivalence can be deduced by looking at the mapping spaces from a third $\ell$-adic sheaf to both sides of (\ref{equation-commutetwistspushforward}) and using the adjunction $(j^*, j_*)$ together with the fact that $j^*$ is monoidal, and that the Tate twist is an invertible object, stable under base change. The equivalence $i_*(E(1))\simeq i_*(E)(1)$ follows by the same argument. In this case, the cofiber sequence (\ref{equation-cofibersequencetwisted}) is equivalent to

\begin{equation}
\label{equation-cofibersequencetwisted2}
(v_\sigma^* i_*(E))(1)\to (\bar{i}^*\bar{j}_* \bar{j}^* v^* E)(1)\to \mathcal{V}_{p}(E)(1)
\end{equation}
Finally to deduce the equivalence $\mathcal{V}_{p}(\beta)\simeq \mathcal{V}_{p}\otimes \mathbb{Q}_{\ell}(\beta)$ one uses the equivalence 
$\mathbb{Q}_{\ell}(\beta)\simeq \bigoplus_{i\in \mathbb{Z}}\mathbb{Q}_{\ell}(i)[2i]$ together with the fact that both $^*$-pullbacks and $_*$-pushforwards preserve arbitrary colimits (see the discussion in \cite[Example 9.4.6]{robalo-thesis} for pushfowards, which is the only non obvious verification to be made).
\end{proof}
\end{prop}

\vspace{1cm}

Let us now proceed to investigate the sequence defining vanishing cycles for $E:=\mathrm{R}^{\ell}_X(\BU_{X})\simeq \mathbb{Q}_{\ell}(\beta)_X$, associated to diagram (\ref{eq-Xwithuniformizervanishingcycles}). By definition it lives in the equivariant derived $\infty$-category $\mathrm{Sh}_{\mathbb{Q}_{\ell}}(\bar{\sigma})^{ \mathrm{Gal}(\bar{\eta}/\eta)}$, and after shifting, it may be written as
\medskip

\begin{equation}
\label{eq-equivariantexactsequencevanishingsequences}
\xymatrix{(p_{\bar{\sigma}})_*\mathcal{V}_p(\beta)[-1]\ar[r]& (p_{\bar{\sigma}})_*\mathrm{R}^{\ell}_{X_{\bar{\sigma}}}(\BU_{X_{\bar{\sigma}}})\ar[r]^{\mathrm{sp}}& (p_{\bar{\sigma}})_*\bar{i}^*\bar{j}_*\mathrm{R}^{\ell}(\BU_{X_{\bar{\eta}}})}
\end{equation}
\medskip

In particular, by standard lax monoidal considerations, the map $\mathrm{sp}$ is in fact a map of commutative algebra objects in the equivariant category. 

Now, thanks to Remark \ref{remark-inertiainvariantscanbetaken}, it makes sense to take homotopy fixed points under the action of the inertia group $\mathrm{I}\subseteq \Gal(\bar{\eta}/\eta)$. Consider the adjunction $(\mathrm{Triv},  (-)^{\mathrm{hI}})$, 

\begin{equation}
\label{17ago1448}
\xymatrix{
\mathrm{Sh}_{\mathbb{Q}_\ell}(\sigma)^{\mathrm{I}}\ar@<-.5ex>[r]_{(-)^{\mathrm{hI}}}& \ar@<-.5ex>[l]_-{\mathrm{Triv}}\mathrm{Sh}_{\mathbb{Q}_\ell}(\sigma)
}
\end{equation}

We have that
\begin{prop}
\label{17ago1451}
The construction of homotopy fixed points $(-)^{\mathrm{hI}}$ is lax monoidal. In particular $\mathrm{sp}^{\mathrm{hI}}$ is a map of commutative algebra objects.
\begin{proof}
Indeed, $(-)^{\mathrm{hI}}$ is right adjoint to the trivial representation functor, which is monoidal.
\end{proof}
\end{prop}

Passing to inertia invariants, we obtain a cofiber-fiber sequence in $\mathrm{Sh}_{\mathbb{Q}_{\ell}}(\bar{\sigma})^{ \mathrm{Gal}(\bar{\sigma}/\sigma)}$:

\begin{equation}
\label{eq-equivariantexactsequencevanishingsequences2}
\xymatrix{((p_{\bar{\sigma}})_*\mathcal{V}_p(\beta)[-1])^{\mathrm{h}\mathrm{I}}\ar[r]& ((p_{\bar{\sigma}})_*\mathrm{R}^{\ell}_{X_{\bar{\sigma}}}(\BU_{X_{\bar{\sigma}}}))^{\mathrm{h}\mathrm{I}}\ar[r]^{(\mathrm{sp})^{\mathrm{hI}}}& ((p_{\bar{\sigma}})_*\bar{i}^*\bar{j}_*\mathrm{R}^{\ell}_{X_{\bar{\eta}}}(\BU_{X_{\bar{\eta}}}))^{\mathrm{h}\mathrm{I}}}
\end{equation}

\vspace{1cm}

We will now give an explicit description of the middle and last term of the cofiber-fiber sequence (\ref{eq-equivariantexactsequencevanishingsequences2})  (see Corollary \ref{newcorreferee4} below) and, in the process, we establish in Lemma \ref{newlabelreferee142} a relation between the cohomology of the punctured disk (Definition \ref{definition-cohomologypunctureddisk}) and inertia invariants vanishing cycles.
\begin{context}
\label{context-Sstrictlylocal}
From now on, for simplicity, we assume that $S$ is strictly local, so that $\bar{\sigma}=\sigma$. In particular, $\mathrm{I}=\mathrm{Gal}(\bar{\eta}/\eta)$ in this Context.
\end{context}
\medskip

We start with a description of the last term in (\ref{eq-equivariantexactsequencevanishingsequences2}) using Galois descent for the sheaf of $\infty$-categories $\mathrm{Sh}_{\mathbb{Q}_\ell}(-)$:
\medskip
\begin{prop}
\label{prop-lasterm}
The canonical map of commutative algebra-objects
\begin{equation}
\label{15ago1}
\mathbb{H}^{\mathrm{R}^\ell(\BU)}_{\mathbb{Q}_\ell}(X_\eta):=(p_{\sigma})_*i^*j_*\mathrm{R}^{\ell}_{X_{{\eta}}}(\BU_{X_{\eta}})\to ((p_{\sigma})_*i^*\bar{j}_*\mathrm{R}^{\ell}_{X_{\bar{\eta}}}(\BU_{X_{\bar{\eta}}}))^{\mathrm{h}\mathrm{I}}
\end{equation}
\noindent is an equivalence.
\begin{proof}
The result follows by adjunction and Galois descent. Let $v_\eta: X_{\bar{\eta}}\to X_\eta$ be the canonical map. Then by Galois theory we know that étale sheaves on $X_\eta$ are equivalent to étale sheaves on $X_{\bar{\eta}}$ equivariant with respect to the continuous action of the Galois group $\mathrm{Gal}(\bar{\eta}/\eta)$. This equivalence is given by 
$$v_\eta^\ast:\mathrm{Sh}_{\mathbb{Q}_\ell}(X_\eta)\to \mathrm{Sh}_{\mathbb{Q}_\ell}(X_{\bar{\eta}})^{\mathrm{Gal}(\bar{\eta}/\eta)} $$
\noindent with inverse given by taking fixed points of the global sections
$$(v_\eta)_*(-)^{\mathrm{hGal}(\bar{\eta}/\eta)}: \mathrm{Sh}_{\mathbb{Q}_\ell}(X_{\bar{\eta}})^{\mathrm{Gal}(\bar{\eta}/\eta)}\to \mathrm{Sh}_{\mathbb{Q}_\ell}(X_\eta)$$
In particular, the unit of this equivalence induces
$$\xymatrix{\mathrm{Id}_{X_{\eta}}\ar[r]^-{\sim}& ((v_\eta)_\ast v_\eta^*(-))^{\mathrm{hGal}(\bar{\eta}/\eta)}}$$
Using the commutativity of (\ref{eq-diagramintegralclosures2}) together with the proper base change formula and the commutativity of pushfowards with homotopy fixed points (taking homotopy fixed points is a pushfoward), we find  
$$
i^*j_\ast\simeq i^*j_\ast \mathrm{Id}_{X_{\eta}}\simeq  i^*j_\ast ((v_\eta)_\ast v_\eta^*(-))^{\mathrm{hGal}(\bar{\eta}/\eta)}\simeq   i^*(j_\ast (v_\eta)_\ast v_\eta^*(-))^{\mathrm{hGal}(\bar{\eta}/\eta)}\simeq 
$$
$$
\simeq  i^*(v_\ast \bar{j}_\ast v_\eta^*(-))^{\mathrm{hGal}(\bar{\eta}/\eta)}\simeq i^* v_\ast( \bar{j}_\ast v_\eta^*(-))^{\mathrm{hGal}(\bar{\eta}/\eta)}\simeq  (v_{\sigma})_\ast \bar{i}^\ast ( \bar{j}_\ast v_\eta^*(-))^{\mathrm{hGal}(\bar{\eta}/\eta)}
$$

In particular, in the strictly local case we find $X_{\bar{\sigma}}=X_\sigma$ and $v_\sigma$ is the identity. Thus the last term in the chain of equivalences becomes  $i^\ast ( \bar{j}_\ast v_\eta^*(-))^{\mathrm{hGal}(\bar{\eta}/\eta)}$. To conclude, it remains to explain the formula
$$i^\ast ( \bar{j}_\ast v_\eta^*(-))^{\mathrm{hGal}(\bar{\eta}/\eta)}\simeq (i^\ast  \bar{j}_\ast v_\eta^*(-))^{\mathrm{hGal}(\bar{\eta}/\eta)}$$
This follows because the continuity of the action of the profinite group $\mathrm{Gal}(\bar{\eta}/\eta)$ tells us that

$$(v_\eta)_\ast v_\eta^\ast(-)^{\mathrm{h}\mathrm{Gal}(\bar{\eta}/\eta)}\simeq \mathrm{colim}_i \,(v_{\eta_i})_\ast v_{\eta_i}^\ast(-)^{\mathrm{h}\mathrm{Gal}(\eta_i/\eta)}$$

\noindent where the colimit runs through all finite Galois extensions. In the same way we can also write 

$$ (\bar{j}_\ast v_\eta^*(-))^{\mathrm{hGal}(\bar{\eta}/\eta)}\simeq  \mathrm{colim}_i \, (\bar{j}_\ast v_{\eta_i}^*(-))^{\mathrm{hGal}(\eta_i/\eta)}$$

Finally, on each term of this colimit, the group $\mathrm{Gal}(\eta_i/\eta)$ is finite, so taking invariants is a finite limit. We conclude using the fact that $i^\ast$ commutes with all colimits and with finite limits.

To finish the proof of the proposition, apply this natural equivalence to $\mathrm{R}^{\ell}_{X_{{\eta}}}(\BU_{X_{\eta}})\in \mathrm{Sh}_{\mathbb{Q}_\ell}(X_\eta)$ and composing it with $(p_\sigma)_\ast$. See  also \cite[Exp XIII pag 7]{MR0354657}.
\vspace{0.5cm}
\end{proof}
\end{prop}

\medskip

Let us now discuss the middle term of (\ref{eq-equivariantexactsequencevanishingsequences2}). 

\medskip

\begin{prop} 
\label{prop-takinginvariants}
There is a canonical equivalence of commutative algebra objects
\begin{equation}
\label{15ago2}
((p_{\sigma})_*\mathrm{R}^{\ell}_{X_{\sigma}}(\BU_{X_{\sigma}}))^{\mathrm{h}\mathrm{I}}\simeq (p_{\sigma})_*\mathrm{R}^{\ell}_{X_{\sigma}}(\BU_{X_{\sigma}})\otimes (\mathbb{Q}_{\ell, \sigma})^{\mathrm{h}\mathrm{I}}
\end{equation}
\medskip
In particular, at the level of underlying objects, we recover
\begin{equation}
\label{15ago3}
((p_{\sigma})_*\mathrm{R}^{\ell}_{X_{\sigma}}(\BU_{X_{\sigma}}))^{\mathrm{h}\mathrm{I}}\simeq (p_{\sigma})_*\mathrm{R}^{\ell}_{X_{\sigma}}(\BU_{X_{\sigma}})\oplus (p_{\sigma})_*\mathrm{R}^{\ell}_{X_{\sigma}}(\BU_{X_{\sigma}})(-1)[-1]
\end{equation}
\medskip

\begin{proof}
By construction of vanishing cycles, the action of $\mathrm{I}$ on $(p_{\sigma})_*\mathrm{R}^{\ell}_{X_{\sigma}}(\BU_{X_{\sigma}})$ is the trivial action.  The projection formula for the projection $\mathrm{BI}\to \sigma$  (see \cite{1211.5948}) gives us the equivalence (\ref{15ago2}). The formula (\ref{15ago3}) then follows from the Lemma \ref{newlabelreferee141} and the key Lemma \ref{newlabelreferee142} below.\end{proof} \end{prop}

\vspace{0.5cm}

Finally, as a consequence of the Prop. \ref{prop-takinginvariants} and \ref{prop-lasterm}, we have 
\begin{cor}
\label{newcorreferee4}
Assume \ref{context-Sstrictlylocal}. The cofiber-fiber sequence (\ref{eq-equivariantexactsequencevanishingsequences2}) is equivalent to a cofiber-fiber sequence 
\begin{equation}
\label{eq-maincofiberfiberseq3}
\xymatrix{
\ar[d]((p_{\sigma})_*\mathcal{V}_p(\beta)[-1])^{\mathrm{h}I}\ar[r]&   (p_{\sigma})_*\mathrm{R}^{\ell}_{X_{\sigma}}(\BU_{X_{\sigma}})\otimes (\mathbb{Q}_{\ell, \sigma})^{\mathrm{h}\mathrm{I}}\ar[d]^{(sp)^{\mathrm{hI}}}\\
0\ar[r]&\mathbb{H}^{\mathrm{R}^\ell(\BU)}_{\mathbb{Q}_\ell}(X_\eta)
}
\end{equation}
\noindent where $\mathrm{sp}^{\mathrm{hI}}$ is a map of commutative algebra objects
\end{cor}

\vspace{1cm}

The next lemma is crucial not only in the proof of Proposition \ref{prop-takinginvariants} above but is also the key result behind our main Theorem \ref {theorem-maincomparisonvanishingmf} below. The lemma explains how the algebra structure on the cohomology of the punctured disk is related to inertia invariants. It is essentailly an algebraic version of the fact that the cohomology of the topological circle $\mathrm{S}^1= \mathrm{B} \mathbb{Z}$ with its cup product is quasi-isomorphic to the commutative differential graded algebra (cdga) of homotopy fixed points of the trivial $\mathbb{Z}$-representation. See Remark \ref{16agotopologicalanalogycircle} for more details on this analogy.

\begin{lem}
\label{newlabelreferee142}
Assume Context \ref{context-Sstrictlylocal}. The object $(\mathbb{Q}_{\ell, \bar{\sigma}})^{\mathrm{hI}}$ carries a canonical structure of commutative algebra object in $\mathrm{Sh}_{\mathbb{Q}_{\ell}}(\sigma)$. Furthermore, we have a canonical equivalence of commutative algebra objects
\begin{equation}
\label{newlabelreferee143}
(\mathbb{Q}_{\ell, \sigma})^{\mathrm{hI}}\simeq \mathbb{H}_{\mathbb{Q}_\ell}(\eta)
\end{equation} 
\noindent where on the r.h.s we have the algebra structure of the Lemma \ref{newlabelreferee141}. In particular, the algebra structure on $(\mathbb{Q}_{\ell, \sigma})^{\mathrm{hI}}$ is obtained by transferring the canonical algebra structure on $\mathbb{Q}_{\ell, \eta}$ via the lax monoidal functor $i_{\sigma}^\ast (j_{\eta})_\ast$ and at the level of the underlying objects we have 
\begin{equation}
\label{eq-invariantstrivial}
(\mathbb{Q}_{\ell, \sigma})^{\mathrm{h}\mathrm{I}}\simeq \mathbb{Q}_{\ell, \sigma}\oplus \mathbb{Q}_{\ell, \sigma}(-1)[-1]
\end{equation}
\begin{proof}
In Context \ref{context-Sstrictlylocal}, we have $\bar{\sigma}=\sigma$. The fact that taking homotopy fixed points is a lax monoidal functor (Prop. \ref{17ago1451}) guarantees that $(\mathbb{Q}_{\ell, \sigma})^{\mathrm{hI}}$ is an algebra object. It remains to show the equivalence of algebras (\ref{newlabelreferee143}).
As $S$ is smooth over itself via the identity map, the specialization map defining vanishing cycles relative to $S$

\begin{equation}
\label{eq-vanishingcyclestrivialoverthetrait}
\mathbb{Q}_{\ell,\sigma} \to i_\sigma^* (j_{\bar{\eta}})_*\mathbb{Q}_{\ell,\bar{\eta}} 
\end{equation}

\noindent is an equivalence in $\mathrm{Sh}_{\mathbb{Q}_\ell}(\sigma)^{\mathrm{I}}$. Note that in the l.h.s. of (\ref{eq-invariantstrivial}), the inertia $I$ acts trivially on $\mathbb{Q}_{\ell, \sigma}$. Note also that as we work under the assumption that $S$ is strictly local, we have $\mathrm{I}=\mathrm{Gal}(\bar{\eta}, \eta)$.
 For tame nearby cycles this follows by the explicit computation of the r.h.s of  \cite[Formula (102)]{MR3205601} for torsion coefficients. For the total vanishing cycles functor this follows by passing to the colimit over all wild extensions. \footnote{One can also give a simpler argument for the equivalence (\ref{eq-vanishingcyclestrivialoverthetrait}) in terms of étale cohomology groups. Indeed, one can give an explicit description of the étale sheaves $(j_{\bar{\eta}})_*\mathbb{Z}/n\mathbb{Z}$ by noticing that its cohomology groups are the étale sheafification of the presheaf of abelian groups sending an étale map  $V\to S$ to $H^i(V\times_S \bar{\eta}, \mathbb{Z}/n\mathbb{Z})$. As $\bar{\eta}$ is separably closed, $V\times_S \bar{\eta}$ is a disjoint union of copies of $\bar{\eta}$ so that its étale cohomology groups vanish for $i\geq 1$. For $i=0$ we get $\mathbb{Z}/n\mathbb{Z}$ as the set of (underived) global sections of its assocated constant sheaf. }

The specialization map being equivariant, implies that after passing to $\mathrm{I}$-invariants we still get an equivalence

\begin{equation}
\label{korokseed1}
\xymatrix{\mathbb{Q}_{\ell,\sigma}^\mathrm{hI} \ar[r]^-{\sim} &(i_\sigma^* (j_{\bar{\eta}})_*\mathbb{Q}_{\ell,\bar{\eta}} )^\mathrm{hI}}
\end{equation}

Finally, by the same argument in the proof of the Prop. \ref{prop-lasterm} we find 

\begin{equation}
\label{korokseed2}
\xymatrix{(i_\sigma^* (j_{\bar{\eta}})_*\mathbb{Q}_{\ell,\bar{\eta}} )^\mathrm{hI}&\ar[l]_-{\sim} i_\sigma^* (j_\eta)_*\mathbb{Q}_{\ell,\eta}=:\mathbb{H}_{\mathbb{Q}_\ell}(\eta)}
\end{equation}


\end{proof}
\end{lem}

\begin{rem}
\label{17ago1617}
As seen in Remark (\ref{newrefereelabel141713}),  multiplication by elements in Tate degree $(1,1)$
can  be presented as a map of $(\mathbb{Q}_{\ell, \sigma})^{\mathrm{hI}}$-modules
\begin{equation}
\label{korokseed5}
(\mathbb{Q}_{\ell, \sigma})^{\mathrm{hI}}\to (\mathbb{Q}_{\ell, \sigma})^{\mathrm{hI}}(1)[1]
\end{equation}
Moreover, being a map of $(\mathbb{Q}_{\ell, \sigma})^{\mathrm{hI}}$-modules, (\ref{korokseed5}) is determined, via base-change, by a map of $\mathbb{Q}_{\ell, \sigma}$-modules
\begin{equation}
\mathbb{Q}_{\ell, \sigma}\to (\mathbb{Q}_{\ell, \sigma})^{\mathrm{hI}}(1)[1]
\end{equation}
\noindent which under (\ref{newlabelreferee143}), corresponds (\ref{16ago1353}).
\end{rem}

\vspace{1cm}

The following result describes how the class $\theta_\ell^{(S,\sigma)}$ transfers along the equivalence (\ref{newlabelreferee143}).

\begin{prop}
\label{remark-invariantslaxmonoidalfundamentalclass}
\hfill
\begin{enumerate}
\item  There exists an $\ell$-adic class 
\begin{equation}
\label{korokseed6757575040}
[\theta_I^{(S,\sigma)}: \mathbb{Q}_{\ell, \eta}\to \mathbb{Q}_{\ell, \eta}(1)[1]]\in \mathrm{H}^1(\eta, \mathbb{Q}_{\ell}(1))
\end{equation} such that its image under $i_\sigma^\ast (j_\eta)_\ast$ is the map (\ref{korokseed5}).
\medskip
\item The classes $\theta_I^{(S,\sigma)}$ and $\theta_\ell^{(S,\sigma)}$ coincide.
\end{enumerate}

\begin{proof}
(i) is the description in  \cite[1.2]{MR666636},  \cite[\S 3.6]{MR1293970}. (ii) is the computation of \cite[Exp XVI Lemmas 3.4.6, 3.4.7 and 3.4.8]{MR3329769}.
\end{proof}
\end{prop}

\medskip

\begin{rem}
\label{17ago1216}
It follows from the lax-monoidality that for every object $\Psi\in \mathrm{Sh}_{\mathbb{Q}_\ell}(\sigma)^{\mathrm{I}}$, the object $\Psi^\mathrm{hI}$ carries a canonical structure of $(\mathbb{Q}_{\ell, \sigma})^{\mathrm{hI}}$-module 
$$(\mathbb{Q}_{\ell, \sigma})^{\mathrm{hI}}\otimes \Psi^\mathrm{hI}\to  \Psi^\mathrm{hI}$$
\noindent and by the arguments above, is acted by the class $\theta_\ell^{(S,\sigma)}$
$$\xymatrix{\Psi^\mathrm{hI}\ar[rr]^-{-\ast \theta_\ell^{(S,\sigma)}}&&  \Psi^\mathrm{hI}(1)[1]}$$
It is a key idea used in \cite[1.4]{MR666636} that in the unipotent case, $\Psi$ with its $\mathrm{I}$-action can be completely recovered from $\Psi^\mathrm{hI}$ together with the information of the action of $\theta_\ell^{(S,\sigma)}$.  See Remark \ref{16agotopologicalanalogycircle} below for a topological analogy of this.
 \end{rem}

\begin{rem}
\label{16agotopologicalanalogycircle}
The equivalence of algebras in the Lemma \ref{newlabelreferee142} is an algebraic version of a more familiar topological situation where the role of $\eta$ is played by the topological circle $\mathrm{S}^1$ 
and the choice of a closure $\bar{\eta}\to \eta$ is replaced by the choice of a universal cover of $\mathrm{S}^1$. Indeed, the description of the circle $\mathrm{S}^1\simeq \mathrm{B}\mathbb{Z}$ presents the choice of a point $\ast\to \mathbb{S}^1$ as the choice of a universal cover and characterizes $\mathrm{S}^1$ as the homotopy orbits for the trivial action of $\mathbb{Z}$ on the trivial space $\ast$
$$\mathrm{S}^1\simeq \mathrm{B}\mathbb{Z}=\ast/\mathbb{Z}\simeq \mathrm{colim}_{\mathrm{B}\mathbb{Z}} \,\,\ast $$
As a consequence, after passing to singular cochains one finds 
\begin{equation}
\label{analogytopological16ago1432}
C^\ast(\mathrm{S}^1, \mathbb{C})\simeq C^\ast(\mathrm{colim}_{\mathrm{B}\mathbb{Z}}\,\, \ast , \mathbb{C})\simeq \mathrm{lim}_{\mathrm{B}\mathbb{Z}}\, C^\ast( \ast , \mathbb{C})\simeq  C^\ast( \ast , \mathbb{C})^{\mathrm{h}\mathbb{Z}}\simeq \mathbb{C}^{\mathrm{h}\mathbb{Z}}
\end{equation}
\noindent as an equivalence in $\Mod_{\mathbb{C}}\Sp$. By formality for the circle we have $C^\ast(\mathrm{S}^1, \mathbb{C})\simeq \mathrm{H}^\ast (\mathrm{S}^1,\mathbb{C})$. The resulting formula
\begin{equation}
\label{17ago1110}
\mathrm{H}^\ast (\mathrm{S}^1,\mathbb{C})\simeq \mathbb{C}^{\mathrm{h}\mathbb{Z}}
\end{equation}
\noindent is analogous to (\ref{newlabelreferee143}).  The class $\theta$ is analogous to the topological generator $\epsilon$ of degree -1.  Notice that we can recover the trivial representation $\mathbb{C}$ via the Koszul-Tate resolution of $\mathbb{C}$ as a $\mathrm{H}^\ast (\mathrm{S}^1,\mathbb{C})$-module, namely
$$
\mathbb{C}\simeq \bigoplus_{i\in \mathbb{Z}} \mathrm{H}^\ast (\mathrm{S}^1,\mathbb{C})[-i]
$$
In the case of a $\mathbb{Z}$-representation $M$ with unipotent action, one can use the Eilenberg-Moore spectral sequence  \cite[1.1.10]{lurie-rational} combined with Koszul-Tate resolution to reconstruct $M$ from the pair $M^{\mathrm{h}\mathbb{Z}}$ equipped with its canonical $\mathrm{H}^\ast (\mathrm{S}^1,\mathbb{C})$-action.
\end{rem}

\vspace{1cm}

\subsection{Comparison between Vanishing Cycles and the singularity category}
\label{section-comparisonvanishingcyclesandMF}
We are now ready to state and prove our main theorem. 

\begin{thm}
\label{theorem-maincomparisonvanishingmf}
Let $p:X\to S$ with $X$ regular, $p$ a proper flat morphism over a strictly local excellent henselian trait $S=\Spec A$. Set 
\begin{equation}
\mathbb{H}^{\mathrm{R}^\ell(\BU)}_{\mathbb{Q}_\ell}(X_\sigma):=(p_{\sigma})_*\mathrm{R}^{\ell}(\BU_{X_{\sigma}}).\end{equation}
Then, the equivalence of $\mathbb{Q}_{\ell,\sigma}$-algebras (\ref{newlabelreferee143}) induces an homotopy  between the two maps of commutative $\mathbb{Q}_{\ell,\sigma}$-algebra objects (\ref{newref146774jajfunnotfunjcaclnl833943}) and (\ref{17ago1451}), namely:

\begin{equation}
\xymatrix{
\mathrm{can}\,\,, \mathrm{sp}^{\mathrm{hI}}\,\,:\mathbb{H}_{\mathbb{Q}_\ell}(\eta)\otimes_{\mathbb{Q}_{\ell,\sigma}} \mathbb{H}^{\mathrm{R}^\ell(\BU)}_{\mathbb{Q}_\ell}(X_\sigma) \ar[r]&\mathbb{H}^{\mathrm{R}^\ell(\BU)}_{\mathbb{Q}_\ell}(X_\eta)}
\end{equation}
\noindent 
\medskip
In particular, the cofiber-fiber sequences (\ref{newref146774jajfunnotfunnotfunagainjcaclnl833943}) and (\ref{eq-maincofiberfiberseq3}) are equivalent, and we deduce an equivalence 
\begin{equation}\label{one}
(i_\sigma)^*\mathrm{R}^{\ell}_S((\mathcal{M}_S^\vee(\Sing(X,\pi\circ p)))\simeq ((p_{\sigma})_*\mathcal{V}_p(\beta)[-1])^{\mathrm{hI}}
\end{equation}
\noindent of $\mathbb{Q}_\ell$-adic sheaves over $\sigma$ which is compatible with the action of $i_\sigma^\ast \mathcal{M}_S^\vee(\Sing(S,0))$ on the l.h.s and the action of $(\mathbb{Q}_{\ell, \bar{\sigma}})^{\mathrm{hI}}$ on the r.h.s. 
\end{thm}

\vspace{1cm}

Before addressing the proof of this theorem, let us collect some remarks.

\medskip

\begin{rem}
Thanks to the Proposition \ref{prop-exactsequenceMFoverSlocalring}, the equivalence (\ref{one}) can be formulated as an equivalence of $\mathbb{Q}_\ell$-adic sheaves over $S$
\begin{equation}\label{thisone}
\mathrm{R}^{\ell}_S((\mathcal{M}_S^\vee(\Sing(X,\pi\circ p)))\simeq (i_{\sigma})_*((p_{\sigma})_*\mathcal{V}_p(\beta)[-1])^{\mathrm{hI}}.
\end{equation}
\end{rem}

\medskip

\begin{rem} Note that if $p:X\to S$ is a proper morphism, and $X$ regular, $p$  then $\Sing(X,\pi\circ p) \simeq \Sing(X_0) $, where $X_0$ is the derived zero locus of $\pi \circ p$. If, moreover, $p$ is flat, then $X_0 \simeq t(X_0)$ (i.e. the derived zero locus coincides with the scheme theoretic zero-locus). Therefore, the equivalences (\ref{one}) and (\ref{thisone}) can be equivalently re-written as
\begin{equation}\label{one'}
(i_\sigma)^*\mathrm{R}^{\ell}_S((\mathcal{M}_S^\vee(\Sing(X_0)))\simeq ((p_{\sigma})_*\mathcal{V}_p(\beta)[-1])^{\mathrm{hI}},
\end{equation}
\noindent respectively,
\begin{equation}\label{thisone'}
\mathrm{R}^{\ell}_S((\mathcal{M}_S^\vee(\Sing(X_0)))\simeq (i_{\sigma})_*((p_{\sigma})_*\mathcal{V}_p(\beta)[-1])^{\mathrm{hI}}.
\end{equation}

\end{rem}

\medskip

\begin{rem} Notice that, since the action of $I$ on $X_{\sigma}$ is trivial and taking homotopy invariants can be represented as a limit, both $(i_{\sigma})_*$ and $(p_{\sigma})_*$ commute with taking $\mathrm{hI}$-invariants, being both right adjoints\footnote{This can also be deduced from the monadic argument producing the equivalence $\mathrm{Sh}_{\mathbb{Q}_\ell}(\sigma)^\mathrm{I}\simeq \Mod_{\mathbb{Q}_{\ell, \sigma}[\mathrm{I}]}(\mathrm{Sh}_{\mathbb{Q}_\ell}(\sigma))$ with $\mathbb{Q}_{\ell, \sigma}[\mathrm{I}]$ the internal group-ring of $\mathrm{I}$}. In particular, we have $$((p_{\sigma})_*\mathcal{V}_p(\beta)[-1])^{\mathrm{hI}} \simeq (p_{\sigma})_*((\mathcal{V}_p(\beta)[-1])^{\mathrm{hI}})$$ and $$(i_{\sigma})_*((p_{\sigma})_*\mathcal{V}_p(\beta)[-1])^{\mathrm{hI}} \simeq (i_{\sigma})_*(p_{\sigma})_*(\mathcal{V}_p(\beta)[-1])^{\mathrm{hI}}).$$ 
\end{rem}

\medskip

\begin{proof}[Proof of Theorem \ref{theorem-maincomparisonvanishingmf}]
We start with by comparing the two maps of commutative algebras $\mathrm{sp}^\mathrm{hI}$ and $\mathrm{can}$. For this purpose one has to provide commutativity for  the following diagram of commutative algebra-objects:

$$
\resizebox{1. \hsize}{!}{
\xymatrix{
(i_\sigma^* (j_{\bar{\eta}})_*\mathbb{Q}_{\ell,\bar{\eta}} )^\mathrm{hI}\otimes (p_{\sigma})_*\mathrm{R}^{\ell}_{X_{\sigma}}(\BU_{X_{\sigma}})&
 (\mathbb{Q}_{\ell, \sigma})^{\mathrm{h}\mathrm{I}}\otimes (p_{\sigma})_*\mathrm{R}^{\ell}_{X_{\sigma}}(\BU_{X_{\sigma}}) \ar[l]^-{(\ref{korokseed1})}_-{\sim}&
  \ar[l]^-{(\ref{15ago2})}_-{\sim} ((p_{\sigma})_*\mathrm{R}^{\ell}_{X_{\sigma}}(\BU_{X_{\sigma}}))^{\mathrm{h}\mathrm{I}} \ar[r]^-{(\mathrm{sp})^{\mathrm{hI}}}&
   ((p_{\sigma})_*i^*\bar{j}_*\mathrm{R}^{\ell}_{X_{\bar{\eta}}}(\BU_{X_{\bar{\eta}}}))^{\mathrm{h}\mathrm{I}}\\
 \mathbb{H}_{\mathbb{Q}_\ell}(\eta)\otimes (p_{\sigma})_*\mathrm{R}^{\ell}_{X_{\sigma}}(\BU_{X_{\sigma}})\ar[u]^{\sim}_{(\ref{korokseed2})\otimes Id}\ar[rrr]^-{\mathrm{can}}&&& \ar[u]^{\sim}_{(\ref{15ago1})} 
\mathbb{H}^{\mathrm{R}^\ell(\BU)}_{\mathbb{Q}_\ell}(X_\eta)
}
}
$$

\medskip

As coproducts of commutative algebra objects are given by tensor products \cite[3.2.4.7]{lurie-ha}, providing a 2-cell witnessing the commutativity of the diagram above, it is equivalent to providing 2-cells witnessing the commutativity of the following two diagrams of commutative algebras

\medskip

\noindent (A):
$$
\resizebox{1. \hsize}{!}{
\xymatrix{
\ar[r]^-{Id\otimes 1}(i_\sigma^* (j_{\bar{\eta}})_*\mathbb{Q}_{\ell,\bar{\eta}} )^\mathrm{hI}&(i_\sigma^* (j_{\bar{\eta}})_*\mathbb{Q}_{\ell,\bar{\eta}} )^\mathrm{hI}\otimes (p_{\sigma})_*\mathrm{R}^{\ell}_{X_{\sigma}}(\BU_{X_{\sigma}}) \ar[r]^-{(\mathrm{sp})^{\mathrm{hI}}}&
   ((p_{\sigma})_*i^*\bar{j}_*\mathrm{R}^{\ell}_{X_{\bar{\eta}}}(\BU_{X_{\bar{\eta}}}))^{\mathrm{h}\mathrm{I}}\\
\mathbb{H}_{\mathbb{Q}_\ell}(\eta) \ar[u]^{\sim}_{(\ref{korokseed2})} \ar[r]^-{Id\otimes 1}& \mathbb{H}_{\mathbb{Q}_\ell}(\eta)\otimes (p_{\sigma})_*\mathrm{R}^{\ell}_{X_{\sigma}}(\BU_{X_{\sigma}})\ar[r]^-{\mathrm{can}}& \ar[u]^{\sim}_{(\ref{15ago1})} 
\mathbb{H}^{\mathrm{R}^\ell(\BU)}_{\mathbb{Q}_\ell}(X_\eta)
}
}
$$
\medskip

\noindent (B): 

$$
\resizebox{1. \hsize}{!}{
\xymatrix{
\ar[r]^-{1\otimes Id} (p_{\sigma})_*\mathrm{R}^{\ell}_{X_{\sigma}}(\BU_{X_{\sigma}})&(i_\sigma^* (j_{\bar{\eta}})_*\mathbb{Q}_{\ell,\bar{\eta}} )^\mathrm{hI}\otimes (p_{\sigma})_*\mathrm{R}^{\ell}_{X_{\sigma}}(\BU_{X_{\sigma}}) \ar[r]^-{(\mathrm{sp})^{\mathrm{hI}}}&
   ((p_{\sigma})_*i^*\bar{j}_*\mathrm{R}^{\ell}_{X_{\bar{\eta}}}(\BU_{X_{\bar{\eta}}}))^{\mathrm{h}\mathrm{I}}\\
(p_{\sigma})_*\mathrm{R}^{\ell}_{X_{\sigma}}(\BU_{X_{\sigma}}) \ar@{=}[u]\ar[r]^-{1\otimes Id}& \mathbb{H}_{\mathbb{Q}_\ell}(\eta)\otimes (p_{\sigma})_*\mathrm{R}^{\ell}_{X_{\sigma}}(\BU_{X_{\sigma}})\ar[r]^-{\mathrm{can}}& \ar[u]^{\sim}_{(\ref{15ago1})} 
\mathbb{H}^{\mathrm{R}^\ell(\BU)}_{\mathbb{Q}_\ell}(X_\eta)
}
}
$$

\medskip

Concerning (A), the 2-cell is established by the naturally of the diagram in algebras

$$
\xymatrix{
\ar[r] (i_\sigma^* (j_{\bar{\eta}})_*\mathbb{Q}_{\ell,\bar{\eta}} )^\mathrm{hI}&
   ((p_{\sigma})_*i^*\bar{j}_*\mathrm{R}^{\ell}_{X_{\bar{\eta}}}(\BU_{X_{\bar{\eta}}}))^{\mathrm{h}\mathrm{I}}\\
 i_\sigma^* (j_\eta)_*\mathbb{Q}_{\ell,\eta}\ar[u]^{\sim}_{(\ref{korokseed2})} \ar[r]& \ar[u]^{\sim}_{(\ref{15ago1})} 
i_\sigma^\ast (j_\eta)_\ast (p_\eta)_\ast \mathrm{R}^\ell(\BU_\eta)
}
$$
\medskip

For (B), it comes from the adjunction $(\ref{17ago1448})$, as in fact, by definition, the specialization map is of the form

$$
\xymatrix{\mathrm{Triv}((p_{\sigma})_*\mathrm{R}^{\ell}_{X_{\sigma}}(\BU_{X_{\sigma}}))\ar[r]^{\mathrm{sp}}   & (p_{\sigma})_*i^*\bar{j}_*\mathrm{R}^{\ell}_{X_{\bar{\eta}}}(\BU_{X_{\bar{\eta}}})}
$$

\noindent and the adjunction $(\mathrm{Triv},  (-)^{\mathrm{hI}})$ guarantees a commutative cell in algebras

$$
\xymatrix{
(p_{\sigma})_*\mathrm{R}^{\ell}_{X_{\sigma}}(\BU_{X_{\sigma}})\ar[r]\ar[d]&  (\mathrm{Triv}((p_{\sigma})_*\mathrm{R}^{\ell}_{X_{\sigma}}(\BU_{X_{\sigma}})))^{\mathrm{hI}}\ar[d]^{\mathrm{sp}^{\mathrm{hI}}}\\
i_\sigma^\ast (j_\eta)_\ast (p_\eta)_\ast \mathrm{R}^\ell(\BU_\eta)  \ar[r]^-{\sim}_{(\ref{15ago1})} & [\,(p_{\sigma})_*i^*\bar{j}_*\mathrm{R}^{\ell}_{X_{\bar{\eta}}}(\BU_{X_{\bar{\eta}}})\,]^{\mathrm{hI}}
}
$$

This concludes the construction of an homotopy between the two maps of commutative algebras $\mathrm{sp}^\mathrm{hI}$ and $\mathrm{can}$.  To  prove the remaining statement in the theorem, we need to explain the formula (\ref{one}). Now that we have the identification of the two maps, the formula follows by passing to the fibers of the underlying map of modules and the Proposition \ref{16ago2244}. Finally to justify why (\ref{one}) is compatible with the canonical action of $\mathrm{R}_S^{\ell}(\mathcal{M}_S^\vee(\Sing(S,0)))\simeq \mathrm{R}_S^{\ell}(B\mathbb{U}_S)\oplus \mathrm{R}_S^{\ell}(B\mathbb{U}_S)[1]$ on the l.h.s, and the action of $\mathbb{Q}_\ell(\beta)\oplus \mathbb{Q}_\ell(\beta)(-1)[-1]$ on the r.h.s, we start by noticing that by the universal property of base change, the maps $\mathrm{sp}^\mathrm{hI}$ and $\mathrm{can}$ may also be written as maps of $\BU_\sigma$-algebras

\begin{equation}
\label{21ago0921}
\xymatrix{
\mathrm{can}\,\,, \mathrm{sp}^{\mathrm{hI}}\,\,: [\,\mathrm{R}^\ell(\BU_\sigma)\otimes_{\mathbb{Q}_{\ell, \sigma}}\mathbb{H}_{\mathbb{Q}_\ell}(\eta)\,]\otimes_{\mathrm{R}^\ell(\BU_\sigma)} \mathbb{H}^{\mathrm{R}^\ell(\BU)}_{\mathbb{Q}_\ell}(X_\sigma) \ar[r]&\mathbb{H}^{\mathrm{R}^\ell(\BU)}_{\mathbb{Q}_\ell}(X_\eta)}
\end{equation}

Now we use the following general fact: the fiber $K$ of a morphism of commutative algebras $A\to B$ has a canonical structure of $A$-module and this is functorial under equivalences of algebras $A\simeq C$. 
In our case, we apply this to the map of algebras (\ref{21ago0921}) and the equivalence of algebras (\ref{20ago2120}).

\end{proof}

\begin{cor}\label{rem-downtoearchformulacomparisonMFvanishing} Under the hypotheses and notations of Theorem \ref{theorem-maincomparisonvanishingmf}, we have an equivalence of \'etale $\ell$-adic hyper-cohomologies (ie, derived global sections) $$\mathbb{H}_{\'et}(S, \mathrm{R}^\ell_S \mathcal{M}_S^\vee(\Sing(X,\pi\circ p)))\simeq \mathbb{H}_{\'et}(X_{\sigma}, \mathcal{V}_{p}(\beta)[-1])^{\mathrm{hI}}$$ in the $\infty$-category of $\mathbb{Q}_{\ell}$-dg-modules.\footnote{Note that in the literature the r.h.s of is often denoted as $\mathbb{R}\Gamma(\mathrm{I},\mathbb{H}_{\'et}(X_{\sigma}, \mathcal{V}_{p}(\beta)[-1]))$.} \end{cor} \begin{proof} The statement follows by applying the hypercohomology functor $\mathbb{H}_{\'et}(S,-)$ to the equivalence (\ref{thisone}), and using in the r.h.s. that $$\mathbb{H}_{\'et} (X_{\sigma}, -) \simeq \mathbb{H}_{\'et} (S, (i_{\sigma})_*(p_{\sigma})_*(-)).$$ Also observe that, since  taking homotopy invariants can be represented as a limit, it commutes with taking hyper-cohomology, so that  we have $$\mathbb{H}_{\'et}(X_{\sigma}, \mathcal{V}_{p}(\beta)[-1])^{\mathrm{hI}} \simeq \mathbb{H}_{\'et}(X_{\sigma}, (\mathcal{V}_{p}(\beta)[-1])^{\mathrm{hI}}).$$
\end{proof}

  \medskip

\begin{cor}\label{nonstrictlylocalbutexcellent}
In the situation of Theorem \ref{theorem-maincomparisonvanishingmf}, let us suppose that $S=\Spec \, A$ is an excellent henselian trait (so that its residue field is not necessarily separably closed). Let us fix a separable closure $\bar{k}$ of $k$, and let $\bar{S}=\mathrm{Spec}\, A^{sh}$ be the corresponding strict henselization of $S$. Then, in the notations of diagram (\ref{eq-diagramintegralclosures2}), we have an equivalence in $\mathrm{Sh}_{\Ql}(S)^{\mathrm{Gal}(\bar{\sigma}|\sigma)}$
\begin{equation}\label{thisotherone-equiv}
\mathrm{R}^{\ell}_S((\mathcal{M}_S^\vee(\Sing(X_{\bar{\sigma}})))\simeq u_*(i_{\bar{\sigma}})_*(p_{\bar{\sigma}})_*(\mathcal{V}_p(\beta)[-1])^{\mathrm{hI}}).
\end{equation}
 \end{cor}  
\begin{proof} We borrow our notations from diagram (\ref{eq-diagramintegralclosures2}). First of all, observe that since $S$ is an excellent henselian trait, the \'etale topos of $\ell$-torsion sheaves of any $S$-scheme of finite type is of finite cohomological dimension (Remark \ref{rem-finiteetaledimension}), therefore we still have an $\ell$-adic realization functor $$\mathrm{R}^{nc}_{\ell,S} := \mathrm{R}^\ell_S \circ \mathcal{M}_S^\vee : \dgcat^{\mathrm{idem}}_S \to \mathrm{Sh}_{\Ql}(S).$$  Analogously, we will write  $$\mathrm{R}^{nc}_{\ell,\bar{S}} := \mathrm{R}^\ell_{\bar{S}} \circ \mathcal{M}_{\bar{S}}^\vee : \dgcat^{\mathrm{idem}}_{\bar{S}} \to \mathrm{Sh}_{\Ql}(\bar{S}).$$  Notice that $\mathfrak{m}_A A^{sh}=\mathfrak{m}_{A^{sh}}$, so that any uniformizer $\pi$ of $A$ gives a uniformizer of $A^{sh}$ (i.e. its image via $A \to A^{sh}$ is a uniformizer in $A^{sh}$). Since $u: \bar{S} \to S$ is formally \'etale and  the ramification locus of $u: \bar{S} \to S$ is disjoint from the singularity locus of $p: X \to S$, in the base change $\bar{X}/\bar{S}$ of $X/S$ along $u$, $\bar{X}$ is still regular (since $X$ is).  

We now argue for the commutativity of the diagram
\begin{equation}
\label{16ago1045}\xymatrix{\dgcat^{\mathrm{idem}}_{\bar{S}} \ar[d]_-{u_\ast:=\mathrm{Restr}_u} \ar[r]^-{\mathrm{R}^{nc}_{\ell,\bar{S}}} & \mathrm{Mod}_{\mathrm{R}^{nc}_{\ell,\bar{S}}(A^{sh})}(\mathrm{Sh}_{\Ql}(\bar{S})) \ar[d]^-{u_*} \\ 
\dgcat^{\mathrm{idem}}_S \ar[r]_-{\mathrm{R}^{nc}_{\ell,S}} &  \mathrm{Mod}_{\mathrm{R}^{nc}_{\ell,S}(A)}(\mathrm{Sh}_{\Ql}(S))
 } \end{equation}
 
 Indeed, using the definitions, this amounts to check first that for any $\bar{S}$-dg-category $T$ there is a natural equivalence of functors
  to spectra
  
$$
u_\ast  \mathcal{M}_{\bar{S}}^\vee (T) \simeq \mathcal{M}_{S}^\vee (\mathrm{Restr}_u T)
$$

\medskip
The first is the presheaf on $S$-dg-categories sending 

$$T'\mapsto \mathrm{KH}_{A^{\mathrm{sh}}}(T\otimes_{A^{\mathrm{sh}}}(T'\otimes_{A} A^{\mathrm{sh}})$$

The second is defined by 

$$T'\mapsto \mathrm{KH}_{A}(\mathrm{Restr}_u T\otimes_{A} T')$$

The two agree by the projection formula for K-theory. Finally, the commutativity of (\ref{16ago1045}) follows from this computation together with the fact that the $\ell$-adic realization (\ref{eq-ladicrealizationBUmodules}) is strongly compatible with the six operations  and  $u_*: \mathrm{Sh}_{\Ql}(\bar{S}) \to \mathrm{Sh}_{\Ql}(S)$ is lax monoidal.

 We can now conclude the proof of the proposition: as already observed for $\mathrm{R}^{nc}_{\ell,\bar{S}}(A^{sh}) \simeq \mathbb{Q}_{\ell, \bar{S}}(\beta)$, we have $\mathrm{R}^{nc}_{\ell,S}(A) \simeq \mathbb{Q}_{\ell, S}(\beta)$. 
 Since $\bar{X}$ is regular, we have $u_*(\Sing(\bar{X}, \pi \circ \bar{p})) \simeq \Sing (X_{\bar{\sigma}})$, and recall that, by definition of vanishing cycles for $p$ and the fact that $u^*(\mathbb{Q}_{\ell, S}(\beta))\simeq \mathbb{Q}_{\ell, \bar{S}}(\beta)$, we have $\mathcal{V}_p(\mathbb{Q}_{\ell, S}(\beta)) \simeq \mathcal{V}_{\bar{p}}(\mathbb{Q}_{\ell, \bar{S}}(\beta))$ inside $\mathrm{Sh}_{\Ql}(X_{\bar{\sigma}})^{\mathrm{Gal}(\bar{\eta} | \eta)}$. Now, the above commutative diagram of $\ell$-adic realizations combined with Theorem  \ref{theorem-maincomparisonvanishingmf}, yields an equivalence in $\mathrm{Sh}_{\Ql}(S)^{\mathrm{Gal}(\bar{\sigma}|\sigma)}$ 
 \begin{equation}
 \label{almost}
 \mathrm{R}^{nc}_{\ell,S}(\Sing (X_{\bar{\sigma}})) \simeq u_* (i_{\bar{\sigma}})_*(p_{\bar{\sigma}})_*(\mathcal{V}_p(\beta)[-1])^{\mathrm{hI}}).
  \end{equation} 
  \end{proof}

\begin{rem}
\label{rem-chern-character}
The equivalence of the Theorem \ref{theorem-maincomparisonvanishingmf} also provides a \emph{Chern character} map from the $\mathrm{K}$-theory of matrix factorizations to $\ell$-adic 2-periodic inertia invariant vanishing cohomology. Indeed, using the functoriality of the $\ell$-adic realization (\ref{eq-ladicrealizationBUmodules}) one obtains a map

\begin{equation}
\label{eq-chern-character}
\resizebox{1. \hsize}{!}{
\xymatrix{
\mathrm{KH}(\MF(X,\pi\circ p))\simeq \Map_{\Mod_{\BU_S}}(\BU_S, \mathcal{M}^{\vee}_S(\MF(X,\pi\circ p))\ar[r]^-{i_\sigma^\ast}&\Map_{\Mod_{\BU_\sigma}}(\BU_\sigma, i_\sigma^\ast\mathcal{M}^{\vee}_S(\MF(X,\pi\circ p))\ar[d]^{(\ref{eq-ladicrealizationBUmodules})+ (\ref{one})}\\
\mathbf{H}_{\'et}(X_{k}, \mathcal{V}_{p}(\beta)[-1])^{\mathrm{hI}} &\Map_{\Mod_{\mathbb{Q}_{\ell,\sigma}(\beta)}}(\mathbb{Q}_{\ell,\sigma}(\beta), ((p_{\sigma})_*\mathcal{V}_p(\beta)[-1])^{\mathrm{hI}}) \ar[l]^{\sim}
}
}
\end{equation}
\end{rem}

\vspace{0.5cm}

\begin{rem}
\label{rem-crucialregularsingularitycategory}
Notice that in the proof of the Theorem \ref{theorem-maincomparisonvanishingmf}, the hypothesis that $X$ is regular is crucial. Indeed, if $X$ is not regular, the relative derived category of singularities $\Sing(X,\pi)$ is not equivalent to the absolute one $\Sing(X_0)$ as explained in the Remark \ref{remark-differentdefinitionsrelativesingularities}. In the non-regular case we would need to provide a proof for all the statements in Section \ref{section-motivesPerfCohgeneralities} replacing $\Coh(X_0)$ by $\Coh(X_0)_{\Perf(X)}$.
\end{rem}

\appendix
\section{The formalism of six Operations in the Motivic Setting}
\label{Section:formalismsixoperations}
\vspace{0.5cm}
The $\infty$-functor $\rmSH^\otimes$ carries a system of extra functorialities known as the six operations. This means:
\begin{enumerate} 
\label{listsixoperations}
\item For every smooth morphisms $f:X\to Y$ of base schemes, the assignment $f^*: \rmSH_Y\to \rmSH_X$ has a left adjoint $f_\sharp$ which is a map of $\rmSH_Y$-modules with respect to the natural map induced from the unit of the adjunctions $(f_\sharp, f^*)$
$$f_\sharp(- \otimes f^*(-))\to f_\sharp(-)\otimes -$$
where $\rmSH_X$ is seen as a $\rmSH_Y$-module via the monoidal functoriality $f^*$. Moreover, $(-)_\sharp$ should verify smooth base-change.
\item The existence of a second functoriality for the assignment $X\mapsto \rmSH_X$ encoded by an $\infty$-functor
$$\rmSH_!:\bsch^{\mathrm{sep, ft}}\to \Prlstb$$
defined in $\bsch^{\mathrm{sep, ft}}$ - the subcategory of all $S$-schemes together with  separated morphisms of finite type between them.
\item The existence of natural transformation $(-)_*\to (-)_!$  defined in $\bsch^{\mathrm{sep, ft}}$  which is an equivalence for proper maps.\footnote{Here $(-)_*$ denotes the right adjoints of the functoriality $\rmSH^\otimes$};
\item (Projection Formula) The functoriality $\rmSH_!$ has a module structure over the functoriality $\rmSH^\otimes$. More precisely, for any map $f:X\to Y$ separated of finite type, we ask for $f_!: \rmSH_X\to \rmSH_Y$ to be a map of $\rmSH_Y$-modules as in (1). Here $\rmSH_X$ is seen as a $\rmSH_Y$-module via the monoidal functoriality $f^*$.
\item For any \emph{cartesian square} of schemes
 \begin{equation}
 \xymatrix{
 Y'\ar[r]^{p'} \ar[d]^{f'}& X'\ar[d]^f\\
 Y\ar[r]^p& X
 }
 \end{equation}
 \noindent with $f$ \emph{separated of finite type}, we ask for natural equivalences of $\infty$-functors 
\begin{equation}
 p^*\circ f_! \simeq (f')_!\circ (p')^{*}
 \end{equation}
 \noindent and
\begin{equation}
 f^!\circ p_*\simeq (p')_*\circ (f')^!
\end{equation}
    \item For any smooth morphism of relative dimension $d$, $f:Y\to X$  the adjunctions $(f_!, f^!)$ and $(f_\sharp, f^*)$ are related by a natural equivalence 
\begin{equation}
\label{eq-smoothsharpisshriek1}
f_{\sharp}\simeq f_!(-\otimes \mathrm{Thom}_Y(N_{Y/Y\times_X Y}))
\end{equation}
\noindent where $\mathrm{Thom}_Y(N_{Y/Y\times_X Y})$ is the motivic thom spectrum of the normal bundle $N_{Y/Y\times_X Y}$.
\medskip
 In the presence of an \emph{orientation data} (see \cite[2.4.c]{cisinski-tcmm}) we have $\mathrm{Thom}_Y(N_{Y/Y\times_X Y})\simeq (\mathbb{P}_Y^{1}, \infty)^{\otimes^d} $ so that
\begin{equation}
\label{eq-smoothsharpisshriek}
f_{\sharp}\simeq f_!(-\otimes (\mathbb{P}_Y^{1}, \infty)^{\otimes^d})
\end{equation}
In particular, whenever $f$ is an open immersion we have
$$
f_\sharp\simeq f_!
$$
\end{enumerate}
\medskip
\medskip
Thanks to the main results of \cite{ayoub1, cisinski-tcmm}  (see Prop. \ref{prop-Ayoub-Cisinski}), all these operations and coherences can be constructed from the initial $\s$-functor $\rmSH^\otimes$. In the setting of higher categories this can be done using the theory of correspondences developed in  \cite[Chapter 7]{MR3701352}. Alternatively, one can also use results of \cite{1211.5948, 1211.5294} as explained in \cite[Section 9.4]{robalo-thesis}. In this appendix we give a brief survey of the construction based on the techniques of \cite[Chapter 7]{MR3701352} but we do not look at the necessary Beck-Chevalley conditions. These have now been carefully treated in \cite{adeelthesis}, in a more general setting where base schemes are allowed to be derived.\\

\noindent Fix a base Noetherian scheme $S$ and fix $\bsch$ a full subcategory of the category of all Noetherian schemes as in \cite[2.0]{cisinski-tcmm}.  Suppose we are given an  $\infty$-functor  $\rmF:\bsch^{\mathrm{op}}\to \CAlg(\Prlstb)$. From this it is possible to extract a new functor $\mathrm{Arr}(\bsch)^{\mathrm{op}}\to \Mod(\Prlstb)$ where $\Mod(\Prlstb)$ is the $\infty$-category of pairs $(\mathrm{C}, \mathrm{M})$ with $\mathrm{C}$ a symmetric monoidal stable presentable $\infty$-category and $\mathrm{M}$ a stable presentable category endowed with a structure of $\mathrm{C}$-module. The new functor sends $f:Y\to X$ to the pair $(\rmF(X), \rmF(Y))$ with $\rmF(Y)$ seen as a module via $\rmF(f)$. We will also denote by $\rmF$. The reader can consult \cite[Section 9.4.1.2]{robalo-thesis} for a precise description of this assignment.\\
 The $\infty$-category $\Mod(\Prlstb)$ is a non-full subcategory of $\Mod(\inftyCatbig)$ which is the maximal $(\infty,1)$-category of the $(\infty,2)$-category $\Mod(\inftyCatbig)^{\mathrm{2-cat}}$ where we also include natural transformations of functors. In reality, the initial data we are interested in, is the new $\infty$-functor
$$\rmF:\mathrm{Arr}(\bsch)^{\mathrm{op}}\to  \Mod(\inftyCatbig)^{\mathrm{2-cat}}$$
\noindent and the six operations will express the coherences between $\rmF$ and the following five distinct classes of maps in $\bsch$:
\begin{itemize}
\item  $spft:= \text{ separated morphisms of finite type}$
\item $all:= \text{ all morphisms}$
\item $isom:= \text{ isomorphisms}$
\item $proper:= \text{ proper morphisms}$
\item $smooth:= \text{ smooth morphisms}$
\item  $open:= \text{ open morphisms}$
\end{itemize}
These classes verify some standard stability assumptions - see  \cite[Chapter 7, Sect 1.1.1]{MR3701352}. We will use the same notations for the classes of maps in $\mathrm{Arr}(\bsch)$ given by natural transformations where the maps belong to the respective classes.
\medskip
We now explain the conditions and their consequences. The reader should consult \cite[Chapter 7]{MR3701352} for the notations. The first two conditions are:\\
\begin{enumerate}[(i)]
\item $\rmF$ satisfies the \emph{right Beck-Chevalley} condition with respect to the inclusion $vert:=smooth\subseteq horiz:=all$ \cite[Def. 3.1.5 Chapter 7]{MR3701352}. By the universal property of correspondences \cite[Theorem 3.2.2-b) Chapter 7]{MR3701352} $\rmF$ extends in a unique way to an $(\infty,2)$-functor
$$
\rmF^{smooth}_{smooth, all}:\Corr(\mathrm{Arr}(\bsch))^{smooth}_{smooth, all}\to \Mod(\inftyCatbig)^{\mathrm{2-cat}}
$$
\noindent whose restriction along the inclusion 
$$((\bsch)_{horiz})^{\mathrm{op}}\subseteq \mathrm{Arr}(\bsch)^{\mathrm{op}}\subseteq \Corr(\mathrm{Arr}(\bsch))^{smooth}_{smooth, all}$$
\noindent  recovers $\rmF$.\\
Using the fact that $(\Corr(\mathrm{Arr}(\bsch))^{smooth}_{smooth, all})^\mathrm{1-op}=\Corr(\mathrm{Arr}(\bsch))^{smooth}_{all, smooth}$, passing to the 1-opposite we obtain a new $(\infty,2)$-functor
$$((\rmF)^{smooth}_{smooth, all})^{\mathrm{1-op}}:\Corr(\mathrm{Arr}(\bsch))^{smooth}_{all, smooth}\to (\Mod(\inftyCatbig)^{\mathrm{2-cat}})^{\mathrm{1-op}}$$ 
whose restriction along 
$$(\bsch)_{all}\subseteq \Corr(\mathrm{Arr}(\bsch))^{smooth}_{all, smooth}$$
\noindent recovers $\rmF^{\mathrm{op}}$.\\
\item $\rmF^\mathrm{op}:\mathrm{Arr}(\bsch)\to (\Mod(\inftyCatbig)^{\mathrm{2-cat}})^{\mathrm{1-op}}$ satisfies the \emph{left Beck-Chevalley} with respect to the inclusion $horiz:=proper\subseteq vert:=all$ \cite[Chapter 7, Def. 3.1.2]{MR3701352}. By the universal property of correspondences \cite[Theorem 3.2.2-a) Chapter 7]{MR3701352} $\rmF^{\mathrm{op}}$ extends in a unique way to an $(\infty,2)$-functor
$$
(\rmF^{\mathrm{op}})^{proper}_{all, proper}:\Corr(\mathrm{Arr}(\bsch))^{proper}_{all, proper}\to (\Mod(\inftyCatbig)^{\mathrm{2-cat}})^{\mathrm{1-op}}
$$
\noindent whose restriction along the inclusion 
$$((\bsch)_{vert})=\mathrm{Arr}(\bsch) \subseteq \Corr(\mathrm{Arr}(\bsch))^{proper}_{all, proper}$$
\noindent recovers $\rmF^{\mathrm{op}}$. \\
\end{enumerate}
We remark that these Beck-Chevalley conditions need to be verified at the level of modules. \\
\vspace{0.5cm}
We consider now the restriction $((\rmF)^{smooth}_{smooth, all})^{\mathrm{1-op}}$ along the inclusion
$$\Corr(\mathrm{Arr}(\bsch))^{isom}_{all, open}\subseteq \Corr(\mathrm{Arr}(\bsch))^{smooth}_{all, smooth}$$ 
\noindent and observe that we have built a commutative diagram of $(\infty,2)$-functors
$$
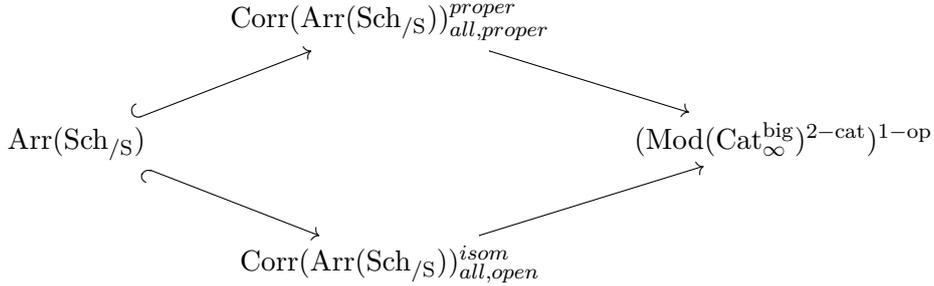
\xymatrix{
&\Corr(\mathrm{Arr}(\bsch))^{proper}_{all, proper}\ar[dr]&\\
\mathrm{Arr}(\bsch)\ar@{_{(}->}[dr]\ar@{^{(}->}[ur]&&(\Mod(\inftyCatbig)^{\mathrm{2-cat}})^{\mathrm{1-op}}\\
&\Corr(\mathrm{Arr}(\bsch))^{isom}_{all, open}\ar[ur]&\\
}
$$
The formalism of six operations is constructed by gluing these two functors, merging open immersions and proper maps. More precisely, it follows from Nagata's compactification that any morphism in $spft$ can be written as a morphism in $open$ composed with a morphism in $proper$. In this sense what we would like is to produce a new $(\infty,2)$-functor completing the commutativity of the diagram
$$
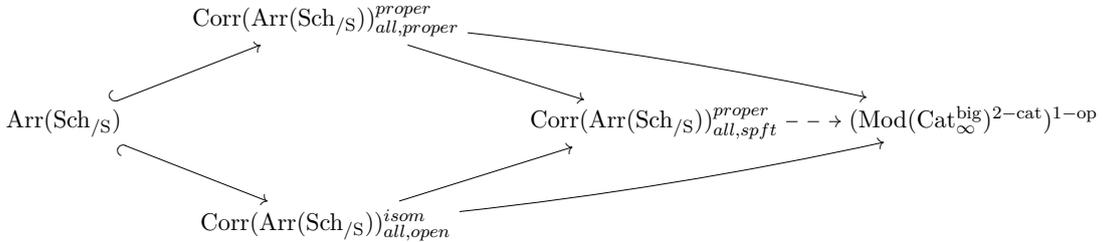
\resizebox{1. \hsize}{!}{
\xymatrix{
&\Corr(\mathrm{Arr}(\bsch))^{proper}_{all, proper}\ar[dr]\ar@/^/[drr]&&\\
\mathrm{Arr}(\bsch)\ar@{_{(}->}[dr]\ar@{^{(}->}[ur]&&\Corr(\mathrm{Arr}(\bsch))^{proper}_{all, spft}\ar@{-->}[r]&(\Mod(\inftyCatbig)^{\mathrm{2-cat}})^{\mathrm{1-op}}\\
&\Corr(\mathrm{Arr}(\bsch))^{isom}_{all, open}\ar[ur]\ar@/_/[urr]&&\\
}
}
$$
This is solved by the theorem \cite[Thm 5.2.4 Chapter 7]{MR3701352} which gives necessary and sufficient conditions for the existence and uniqueness of the dotted map.  These are the following:
\begin{enumerate}[(i)]
\addtocounter{enumi}{2}
\item \cite[5.1.2 Chapter 7]{MR3701352}: the class $open\cap proper$ consists of embeddings of connected components (therefore monomorphisms).
\item \cite[5.1.4 Chapter 7]{MR3701352}: For any map $f:X\to Y$ separated of finite type, the space $\mathrm{Fact}(f) $ of factorizations of $f$ as an open immersion followed by a proper map, is contractible. This is a consequence of Nagata's compactification as explained in \cite[Prop 2.1.6 Chapter 5]{MR3701352}.
\item \cite[5.2.2 Chapter 7]{MR3701352} It is the well-known \emph{support property} of Deligne.
\end{enumerate}
\personal{Gaitsgory works to give an equivalence of 2-cats}
\begin{defn}
We say that an $(\infty,1)$-functor $\rmF:\bsch^{\mathrm{op}}\to \CAlg(\Prlstb)$ \emph{has the six operations} if it verifies the conditions (1) to (5) above. We denote by $\rmF^{proper}_{all, spft}$ its unique extension
$$
\Corr(\mathrm{Arr}(\bsch))^{proper}_{spft, all}\to \Mod(\inftyCatbig)^{\mathrm{2-cat}}
$$
\end{defn}
\medskip
Let $\rmF$ verify the six operations. We will use the following notations:\\
\begin{itemize}
\item $F_\sharp:= \rmF^{smooth}_{smooth,all}|_{\mathrm{Arr}(\bsch)_{smooth}}:\mathrm{Arr}(\bsch)_{smooth} \to \Mod(\inftyCatbig)^{\mathrm{2-cat}}$\\
\item $F_*:=((\rmF^{\mathrm{op}})^{proper}_{all,proper}|_{\mathrm{Arr}(\bsch)_{proper}^{\mathrm{op}}}))^{\mathrm{op}}:\mathrm{Arr}(\bsch)_{proper} \to \Mod(\inftyCatbig)^{\mathrm{2-cat}}$\\
\item $F_!:=\rmF^{proper}_{all, spft}|_{\mathrm{Arr}(\bsch)_{spft}}:\mathrm{Arr}(\bsch)_{spft} \to \Mod(\inftyCatbig)^{\mathrm{2-cat}}$\\
\end{itemize}
\medskip
The following result of Ayoub and Cisinski-Deglise  gives sufficient conditions for a given $\rmF$ to have the six operations:
\begin{prop}[Ayoub and Cisinski-Deglise]
\label{prop-Ayoub-Cisinski}
Let $\rmF:\bsch^\mathrm{op}\to \CAlg(\Prlstb)$ be an $\infty$-functor satisfying the following conditions:
\begin{enumerate}[a)]
\item $\rmF$ satisfies (i);
\item for each proper map $f:X\to Y$, $\rmF_*(f)$ has a right adjoint;
\item (\emph{Localization}) For every closed immersion $i:Z\hookrightarrow X$ of base schemes with open complementary $U:=X-Z\hookrightarrow X$ the commutative diagram
$$
\xymatrix{
\rmF(Z)\ar[d]\ar[r]^{\rmF_*(i)}& \rmF(X)\ar[d]^{\rmF(f)}\\
0\ar[r]& \rmF(U)
}
$$
\noindent is a pullback in $\Prlstb$. In particular $\rmF_*(i)$ is fully faithful \footnote{See for instance \cite[9.4.20]{robalo-thesis}.}.
\item (\emph{Homotopy Invariance}) For any base scheme $X$, the map $\rmF(\pi): \rmF(X)\to \rmF(\mathbb{A}^1_X)$ is fully faithful. Here $\pi:\mathbb{A}^1_X \to X$ is the canonical projection.
\item (\emph{Stability}) For any base scheme $X$, the composition $\rmF_\sharp(\pi)\circ \rmF_*(s):\rmF(X)\to \rmF(X)$ maps the tensor unit to a $\otimes$-invertible object (the Tate motive). \footnote{$s:X\to \mathbb{A}^1_X$ being the zero section.}
\end{enumerate}
Then $\rmF$ satisfies all the conditions (1)-(5).
\end{prop}
\vspace{0.5cm}

\personal{
\begin{rem}
\label{rem-gluinglaxlimitlocalization}
The localization assumption of the Prop. \ref{prop-Ayoub-Cisinski} is actually equivalent to for the canonical map
$$
\rmF(X)\simeq \mathrm{laxlim}(i^*j_*: \rmF(U)\to \rmF(Z))
$$
\noindent to be an equivalence.
\end{rem}
}

\begin{prop}
\label{prop-AdeelAyoubCisinski}
The $\infty$-functor $\rmSH^\otimes$ verifies all the conditions of the Prop \ref{prop-Ayoub-Cisinski}.
\begin{proof}
The localization property was proved by Morel-Voevodsky in \cite[Thm 2.21 pag. 114]{voevodsky-morel} in the unstable setting. The property in the stable setting as formulated above follows as a consequence. The other conditions follow from the results of Cisinski-Deglise in \cite{cisinski-tcmm} and Ayoub \cite{ayoub1, ayoub2}.  The fact it satisfies the necessary Beck-Chevalley conditions in the correct $\infty$-sense is proved in \cite{adeelthesis}. See also the survey in \cite[Sections 9.3 and  9.4.1]{robalo-thesis} for an overview and more precise references.
\end{proof}
\end{prop}

To conclude this Appendix we discuss the compatibility of the six operations under natural transformations. We have the following result due to Ayoub (for the projective case) and Cisinski-Deglise (for the generalization to proper morphisms)
\begin{prop}(See \cite[Theorem 3.4]{MR2602027} and \cite{cisinski-tcmm})
\label{prop-compsixoperations}
Let $\phi:\rmF\to \rmG$ be a natural transformation of $\infty$-functors $\bsch^\mathrm{op}\to \CAlg(\Prlstb)$ such that
\begin{enumerate}
\item both $\rmF$ and $\rmG$ satisfy the hypothesis of the Proposition \ref{prop-Ayoub-Cisinski}.
\item if $f:X\to Y$ is a smooth map in $\bsch$, then the diagram 
\begin{equation}
\xymatrix{
\rmF(X)\ar[r]^{\phi_X}& \rmG(X)\\
\rmF(Y) \ar[r]^{\phi_Y}\ar[u]_{\rmF(f)} &  \ar[u]_{\rmG(f)} \rmG(Y)
}
\end{equation}
\noindent is left-adjointable
\begin{equation}
\xymatrix{
\rmF(X)\ar[r]^{\phi_X}\ar[d]^{\rmF_{\sharp}(f)}&\ar[d]^{\rmG_{\sharp}(f)} \rmG(X)\\
\rmF(Y) \ar[r]^{\phi_Y}&\rmG(Y)
}
\end{equation}
\end{enumerate}
Then, the natural transformation induced from the adjunctions
\begin{equation}
\label{eq-transferlowershriek}
\rmG_!\circ \phi\to \phi\circ \rmF_! 
\end{equation}
\noindent is an equivalence. Moreover, the natural transformations
\begin{equation}
\label{eq-transferlowerstartuppershriek}
\phi\circ \rmF_*\to \rmG_*\circ \phi \, \, \, \, \, \, \, \, \, \, \, \,  \phi\circ \rmF^!\to \rmG^!\circ \phi
\end{equation}
\noindent are given by equivalences whenever $f$ is proper, respectively, smooth.
\end{prop}

\bibliographystyle{alpha}
\bibliography{biblio}

\end{document}

%% file: newcommands_all.tex
\ifpersonal
\newcommand*{\personal}[1]{\textcolor{blue}{(COMMENT: #1)}}
\newcommand*{\todos}[1]{\textcolor{red}{#1}}
\else
\newcommand*{\personal}[1]{\ignorespaces}
\newcommand*{\todos}[1]{\ignorespaces}
\fi

\newcommand  {\dgcat}     {\mathbf{dgcat}}

\newcommand  {\BU}   {\mathrm{B}\mathbb{U}}
\newcommand{\s}{\infty}

\newcommand{\rmF}{\mathrm{F}}
\newcommand{\rmG}{\mathrm{G}}

\newcommand{\rmSH}{\mathrm{SH}}
\newcommand{\LG}{\mathrm{LG}}

\newcommand{\rmSHNC}{\mathrm{SHnc}}

\newcommand{\K}{\mathrm{K}}

\newcommand{\Ql}{\mathbb{Q}_{\ell}}

\newcommand{\Mod}{\mathrm{Mod}}
\newcommand{\Coh}{\mathrm{Coh}^\mathrm{b}}

\newcommand{\Sch}{\mathrm{Sch}}

\newcommand{\bsch}{\mathrm{Sch}_{/\mathrm{S}}}

\usetikzlibrary{decorations.markings} 
\tikzset{
  closed/.style = {decoration = {markings, mark = at position 0.5 with { \node[transform shape, xscale = .8, yscale=.4] {/}; } }, postaction = {decorate} },
  open/.style = {decoration = {markings, mark = at position 0.5 with { \node[transform shape, scale = .7] {$\circ$}; } }, postaction = {decorate} }
}

\newcommand{\inftyCatbig}{\mathrm{Cat}^{\mathrm{big}}_{\infty}}
\newcommand{\Prl}{\mathrm{Pr}^{\mathrm{L}}}
\newcommand{\Prlstb}{\mathrm{Pr}^{\mathrm{L}}_{\mathrm{Stb}}}

\newcommand{\rmSm}{{\mathrm{Sm}}}

\DeclareMathOperator{\CAlg}{CAlg}
\DeclareMathOperator{\Corr}{Corr}

\DeclareMathOperator{\Gal}{Gal}

\DeclareMathOperator{\Map}{Map}

\DeclareMathOperator{\Perf}{Perf}

\DeclareMathOperator{\Sp}{Sp}

\DeclareMathOperator{\Spec}{Spec}

\DeclareMathOperator{\MF}{\mathrm{MF}}
\DeclareMathOperator{\MFPairs}{\mathrm{MF}}

\DeclareMathOperator{\Sing}{\mathrm{Sing}}
\DeclareMathOperator{\Qcoh}{\mathrm{Qcoh}}